\documentclass[12pt,oneside]{report}
\usepackage[margin=1in]{geometry}
\usepackage{amsthm,amsmath}
\usepackage{amssymb,latexsym,amscd}
\usepackage{mathrsfs}
\usepackage[all,cmtip]{xy} 
\usepackage{tikz-cd}
\usepackage{hyperref}
\usepackage{bookmark}
\usepackage{comment}
\usepackage{setspace}

\usepackage{datetime}
\renewcommand{\dateseparator}{-}
\renewcommand{\today}{\the\year \dateseparator \twodigit\month
\dateseparator \twodigit\day}

\newcommand{\Bf}[1]{\mathbf{#1}}
\newcommand{\Ca}[1]{\mathcal{#1}}

\newcommand{\Fr}[1]{\mathfrak{#1}}
\newcommand{\Sc}[1]{\mathscr{#1}}

\def\A{\mathbf{A}}

\def\C{\mathbf{C}}
\def\I{\mathscr{I}}

\def\F{\mathscr{F}}

\def\FF{\mathbf{F}}

\def\G{\mathscr{G}}

\def\L{\mathscr{L}}

\def\O{\mathcal{O}}
\def\p{\mathfrak{p}}

\def\Q{\mathbf{Q}}
\def\R{\mathbf{R}}
\def\Z{\mathbf{Z}}

\def\GL{\mathrm{GL}}

\def\GSp{\mathrm{GSp}}

\def\BT{\mathrm{BT}_1}

\DeclareMathOperator{\coker}{coker}
\DeclareMathOperator{\im}{im}

\DeclareMathOperator{\tr}{tr}

\renewcommand{\hom}{\operatorname{Hom}}
\DeclareMathOperator{\End}{\operatorname{End}}

\DeclareMathOperator{\Lie}{Lie}

\DeclareMathOperator{\spec}{Spec}

\numberwithin{equation}{section} 
\theoremstyle{plain}
\newtheorem{thm}{Theorem}[section]
\newtheorem{prop}[thm]{Proposition}
\newtheorem{lem}[thm]{Lemma}
\newtheorem{cor}[thm]{Corollary}
\newtheorem*{key claim}{Key Claim}
\newtheorem{mainthm}{Theorem}[chapter]

\theoremstyle{definition}
\newtheorem{defn}[thm]{Definition}
\newtheorem{conv}[thm]{Convention}
\newtheorem{exmp}[thm]{Example}

\theoremstyle{remark}
\newtheorem{rem}[thm]{Remark}
\newtheorem{cond}[thm]{Condition}

\title{Torsion in the Coherent Cohomology of Shimura Varieties and Galois Representations}
\author{George Boxer}

\begin{document}
\pagenumbering{roman}
% ----------------------------------------------------------------

\thispagestyle{empty}

\vspace*{\fill}

\begin{center}
Torsion in the Coherent Cohomology of\\ Shimura Varieties and Galois Representations\\
\vspace{0.2in}
A dissertation presented\\
\vspace{0.2in}
by\\
\vspace{0.2in}
George Andrew Boxer\\
\vspace{0.2in}
to\\
\vspace{0.2in}
The Department of Mathematics\\
\vspace{0.2in}
in partial fulfillment of the requirements\\
for the degree of\\
Doctor of Philosophy\\
in the subject of\\
Mathematics\\
\vspace{0.2in}
Harvard University\\
Cambridge, Massachusetts\\
\vspace{0.2in}
April 2015\\
\end{center}

\vspace*{\fill}

\pagebreak

%\maketitle

% ----------------------------------------------------------------

\thispagestyle{empty} 

\vspace*{\fill}

\begin{center}
\copyright \, 2015 -- George A. Boxer \\
All rights reserved.
\end{center}

\vspace*{\fill}

\pagebreak

% ----------------------------------------------------------------

\doublespacing

% ----------------------------------------------------------------

\noindent Dissertation Advisor: Richard Taylor \hfill
George Andrew Boxer

\vspace{0.5in}

\centerline{Torsion in the Coherent Cohomology of Shimura Varieties and Galois Representations}

\vspace{0.8in}

\centerline{Abstract}

\vspace{0.3in}

We introduce a method for producing congruences between Hecke eigenclasses, possibly torsion, in the coherent cohomology of automorphic vector bundles on certain good reduction Shimura varieties.  The congruences are produced using some ``generalized Hasse invariants'' adapted to the Ekedahl-Oort stratification of the special fiber.

\pagebreak

% ----------------------------------------------------------------}

\tableofcontents

\pagebreak

% ----------------------------------------------------------------

\section*{Acknowledgements}

First and foremost, I would like to thank my advisor Richard Taylor for all that he has taught me about Shimura varieties, automorphic forms, and Galois representations, and for his support, encouragement, and patience throughout this project.

Next it is my pleasure to thank Mark Kisin for many valuable conversations related to the topic of this thesis, and in particular for pointing me to Serre's letter \cite{Se96}.

I would like to thank Wushi Goldring and David Geraghty for explaining to me their beautiful construction of a Hasse invariant on the non ordinary locus of the Siegel threefold.  I am extremely grateful to Kai-Wen Lan for answering many questions about compactifications of Shimura varieties.  I would like to thank Bao Le Hung, Peter Scholze, and Anand Patel for valuable conversations related to this work.

I must also thank the members and honorary members of the \emph{true} alcove for all that I have learned from them over the last four years, as well as the disciples of Oleg for helping me stay sane.

I am extremely grateful to the math department staff for making the department a pleasant and productive place to work, and I am especially appreciative of Susan Gilbert for keeping me on track all these years.

Last but certainly not least, I would like to thank my family for all their support.

\pagebreak

% ----------------------------------------------------------------

\pagenumbering{arabic}

\chapter{Introduction}

\section{Motivation: Weight 1 Modular Forms mod $p$}

This thesis is concerned with the conjectural correspondence between modular forms mod $p$ and mod $p$ Galois representations.  We will begin by reviewing the first example where ``torsion'' phenomena arise: weight 1 modular forms mod $p$.

Fix an odd prime $p$ and an integer $N\geq 3$ relatively prime to $p$.  Let $\Ca{X}=\Ca{X}(N)/\Z_p$ be the proper modular curve with full level $N$ structure.  We have a universal generalized elliptic curve $\Ca{E}/\Ca{X}$, from which we may define a line bundle
\begin{equation*}
\omega=e^*\Omega^1_{\Ca{E}/\Ca{X}}
\end{equation*}
where $e$ denotes the identity section of $\Ca{E}$.

Then Katz \cite{Ka73} has defined the space of geometric modular cusp forms of level $N$, weight $k$, and with coefficients in a $\Z_p$-algebra $R$ as
\begin{equation*}
S_k(R)=H^0(\Ca{X},\omega^{\otimes k}(-\infty)\otimes_{\Z_p}R)
\end{equation*}
where $\infty$ denotes the divisor of cusps of $\Ca{X}$.  When $k\geq 2$, Katz proves that for each $\Z_p$-algebra $R$
\begin{equation*}
S_k(R)=S_k(\Z_p)\otimes_{\Z_p}R.
\end{equation*}

However when $k=1$, this does not hold, and in fact a new phenomenon arises: it can happen that the reduction mod $p$ map
\begin{equation*}
S_1(\Z_p)\to S_1(\FF_p) 
\end{equation*}
is not surjective.  Modular forms mod $p$ which do not lift to characteristic 0 are termed ``ethereal'' by some authors.  The first example of an ethereal form was discovered by Mestre in 1987 for $N=1429$, $p=2$ (see the appendix of \cite{Ed06}.)  Further examples were found by Buzzard in the early 2000's, including the first examples with $p$ odd \cite{Bu13}.  More recently, Schaeffer \cite{Sc14} has developed an algorithm for finding ethereal forms and has produced extensive tables.

Let us explain what the existence of ethereal forms has to do with the ``torsion'' in the title of this thesis.  We have a short exact sequence of sheaves on $\Ca{X}$
\begin{equation*}
0\to\omega(-\infty)\overset{p}{\to}\omega(-\infty)\to \omega(-\infty)/p\to 0
\end{equation*}
and upon taking cohomology we conclude that
\begin{equation*}
\coker(S_1(\Z_p)\to S_1(\FF_p))=H^1(\Ca{X},\omega(-\infty))[p].
\end{equation*}

The space $S_1(\overline{\FF}_p)$ has an action of Hecke operators $T_l,S_l$ for $l\nmid Np$.  If $f\in S_1(\overline{\FF}_p)$ is an eigenform, then there is a mod $p$ Galois representation
\begin{equation*}
\rho_f:G_\Q\to\GL_2(\overline{\FF}_p)
\end{equation*}
associated to $f$, even if $f$ is ethereal (we will explain why in the next section).  These Galois representations have the unusual property that they are unramified at $p$.  In fact there is a diagram

\begin{equation*}
\begin{tikzcd}
\left\{\text{Eigenforms in $S_1(\overline{\Q}_p)$}\right\}\arrow[leftrightarrow]{r}\arrow{d}&\left\{\rho:G_\Q\to\GL_2(\overline{\Q}_p)\text{ odd, unramified at $p,\ldots$}\right\}\arrow{d}\\
\left\{\text{Eigenforms in $S_1(\overline{\FF}_p)$}\right\}\arrow[leftrightarrow]{r}&\left\{\rho:G_\Q\to\GL_2(\overline{\FF}_p)\text{ odd, unramified at $p,\cdots$}\right\}
\end{tikzcd}
\end{equation*}
where the top horizontal arrow is the usual Langlands correspondence between weight 1 modular forms and odd two dimensional Artin representations, and the bottom horizontal arrow is part of Serre's conjecture.  The vertical arrows are reduction mod $p$.

The fact that the left vertical arrow needn't be surjective can also be seen on the Galois side: there exists odd mod $p$ representations unramified at $p$ whose projective image contains $\text{PSL}_2(\FF_p)$.  Such a representation cannot be the reduction of a two dimensional Artin representation for $p>5$.

The basic goal of this thesis is to try to understand this picture for groups beyond $\GL_2$.  In particular we want to construct the bottom arrow from left to right.

\section{Constructing Congruences Using Hasse Invariants}

We maintain the notation of the last section.  For convenience we let $X=\Ca{X}_{\FF_p}$.  Let us suppose we have a Hecke eigenform $f\in S_1(\overline{\FF}_p)$.  We will sketch how to construct a Galois representation
\begin{equation*}
\rho_f:G_\Q\to\GL_2(\overline{\FF}_p)
\end{equation*}
associated to $f$.

The key tool will be the Hasse invariant
\begin{equation*}
A\in H^0(X,\omega^{\otimes p-1}).
\end{equation*}
We will recall its construction in section \ref{S:EOGH} below.

We may form the product $Af\in S_p(\overline{\FF}_p)$.  By a crucial property of the Hasse invariant reviewed below, $Af$ is also a Hecke eigenform with the same Hecke eigenvalues as $f$.  But as it has weight $p>1$, it admits a lift to characteristic 0.  By the lemma of Deligne-Serre \cite[6.11]{DS74}, we may even pick a lift $\tilde{f}\in S_1(\overline{\Q}_p)$ which is a Hecke eigenform.  Then we make take $\rho_f=\overline{\rho}_{\tilde{f}}$, the reduction of the Galois representation associated to $\tilde{f}$.

In the rest of this thesis we will also want to consider higher coherent cohomology.  Let us suppose then that we have a Hecke eigenform
\begin{equation*}
f\in H^1(X,\omega(-\infty))
\end{equation*}
that we would like to associate a Galois representation to.  Of course, we could use Serre duality to reduce to the case above, but let us explain a different method which will be a sort of baby case of the more general argument to be explained in section \ref{S:introcong} below.

We may again try to multiply by the Hasse invariant, obtaining
\begin{equation*}
Af\in H^1(X,\omega^{\otimes p}(-\infty))
\end{equation*}
but this space is easily seen to be trivial (recall that $p>2$), and hence $Af=0$.  But all is not lost: we may consider the short exact sequence of sheaves on $X$
\begin{equation*}
0\to \omega(-\infty)\overset{A}{\to}\omega^{\otimes p}(-\infty)\to \omega^{\otimes p}|_{SS}\to 0
\end{equation*}
where $SS=V(A)\subset X$ is the supersingular locus.  The resulting long exact sequence reads
\begin{equation*}
H^0(X,\omega^{\otimes p}(-\infty))\to H^0(SS,\omega^{\otimes p}|_{SS})\overset{\delta}{\to} H^1(X,\omega(-\infty))\to 0
\end{equation*}
One may show that the surjective map $\delta$ commutes with all the Hecke operators.  Hence there exists a Hecke eigenform $g\in H^0(SS,\omega^{\otimes p}|_{SS})$ with the same Hecke eigenvalues as $f$ (note that we may not be able to pick $g$ with $\delta(g)=f$).

Now if we are lucky, $\delta(g)=0$ and so we may find some $\tilde{g}\in H^0(X,\omega^{\otimes p})$ with the same Hecke eigenvalues as $g$, and so we are done as before.

If not, then we use a ``generalized Hasse invariant''
\begin{equation*}
B\in H^0(SS,\omega^{\otimes p^2-1}|_{SS})
\end{equation*}
whose construction is explained in section \ref{S:EOGH} below (see also \cite{Se96}).  Multiplication by $B$ gives a Hecke equivariant isomorphism
\begin{equation*}
H^0(SS,\omega^{\otimes p}|_{SS})\simeq H^0(SS,\omega^{\otimes p^2+p-1}|_{SS})
\end{equation*}
But now the weight is sufficiently large that $Bg$ may be lifted first to $H^0(\Ca{X},\omega^{\otimes p^2+p-1}(-\infty))$ and then to characteristic 0.

The idea of using sections of powers of $\omega$ on the supersingular locus, as well as the ``Hasse invariant'' $B$ to raise the weight comes from the paper of Serre \cite{Se96} (see also the work of Ghitza \cite{Gh04}).

\section{Geometric Siegel Modular Forms}
We will now explain the methods and results of this thesis.  We begin by establishing some notation.  In the body of the thesis we will work with general PEL type modular varieties of type A and C (to be introduced in Chapter \ref{PEL}) but for simplicity in this introduction we will only consider Siegel modular varieties.

Fix a prime $p$, an integer $N\geq 3$ relatively prime to $p$, and a positive integer $g$.  Let $\Ca{X}/\Z_p$ be the moduli space of principally polarized abelian varieties of dimension $g$ with a principal level $N$ structure.

Over $\Ca{X}$ we have a universal family $\pi:A\to\Ca{X}$ and from it we may form the \emph{Hodge bundle}
\begin{equation*}
\Ca{E}=\pi_*\Omega^1_{A/\Ca{X}}=e^*\Omega^1_{A/\Ca{X}}
\end{equation*}
where $e:\Ca{X}\to A$ denotes the identity section.  It is locally free sheaf of rank $g$.  We may also consider its determinant
\begin{equation*}
\omega=\det\Ca{E}.
\end{equation*}

To each algebraic representation $\rho$ of $\GL_g$ on a finite free $\Z_p$-module we may form a vector bundle $V_\rho/\Ca{X}$ by ``applying $\rho$ to the transition functions of $\Ca{E}$.''  For example this procedure leads to familiar tensor constructions like $\text{Sym}^n\Ca{E}$.  The $V_\rho$ are sometimes called \emph{automorphic vector bundles}, and they are the natural generalizations of the line bundles $\omega^{\otimes k}$ when $g=1$.  When $g>1$ one speaks of sections of $V_\rho$ over $\Ca{X}$\footnote{Indeed by Koecher's principal when $g>1$ it is not necessary to impose any condition of ``holomorphy at the cusps.''} as holomorphic Siegel modular forms of genus $g$, weight $\rho$, and level $N$ (of course classically one would work with complex, rather than $p$-adic coefficients).

In this thesis we will also be interested in the higher coherent cohomology of the vector bundles $V_\rho$.  The Siegel modular varieties $\Ca{X}/\Z_p$ are not proper, and in order to have a reasonable theory of coherent cohomology we need to introduce compactifications, due to Faltings-Chai \cite{FC90} in the arithmetic setting.

First, there is a family of \emph{toroidal compactifications} $j^{\text{tor}}:\Ca{X}\hookrightarrow\Ca{X}^{\text{tor}}$ indexed by a suitable choice of combinatorial data (a so called compatible family of rational polyhedral cone decompositions.)  This choice will be suppressed in what follows, but when it is made appropriately, $\Ca{X}^{\text{tor}}/\Z_p$ is smooth and proper, with boundary $D=\Ca{X}^{\text{tor}}-\Ca{X}$ a relative simple normal crossing divisor.  There is also a \emph{minimal} (or Baily-Borel or Satake) compactification $j^{\text{min}}:\Ca{X}\to\Ca{X}^{\text{min}}$.  The minimal compactification $\Ca{X}^{\text{min}}/\Z_p$ is projective and normal, but not smooth when $g>1$.  When $g>1$ its boundary is not a divisor.  In fact it has codimension $g$.  There is a map $\Ca{X}^{\text{tor}}\to\Ca{X}^{\text{min}}$ fitting into a commutative diagram
\begin{equation*}
\begin{tikzcd}
&\Ca{X}^{\text{tor}}\arrow{d}\\
\Ca{X}\arrow[hook]{ur}{j^{\text{tor}}}\arrow[hook]{r}{j^{\text{min}}}&\Ca{X}^{\text{min}}.
\end{tikzcd}
\end{equation*}

The theory of the compactifications $\Ca{X}^{\text{tor}}$ and $\Ca{X}^{\text{min}}$ is made somewhat complicated by the fact that neither can be interpreted as moduli spaces in general.  Nonetheless there is a semiabelian scheme $A/\Ca{X}^{\text{tor}}$ extending the universal family over $\Ca{X}$.  Using it we may define
\begin{equation*}
\Ca{E}^{\text{can}}=e^*\Omega^1_{A/\Ca{X}^{\text{tor}}},
\end{equation*}
the so called \emph{canonical extension} of the Hodge bundle.

We may then define, for each representation $\rho$ as above, a \emph{canonical extension} $V_{\rho}^{\text{can}}/\Ca{X}^{\text{tor}}$ of $V_\rho$ to $\Ca{X}^{\text{tor}}$, as well as a so called \emph{subcanonical extension}
\begin{equation*}
V_{\rho}^{\text{sub}}=V_\rho^{\text{can}}(-D).
\end{equation*}
One should think of $V_{\rho}^{\text{sub}}$ as the sheaf whose sections are cusp forms.

Then we may define spaces of geometric Siegel modular forms\footnote{A priori these spaces could depend on the choice of toroidal compactification, but it turns out that they don't.  See the discussion in section \ref{S:heckeaction}}
\begin{equation*}
H^n(\Ca{X}^{\text{tor}},V_{\rho}^{\text{can}})\qquad\text{and}\qquad H^n(\Ca{X}^{\text{tor}},V_{\rho}^{\text{sub}})
\end{equation*}
as well as mod $p$ and mod $p^r$ variants
\begin{equation*}
H^n(\Ca{X}^{\text{tor}},V_{\rho}^{\text{can}}/p^r)\qquad\text{and}\qquad H^n(\Ca{X}^{\text{tor}},V_{\rho}^{\text{sub}}/p^r).
\end{equation*}
These spaces carry an action of the Hecke algebra
\begin{equation*}
\Bf{T}=\bigotimes_{l\nmid Np}\Z_p[\GSp_{2g}(\Q_l)//\GSp_{2g}(\Z_l)]
\end{equation*}
and it is this action that makes them interested.

As explained above, when $n=0$ one should think of these spaces as holomorphic Siegel modular forms and holomorphic Siegel cusp forms respectively.  When $n>0$ they should be viewed as some sort of ``non holomorphic'' Siegel modular forms.  Indeed by a theorem of Harris \cite{Ha88}, with $\C$ coefficients, these spaces can essentially be computed in terms of automorphic representations on the group $\GSp_{2g}/\Q$.  A cuspidal automorphic representation $\pi=\otimes_v \pi_v$ contributes according to its archimedean component $\pi_\infty$.  Following \cite{Ha88}, those that do contribute are called $\overline{\partial}$-cohomological.   By a theorem of Mirkovi\'{c} \cite[3.5]{Ha88}, if $\pi_\infty$ is tempered and $\overline{\partial}$-cohomological then it is either discrete series or a non degenerate limit of discrete series.  The former class of representations generalize classical modular forms of weight $k\geq 2$, while the latter generalize modular forms of weight 1.

Let us now give three reasons why it is interesting to consider the higher coherent cohomology rather than just $H^0$.
\begin{enumerate}
\item For $g\geq 2$ there should exist $L$-packets of automorphic representations which contribute to the higher coherent cohomology of some $V_{\rho}$ but not to the $H^0$ of any vector bundle.  In order to associate Galois representations to such automorphic representations one needs to work with higher coherent cohomology.
\item In a recent breakthrough Calegari and Geraghty \cite{CG14} have extended the Taylor-Wiles method to situations where the automorphic forms of interest contribute to cohomology in more than one degree (like the weight 1 modular forms considered earlier.)  In order for their method to succeed, they require the existence of Galois representations attached to all cohomology classes, including torsion, in the entire range where there is cohomology.  One of the main goals of this work is to produce some of the Galois representations they require.
\item By modifying a construction of Harris-Lan-Taylor-Thorne \cite{HLTT12} one may construct certain ``boundary cohomology classes'' in coherent cohomology out of torsion classes in the betti cohomology of certain arithmetic locally symmetric spaces.  This leads to a new proof of cases of Scholze's spectacular work \cite{Sc13}.  This is the subject of a forthcoming paper.
\end{enumerate}

\section{The Ekedahl-Oort Stratification and Generalized Hasse Invariants}\label{S:EOGH}

In this section we let $X/\FF_p$ denote the special fiber of $\Ca{X}$.  The classical Hasse invariant is a section
\begin{equation*}
A\in H^0(X,\omega^{\otimes p-1})
\end{equation*}
which plays an important role in the construction of congruences between automorphic forms.  We briefly recall one of its constructions.  We begin with the relative Verschiebung
\begin{equation*}
V=V_{A/X}:A^{(p)}\to A
\end{equation*}
where $A^{(p)}$ is defined by the Cartesian diagram
\begin{equation*}
\begin{CD}
A^{(p)}@>>> A\\
@VVV @VVV\\
X@>F_X>> X
\end{CD}
\end{equation*}
where $F_X$ denotes the absolute Frobenius on $X$.

$V$ induces a map on cotangent spaces along the identity section
\begin{equation*}
V^*:\Ca{E}\to \Ca{E}^{(p)}=F_X^*\Ca{E}
\end{equation*}
where we have used the isomorphism $e^*\Omega^1_{A^{(p)}/X}\cong (e^*\Omega^1_{A/X})^{(p)}$.  We take its determinant to obtain
\begin{equation*}
V^*:\omega\to \omega^{(p)}\cong\omega^{\otimes p}
\end{equation*}
which gives the Hasse invariant
\begin{equation*}
A\in H^0(X,\omega^{\otimes p-1}).
\end{equation*}

An abelian variety with non zero Hasse invariant is called \emph{ordinary}.  We have a (set theoretic) decomposition of $X$ into its ordinary and non ordinary loci
\begin{equation*}
X=X^{\text{ord}}\bigcup X^{\text{NO}}
\end{equation*}
where $X^{\text{NO}}$ is exactly the zero locus of $A$.

There are several ways to further stratify $X$ according to finer invariants of abelian varieties in characteristic $p$.  The one that will play a role in this thesis is the Ekedahl-Oort stratification \cite{Oo01}.  Its definition begins with the surprising observation that for a principally polarized abelian variety $A$ over an algebraically closed field $k$ of characteristic $p$, there are only finite many possibilities for its p-torsion $A[p]$ as a finite group scheme over $k$ up to isomorphism.

This observation leads to a set theoretic decomposition
\begin{equation*}
X=\coprod_{w\in W^I}X_w
\end{equation*}
of $X$ by reduced locally closed subschemes $X_w$ which is characterized in the following way: two geometric points $x,y\in A(k)$ lie in the same $X_w$ if and only if $A_x[p]\cong A_y[p]$.

The indexing set $W^I$ is a certain set of Weyl group Coset representatives which we now describe.  Let $W$ be the Weyl group of type $C_g$.  Concretely, we realize this as the subgroup of $S_{2g}$ given by
\begin{equation*}
W=\{w\in S_{2g}\mid w(2g+1-i)=2g+1-w(i)\}.
\end{equation*}
The group $W$ is generated by the simple reflections
\begin{align*}
s_i&=(i\, i+1)(2g-i\,2g+1-i)\quad i=1,\ldots,g-1\\
s_g&=(g\,g+1)
\end{align*}
Then let $I=\{1,\ldots,g-1\}$ and let $W_I$ be the subgroup of $W$ generated by $s_i$ for $i\in I$ (this is nothing but the symmetric group on $g$ letters.)  Then the indexing set $W^I$ is the set of minimal length coset representatives for $W/W_I$.  Explicitly it is the set of $w\in W^I$ satisfying
\begin{equation*}
w(1)<w(2)<\cdots<w(g).
\end{equation*}
We will explain how this indexing comes about below.

The classical Hasse invariant recalled above lives on all of $X$ and its non vanishing locus is the ordinary locus $X^{\text{ord}}$.  Our generalized Hasse invariants live on the closed strata $\overline{X}_w$ and they are non vanishing on $X_w$.  That is, they cut out $\overline{X}_w-X_w$, the union of the strata lying in the closure of the open stratum $X_w$.

\begin{mainthm}[Existence of generalized Hasse invariants]
For each $w\in W^I$ there is an integer $N_w>0$ and a section
\begin{equation*}
A_w\in H^0(\overline{X}_w,\omega^{\otimes N_w})
\end{equation*}
with the following properties:
\begin{enumerate}
\item $A_w$ is non vanishing precisely on $X_w$.
\item For every Hecke correspondence $X\overset{p_1}{\leftarrow}Y\overset{p_2}{\to}X$ we have
\begin{equation*}
p_1^* A_w=p_2^*A_w
\end{equation*}
thought of as sections of $p_1^*\omega\simeq p_2^*\omega$ over $p_1^{-1}(\overline{X}_w)=p_2^{-1}(\overline{X}_w)$.
\end{enumerate}
\end{mainthm}

In fact we prove an analogous result for any PEL modular variety of type A or C, see Theorem \ref{T:hasseext}.

We now give an overview of the construction of the Hasse invariants in the theorem.  We use the theory of the canonical filtration due to Oort \cite{Oo01}.  In fact, our construction is directly motivated by the ``generalized Raynaud trick'' of Ekedahl and Oort.  Let $(A,\lambda)$ be a principally polarized abelian variety of dimension $g$ over an algebraically closed field $k$ of characteristic $p$.  Then there is a \emph{canonical filtration}
\begin{equation*}
0=G_0\subset G_1\subset G_2\subset \cdots\subset G_n=A[p]
\end{equation*}
of the $p$-torsion subgroup scheme $A[p]$ by finite subgroup schemes which is defined as the coarsest filtration such that for every $i$, $F^{-1}(G_i^{(p)})$, $V(G_i^{(p)})$ are both terms in the filtration.  Here $F$ and $V$ are the Frobenius
\begin{equation*}
F:A[p]\to (A[p])^{(p)}
\end{equation*}
and Verschiebung 
\begin{equation*}
V:(A[p])^{(p)}\to A[p]
\end{equation*}
respectively.  Note that $G_c=\ker F=\im V$ is always a term in the canonical filtration.  Moreover the canonical filtration is always self dual in the sense that $G_i=G_{n-i}^\perp$ under the Weil pairing induced by $\lambda$.  In particular $n=2c$ is even and for every $i>j$ we have isomorphisms $G_i/G_j\simeq (G_{2c-j}/G_{2c-i})^D$ (here for a finite flat group scheme $G$ we denote its cartier dual by $G^D$.)  Note that while the canonical filtration itself does not depend on the principal polarization $\lambda$, these isomorphisms do.

A crucial property of the canonical filtration is that there exists a permutation $\sigma:\{1,\ldots,2c\}\to\{1,\ldots,2c\}$ with the property that
\begin{enumerate}
\item For $i=1,\ldots,c$ we have
\begin{equation*}
V((G_{\sigma(i)})^{(p)})=G_i,\quad V((G_{\sigma(i)-1})^{(p)})=G_{i-1},
\end{equation*}
and
\begin{equation*}
V:(G_{\sigma(i)}/G_{\sigma(i)-1})^{(p)}\to G_i/G_{i-1}.
\end{equation*}
is an isomorphism.
\item For $i=c+1,\ldots,2c$,
\begin{equation*}
F^{-1}((G_{\sigma(i)})^{(p)})=G_i,\quad F^{-1}((G_{\sigma(i)-1})^{(p)})=G_{i-1},
\end{equation*}
and
\begin{equation*}
F:G_i/G_{i-1}\to (G_{\sigma(i)}/G_{\sigma(i)-1})^{(p)}
\end{equation*}
is an isomorphism.
\end{enumerate}

Let the order of $G_i$ be $p^{k_i}$.  It turns out that the $k_i$ and the permutation $\sigma$ determine the group $A[p]$ up to isomorphism, and so determine which Ekedahl-Oort stratum $(A,\lambda)$ lies in.  The Weyl group element $w$ associated to this Ekedahl-Oort stratum is the unique $w\in W^I$ with
\begin{equation*}
wx(k_i)=k_{\sigma(i)}
\end{equation*}
for $i=1,\ldots,2c$, where $x\in W$ is the permutation given by
\begin{equation*}
x(i)=i+g\quad \text{for $i=1,\ldots,g$}
\end{equation*}
and
\begin{equation*}
x(i)=i-g\quad\text{for $i=g+1,\ldots,2g$}.
\end{equation*}

Now consider an Ekedahl-Oort stratum $X_w$ and the universal abelian scheme $A$ over it.  It is not difficult to show that over $X_w$, $A[p]$ has a canonical filtration
\begin{equation*}
0=G_0\subset G_1\subset\cdots\subset G_{2c}=A[p]
\end{equation*}
by finite flat subgroup schemes $G_i$ over $X_w$ which over each geometric point gives the canonical filtration considered above.  For this it is crucial that we are working in a fixed EO stratum.  We remind the reader that the category of finite flat group schemes over a general base is not abelian, and that a finite flat subgroup scheme $H\subset G$ is a homomorphism of finite flat groups schemes which is a closed immersion.

Moreover, over all of $X_w$ we still have isomorphisms
\begin{equation*}
V:(G_{\sigma(i)}/G_{\sigma(i)-1})^{(p)}\to G_i/G_{i-1}
\end{equation*}
for $i=1,\ldots,c$ and
\begin{equation*}
F:G_i/G_{i-1}\to (G_{\sigma(i)}/G_{\sigma(i)-1})^{(p)}
\end{equation*}
for $i=c+1,\ldots,2c$.

For a finite flat group scheme $G/S$ we let
\begin{equation*}
\omega_G=e^*\Omega^1_{G/S}.
\end{equation*}

It can be shown that $\omega_{G_i/G_{i-1}}$ is locally free for $i=1,\ldots,2c$ unless $i=2c$ and $\sigma(2c)=2c$ (in which case $A[p]/G_{2c-1}$ is \'{e}tale.)  Hence we can define line bundles $\omega_i=\det\omega_{G_i/G_{i-1}}$, again unless $i=2c$ and $\sigma(2c)=2c$.  Cartier duality defines isomorphisms
\begin{equation*}
\omega_i\simeq\omega_{2c+1-i}^\vee
\end{equation*}
unless $i=1$ and $\sigma(1)=1$ or $i=2c$ and $\sigma(2c)=2c$.

Now differentiating Verschiebung and taking determinants we get isomorphisms
\begin{equation*}
V^*:\omega_i\to\omega_{\sigma(i)}^{(p)}
\end{equation*}
or in other words, nonvanishing sections
\begin{equation*}
B_i\in H^0(X_w,\omega_i^{\otimes -1}\otimes\omega_{\sigma(i)}^{\otimes p}).
\end{equation*}

Now $G_c=\ker F:A\to A^{(p)}$ and hence $\omega_{G_c}=\omega_A=\omega$, the determinant of the Hodge bundle.  From this we obtain an isomorphism
\begin{equation*}
\omega=\bigotimes_{i=1}^c\omega_i.
\end{equation*}

Now for suitable integers $r_i$, $i=1,\ldots, c$ (some possibly negative!) we can form the alternating product
\begin{equation*}
A_w'=\prod_{i=1}^c B_i^{r_i}\in H^0(X_w,\omega^{\otimes N'_w}).
\end{equation*}
The $r_i$ are chosen to make this combination of the $B_i$ a section of a positive power of $\omega$.  Then what we want to prove is
\begin{key claim}
Some power $A_w$ of $A_w'$ extends to a section $A_w\in H^0(\overline{X}_w,\omega^{\otimes N_w})$ which is non vanishing precisely on $X_w$.
\end{key claim}

Here is the idea: we would like to extend the canonical filtration on $X_w$ to $\overline{X}_w$.  Unfortunately this is impossible: it need not attain a limit, even in codimension 1.  To remedy this we introduce some auxiliary moduli spaces of abelian varieties with parahoric level structure.

Let $\tilde{X}$ be the moduli space of principally polarized abelian varieties $(A,\lambda)$ in characteristic $p$ along with a self dual filtration of $A[p]$ by finite flat group schemes $G_i$ with $|G_i|=p^{k_i}$.  Then there is a proper map $\pi:\tilde{X}\to X$ given by ``forgetting the level structure.''  Over the Ekedahl-Oort stratum $X_w$, $\pi$ admits a section $s$ given by the canonical filtration.

We let $\tilde{X}_w$ be the Zariski closure of $s(X_w)$ .  Now the situation has improved for the following reasons:

\begin{enumerate}
\item Because the canonical filtration over $s(X_w)$ extends to a filtration over $\tilde{X}_w$, we can extend the line bundles $\omega_i$ and the sections $B_i$ as well (although they will no longer be non vanishing.)
\item $\tilde{X}_w$ can be studied using Grothendieck-Messing theory \cite{dJ93}.  In particular it is normal and $\tilde{X}_w-s(X_w)$ is a union of divisors.
\item By a formal argument, in order to prove the key claim it suffices to show that $A'_w$ extends to a section of $\omega^{N_w}$ over $\tilde{X}_w$ whose non vanishing locus is precisely $s(X_w)$.
\end{enumerate}

Thus in order to complete the proof we have to
\begin{enumerate}
\item Compute the order of vanishing $\text{ord}_D(B_i)$ for $i=1,\ldots,c$ and each irreducible component $D$ of $\tilde{X}_w-s(X_w)$.
\item Show that for each irreducible component $D$ of the boundary
\begin{equation*}
\text{ord}_D(A_w)=\sum_{i=1}^c r_i\text{ord}_D(B_i)>0
\end{equation*}
\end{enumerate}
The first is reduced, via Grothendieck-Messing theory, to a completely explicit problem about sections of line bundles on Schubert varieties.  The second is combinatorics.

\section{Hasse Invariants at the Boundary}

For our applications to the construction of congruences, it is important to understand how the Hasse invariants $A_w$ of the previous section behave at the boundary of the minimal compactification $X$.  This subject is rather technical, so let us only mention some of the results.

First of all, the Ekedahl-Oort stratification extends to the minimal compactification (see Theorem \ref{T:EOmin}.)  Let $X^{\text{min}}=\Ca{X}^{\text{min}}_{\FF_p}$.  Then we have
\begin{equation*}
X^{\text{min}}=\coprod_{w\in W^I}X^{\text{min}}_w
\end{equation*}

The following theorem is somewhat technical and we refer the reader to the text for a precise statement.

\begin{mainthm}\label{MT:bdbehavior}
\begin{enumerate}
\item The Hasse invariants $A_w\in H^0(\overline{X}_w,\omega^{N_w})$ extends to a section $A_w\in H^0(\overline{X}_w^\text{min},\omega^{N_w})$ whose nonvanishing locus is $X_w^\text{min}$.
\item $A_w$ is ``nice'' at the boundary.  See Theorem \ref{T:hassemin} for a precise statement.
\end{enumerate}
\end{mainthm}

Roughly what ``nice'' means in the statement of the theorem is that the Fourier-Jacobi expansion of $A_w$ along each boundary component has only a constant term, and that constant term is a suitable Hasse invariant on a lower dimensional PEL modular variety.

As a corollary we have the following theorem which answers in the affirmative a question of Oort \cite{Oo01}:
\begin{mainthm}
The open Ekedahl-Oort strata of the minimal compactification $X^\text{min}_w$ are affine.
\end{mainthm}
Indeed this follows immediately from the fact that the non vanishing locus of an ample line bundle on a proper variety is affine.

\section{Constructing Congruences}\label{S:introcong}

In this section we explain how to use the generalized Hasse invariants of the last sections to produce congruences.  The  idea is this: for any automorphic vector bundle $V_\rho$ and nonnegative integers $r$ and $n$, we want to study the Hecke modules $H^n(\Ca{X}^\text{tor},V_\rho^{\text{sub}}/p^r)$.  We would like to relate it to a space of modular forms in characteristic 0 that we can already attach Galois representations to.  Here is the result, Theorem \ref{T:abscong} in the text, which may be regarded as the main result of this thesis.

\begin{mainthm}\label{MT:cong}
For any automorphic vector bundle $V_\rho$ and integers $r$ and $n$, there is an integer $C$ (which can be taken to be as large as desired) such that $H^n(\Ca{X}^{\text{\rm tor}},V_\rho^{\text{\rm sub}}/p^r)$ is a Hecke equivariant sub quotient of $H^0(\Ca{X}^{\text{\rm tor}},V_\rho^{\text{\rm sub}}\otimes\omega^{C})$.
\end{mainthm}

Before discusses the proof of Theorem \ref{MT:cong}, let us mention some related work.
\begin{itemize}
\item For Hilbert modular varieties, the analogous result was proved by Emerton-Reduzzi-Xiao \cite{ERX13}
\item Building on methods of Scholze \cite{Sc13}, Pilloni and Stroh \cite{PS15} have proved an analog of Theorem \ref{MT:cong} where $\Ca{X}^{\text{tor}}$ is replaced by a ``strange integral model'' coming from the theory of the Hodge-Tate period map.  At the present time, it is not clear if there is any relation between the two results.
\item A similar result has also been announced by Goldring and Geraghty.
\end{itemize}

The first step in the proof of Theorem \ref{MT:cong} is to pass from the toroidal compactification to the minimal compactification.  By a Theorem discovered independently by Harris, Lan, Taylor, and Thorne \cite{HLTT12} and by Andreatta, Iovita, and Pilloni \cite{AIP15} (see Theorem \ref{T:HLTT}) if $\pi:X^{\text{tor}}\to X^{\text{min}}$ is the canonical map, then for $i>0$
\begin{equation*}
R^i\pi_*V_\rho^{\text{sub}}=0.
\end{equation*}

Let $V=\pi_*V_\rho^{\text{sub}}$.  We warn that $V$ is not usually a vector bundle.  This, along with the fact that $\Ca{X}^{\text{min}}/\Z_p$ is not smooth at the boundary is the source of technical complications.  On the other hand there is something important to be gained from working on the minimal compactification: $\omega$ is ample.

Now in order to prove Theorem \ref{MT:cong}, we need to show that $H^n(\Ca{X}^{\text{min}},V/p^r)$ is a Hecke equivariant sub quotient of $H^0(\Ca{X}^{\text{min}},V\otimes\omega^C)$ for some large $C$.

Let us explain the idea of the proof.  Recall that $X^{\text{min}}=\Ca{X}^{\text{min}}_{\FF_p}$ has an Ekedahl-Oort stratification
\begin{equation*}
X^{\text{min}}=\coprod_{w\in W^I}X^{\text{min}}_w
\end{equation*}
and we denote the closure of $X^{\text{min}}_w$ by $\overline{X}^{\text{min}}_w$.

For $i=0,\ldots, \dim X=\frac{g(g+1)}{2}$ we let
\begin{equation*}
X_i=\bigcup_{l(w)=\dim X-i} \overline{X}_w^{\text{min}}
\end{equation*}
be the union of the codimension $i$ EO strata.  Then each $X_i$ is a reduced closed subscheme of $X^{\text{min}}$ and every irreducible component has codimension $i$ in $X$.  By a formal argument, on each $X_i$ we can construct a ``glued Hasse invariant'' $A_i\in H^0(X_i,\omega^{\otimes N_i})$ which has the property that its restriction to each $\overline{X}_w^\text{min}$ in the union above is a suitable power of the Hasse invariant $A_w$ on $\overline{X}_w^{\text{min}}$ .  In particular $A_i$ has the following properties:
\begin{itemize}
\item The zero locus of $A_i$ is set theoretically $X_{i+1}$.
\item The section $A_i$ is ``compatible'' with all Hecke correspondences.
\end{itemize}

Now in order to prove Theorem \ref{MT:cong} we will inductively construct the following objects
\begin{enumerate}
\item Closed subschemas $\tilde{X}_i$ of $\Ca{X}^{\text{min}}$ with reduction $(\tilde{X}_i)_{\text{red}}=X_i$ and which are Hecke stable in the sense that for each Hecke correspondence $\Ca{X}\overset{p_1}{\leftarrow}\Ca{Y}\overset{p_2}{\to}\Ca{X}$ we have
\begin{equation*}
p_1^{-1}(\tilde{X}_i)=p_2^{-1}(\tilde{X}_i)
\end{equation*}
\item Sections $\tilde{A}_i\in H^0(\tilde{X}_i,\omega^{\tilde{N}_i})$ whose restriction to the reduction $X_i$ is a suitable power of $A_i$ and which are compatible with the Hecke correspondences in the natural sense.
\item Hecke equivariant surjections 
\begin{equation*}
H^{n-i-1}(\tilde{X}_{i+1},V\otimes\omega^{\otimes M_{i+1}}|_{\tilde{X}_{i+1}})\to H^{n-i}(\tilde{X}_i,V\otimes \omega^{\otimes M_i}|_{\tilde{X}_i})
\end{equation*}
\end{enumerate}
The argument proceeds in the following steps:
\begin{itemize}
\item To start the induction we take $\tilde{X}_0=\Ca{X}^{\text{min}}\times_{\Z_p}\Z/p^r$ and $M_0=0$, so that the target of the chain of surjections produced in 3 above is $H^n(\Ca{X}^*,V/p^r)$, the space we want to understand in Theorem \ref{MT:cong}.
\item The construction of $\tilde{A}_i$ given $\tilde{X}_i$ is formal: $A_i^{p^m}$ has a canonical lift to $\tilde{X}_i$ for $m$ sufficiently large.  The compatibility with Hecke correspondences also follows from the canonicity.

\item Now consider the following sequence of coherent sheaves on $\Ca{X}$:
\begin{equation}\label{E:keyseq}
0\to V\otimes\omega^{\otimes M_i}|_{\tilde{X}_i}\overset{\tilde{A}_i^k\cdot}{\to}V\otimes\omega^{\otimes M_i+k\tilde{N}_i}|_{\tilde{X}_i}\to V\otimes\omega^{\otimes M_i+k\tilde{N}_i}|_{V(\tilde{A}_i^k)}\to 0\tag{*}
\end{equation}
where $V(\tilde{A}_i^k)$ denotes the vanishing locus of the section $\tilde{A}_i^k$ on $\tilde{X}_i$.

\begin{key claim}
This sequence is exact.
\end{key claim}
Admitting this for the moment we consider a part of the long exact sequence in cohomology:
\begin{equation*}
 H^{n-i-1}(V(\tilde{A}_i^k),V\otimes\omega^{\otimes M_i+k\tilde{N}_i})\to H^{n-i}(\tilde{X}_i,V\otimes\omega^{\otimes M_i})\to H^{n-i}(\tilde{X}_i,V\otimes\omega^{\otimes M_i+k\tilde{N}_i})
\end{equation*}
Now recall that $\omega$ is ample on $\Ca{X}$ and as $n-i>0$ we can pick $k$ sufficiently large that
\begin{equation*}
H^{n-i}(\tilde{X}_i,V\otimes\omega^{\otimes M_i})=0
\end{equation*}
by Serre vanishing.  Then we can take $M_{i+1}=M_i+k\tilde{N}_i$ and $\tilde{X}_{i+1}=V(\tilde{A}_i^k)$ and the map
\begin{equation*}
H^{n-i-1}(\tilde{X}_{i+1},V_{\text{sub}}\otimes\omega^{\otimes M_{i+1}})\to H^{n-i}(\tilde{X}_i,V_{\text{sub}}\otimes \omega^{\otimes M_i})
\end{equation*}
in the above long exact sequence gives the desired surjection.
\end{itemize}

We now turn to the key claim in the above argument, namely the exactness of (\ref{E:keyseq}).  This is actually quite delicate at the boundary, where the full strength of Theorem \ref{MT:bdbehavior} is used.  Here we content ourselves with explaining why it is true in the interior.  In the interior, $V$ is locally free, so the claim amounts to the fact that $\tilde{A}_i$ is a non zero divisor on the interior of $\tilde{X}_i$.  Now $\tilde{X}_i$ is cut out by the sequence $p^r,\tilde{A}_0^{k_0},\ldots,\tilde{A}_{i-1}^{k_{i-1}}$, which, by induction, is a regular sequence.  Hence as $\Ca{X}$ is regular, the interior of $\tilde{X}_i$ is Cohen-Macaulay (this step in the argument breaks down at the boundary.)  Thus in order to check that $\tilde{A}_i$ is a non zero divisor, it suffices to check that it doesn't vanish identically on any reduced irreducible component.  But this follows from our knowledge of the set theoretic vanishing locus of $A_i$.

To summarize what we have done towards proving Theorem \ref{MT:cong}, we have constructed a closed subscheme $\tilde{X}_n$ of $\Ca{X}$, a section $\tilde{A}_n\in H^0(\tilde{X}_n,\omega^{\tilde{N}_n})$ which is a non zero divisor on $V|_{\tilde{X}_n}$, and a Hecke equivariant surjection
\begin{equation*}
H^0(\tilde{X}_n,V\otimes\omega^{M_n}|_{\tilde{X}_n})\to H^n(\Ca{X}^{\text{min}},V/p^r).
\end{equation*}

Now consider the short exact sequence of coherent sheaves on $\Ca{X}^{\text{min}}$
\begin{equation*}
0\to \Ca{K}\to V\otimes\omega^{\otimes M_n}\overset{\text{res}}{\to} V\otimes\omega^{\otimes M_n}|_{\tilde{X}_n}\to 0.
\end{equation*}
where $\Ca{K}$ is just the kernel of the restriction map on the right.  Tensoring with $\omega^{k\tilde{N}_d}$ and taking cohomology we obtain an exact sequence
\begin{equation*}
H^0(\Ca{X}^*,V\otimes\omega^{\otimes M_n+k\tilde{N}_n})\overset{\text{res}}{\to} H^0(\tilde{X}_n,V\otimes\omega^{\otimes M_n+k\tilde{N}_n})\to H^1(\Ca{X}^{\text{min}},\Ca{K}\otimes\omega^{\otimes k\tilde{N}_n}).
\end{equation*}
By Serre vanishing we can pick $k$ large enough so that the term on the right vanishes.  Then the restriction map is surjective.

Now consider the diagram
\begin{equation*}
\begin{CD}
@. H^0(\tilde{X}_n,V\otimes\omega^{\otimes M_n})@>>> H^n(\Ca{X}^{\text{min}},V/p^r)\\
@. @VV\cdot \tilde{A}_n^kV @.\\
H^0(\Ca{X}^{\text{min}},V\otimes \omega^{\otimes M_n+k\tilde{N}_n})@>\text{res}>> H^0(\tilde{X}_n,V\otimes\omega^{\otimes M_n+k\tilde{N}_n}) @.
\end{CD}
\end{equation*}
where all the maps are Hecke equivariant, the horizontal maps are sujections, and the vertical map is injective.  This gives Theorem \ref{MT:cong} with $C=M_n+k\tilde{N}_n$.

\section{Overview and Advice for the Reader}

Chapter \ref{PEL} is a review of the theory of PEL modular varieties, while Chapter \ref{compact} is concerned with their compactifications.  Nothing except possibly the results of section \ref{S:wellpos} can be considered new, and we follow Lan \cite{La13} closely.  These chapters can probably be skipped and referred back to as necessary (and moreover the results of Chapter \ref{compact} are not used until Chapter \ref{boundaryhasse}).  Chapter \ref{EOPEL} is concerned with the Ekedahl-Oort stratification, the theory of the canonical filtration, and generalized Hasse invariants on open Ekedahl-Oort strata.  The first main result of this thesis is proved in Chapter \ref{siegel}, where the proof of the existence of generalized Hasse invariants is completed.  In Chapter \ref{boundaryhasse} we study the Ekedahl-Oort stratification and generalized Hasse invariants at the boundary.  Finally the argument for constructing congruences is presented in Chapter \ref{cong}.

On first reading, the reader is advised to consider only the case of compact Shimura varieties.  Then chapters 3 and 6 can be skipped, as well as all the details concerning the boundary in chapter 7.  

\section{Notation and Conventions}
Throughout this thesis we fix a rational prime $p$, an algebraic closure $\overline{\Q}_p$ of $\Q_p$, and an isomorphism $\iota:\overline{\Q}_p\simeq \C$.  We let $E$ be a finite extension of $\Q_p$ contained in $\overline{\Q}_p$ with integer ring $R$, uniformizer $\pi$ and residue field $k$.  We will freely enlarge $E$ as necessary throughout the text.

We let
\begin{equation*}
\hat{\Z}^{(p)}=\prod_{l\not=p}\Z_l,\quad\hat{\Z}=\prod_l\Z_l
\end{equation*}
and
\begin{equation*}
\A^{\infty,p}=\hat{\Z}^{(p)}\otimes\Q,\quad \A^{\infty}=\hat{\Z}\otimes\Q,\quad \A=\R\times\A^{\infty}.
\end{equation*}

Throughout the text we will use $\Ca{X}$ with various decorations to denote a PEL modular variety, or a compactification or completion etc over $R$.  We will use $X$ with the same decorations for its base change to $k$.

Throughout this thesis all schemes are assumed to be locally noetherian unless mentioned otherwise.  We make the following conventions regarding group schemes:
\begin{itemize}
\item All group schemes will be commutative.
\item For a flat group scheme $G/S$ for $S/\FF_p$ we have \cite[VIIa 4]{SGA3} relative Frobenius and Verschiebung homomorphisms which we will denote by
\begin{equation*}
F:G\to G^{(p)}
\end{equation*}
and
\begin{equation*}
V:G^{(p)}\to G.
\end{equation*}
They are functorial, compatible with arbitrary base change, and satisfy $VF=[p]_{G}$ and $FV=[p]_{G^{(p)}}$.
\item We embed the category of commutative group schemes over $S$ into the category of fppf sheaves of commutative groups on $S$.  When we speak of injective or surjective group homomorphisms we should always mean in the sense of the corresponding morphisms of fppf sheaves (note that this conflicts with the meaning of these words in scheme theory!)  Similarly we form kernels, cokerenels, and images in the category of fppf sheaves of commutative groups (so in particular cokernels and images need not be representable), and a sequence of maps
\begin{equation*}
G_1\to G_2\to G_3
\end{equation*}
of commutative group schemes is said to be exact at $G_2$ if the corresponding sequence of fppf sheaves is.
\item For $G/S$ a finite flat group scheme, we denote by
\begin{equation*}
G^D=\hom(G,\G_m)
\end{equation*}
its Cartier dual.
\end{itemize}

\chapter{Good Reduction PEL Modular Varieties}\label{PEL}

In this chapter we review the theory of PEL modular varieties and automorphic vector bundles on them.  Basic references for this subject are \cite{Ko92} and \cite{La13}.  We will mostly adopt the approach of \cite{La13}.

We let $\Z(1)=2\pi i\Z$.  For any $\Z$-module $M$ we denote $M(1)=M\otimes\Z(1)$.

\section{PEL Datum}
In this section we recall some definitions that are needed to formulate PEL moduli problems.

\begin{defn}
A \emph{rational PEL datum} is a tuple $(B,*,V,\langle\cdot,\cdot\rangle,h)$ where
\begin{itemize}
\item $B$ is a finite dimensional semisimple $\Q$-algebra.
\item $*$ is a positive involution on $B$ (i.e. $\tr_{B/\Q}(xx^*)>0$ for all nonzero $x\in B$.)
\item $V$ is a finitely generated, left $B$-module (which we do not assume is faithful!)
\item $\langle \cdot,\cdot\rangle:V\times V\to \Q(1)$ is an alternating form such that
\begin{equation*}
\langle bv,w\rangle=\langle v,b^*w\rangle
\end{equation*}
for all $v,w\in V$ and $b\in B$.
\item $h:\C\to \End_{B_\R}(V_\R)$ is a homomorphism of $\R$-algebras such that
\begin{equation*}
\langle h(z)v,w\rangle=\langle v,h(\overline{z})w\rangle
\end{equation*}
for all $z\in \C$ and $v,w\in V$, and such that the symmetric form $1/(2\pi i)\langle v,h(i)w \rangle$ is positive definite.
\end{itemize}
\end{defn}

Let $(B,*)$ be a finite dimensional semisimple $\Q$ algebra with positive involution as above and let $F$ be its center.  We let $\Ca{T}$ denote the set of embeddings $\tau:F\to \C$.  Via the fixed isomorphism $i:\overline{\Q}_p\simeq\C$ we may also view it as the set of embeddings $\tau: F\to\overline{\Q}_p$.  We have a decomposition
\begin{equation*}
F=\prod_{[\tau]}F_{[\tau]}
\end{equation*}
of $F$ into a product of number fields, where the product is indexed by $\text{Aut}(\C)$ orbits of $\Ca{T}$.  We have a corresponding decomposition
\begin{equation*}
B=\prod_{[\tau]}B_{[\tau]}
\end{equation*}
of $B$ where $B_{[\tau]}$ is simple with center $F_{[\tau]}$.  The positivity of the involution $*$ forces it to preserve this decomposition, and hence $(B,*)$ is a product of finite dimensional simple $\Q$-algebras with positive involution.

We now recall that a simple $\Q$-algebra with positive involution $(B,*)$ falls into one of three classes.  We let $F$ denote the center of $B$ and $F^+\subset F$ the subfield fixed by $*$.
\begin{enumerate}
\item(Type A) $F/F^+$ is a totally imaginary quadratic extension of a totally real field $F^+$.
\item(Type C) $F=F^+$ is totally real and for every embedding $\tau:F\to \R$, $B\otimes_{F,\tau}\R\simeq M_n(\R)$ for some integer $n$. 
\item(Type D) $F=F^+$ is totally real and for every embedding $\tau:F\to \R$, $B\otimes_{F,\tau}\R\simeq M_n(\Bf{H})$ for some integer $n$, where $\Bf{H}$ denotes the real quaternion algebra.
\end{enumerate}

For technical reasons we will exclude case D throughout this thesis.  We say that the PEL datum $(B,*,V,\langle\cdot,\cdot,\rangle,h)$ has no factors of type D if in the decomposition of $B$ into simple factors as above, none of the factors is of type D.

The homomorphism $h$ defines a decomposition
\begin{equation*}
V\otimes\C=V_0\oplus V_0^c
\end{equation*}
as $\C$ vector spaces where $h(z)$ acts as $z$ on $V_0$ and $\overline{z}$ on $V_0^c$.  This decomposition is stable under the action of $B$, and each factor is (maximal) isotropic for $\langle\cdot,\cdot\rangle$.

\begin{defn}
The \emph{reflex field} of the PEL datum $(B,*,V,\langle\cdot,\cdot\rangle,h)$ is the subfield $F_0$ of $\C$ over which the $B\otimes\C$ module $V_0$ is defined, i.e. it is the subfield of $\C$ fixed by all those $\sigma\in\text{Aut}(\C)$ such that
\begin{equation*}
V_0^{\sigma}:=V_0\otimes_{\C,\sigma}\C\simeq V_0
\end{equation*}
as $B\otimes\C$ modules.  Equivalently it is the subfield of $\C$ generated by the traces $\tr(b|V_0)$ for all $b\in B$, where the trace is taken with $b$ thought of as an endomorphism of the $\C$ vector space $V_0$.  (We remind the reader that there need not exist a $B\otimes F_0$ module $W$ with $W\otimes_{F_0}\C\simeq V_0$ as $B\otimes\C$-modules.)
\end{defn}

\begin{defn}
An \emph{integral structure} on a rational PEL datum $(B,*,V,\langle\cdot,\cdot\rangle,h)$ is the additional choice of $\O$ and $L$ where
\begin{itemize}
\item $\O$ an order in $B$ which is stable under $*$.
\item $L$ is a lattice in $V$ which is stable by $\O$ and such that for all $x,y\in L$
\begin{equation*}
\langle x,y\rangle\in\Z(1).
\end{equation*}
\end{itemize}
An \emph{integral PEL datum} is a tuple $(\O,*,L,\langle\cdot,\cdot\rangle,h)$ consisting of a rational PEL datum with an integral structure.  We omit $B$ and $V$ from the notation as they can be recovered as $\O\otimes\Q$ and $L\otimes\Q$ respectively.
\end{defn}

\begin{defn}\label{D:G}
Given an integral PEL datum $(\O,*,L,\langle\cdot,\cdot\rangle,h)$ define an algebraic group $G/\Z$ by
\begin{equation*}
G(A)=\{(g,a)\in\GL_{\O\otimes A}(L\otimes A)\times A^\times\mid \langle gx,gy\rangle=a\langle x,y\rangle, \forall x,y\in L\otimes A\}
\end{equation*}
for all rings $A$.  We note that $G\otimes\Q$ depends only on the rational PEL datum.
\end{defn}

Next we recall what it means for a compact open subgroup $K\subset G(\A^{p,\infty})$ to be neat as in Lan \cite[1.4.1.8]{La13} (see also Pink \cite{Pi89}.)

\begin{defn}
Let $g_l=(g,a)\in G(\Q_l)$.  Then let $\Gamma_{g_l}$ be subgroup of $\Q_l^\times$ generated by $a$ and the eigenvalues of $g$ thought of as an element of $\GL(L\otimes\Q_l)$.  For any embedding $\overline{\Q}\to\overline{\Q}_l$ consider the torsion subgroup $(\overline{\Q}^\times\cap\Gamma_{g_l})_{\text{tors}}$.  This does not depend on the choice of the embedding.

Then an element $g=(g_l)\in G(\A^{\infty,p})$ is said to be \emph{neat} if
\begin{equation*}
\cap_{l\not=p}(\overline{\Q}^\times\cap\Gamma_{g_l})_{\text{tors}}=\{1\}.
\end{equation*}
An open compact subgroup $K\subset G(\A^{p,\infty})$ is said to be \emph{neat} if every $g\in K$ is.
\end{defn}

\begin{defn}
We say that a rational prime $p$ is a good prime for an integral PEL datum $(\O,*,L,\langle\cdot,\cdot,\rangle,h)$ with no factors of type D if $p\nmid\text{Disc}_\O[L^\vee:L]$, where $\text{Disc}_{\O}$ is the discriminant of $\O$, and $L^\vee$ is the dual lattice
\begin{equation*}
L^\vee=\{x\in V\mid \langle x,y\rangle\in\Z(1), \forall y\in L\}.
\end{equation*}
\end{defn}

For the rest of this chapter assume that $p$ is a good prime for the integral PEL datum $(\O,*,L,\langle\cdot,\cdot\rangle,h)$.  The fact that $p\nmid \text{Disc}_\O$ implies (see e.g. \cite[1.1.1.21]{La13}) that 
\begin{equation*}
\O\otimes\Z_p=\prod_{[\tau]}M_{n_{[\tau]}}(\O_{F_{[\tau]}}\otimes\Z_p)
\end{equation*}
where $n_{[\tau]}=\dim_{F_{[\tau]}}B_{[\tau]}^{1/2}$ and moreover that $p$ is unramified in each $F_{[\tau]}$.  In particular
\begin{equation*}
\overline{\O}=\O\otimes\FF_p=\prod_{[\tau]}M_{n_[\tau]}(\O_{F_{[\tau]}}\otimes\FF_p)
\end{equation*}
is a semisimple $\FF_p$ algebra with involution, and we may identify $\Ca{T}$ with the set of ring homomorphisms $\O_F\to\overline{\FF}_p$.

The moduli problem associated to an integral PEL datum is naturally defined over the reflex field $F_0$.  Then if $p$ is a good prime, one can even work integrally over $\O_{F_0,(p)}$.  For studying rationality properties of automorphic forms, it is important to work over a number field like $F_0$ or a more integral variant.  However for the purposes of studying congruences, it is more convenient to work over a $p$-adic base.  Moreover we have no reason to work over a small field or ring of coefficients and we will readily enlarge it whenever it is convenient.  Let us now introduce the base we will work over.

Let $E\subset\overline{\Q}_p$ be a finite extension of $\Q_p$ which contains the images of all the embeddings $\tau:F\to\Q_p$.  We can take $E$ to be unramified over $\Q_p$, but this is not necessary.  We note that $E$ contains the reflex field $F_0$ thought of as a subfield of $\overline{\Q}_p$ via $i:\overline{\Q}_p\simeq\C$.  We let $R$ be the integer ring of $E$, denote a uniformizer by $\pi$ and let $k$ be its residue field, which has a canonical embedding into $\overline{\FF}_p$.

By our choice of $R$, it is easy to see that there is an $\O\otimes R$ module $L_0$ such we have
\begin{equation*}
L_0\otimes_R\C\simeq V_0
\end{equation*}
as $\O\otimes\C$-modules.  Moreover $L_0$ is unique up to isomorphism.

Next we introduce some notation for formulating the Kottwitz determinant condition.  We follow the coordinate free approach of Lan \cite{La13}, which is based on that of Rapoport-Zink \cite{RZ96}.

\begin{defn}
Let $S$ be any scheme.
\begin{enumerate}
\item We have a quasi-coherent sheaf of $\O_S$-algebras
\begin{equation*}
\O_S[\O^\vee]:=\O_S\otimes\Z[\O^\vee]
\end{equation*}
where $\Z[\O^\vee]$ is the symmetric algebra
\begin{equation*}
\Z[\O^\vee]=\text{Sym}_\Z^*(\O^\vee).
\end{equation*}
A choice of a $\Z$-basis $\alpha_1,\ldots,\alpha_n$ of $\O$ with dual basis $\alpha_1^\vee,\ldots,\alpha_n^\vee$ of $\O^\vee$ defines an isomorphism
\begin{equation*}
\Z[\O^\vee]\simeq\Z[x_1,\ldots,x_n]
\end{equation*}
sending $\alpha_i^\vee$ to $x_i$.
\item Given a locally free sheaf of finite rank $\Ca{E}$ on $S$ with an $\O_S$ linear action of $\O$, we define a section $\text{Det}_{\O|\Ca{E}}$ of $\O_S[\O^\vee]$ as follows: choose a basis $\alpha_1,\ldots,\alpha_n$ for $\O$ and dual basis $\alpha_1^\vee,\ldots,\alpha_n^\vee$ of $\O^\vee$ and consider
\begin{equation*}
\det(x_1\alpha_1+\cdots+x_n\alpha_n|\Ca{E})\in\O_S\otimes\Z[x_1,\ldots,x_n].
\end{equation*}
Via the isomorphism $\Z[\O^\vee]\simeq\Z[x_1,\ldots,x_n]$ above this gives the desired section.  One readily checks that it does not depend on the choice of basis of $\O$.
\end{enumerate}
\end{defn}

This definition is motivated by the following easy lemma, which we refer to \cite[1.1.2.20]{La13} for a proof.
\begin{lem}\label{L:det}
Let $k$ be a field such that $\O\otimes k$ is a separable $k$ algebra.  Then if $V_1$ and $V_2$ are two finite dimensional $\O\otimes k$ modules which are finite dimensional as $k$-vector spaces such that
\begin{equation*}
\text{\rm Det}_{\O|V_1}=\text{\rm Det}_{\O|V_2}
\end{equation*}
then $V_1\simeq V_2$ as $\O\otimes k$ modules.
\end{lem}
We note that the reason for working with determinants rather than traces is that otherwise, this lemma would be false if $k$ has positive characteristic.

\section{PEL Modular Varieties}

Let $(\O,*,L,\langle\cdot,\cdot\rangle,h)$ be an integral PEL datum without factors of type D for which $p$ is a good prime.  Following Lan \cite{La13}, we give two different descriptions of the associated PEL moduli problem.  The first moduli problem involves abelian schemes with extra structure up to isomorphism.

\begin{defn}\label{D:modisom}
For a compact open subgroup $K\subset G(\hat{\Z}^{(p)})$ we define a functor $M_K^{\text{isom}}$ which sends an $R$-scheme $S$ to the set of isomorphism classes of tuples $(A,\lambda,i,\alpha_K)$ where
\begin{enumerate}
\item $A/S$ is an abelian scheme.
\item $\lambda:A\to A^\vee$ is a prime to $p$ polarization of $A$.
\item $i:\O\to\text{End}_S(A)$ is a ring homomorphism compatible with $\lambda$ in the sense that
\begin{equation*}
\lambda i(x^*)=i(x)^\vee\lambda
\end{equation*}
for all $x\in \O$ and such that the Kottwitz condition is satisfied:
\begin{equation*}
\text{Det}_{\O|\Lie(A/S)}=\text{Det}_{\O|L_0}\in\O_S[\O^\vee]
\end{equation*}
where the finite locally free $\O_S$-module $\Lie(A/S)$ has an $\O_S$-linear action of $\O$ via $i$. 
\item $\alpha_K$ is an integral level $K$ structure of $(A,\lambda,i)$ in these sense of Defnition 1.3.7.6 of Lan \cite{La13}.
\end{enumerate}
and an isomorphism between $(A,\lambda,i,\alpha_K)$ and $(A',\lambda',i',\alpha_K')$ is given by an isomorphism $f:A\to A'$ of abelian schemes which is compatible with $\lambda$ and $\iota$ in the sense that
\begin{equation*}
\lambda=f^\vee\lambda'f
\end{equation*}
and
\begin{equation*}
fi(x)=i'(x)f
\end{equation*}
for all $x\in\O$, and moreover $f$ sends the level $K$ structure $\alpha_K$ to $\alpha'_K$ in the sense explained in Definition 1.4.1.4 of \cite{La13}.
\end{defn}

We do not recall Lan's somewhat complicated definition of an integral level $K$ structure in 4 above.  However below we will give an equivalent but simpler definition that works when the base $S$ is locally noetherian.

Now we give the second version of the moduli problem, involving abelian schemes with extra structure up to prime to $p$ isogeny, which is however only defined on locally noetherian schemes.

\begin{defn}\label{D:modisog}
For a compact open subgroup $K\subset G(\A^{\infty,p})$ we define a functor $M_K^{\text{isog}}$ which sends a \emph{locally noetherian} $R$-scheme $S$ to the set of isomorphism classes of tuples $(A,\lambda,i,\alpha_K)$ where
\begin{enumerate}
\item $A/S$ is an abelian scheme.
\item $\lambda:A\to A^\vee$ is a prime to $p$ quasi-polarization of $A$.  (We remind the reader that a prime to $p$ quasi-polarization is a prime to $p$ quasi-isogeny such that there exists an integer $N>0$ with $N\lambda$ a polarization.)
\item $i:\O\otimes\Z_{(p)}\to\text{End}_S(A)\otimes\Z_{(p)}$ is a ring homomorphism compatible with $\lambda$ in the sense that
\begin{equation*}
\lambda i(x^*)=i(x)^\vee\lambda
\end{equation*}
for all $x\in \O\otimes\Z_{(p)}$ and such that the Kottwitz condition is satisfied:
\begin{equation*}
\text{Det}_{\O|\Lie(A/S)}=\text{Det}_{\O|L_0}\in\O_S[\O^\vee]
\end{equation*}
where the locally free of finite rank $\O_S$-module $\Lie(A/S)$ has an $\O_S$-linear action of $\O$ via $i$. 
\item $\alpha_K$ is a rational level $K$ structure of $(A,\lambda,i)$.
\end{enumerate}
and an isomorphism between $(A,\lambda,i,\alpha_K)$ and $(A',\lambda',i',\alpha_K')$ is given by a prime to $p$ quasi-isogeny $f:A\to A'$ of abelian schemes which is compatible with $\lambda$ and $\iota$ in the sense that
\begin{equation*}
\lambda=rf^\vee\lambda'f
\end{equation*}
for some locally constant function $r:S\to\Z_{(p)}^{\times,>0}$ and
\begin{equation*}
fi(x)=i'(x)f
\end{equation*}
for all $x\in\O\otimes\Z_{(p)}$, and moreover $f$ sends the level $K$ structure $\alpha_K$ to $\alpha'_K$ in the sense explained below.
\end{defn}

Let us now recall the definitions of the rational and integral level $K$ structures appearing in the moduli problem, at least when the base $S$ is locally noetherian.
\begin{defn}
\begin{enumerate}
\item Let $S$ be a locally noetherian scheme, let $(A,\lambda,i)$ be as in definition \ref{D:modisog}, and let $K\subset G(\A^{\infty,p})$ be a compact open subgroup.  

First assume that $S$ is connected and pick a geometric point $\overline{s}$ of $S$.  Then a \emph{rational level $K$ structure} $\alpha_K$ on $(A,\lambda,i)$ is a $\pi_1(S,\overline{s})$ invariant $K$-orbit of pairs $(\alpha,\nu)$ where
\begin{equation*}
\alpha:L\otimes\A^{\infty,p}\simeq V^pA_{\overline{s}}
\end{equation*}
is an $\O$ invariant isomorphism of $\A^{\infty,p}$-modules, and
\begin{equation*}
\nu:\A^{\infty,p}(1)\simeq V^p\Bf{G}_m
\end{equation*}
is an isomorphism such that
\begin{equation*}
\begin{tikzcd}
(L\otimes\A^{\infty,p})\times(L\otimes\A^{\infty,p})\arrow{r}{\langle\cdot,\cdot\rangle}\arrow{d}{\alpha\times\alpha}&\A^{\infty,p}(1)\arrow{d}{\nu}\\
V^pA_{\overline{s}}\times V^pA_{\overline{s}}\arrow{r}{e^\lambda}&V^p\Bf{G}_m
\end{tikzcd}
\end{equation*}
commutes.  Here $V^p$ denotes the rational, prime to $p$ adelic Tate module, $e^\lambda$ is the Weil pairing induced by the polarization, and $G(\hat{\A}^{\infty,p})$ acts on the set of $(\alpha,\nu)$ on the right via its action on $L\otimes\A^{\infty,p}$ and $\A^{\infty,p}(1)$ (the later being via the similitude factor.)

If $f:(A,\lambda,i)\to(A',\lambda',i')$ is a prime to $p$ quasi-isogeny with similitude factor $r\in\Z^{\times,>0}_{(p)}$ as in the definition, then from a pair $(\alpha,\nu)$ for $(A,\lambda,i)$ as above we define $(\alpha',\nu')$ via $\alpha'=V^p(f)\alpha$ and $\nu'=r^{-1}\nu$.  In this way $f$ sends a level $K$-structure $\alpha_K$ on $(A,\lambda,i)$ to a level $K$ structure $\alpha'_K$ on $(A',\lambda',i')$.

By a standard argument which we don't recall here the notion of a rational level $K$ structure is canonically independent of the base point $\overline{s}$.  Finally if $S$ is not necessarily connected then a rational level $K$ structure on $(A,\lambda,i)$ is just a rational level $K$ structure in the sense above on each connected component of $S$.

\item Continue to assume that $S$ is a locally noetherian scheme and now let $(A,\lambda,i)$ be as in definition \ref{D:modisom}, and let $K\subset G(\hat{\Z}^{\infty,p})$ be open compact.

Let $\alpha_K$ be a rational level $K$ structure on $(A,\lambda,i)$ as above.  Then $\alpha_K$ is said to be an \emph{integral level $K$ structure} if for every geometric point $\overline{s}$ of $S$, and every pair $(\alpha,\nu)$ in the $\pi_1(S,\overline{s})$ invariant $K$ orbit as above, $\alpha$ and $\nu$ define isomorphisms between the natural $\hat{\Z}^{(p)}$ lattices on both sides, i.e. we have
\begin{equation*}
\alpha:L\otimes\hat{\Z}^{(p)}\simeq T^pA_{\overline{s}}
\end{equation*}
and
\begin{equation*}
\nu:\hat{\Z}^{(p)}(1)\simeq T^p\Bf{G}_m
\end{equation*}
isomorphisms of $\hat{\Z}^{(p)}$-modules, where $T^p$ denotes the integral prime to $p$ adelic Tate module.
\end{enumerate}
\end{defn}

\begin{rem}
Lan \cite[1.3.7.6]{La13} has given a definition of an integral level $K$ structure which does not require the base $S$ to be locally noetherian.  Then he proves in Lemma 1.3.8.5 and the discussion before it that this definition is the same as the one given above when the base is locally noetherian.  As the functor $M_K^{\text{isom}}$ will turn out to be finitely presented (see below) it is determined by its points on locally noetherian schemes.  Nonetheless, Lan's definition seems to be better adapted to his study of level structures on degenerating abelian varieties.
\end{rem}

Let $M_K^{\text{isom,LN}}$ denote the restriction of the functor $M_K^{\text{isom}}$ to the category of locally noetherian $R$-schemes.  Then there is an obvious natural transformation
\begin{equation*}
M_K^{\text{isom,LN}}\to M_K^{\text{isog}}
\end{equation*}
sending a tuple $(A,\lambda,i,\alpha_K)$ to its prime to $p$ quasi-isogeny class.  It is not difficult to see that this defines an isomorphism of functors (see \cite[1.4.3.4]{La13}.)

We now make some remarks about these definitions.
\begin{rem}\label{R:modcompare}
\begin{enumerate}
\item The reader may wonder why we would want to have both descriptions of the moduli problem.  For many purposes, for example for defining Hecke actions as in Section \ref{S:heckepel} below, it is more convenient to work with abelian schemes up to isogeny.  However for many questions of a local nature, such as the theory of degeneration, it seems more suitable to work with abelian varieties up to isomorphism (at least that is the approach of \cite{FC90} and \cite{La13}.)  In particular the results on compactifications from \cite{La13} which we refer to extensively are written in this way.

The reader may note that in order to consider $M_K^{\text{isom}}$ we assumed that $K\subset G(\hat{\Z}^{(p)})$, while in order to consider $M_K^{\text{isog}}$ we could work with any $K\subset G(\A^{\infty,p})$.  However there is not actually any generality lost by working with the first moduli problem.  Indeed it is not difficult to show that for any fixed $K\subset G(\A^{\infty,p})$ one can pick a new lattice $L'\subset V$ so that $(\O,*,L',\langle\cdot,\cdot\rangle,h)$ is also an integral PEL datum with $p$ as a good prime, and $K$ stabilizes $L'$ in the sense that if $G'$ denotes the new $\Z$ structure on $G_\Q$ determined by $L$, then $K\subset G'(\hat{\Z}^{(p)})$.  The moduli problems $M_K^{\text{isog}}$ associated to these two integral PEL data are the same.
\item The reader may also note that the definition of $M_K^{\text{isog}}$ only involves $\O\otimes\Z_{(p)}$, $L\otimes\A^{\infty,p}$, and $L\otimes\R$ and not $\O$ and $L$.  Similarly in the definition of $M_K^{\text{isom}}$ we only need the adelic object $L\otimes\Z^{(p)}$ and $L\otimes \R$ and not the lattice $L$.  %As Lan explains in Remarks 1.4.3.12 and 1.4.3.14 of \cite{La13} there are two reasons for not just working with the corresponding adelic objects.  The first is that without beginning with a lattice $L$, there is no guarantee that the moduli problem has any points at all.

%We note that this is also why we have been careful to talk about PEL modular varieties instead of PEL Shimura varieties.
\end{enumerate}
\end{rem}

Next we recall the following theorem of Kottwitz.  A detailed proof (which is also different from that of Kottwitz) can be found in \cite{La13}.

\begin{thm}[{\cite[1.4.1.11, 7.2.3.10]{La13}}]
Suppose $K\subset G(\A^{\infty,p})$ is a neat compact open.  Then:
\begin{enumerate}
\item The objects parameterized by $M_K^{\text{\rm isog}}$ have no non trivial automorphisms.
\item $M_K^{\text{\rm isog}}$ is represented by a smooth quasi-projective scheme $\Ca{X}_K/R$.
\end{enumerate} 
\end{thm}

\section{Automorphic Vector Bundles}\label{S:pelvb}

The goal of this section is to define some automorphic vector bundles on our PEL moduli spaces.  Let $K\subset G(\A^{\infty,p})$ be a neat open compact subgroup.  Let $(A,\lambda,i,\alpha_K)$ be a member of the universal isogeny class over $\Ca{X}_K$.  To it we associate a pair $(\omega_A,\O_{\Ca{X}_K})$ of a vector bundle with an $\O\otimes\Z_{(p)}$ action and a line bundle.  If $(A',\lambda',i',\alpha_K')$ is another member of the same isogeny class, then by the neatness of $K$ there is a unique prime to $p$ quasi-isogeny $f:(A,\lambda,i,\alpha_K)\to(A',\lambda',i',\alpha_K')$, as in definition \ref{D:modisog} which defines an isomorphism of pairs
\begin{equation*}
(\omega_A,\O_{\Ca{X}_K})\simeq(\omega_{A'},\O_{\Ca{X}_K})
\end{equation*}
which is $(f^*)^{-1}$ on the first factor, and multiplication by $r$ on the second factor, where $r$ is the locally constant $\Z_{(p)}^{\times,>0}$ valued function on $S$ such that $\lambda=rf^\vee\lambda'f$.

In this way we obtain a canonical pair $(\Ca{E}_K,\Xi_K)$ of a vector bundle with an $\O\otimes\Z_{(p)}$ action and a line bundle on $\Ca{X}_K$ which is independent of the choice of $(A,\lambda,i,\alpha_K)$ in the universal isogeny class.  (The reason for this slightly convoluted definition will become clear in the next section when we define Hecke actions.  In particular the trivial line bundle $\Xi_K$ will have a non trivial Hecke action.)

\begin{defn}
\begin{enumerate}
\item Let $M/R$ be the affine algebraic group representing the functor
\begin{equation*}
M(A)=\GL_{\O\otimes_\Z A}(L_0^\vee\otimes_RA)\times A^\times
\end{equation*}
for $R$-algebras $A$.
\item The principal $M$-bundle on $\Ca{X}_K$ is defined by
\begin{equation*}
P_K(S)=\text{Isom}_{\O\otimes\O_S}((\Ca{E}_K\otimes\O_S,\Xi_K\otimes\O_S),(L_0\otimes\O_S,R\otimes\O_S))
\end{equation*}
%TODO give a reference for why this is a principal bundle
\item For any algebraic representation $\rho$ of $M$ on a finite $R$-module $W$ define the coherent sheaf
\begin{equation*}
V_{\rho,K}=P_K\times^M W
\end{equation*}
It is functorial in $\rho$.
\end{enumerate}
\end{defn}

For example we have
\begin{equation*}
V_{L_0^\vee,K}=\Ca{E}_K
\end{equation*}
and
\begin{equation*}
V_{\det L_0^\vee,K}=\det\Ca{E}_K=:\omega_K.
\end{equation*}

\begin{prop}\label{P:autprop}
\begin{enumerate}
\item If $\rho$ is an algebraic representation on a finite free $R$-module (resp. a finite free $R/\pi^r$-module) then $V_{\rho,K}$ is a locally free sheaf on $\Ca{X}_K$ (resp. a locally free sheaf on $\Ca{X}_K\times R/\pi^r$.)
\item If $0\to \rho'\to\rho\to\rho''\to 0$ is a short exact sequence of algebraic representations of $M$ then we have a short exact sequence
\begin{equation*}
0\to V_{\rho',K}\to V_{\rho,K}\to V_{\rho'',K}\to 0
\end{equation*}
of sheaves on $\Ca{X}_K$.
\item If $\rho$ is an algebraic representation of $M$ on  a finite $R$-module and $k$ is any integer then
\begin{equation*}
V_{\rho\otimes\det^k L_0^\vee,K}=V_{\rho,K}\otimes\omega_K^{\otimes k}
\end{equation*}
\end{enumerate}
\end{prop}

\section{Hecke Action}\label{S:heckepel}
Suppose we have $g\in G(\A^{\infty,p})$ and $K,K'\subset G(\A^{\infty,p})$ compact open subgroups satisfying $g^{-1}Kg\subset K'$.  Then we can define a natural transformation
\begin{equation*}
M_K^{\text{isog}}\to M_{K'}^{\text{isog}}
\end{equation*}
as follows: send a point $(A,\lambda.i,\alpha_K)$ to $(A,\lambda,i,\alpha_{K'})$ where if $\overline{s}$ is a geometric point of $S$ and $(\alpha,\nu)K$ is the $\pi_1(S,\overline{s})$ stable $K$ orbit given by $\alpha_K$, then $(\alpha,\nu)gK'$ is the $\pi_1(S,\overline{s})$ stable $K'$ orbit corresponding to $\alpha_{K'}$.  If $K'$ is neat then so is $K$ and the above natural transformation defines a map
\begin{equation*}
[g]:\Ca{X}_{K}\to\Ca{X}_{K'}.
\end{equation*}

Here are the basic facts about Hecke actions on the modular varieties and automorphic vector bundles.
\begin{prop}\label{P:hecke}
Let $g\in G(\A^{\infty,p})$ and let $K,K'\subset G(\A^{\infty,p})$ be open compact subgroups satisfying $g^{-1}Kg\subset K'$.  Suppose $K'$ is neat.  Then
\begin{enumerate}
\item The map $[g]:\Ca{X}_{K}\to\Ca{X}_{K'}$ is finite, \'{e}tale, and surjective.
\item If $\rho$ is any algebraic representation of $M$ as in the last section, then there is a canonical isomorphism
\begin{equation*}
g:[g]^*V_{\rho,K'}\to V_{\rho,K}.
\end{equation*}
\item If $g'\in G(\A^{\infty,p})$ and $K''\subset G(\A^{\infty,p})$ are such that ${g'}^{-1}K'g'\subset K''$ and $K''$ is neat then
\begin{equation*}
[gg']=[g']\circ[g]:\Ca{X}_K\to\Ca{X}_{K''}
\end{equation*}
and for any $\rho$ we have
\begin{equation*}
gg'=g\circ[g]^*(g'):[gg']^*V_{\rho,K''}\to V_{\rho,K}.
\end{equation*}
\end{enumerate}
\end{prop}

\section{Morphism to Siegel Space}\label{S:morsieg}
We continue to let $(\O,*,L,\langle\cdot,\cdot\rangle,h)$ be an integral PEL datum without factors of type D for which $p$ is a good prime.  In this section we recall that there is are canonical maps from general PEL type modular varieties to Siegel modular varieties given by ``forgetting the extra endomorphisms.''

Let us consider a new integral PEL datum
\begin{equation*}
(\Z,\text{id},L,\langle\cdot,\cdot\rangle,h).
\end{equation*}
This PEL datum has type C, reflex field $\Q$, and $p$ is still a good prime for it.  We let $\tilde{G}$ be the corresponding group as in definition \ref{D:G}, and for $\tilde{K}\subset\tilde{G}(\A^{\infty,p})$ open and compact, let $\tilde{M}_{\tilde{K}}^{\text{isog}}$ be the corresponding moduli problems.  Moreover if $\tilde{K}\subset\tilde{G}(\A^{\infty,p})$ is neat let $\tilde{\Ca{X}}_{\tilde{K}}/R$ denote the corresponding PEL modular variety.

The group $G$ is naturally a closed subgroup scheme of $\tilde{G}$.  Suppose we have open compact subgroups $K\subset G(\A^{\infty,p})$ and $\tilde{K}\subset \tilde{G}(\A^{\infty,p})$ such that $K\subset\tilde{K}$.  Then there is a natural transformation
\begin{equation*}
M^{\text{isog}}_K\to \tilde{M}^{\text{isog}}_{\tilde{K}}
\end{equation*}
defined as follows: send a tuple $(A,\lambda,i,\alpha_K)$ to $(A,\lambda,i_0,\alpha_{\tilde{K}})$ where $A$ and $\lambda$ are unchanged, $i_0:\Z_{(p)}\to\End_S(A)\otimes \Z_{(p)}$ is the unique such homomorphism (note that the determinant condition is trivial in this case) and $\alpha_{\tilde{K}}$ is defined by sending a $\pi_1(S,\overline{s})$ invariant $K$ orbit of pairs $(\alpha,\nu)$ to its $\tilde{K}$ orbit (which is still $\pi_1(S,\overline{s})$ invariant as the action of $\pi_1(S,\overline{s})$ and $G(\A^{\infty,p})$ commute.)

To summarize the situation we have the following.
\begin{prop}\label{P:morsieg}
\begin{enumerate}
\item Let $K\subset G(\A^{\infty,p})$ and $\tilde{K}\subset \tilde{G}(\A^{\infty,p})$ be open compact subgroups.  Suppose that $K\subset \tilde{K}$ and that $\tilde{K}$ is neat.  Then $K$ is neat and the natural transformation defined above gives a finite morphism
\begin{equation*}
\phi_{K,\tilde{K}}:\Ca{X}_K\to \tilde{\Ca{X}}_{\tilde{K}}.
\end{equation*}
\item Let $g\in G(\A^{\infty,p})$ and suppose we have $K,K'\subset G(\A^{\infty,p})$ and $\tilde{K},\tilde{K}'\subset \tilde{G}(\A^{\infty,p})$ with $K\subset \tilde{K}$, $K'\subset\tilde{K}'$, $g^{-1}Kg\subset K'$, and $g^{-1}\tilde{K}g\subset\tilde{K}'$.  Suppose $\tilde{K}'$ is neat.  Then we have a commutative diagram
\begin{equation*}
\begin{tikzcd}
\Ca{X}_K\arrow{r}{\phi_{K,\tilde{K}}}\arrow{d}{[g]}&\tilde{\Ca{X}}_{\tilde{K}}\arrow{d}{[g]}\\
\Ca{X}_{K'}\arrow{r}{\phi_{K',\tilde{K}'}}&\tilde{\Ca{X}}_{\tilde{K}'}
\end{tikzcd}
\end{equation*}
\end{enumerate}
\end{prop}
The only point which isn't obvious is why $\phi_{K,\tilde{K}}$ is finite.  For this we refer to Proposition 1.3.3.7 of \cite{La13}.

\chapter{Arithmetic Compactifications of PEL Modular Varieties}\label{compact}

The goal of this chapter is to recall the theory of arithmetic compactifications (toroidal and minimal) of good reduction PEL modular varieties due to Faltings-Chai \cite{FC90} (for moduli spaces of principally polarized abelian varieties with principal level structure) and Lan \cite{La13} in general.  We will follow Lan closely.

The objects introduced in this chapter will not be used until chapter \ref{boundaryhasse}.  The reader who is only interested in the coherent cohomology of compact Shimura varieties may skip this chapter.

In Section \ref{S:charts} we will recall the description of certain ``boundary charts'' which model the formal completion of a toroidal compactification along its boundary strata.  In section \ref{S:compact} we will recall the compactifications themselves.  Then in section \ref{S:canext} we turn to the subject of extensions of automorphic vector bundles to the compactifications, as well as Hecke actions on them.  Finally in section \ref{S:wellpos} we prove some technical results that will be used in Chapters \ref{boundaryhasse} and \ref{cong} to deal with difficulties at the boundary in the construction of congruences in Chapter \ref{cong}.

Throughout this chapter we will fix an integral PEL datum $(\O,*,L,\langle\cdot,\cdot\rangle,h)$ without factors of type D and for which $p$ is a good prime.  For technical reasons, we will assume it satisfies the following condition, which is Condition 1.4.3.10 of \cite{La13}.

\begin{cond}\label{extcond}
We require that the action of $\O$ on $L$ extends to an action of a maximal order in $\O\otimes\Q$ containing $\O$.
\end{cond}

As explained in Remark 1.4.3.9 of \cite{La13} this condition does not limit which PEL moduli problems we may consider (see also Remark \ref{R:modcompare}) and it is used by Lan in his study of degenerations of abelian varieties with PEL structures.

\section{Toroidal Boundary Charts}\label{S:charts}

Lan defines \cite[5.4.2.4]{La13} a set $\text{Cusp}_K$ of cusp labels of level $K$.  They are defined to be certain equivalence classes of triples (denoted there by $(Z_\Ca{H},\Phi_{\Ca{H}},\delta_{\Ca{H}})$) and for convenience we fix once and for all a representative of each equivalence class.  We now recall a long list of objects associated to a cusp label $\Sc{C}\in\text{Cusp}_K$.  The definitions of these objects are rather elaborate and mostly won't be recalled here.

For each cusp label $\Sc{C}$ (or rather its chosen representative) we have
\begin{enumerate}
\item An $\O$-lattice $X$ (which is part of the data given by the chosen representative of $\Sc{C}$, see \cite[5.4.2.1]{La13}).
\item A subgroup $\Gamma_{\Sc{C}}\subset \GL_\O(X)$ (denoted $\Gamma_{\Phi_{\Ca{H}}}$ in \cite[6.2.4.1]{La13}) which is neat in the sense that for every $\gamma\in\Gamma_{\Sc{C}}$ the subgroup of $\overline{\Q}^\times$ generated by the eigenvalues $\gamma$ (thought of as an element of $\GL(X\otimes\Q)$) is torsion free.
\item An $\R$-vector space $M_{\Sc{C}}$ of bilinear pairings
\begin{equation*}
(\cdot,\cdot):(X\otimes\R)\times (X\otimes\R)\to \O\otimes\R
\end{equation*}
which are Hermitian in the sense that for all $x,y\in X\otimes\R$ and $b\in \O\otimes\R$ we have
\begin{equation*}
(x,y)=(y,x)^*
\end{equation*}
and
\begin{equation*}
(bx,y)=b(x,y).
\end{equation*}
\item The cones $P_{\Sc{C}}^+\subset P_{\Sc{C}}\subset M_{\Sc{C}}$ where $P_{\Sc{C}}^+$ is the cone of positive definite hermitian pairings in $M_{\Sc{C}}$ and $P_{\Sc{C}}$ is the cone of positive semidefinite hermitian pairings with admissible radical (see 6.2.5.3, 6.2.5.4 of \cite{La13} for the definitions.)  These cones are stable under the action of $\Gamma_{\Sc{C}}$

\item A $\Z$-lattice $\Bf{S}_{\Sc{C}}\subset M_{\Sc{C}}^\vee$ which is stable under the action of $\Gamma_{\Sc{C}}$ on $M_{\Sc{C}}^\vee$.  (See of 6.2.4.4 of \cite{La13} where $\Bf{S}_{\Sc{C}}$ is denoted $\Bf{S}_{\Phi_{\Ca{H}}}$.)

\item A non-canonical integral PEL datum $(\O,*,L_{\Sc{C}},\langle\cdot,\cdot\rangle_{\Sc{C}},h_{\Sc{C}})$ with reflex field $F_0$, the reflex field of our original PEL datum.  If $G_{\Sc{C}}/\Z$ is the corresponding group then $K$ determines a neat open compact subgroup $K_h\subset G_{\Sc{C}}(\hat{\Z}^{(p)})$.  Let $\Ca{X}_{\Sc{C}}/R$ be the corresponding PEL modular variety.  It is canonically associated to $\Sc{C}$ even though the PEL datum is not.  (See 5.4.2.6 of \cite{La13} where $\Ca{X}_{\Sc{C}}$ is denoted by $M_{\Ca{H}}^{Z_{\Ca{H}}}$.)
\item A finite \'{e}tale cover $\tilde{\Ca{X}}_{\Sc{C}}\to\Ca{X}_{\Sc{C}}$ with an action of $\Gamma_{\Sc{C}}$ such that $\tilde{\Ca{X}_{\Sc{C}}}/\Gamma_{\Sc{C}}=\Ca{X}_{\Sc{C}}$.  (See \cite[5.4.2.6,5.1.2.2]{La13} where $\tilde{\Ca{X}}_{\Sc{C}}$ is denoted by $M_{\Ca{H}}^{\Phi_{\Ca{H}}}$.)

\item A torsor under an abelian scheme $C_{\Sc{C}}\to \tilde{\Ca{X}}_{\Sc{C}}$ with a compatible action of $\Gamma_{\Sc{C}}$.  (See 6.2.4.7 of \cite{La13} where $C_{\Sc{C}}$ is denoted by $C_{\Phi_{\Ca{H}},\delta_{\Ca{H}}}$.)

\item A torsor under the split torus with character group $\Bf{S}_{\Sc{C}}$, $\Xi_{\Sc{C}}\to C_{\Sc{C}}$, with an action of $\Gamma_{\Sc{C}}$ compatible with that on $C_{\Sc{C}}$ and $\Bf{S}_{\Sc{C}}$.  We have
\begin{equation*}
\Xi_{\Sc{C}}=\underline{\spec}_{\O_{\C_{\Sc{C}}}}\bigoplus_{l\in\Bf{S}_{\Sc{C}}}\Psi_{\Sc{C}}(l)
\end{equation*}
where for each $l\in\Sc{S}_{\Sc{C}}$, $\Psi_{\Sc{C}}(l)/C_{\Sc{C}}$ is a line bundle and for each pair $l,l'\in\Bf{S}_{\Sc{C}}$ there is an isomorphism
\begin{equation*}
\Psi_{\Sc{C}}(l)\otimes\Psi_{\Sc{C}}(l')\simeq\Psi_{\Sc{C}}(l+l')
\end{equation*}
giving $\bigoplus_{l\in\Bf{S}_{\Sc{C}}}\Psi_{\Sc{C}}(l)$ the structure of a sheaf of $\O_{\Sc{C}}$-algebras.  The action of $\Gamma_{\Sc{C}}$ on $\Xi_{\Sc{C}}$ defines for each $\gamma\in\Gamma_{\Sc{C}}$ and $l\in\Bf{S}_{\Sc{C}}$ an isomorphism
\begin{equation*}
\gamma:\Psi_{\Sc{C}}(l)\to\gamma^*\Psi_{\Sc{C}}(\gamma l)
\end{equation*}
where $\gamma^*$ is the pullback along $\gamma:C_{\Sc{C}}\to C_{\Sc{C}}$.  (See 6.2.4.7 of \cite{La13} where $\Xi_{\Sc{C}}$ is denoted by $\Xi_{\Phi_{\Ca{H}},\delta_{\Ca{H}}}$ and $\Psi_{\Sc{C}}$ is denoted by $\Psi_{\Phi_{\Ca{H}},\delta_{\Ca{H}}}$.)

\item A semiabelian scheme $\tilde{A}_{\Sc{C}}/C_{\Sc{C}}$ with an $\O$ action which sits in an exact sequence
\begin{equation*}
0\to T\to \tilde{A}_{\Sc{C}}\to A_{\Sc{C}}\to 0
\end{equation*}
where $T$ is the constant torus with character group $X$ and $A_{\Sc{C}}$ is the pullback of the universal abelian scheme on $\Ca{X}_{\Sc{C}}$ to $C_{\Sc{C}}$.  $\tilde{A}_{\Sc{C}}$ carries an action of $\Gamma_{\Sc{C}}$ covering that on $C_{\Sc{C}}$ and compatible with the action of $\Gamma_{\Sc{C}}$ on $X$.
\end{enumerate}

Next we would like recall some definitions related to torus embeddings and cone decompositions (see section 6.1 of \cite{La13}).

\begin{defn}
\begin{enumerate}
\item A \emph{rational polyhedral cone} $\sigma\subset M_{\Sc{C}}$ is a subset of the form
\begin{equation*}
\sigma=\R_{>0}v_1+\cdots+\R_{>0}v_n
\end{equation*}
for $v_1,\ldots,v_n\in \Bf{S}_{\Sc{C}}^\vee$.  (Note that by convention the empty sum is $\{0\}$.)
\item For a rational polyhedral cone $\sigma\subset M_{\Sc{C}}$ we have semigroups
\begin{align*}
\sigma^{\vee}&=\{l\in\Bf{S}_{\Sc{C}}\mid l(v)\geq 0,\forall v\in\sigma\}\\
\sigma^{\vee}_0&=\{l\in\Bf{S}_{\Sc{C}}\mid l(v)> 0,\forall v\in\sigma\}\\
\sigma^\perp&=\{l\in\Bf{S}_{\Sc{C}}\mid l(v)=0,\forall v\in\sigma\}.
\end{align*}
\item A rational polyhedral cone $\sigma\subset M_{\Sc{C}}$ is said to be non degenerate if $\overline{\sigma}$ (the closure of $\sigma$ in $M_{\Sc{C}}$ for the real topology) does not contain any non trivial $\R$-vector subspace of $M_{\Sc{C}}$.
\item A rational polyhedral cone $\sigma\subset M_{\Sc{C}}$ is said to be \emph{smooth} if it is of the form
\begin{equation*}
\sigma=\R_{>0}v_1+\cdots+\R_{>0}v_n
\end{equation*}
for $v_1,\ldots,v_n\in \Bf{S}_{\Sc{C}}^\vee$ which extend to a basis for $\Bf{S}_{\Sc{C}}^\vee$.
\item A rational polyhedral cone $\tau$ is said to be a \emph{face} of a rational polyhedral cone $\sigma$ if there exists a linear functional $\lambda:M_{\Sc{C}}\to\R$ with $\lambda(\sigma)\subset\R_{\geq0}$ and $\overline{\tau}=\overline{\sigma}\cap\lambda^{-1}(0)$.  Then $\overline{\sigma}$ (the closure in the real topology on $M_{\Sc{C}}$) is the set theoretic disjoint union of the faces of Note that $\sigma$ is always a face of itself.  We say that $\tau$ is a proper face of $\sigma$ if $\tau$ is a face of $\sigma$ and $\tau\not=\sigma$.  (We note that in \cite{La13}, a face is what we have called a proper face.)
\item A \emph{$\Gamma_{\Sc{C}}$-admissible rational polyhedral cone decomposition} is a set $\Sigma_{\Sc{C}}=\{\sigma_j\}_{j\in J}$ of nondegenerate rational polyhedral cones such that
\begin{enumerate}
\item The $\sigma_j$ are pairwise disjoint, and $P_{\Sc{C}}=\bigcup_{j\in J}\sigma_j$.
\item For each $\sigma_j\in\Sigma_{\Sc{C}}$ and each face $\tau$ of $\sigma_j$, $\tau\in\Sigma_{\Sc{C}}$.
\item The set $\Sigma$ is invariant under $\Gamma$ and the set of orbits $\Sigma/\Gamma$ is finite.
\end{enumerate}
$\Sigma_{\Sc{C}}$ is said to be smooth if each $\sigma_j$ is.
\end{enumerate}
\end{defn}

To each non degenerate rational polyhedral cone $\sigma\subset M_{\Sc{C}}$ we have a relatively affine torus embedding
\begin{equation*}
\Xi_{\Sc{C}}(\sigma)=\underline{\spec}_{\O_{C_{\Sc{C}}}}\bigoplus_{l\in\sigma^\vee}\Psi_{\Sc{C}}(l)
\end{equation*}
We have a sheaf of ideals
\begin{equation*}
\I_\sigma=\bigoplus_{l\in\sigma^\vee_0}\Psi_{\Sc{C}}(l)
\end{equation*}
on $\Xi_{\Sc{C}}(\sigma)$ which defines a reduced closed subscheme $\Xi_{\Sc{C}}(\sigma)_\sigma\subset\Xi_{\Sc{C}}(\sigma)$ which is a relative torus torsor over $C_{\Sc{C}}$ under the split torus with character group $\sigma^\perp$.

If $\sigma,\tau\subset M_{\Sc{C}}$ are non degenerate rational polyhedral cones such that $\tau$ is a face of $\sigma$, then the inclusion $\sigma^\vee\subset\tau^\vee$ induces a map
\begin{equation*}
\Xi_{\Sc{C}}(\tau)\to\Xi_{\Sc{C}}(\sigma)
\end{equation*}
which is an open immersion.  We let $\Xi_{\Sc{C}}(\sigma)_{\tau}$ be the locally closed subscheme which is the image of $\Xi_{\Sc{C}}(\tau)_{\tau}$.  Then set theoretically we have
\begin{equation*}
\Xi_{\Sc{C}}(\sigma)=\coprod_{\tau\text{ face of }\sigma}\Xi_{\Sc{C}}(\sigma)_{\tau}.
\end{equation*}

Now given a $\Gamma_{\Sc{C}}$-admissible rational polyhedral cone decomposition $\Sigma_{\Sc{C}}$ we may glue the $\Xi_{\Sc{C}}(\sigma)$ for $\sigma\in\Sigma_{\Sc{C}}$ to form a separated, locally of finite type, relative torus embedding $\Xi_{\Sc{C},\Sigma_{\Sc{C}}}/C_{\Sc{C}}$.  For each $\sigma\in\Sigma_{\Sc{C}}$ we denote by $\Xi_{\Sc{C},\Sigma_{\Sc{C}},\sigma}$ the image of $\Xi_{\Sc{C}}(\sigma)_{\sigma}$ under the open immersion $\Xi_{\Sc{C}}(\sigma)\subset\Xi_{\Sc{C},\Sigma_{\Sc{C}}}$.  Then $\Xi_{\Sc{C},\Sigma_{\Sc{C}}}$ is stratified by the locally closed subschemes $\Xi_{\Sc{C},\Sigma_{\Sc{C}},\sigma}$ with $\sigma\in\Sigma_{\Sc{C}}$.  More precisely we have a set theoretic decomposition
\begin{equation*}
\Xi_{\Sc{C},\Sigma_{\Sc{C}}}=\coprod_{\sigma\in\Sigma_{\Sc{C}}}\Xi_{\Sc{C},\Sigma_{\Sc{C}},\sigma}
\end{equation*}
and
\begin{equation*}
\overline{\Xi}_{\Sc{C},\Sigma_{\Sc{C}},\sigma}=\coprod_{\substack{\tau\in\Sigma_{\Sc{C}}\\\text{$\sigma$ is a face of $\tau$}}}\Xi_{\Sc{C},\Sigma_{\Sc{C}},\tau}
\end{equation*}

The action of $\Gamma_{\Sc{C}}$ on $\Xi_{\Sc{C}}$ extends to an action of $\Xi_{\Sc{C},\Sigma_{\Sc{C}}}$ covering that on $C_{\Sc{C}}$.  $\gamma\in\Gamma_{\Sc{C}}$ sends the open $\Xi_{\Sc{C}}(\sigma)$ isomorphically to $\Xi_{\Sc{C}}(\gamma\sigma)$ and the stratum $\Xi_{\Sc{C},\Sigma_{\Sc{C}},\sigma}$ isomorphically to $\Xi_{\Sc{C},\Sigma_{\Sc{C}},\gamma\sigma}$.

If $\Sigma_{\Sc{C}}$ is smooth then $\Xi_{\Sc{C}}(\sigma)/C_{\Sc{C}}$ is smooth for each $\sigma\in\Sigma$.

Now let $\partial_{\Sc{C},\Sigma_{\Sc{C}}}\subset\Xi_{\Sc{C},\Sigma_{\Sc{C}}}$ be the reduced closed subscheme whose support is the union of the strata $\Xi_{\Sc{C},\Sigma_{\Sc{C}},\sigma}$ for $\sigma\in\Sigma_{\Sc{C}}$ with $\sigma\subset P^+_{\Sc{C}}$ (that this is closed follows from the fact that $P^+_{\Sc{C}}\subset P_{\Sc{C}}$ is open).  $\partial_{\Sc{C},\Sigma_{\Sc{C}}}$ is stable under the action of $\Gamma_{\Sc{C}}$.  We let $\Fr{X}_{\Sc{C},\Sigma_{\Sc{C}}}/C_{\Sc{C}}$ be the formal completion of $\Xi_{\Sc{C},\Sigma_{\Sc{C}}}$ along $\partial_{\Sc{C},\Sigma_{\Sc{C}}}$.

The formal scheme $\Fr{X}_{\Sc{C},\Sigma_{\Sc{C}}}/C_{\Sc{C}}$ has a cover defined by relatively affine formal schemes as follows: for each $\sigma\in\Sigma_{\Sc{C}}$ with $\sigma\subset P_{\Sc{C}}^+$ we let
\begin{equation*}
\sigma^\vee_\I=\sigma^\vee-\bigcup_{\tau\subset P^+_{\Sc{C}}\text{ a face of }\sigma}\tau^\perp
\end{equation*}
and then the sheaf of ideals defining the closed subscheme 
\begin{equation*}
\partial_{\Sc{C},\Sigma_{\Sc{C}}}\cap \Xi_{\Sc{C}}(\sigma)\subset \Xi_{\Sc{C}}(\sigma) 
\end{equation*}
is
\begin{equation*}
\I_{\partial}=\bigoplus_{l\in\sigma^\vee_\I}\Psi_{\Sc{C}}(l)\subset \bigoplus_{l\in\sigma^\vee}\Psi_{\Sc{C}}(l)
\end{equation*}
Then let
\begin{equation*}
\hat{\bigoplus_{l\in\sigma^\vee}}\Psi_{\Sc{C}}(l)
\end{equation*}
be the $\I_\partial$-adic completion of the direct sum.  This is a sheaf of adic $\O_{C_{\Sc{C}}}$-algebras (not quasi-coherent!) and by definition the formal completion $\Fr{X}_{\Sc{C}}(\sigma)$ of $\Xi_{\Sc{C}}(\sigma)$ along $\partial_{\Sc{C},\Sigma_{\Sc{C}}}$ is the relative formal spectrum
\begin{equation*}
\Fr{X}_{\Sc{C}}(\sigma)=\underline{\text{Spf}}_{\O_{C_{\Sc{C}}}}\hat{\bigoplus_{l\in\sigma^\vee}}\Psi_\Sc{C}(l).
\end{equation*}

The underlying reduced scheme of $\Fr{X}_{\Sc{C}}(\sigma)$ is the union of $\Xi_{\Sc{C}}(\sigma)_{\tau}$ for $\tau\subset P_{\Sc{C}}^+$ a face of $\sigma$.  If $\tau\subset P_{\Sc{C}}^+$ is a face of $\sigma$ then $\Fr{X}_{\Sc{C}}(\tau)\subset\Fr{X}_{\Sc{C}}(\sigma)$ is an open formal subscheme.

The $\Fr{X}_{\Sc{C}}(\sigma)$ for $\sigma\in \Sigma_{\Sc{C}}$ with $\sigma\subset P^+$ form a relatively affine cover of the formal scheme $\Fr{X}_{\Sc{C},\Sigma_{\Sc{C}}}$.  If $\sigma,\sigma'\in\Sigma_{\Sc{C}}$ with $\sigma,\sigma'\subset P_{\Sc{C}}^+$ then $\Fr{X}_{\Sc{C}}(\sigma)$ and $\Fr{X}_{\Sc{C}}(\sigma')$ intersect if and only if $\sigma$ and $\sigma'$ have a common face contained in $P_{\Sc{C}}^+$ (note that by contrast, $\Xi_{\Sc{C}}(\sigma)$ and $\Xi_{\Sc{C}}(\sigma')$ always intersect.)

The action of $\Gamma_{\Sc{C}}$ on $\Xi_{\Sc{C},\Sigma_{\Sc{C}}}$ which preserves $\partial_{\Sc{C},\Sigma_{\Sc{C}}}$, induces an action of $\Gamma_{\Sc{C}}$ on $\Fr{X}_{\Sc{C},\Sigma_{\Sc{C}}}$.  We now impose the following additional condition on the cone decomposition $\Sigma_{\Sc{C}}$, which is Condition 6.2.5.25 of \cite{La13}.

\begin{cond}\label{discond}
For each $\sigma\in\Sigma_{\Sc{C}}$ with $\sigma\subset P_{\Sc{C}}^+$, if we have $\gamma\in\Gamma_{\Sc{C}}$ with $\gamma\overline{\sigma}\cap\overline{\sigma}\not=\{0\}$ then $\gamma=1$.
\end{cond}

Assuming that $\Sigma_{\Sc{C}}$ satisfies Condition \ref{discond} we may form the quotient $\Fr{X}_{\Sc{C},\Sigma_{\Sc{C}}}/\Gamma_{\Sc{C}}$ as a formal scheme in such a way that the quotient map
\begin{equation*}
\Fr{X}_{\Sc{C},\Sigma_{\Sc{C}}}\to\Fr{X}_{\Sc{C},\Sigma_{\Sc{C}}}/\Gamma_{\Sc{C}}
\end{equation*}
is a local isomorphism of formal schemes in the Zariski topology.  Indeed, it follows from Condition \ref{discond} and the discussion preceding it that $\Fr{X}_{\Sc{C},\Sigma_{\Sc{C}}}$ has a cover by the Zariski opens $\Fr{X}_{\Sc{C}}(\sigma)$ for $\sigma\in\Sigma_{\Sc{C}}$ with $\sigma\subset P^+$, which have the property that they are disjoint from all of their translates by $\Gamma_{\Sc{C}}$.

Now recall that from point 10 at the beginning of this section that we have a semiabelian scheme $\tilde{A}_{\Sc{C}}/C_{\Sc{C}}$ with $\O$ action.  By abuse of notation, we will also denote by $\tilde{A}$ the following things:
\begin{enumerate}
\item $\tilde{A}_{\Sc{C}}/\Xi_{\Sc{C},\Sigma_{\Sc{C}}}$, the base change of $\tilde{A}_{\Sc{C}}$ to $\Xi_{\Sc{C},\Sigma_{\Sc{C}}}$, with an induced action of $\Gamma_{\Sc{C}}$ covering that on $\Xi_{\Sc{C},\Sigma_{\Sc{C}}}$.
\item $\tilde{A}_{\Sc{C}}/\Fr{X}_{\Sc{C},\Sigma_{\Sc{C}}}$ the formal completion of $\tilde{A}_{\Sc{C}}/\Xi_{\Sc{C},\Sigma_{\Sc{C}}}$ along the pre image of $\partial_{\Sc{C},\Sigma_{\Sc{C}}}$.
\item $\tilde{A}_{\Sc{C}}/(\Fr{X}_{\Sc{C},\Sigma_{\Sc{C}}}/\Gamma_{\Sc{C}})$ the quotient of $\tilde{A}_{\Sc{C}}$ by $\Gamma_{\Sc{C}}$.
\end{enumerate}

We now recall more definitions from \cite{La13}.

\begin{enumerate}
\item There is a partial order on $\text{Cusp}_K$ which we denote by $\leq$ (see definition  of \cite{La13}).  If $\Sc{C},\Sc{C}'\in\text{Cusp}_K$ with $\Sc{C}\leq\Sc{C}'$ and if $X$ and $X'$ denote the corresponding $\O$-lattices then there is an $\O$-equivariant surjection $X\to X'$ inducing inclusions $M_{\Sc{C'}}\subset M_{\Sc{C}}$ and $P_{\Sc{C'}}\subset P_{\Sc{C}}$.  For a given cusp label $\Sc{C}$ we have
\begin{equation*}
P_{\Sc{C}}=\coprod_{\Sc{C'}\leq\Sc{C}}\Gamma_{\Sc{C}}P^+_{\Sc{C'}}.
\end{equation*}
\item If $\Sc{C},\Sc{C}'\in\text{Cusp}_K$ with $\Sc{C}\leq\Sc{C}'$ and if $\Sigma_{\Sc{C}}$ (resp. $\Sigma_{\Sc{C}'}$) is a $\Gamma_{\Sc{C}}$-admissible (resp. $\Gamma_{\Sc{C}'}$-admissible) rational polyhedral cone decomposition then $\Sigma_{\Sc{C}}$ and $\Sigma_{\Sc{C}'}$ are said to be \emph{compatible} if for each $\sigma\in\Sigma_{\Sc{C}'}$, we also have $\sigma\in\Sigma_{\Sc{C}}$ via the inclusion $M_{\Sc{C}'}\subset M_{\Sc{C}}$.
\item A compatible family of cone decompositions at level $K$ is a collection $\Sigma=\{\Sigma_{\Sc{C}}\}_{\Sc{C}\in\text{Cusp}_K}$ of a $\Gamma_{\Sc{C}}$-admissible rational polyhedral cone decomposition for each cusp label $\Sc{C}$ such that if $\Sc{C}\leq\Sc{C'}$ then $\Sigma_{\Sc{C}}$ and $\Sigma_{\Sc{C}'}$ are compatible.
\item A compatible family $\Sigma=\{\Sigma_{\Sc{C}}\}$ of cone decompositions at level $K$ is said to be \emph{good} if every $\Sigma_{\Sc{C}}$ is smooth and satisfies condition \ref{discond}, and $\Sigma$ is projective in the sense of definition 7.3.1.3 of \cite{La13}.  We recall that good compatible families of cone decompositions at level $K$ exist (see Proposition 7.3.1.4 of \cite{La13}) and any two good compatible families of cone decompositions at level $K$ admit a good common refinement.
\end{enumerate}

\section{Arithmetic Compactifications}\label{S:compact}

\subsection{Toroidal Compactifications}
Here is the the main theorem, due to Lan, on the existence and basic properties of arithmetic toroidal compactifications (see \cite[6.4.1.1,6.4.3.4,7.3.3.4]{La13})

\begin{thm}[Lan]\label{T:toroidal}
Let $K\subset G(\hat{\Z}^{(p)})$ be a neat open compact subgroup and let $\Sigma$ be a good compatible family of cone decompositions at level $K$.  Then there is a smooth projective scheme $\Ca{X}_{K,\Sigma}^{\text{\rm tor}}/R$, along with a semiabelian scheme $A/\Ca{X}_{K,\Sigma}^{\text{\rm tor}}$ with an action $i:\O\to\End_{\Ca{X}_{K,\Sigma}^{\text{\rm tor}}}(A)$ of $\O$ with the following properties
\begin{enumerate}
\item
There is a dense open embedding $j^{\text{\rm tor}}_{K,\Sigma}:\Ca{X}_K\to\Ca{X}_{K,\Sigma}^{\text{\rm tor}}$ such that the pullback of $(A,\iota)$ to $\Ca{X}_K$ is canonically part of the universal object $(A,\lambda,\iota,\alpha_K)$ on $\Ca{X}_K$ (viewed as representing the functor $M_K^{\text{isom}}$).  The (reduced) boundary $D_{K,\Sigma}=\Ca{X}_{K,\Sigma}^{\text{\rm tor}}$ is a Cartier divisor with simple normal crossings and is flat over $R$.
\item There is a set theoretic decomposition
\begin{equation*}
\Ca{X}_{K,\Sigma}^{\text{\rm tor}}=\coprod_{\Sc{C}\in\text{\rm Cusp}_K}\Ca{X}^{\text{\rm tor}}_{K,\Sigma,\Sc{C}}
\end{equation*}
with each $\Ca{X}^{\text{\rm tor}}_{K,\Sigma,\Sc{C}}/R$ flat, reduced, and locally closed.  We do not call this decomposition a stratification because it is not true that the closure of one $\Ca{X}^{\text{\rm tor}}_{K,\Sigma,\Sc{C}}$ is a union of others.

By abuse of notation, let $\hat{\Ca{X}}^{\text{\rm tor}}_{K,\Sigma,\Sc{C}}$ denote the formal completion of $\Ca{X}_{K,\Sigma}^{\text{\rm tor}}$ along $\Ca{X}_{K,\Sigma,\Sc{C}}^{\text{\rm tor}}$ (by convention, by the formal completion of a scheme $X$ along a locally closed subscheme $Z$ we mean the formal completion of the open subscheme $X-(\overline{Z}-Z)$ along its closed subscheme $Z$.)  Then there is a canonical isomorphism of formal schemes
\begin{equation*}
\hat{\Ca{X}}^{\text{\rm tor}}_{K,\Sigma,\Sc{C}}\simeq \Fr{X}_{\Sc{C},\Sigma_{\Sc{C}}}/\Gamma_{\Sc{C}}
\end{equation*}
over which there is a canonical isomorphism of the formal completion of $A$ with $\tilde{A}_\Sc{C}$ compatible with the action of $\O$.
\item There is a stratification
\begin{equation*}
\Ca{X}_{K,\Sigma}^{\text{\rm tor}}=\coprod_{(\Sc{C},[\sigma])}\Ca{X}_{K,\Sigma,(\Sc{C},[\sigma])}^{\text{\rm tor}}
\end{equation*}
the indexing set being the set of all pairs of a cusp label $\Sc{C}\in\text{\rm Cusp}_K$ and a $\Gamma_{\Sc{C}}$ orbit $[\sigma]=\Gamma\cdot\sigma$ with $\sigma\in\Sigma_{\Sc{C}}$ and $\sigma\subset P^+_{\Sc{C}}$.  If $(\Sc{C},[\sigma])$ and $(\Sc{C}',[\sigma'])$ are two such pairs, then $\Ca{X}_{K,\Sigma,(\Sc{C},[\sigma])}^{\text{\rm tor}}$ lies in the closure of $\Ca{X}_{K,\Sigma,(\Sc{C}',[\sigma'])}^{\text{\rm tor}}$ if and only if $\Sc{C}\leq\Sc{C'}$ and there are representatives $\sigma$ of $[\sigma]$ and $\sigma'$ of $[\sigma']$ so that via the inclusion $M_{\Sc{C}'}\subset M_{\Sc{C}}$, $\sigma'$ is a face of $\sigma$.
\item If $K,K'\subset G(\hat{\Z}^{(p)})$ are neat open compact subgroups and $g\in G(\A^{\infty,p})$ is such that $g^{-1}Kg\subset K'$ and $\Sigma$ (resp. $\Sigma'$) is a good compatible family of cone decompositions at level $K$ (resp. $K'$) and $\Sigma$ is a $g$-refinement of $\Sigma'$ then there is a morphism
\begin{equation*}
[g]:\Ca{X}_{K,\Sigma}^{\text{\rm tor}}\to \Ca{X}_{K',\Sigma'}^{\text{\rm tor}}
\end{equation*}
extending the morphism $[g]:\Ca{X}_K\to\Ca{X}_{K'}$ of section \ref{S:heckepel}.
\end{enumerate}
\end{thm}

\subsection{Minimal Compactifications}
Here is the main theorem, due to Lan, on the existence and basic properties of arithmetic minimal compactifications (see \cite[7.2.4.1,7.2.4.3]{La13}.)
\begin{thm}[Lan]\label{T:minimal}
For each neat open compact subgroup $K\subset G(\hat{\Z}^{(p)})$ there is a flat, projective, normal scheme $\Ca{X}_K^{\text{\rm min}}/R$ with the following properties.
\begin{enumerate}
\item There is a dense open embedding $j_K^{\text{\rm min}}:\Ca{X}_K\to\Ca{X}_K^{\text{\rm min}}$
\item For each cusp label $\Sc{C}\in\text{\rm Cusp}_K$ there is a canonical locally closed immersion $\Ca{X}_{\Sc{C}}\hookrightarrow\Ca{X}$.  We will identify $\Ca{X}_{\Sc{C}}$ with the induced locally closed subscheme of $\Ca{X}_K^{\text{\rm min}}$.  These subschemes define a stratification
\begin{equation*}
\Ca{X}_K^{\text{\rm min}}=\coprod_{\Sc{C}\in\text{\rm Cusp}_K}\Ca{X}_{\Sc{C}}
\end{equation*}
such that for $\Sc{C},\Sc{C}'\in\text{\rm Cusp}_K$, $\Ca{X}_{\Sc{C}}$ lies in the closure of $\Ca{X}_{\Sc{C}'}$ if and only if $\Sc{C}\leq\Sc{C}'$.

\item For each choice $\Sigma$ of a good compatible family of cone decompositions at level $K$ there is a map
\begin{equation*}
\pi_{K,\Sigma}:\Ca{X}_{K,\Sigma}^{\text{\rm tor}}\to\Ca{X}_K^{\text{\rm min}}
\end{equation*}
such that $j_K^{\text{\rm min}}=\pi_{K,\Sigma}\circ j_{K,\Sigma}^{\text{\rm tor}}$.  Moreover for each cusp label $\Sc{C}\in\text{\rm Cusp}_K$ we have
\begin{equation*}
\pi_{K,\Sigma}^{-1}(\Ca{X}_{\Sc{C}})=\Ca{X}_{K,\Sigma,\Sc{C}}^{\text{\rm tor}}
\end{equation*}
set theoretically.  Moreover we have
\begin{equation*}
\pi_{K,\Sigma,*}\O_{\Ca{X}_{K,\Sigma}^{\text{\rm tor}}}=\O_{\Ca{X}_K^{\text{\rm min}}}.
\end{equation*}
\item For each cusp label $\Sc{C}\in\text{\rm Cusp}_K$ let $\hat{\Ca{X}}_{K,\Sc{C}}^{\text{\rm min}}$ denote the formal completion of $\Ca{X}_K^{\text{\rm min}}$ along $\Ca{X}_{\Sc{C}}$.  Then there is a canonical \emph{structural morphism}
\begin{equation*}
\hat{\Ca{X}}_{K,\Sc{C}}^{\text{\rm min}}\to \Ca{X}_{\Sc{C}}.
\end{equation*}
For each choice $\Sigma$ of a good compatible family of cone decompositions at level $K$ there is a commutative diagram of morphisms of formal schemes
\begin{equation*}
\begin{tikzcd}
\hat{\Ca{X}}_{K,\Sigma,\Sc{C}}^{\text{\rm tor}}\arrow{r}\arrow{d}{\hat{\pi}_{K,\Sigma}}&\Fr{X}_{\Sc{C},\Sigma_{\Sc{C}}}/\Gamma_{\Sc{C}}\arrow{d}\\
\hat{\Ca{X}}_{K,\Sc{C}}^{\text{\rm min}}\arrow{r}& \Ca{X}_{\Sc{C}}
\end{tikzcd}
\end{equation*}
where the top horizontal arrow is the isomorphism of Theorem \ref{T:toroidal} part 2, the left vertical arrow comes from formally completing $\pi_{K,\Sigma}$.
\item If $K,K'\subset G(\hat{\Z}^{(p)})$ are neat open compact subgroups and $g\in G(\A^{\infty,p})$ is such that $g^{-1}Kg\subset K'$ then there is a finite surjective morphism
\begin{equation*}
[g]:\Ca{X}_K^{\text{\rm min}}\to\Ca{X}_{K'}^{\text{\rm min}}
\end{equation*}
extending the morphism $[g]:\Ca{X}_K\to\Ca{X}_{K'}$ of Section \ref{S:heckepel}.  If $\Sigma$ (resp. $\Sigma'$) is a good compatible family of cone decompositions at level $K$ (resp. $K'$) and $\Sigma$ is a $g$-refinement of $\Sigma'$ then there is a commutative diagram
\begin{equation*}
\begin{tikzcd}
\Ca{X}_{K,\Sigma}^{\text{\rm tor}}\arrow{r}{[g]}\arrow{d}{\pi_{K,\Sigma}}&\Ca{X}_{K',\Sigma'}^{\text{\rm tor}}\arrow{d}{\pi_{K',\Sigma'}}\\
\Ca{X}_K^{\text{\rm min}}\arrow{r}{[g]}&\Ca{X}_{K'}^{\text{\rm min}}.
\end{tikzcd}
\end{equation*}
\end{enumerate}
\end{thm}

\section{Extensions of Automorphic Vector Bundles}\label{S:canext}

In this section we explain how to extend the automorphic vector bundles on $\Ca{X}_K$ defined in section \ref{S:pelvb} to the compactifications of the last section.

\subsection{Canonical and Subcanonical Extensions}

Let $K\subset G(\hat{\Z}^{(p)})$ be a neat open compact subgroup and let $\Sigma$ be a good compatible family of cone decompositions at level $K$.  For each universal object $(A,\lambda,i,\alpha_K)$ in the universal isogeny class over $\Ca{X}_K$, by part 5 of Theorem \ref{T:toroidal} there is a canonical extension to a semiabelian scheme $A/\Ca{X}_{K,\Sigma}^{\text{tor}}$ with a ring homomorphism $i:\O\otimes\Z_{(p)}\to\End_{\Ca{X}_{K,\Sigma}^{\text{tor}}}(A)\otimes\Z_{(p)}$.  As in section \ref{S:pelvb} we associate to this extension $(A,i)$ a pair $(\omega_A,\O_{\Ca{X}_{K,\Sigma}^{\text{tor}}})$ of a vector bundle with an $\O\otimes\Z_{(p)}$ action and a line bundle on $\Ca{X}_{K,\Sigma}^{\text{tor}}$.  Moreover if $(A',\lambda',i',\alpha'_K)$ is another member of the same isogeny class, there is a unique prime to $p$ quasi-isogeny $f:(A,\lambda,i,\alpha_K)\to(A',\lambda',i',\alpha_K')$ as in definition \ref{D:modisog}, which by part 5 of Theorem \ref{T:toroidal} extends to a prime to $p$ quasi-isogeny
\begin{equation*}
f:(A,i)\to(A',i')
\end{equation*}
of semiabelian schemes with $\O$-action, and hence defines an isomorphism of pairs
\begin{equation*}
(\omega_A,\O_{\Ca{X}_{K,\Sigma}^{\text{tor}}})\simeq(\omega_{A'},\O_{\Ca{X}_{K,\Sigma}^{\text{tor}}})
\end{equation*}
which is $(f^*)^{-1}$ on the first factor, and multiplication by $r$ on the second factor, where $r$ is the locally constant $\Z_{(p)}^{\times,>0}$ valued function on $S$ such that $\lambda=rf^\vee\lambda'f$.

Hence we obtain a canonical pair
\begin{equation*}
(\Ca{E}^{\text{can}}_K,\Xi_K^{\text{can}})
\end{equation*}
of a vector bundle with an $\O\otimes\Z_{(p)}$ action and a line bundle on $\Ca{X}_{K,\Sigma}^{\text{tor}}$ whose restriction to $\Ca{X}_K$ is $(\Ca{E}_K,\Xi_K)$.

\begin{defn}
\begin{enumerate}
\item The canonical extension of the principal $M$-bundle $P_K$ is the principal $M$-bundle $P_K^{\text{can}}$ on $\Ca{X}_{K,\Sigma}^{\text{tor}}$ defined by
\begin{equation*}
P_K^{\text{can}}(S)=\text{Isom}_{\O\otimes\O_S}((\Ca{E}_K^{\text{can}}\otimes\O_S,\Xi_K^{\text{can}}),(L_0\otimes\O_S,R\otimes\O_S))
\end{equation*}
for each $\Ca{X}_{K,\Sigma}^{\text{tor}}$-scheme $S$.
\item If $\rho$ is an algebraic representation of $M$ on a finite $R$-module $W$ we define the coherent sheaf
\begin{equation*}
V_{\rho,K,\Sigma}^{\text{can}}=P_K^{\text{can}}\times^{M}W,
\end{equation*}
the \emph{canonical extension} of $V_{\rho,K}$ to $\Ca{X}_{K,\Sigma}^{\text{tor}}$.  We also define the subcanonical extension
\begin{equation*}
V_{\rho,K,\Sigma}^{\text{sub}}=\I_{D_{K,\Sigma}}V_{\rho,K,\Sigma}^{\text{can}}=V_{\rho,K,\Sigma}^{\text{can}}\otimes\I_{D_{K,\Sigma}}
\end{equation*}
where $\I_{D_{K,\Sigma}}$ is the ideal sheaf of the (reduced) boundary $D_{K,\Sigma}=\Ca{X}_{K,\Sigma}^{\text{tor}}-\Ca{X}_K$.  We remark that $\I_D$ is a line bundle because $D$ is a cartier divisor.
\end{enumerate}
\end{defn}

By abuse of notation, we will denote $\det\Ca{E}_K^{\text{can}}=V_{\det L_0^\vee,K,\Sigma}^{\text{can}}$, the canonical extension of the determinant of the Hodge bundle $\omega_K$, also by $\omega_K$.

We have the following analog of Proposition \ref{P:autprop}

\begin{prop}\label{P:autpropext}
\begin{enumerate}
\item If $\rho$ is an algebraic representation on a finite free $R$-module (resp. a finite free $R/\pi^r$-module) then $V_{\rho,K,\Sigma}^{\text{\rm can}}$ and $V_{\rho,K,\Sigma}^{\text{\rm sub}}$ are locally free sheaves on $\Ca{X}_{K,\Sigma}^{\text{\rm tor}}$ (resp. are locally free sheaves on $\Ca{X}_{K,\Sigma}^{\text{\rm tor}}\times R/\pi^r$.)
\item If $0\to \rho'\to\rho\to\rho''\to 0$ is a short exact sequence of algebraic representations of $M$ then we have short exact sequences
\begin{equation*}
0\to V_{\rho',K,\Sigma}^{\text{\rm can}}\to V_{\rho,K,\Sigma}^{\text{\rm can}}\to V_{\rho'',K,\Sigma}^{\text{\rm can}}\to 0
\end{equation*}
and
\begin{equation*}
0\to V_{\rho',K,\Sigma}^{\text{\rm sub}}\to V_{\rho,K,\Sigma}^{\text{\rm sub}}\to V_{\rho'',K,\Sigma}^{\text{\rm sub}}\to 0
\end{equation*}
of sheaves on $\Ca{X}_{K,\Sigma}^{\text{\rm tor}}$.
\item If $\rho$ is an algebraic representation of $M$ on  a finite $R$-module and $k$ is any integer then
\begin{equation*}
V_{\rho\otimes\det^k L_0^\vee,K,\Sigma}^{\text{\rm can}}=V_{\rho,K,\Sigma}^{\text{\rm can}}\otimes\omega_K^{\otimes k}
\end{equation*}
and
\begin{equation*}
V_{\rho\otimes\det^k L_0^\vee,K,\Sigma}^{\text{\rm sub}}=V_{\rho,K,\Sigma}^{\text{\rm sub}}\otimes\omega_K^{\otimes k}.
\end{equation*}
\end{enumerate}
\end{prop}

Next we consider Hecke actions.

\begin{prop}\label{P:exthecke}
Let $K,K'\subset G(\hat{\Z}^{(p)})$ be neat open compact subgroups and $g\in G(\A^{\infty,p})$ be such that $g^{-1}Kg\subset K'$ and let $\Sigma$ (resp. $\Sigma'$) be a good compatible family of cone decompositions at level $K$ (resp. $K'$) such that $\Sigma$ is a $g$-refinement of $\Sigma'$ so that there is a map
\begin{equation*}
[g]:\Ca{X}_{K,\Sigma}^{\text{\rm tor}}\to\Ca{X}_{K',\Sigma'}^{\text{\rm tor}}
\end{equation*}
as in point 4 in theorem \ref{T:toroidal}.  Then the isomorphism
\begin{equation*}
g:[g]^*V_{\rho,K'}\to V_{\rho,K}.
\end{equation*}
of Proposition \ref{P:hecke} extends to an isomorphism
\begin{equation*}
g:[g]^*V_{\rho,K',\Sigma'}^{\text{\rm can}}\to V_{\rho,K,\Sigma}^{\text{\rm can}}
\end{equation*}
and a morphism
\begin{equation*}
g:[g]^*V_{\rho,K',\Sigma'}^{\text{\rm sub}}\to V_{\rho,K,\Sigma}^{\text{\rm sub}}.
\end{equation*}
\end{prop}

\subsection{Higher Direct Images}

In this section we record two results on higher direct images of automorphic vector bundles.

For the first results, let $K,K'\subset G(\hat{\Z}^{(p)})$ be a neat open compact subgroups with $K\subset K'$ and let $\Sigma$ and $\Sigma'$ be good compatible families of cone decompositions at levels $K$ and $K'$ such that $\Sigma$ refines $\Sigma'$ so that we have a map
\begin{equation*}
[1]:\Ca{X}_{K,\Sigma}^{\text{tor}}\to\Ca{X}_{K',\Sigma'}^{\text{tor}}
\end{equation*} 
as in part 4 of Theorem \ref{T:toroidal}.

For the proof of the following proposition, see the proof of Lemma 7.1.1.4 of \cite{La13}.

\begin{prop}
With notation as above, for all $i>0$ we have
\begin{equation*}
R^i[1]_*\O_{\Ca{X}_{K,\Sigma}^{\text{\rm tor}}}=R^i[1]_*\I_{D_{K,\Sigma}}=0.
\end{equation*}
Moreover if $K=K'$ then we have
\begin{equation*}
[1]_*\O_{\Ca{X}_{K,\Sigma}^{\text{\rm tor}}}=\O_{\Ca{X}_{K,\Sigma'}^{\text{\rm tor}}}
\end{equation*}
and
\begin{equation*}
[1]_*\I_{D_{K,\Sigma}}=\I_{D_{K,\Sigma'}}.
\end{equation*}
\end{prop}

Combining this with the projection formula and Proposition \ref{P:exthecke} we obtain
\begin{cor}\label{C:autpush}
With notation as above, let $\rho$ be an algebraic representation on a finite $R$-module.  Then for all $i>0$ we have
\begin{equation*}
R^i[1]_*V_{\rho,K,\Sigma}^{\text{\rm can}}=R^i[1]_*V_{\rho,K,\Sigma}^{\text{\rm sub}}=0.
\end{equation*}
Moreover if $K=K'$ then we have
\begin{equation*}
[1]_*V_{\rho,K,\Sigma}^{\text{\rm can}}=V_{\rho,K,\Sigma'}^{\text{\rm can}}
\end{equation*}
and
\begin{equation*}
[1]_*V_{\rho,K,\Sigma}^{\text{\rm sub}}=V_{\rho,K,\Sigma'}^{\text{\rm sub}}.
\end{equation*}
\end{cor}

The next theorem is deeper and concerns higher direct images of subcanonical extensions from toroidal to minimal compactifications.  In this generality it is Theorem 8.2.1.2 of \cite{La12}.  Special cases of it were discovered independently by Harris, Lan, Taylor, and Thorne \cite{HLTT12} and by Andreatta, Iovita, and Pilloni \cite{AIP15}.  See also the work of Lan and Stroh \cite{LS14} for a more conceptual approach under more restrictive hypotheses.

\begin{thm}\label{T:HLTT}
Let $K\subset G(\hat{\Z}^{(p)})$ be a neat open compact subgroup, and let $\Sigma$ be a good compatible family of cone decompositions at level $K$.  Let $\pi_{K,\Sigma}:\Ca{X}_{K,\Sigma}^{\text{\rm tor}}\to\Ca{X}_{K}^{\text{\rm min}}$ be the map as in part 3 of Theorem \ref{T:minimal}.  Let $\rho$ be an algebraic representation of $M$ on a finite $R$-module.
\begin{equation*}
R^i\pi_{K,\Sigma,*} V_{\rho,K,\Sigma}^{\text{\rm sub}}=0
\end{equation*}
for all $i>0$.
\end{thm}

We remark that this theorem is not true with $V_{\rho,K,\Sigma}^{\text{sub}}$ replaced by $V_{\rho,K,\Sigma}^{\text{can}}$.

\subsection{Pushforwards to the Minimal Compactification}

In this section we consider certain extensions of automorphic vector bundles to minimal compactifications.

\begin{defn}\label{D:minext}
If $K\subset G(\hat{\Z}^{(p)})$ is a neat open compact subgroup and $\rho$ is an algebraic representation on a finite $R$-module, we define a coherent sheaf on $\Ca{X}_K^{\text{min}}$
\begin{equation*}
V_{\rho,K}^{\text{sub}}=\pi_{K,\Sigma,_*}V_{\rho,K,\Sigma}^{\text{sub}}
\end{equation*}
where $\Sigma$ is a choice of a good compatible family of cone decompositions at level $K$.  We claim that this doesn't depend on $\Sigma$.  Indeed if $\Sigma'$ is another choice of a good compatible family of cone decompositions at level $K$, and $\Sigma$ refines $\Sigma'$ then we have
\begin{equation*}
\pi_{K,\Sigma,*}V_{\rho,K,\Sigma}^{\text{sub}}={\pi_{K,\Sigma'}}_*[1]_*V_{\rho,K,\Sigma}^{\text{sub}}={\pi_{K,\Sigma'}}_*V_{\rho,K,\Sigma'}^{\text{sub}}
\end{equation*}
by part 5 of Theorem \ref{T:minimal} and Corollary \ref{C:autpush}.  The general case reduces to this upon using the fact that any two $\Sigma$ and $\Sigma'$ admit a common refinement.
\end{defn}

In general we won't consider push forwards of canonically extended (rather than subcanonically extended) automorphic sheaves to the minimal compactification.  As an exception, we have the following important proposition which is essentially part of the construction of the minimal compactification (see \cite[7.2.4.1]{La13}).

\begin{prop}\label{P:omegaamp}
The push forward $\pi_{K,\Sigma,*}\omega_K$ is a line bundle on $\Ca{X}_K^{\text{min}}$ which is ample.  
\end{prop}

As a further abuse of notation, we will also denote the line bundle ${\pi_{\Sigma,K}}_*\omega_K$ on $\Ca{X}_K$ by $\omega_K$.

We have the following analog of Propositions \ref{P:autprop} and \ref{P:autpropext}.

\begin{prop}\label{P:autpropmin}
\begin{enumerate}
\item If $0\to \rho'\to\rho\to\rho''\to 0$ is a short exact sequence of algebraic representations of $M$ then we have a short exact sequence
\begin{equation*}
0\to V_{\rho',K,\Sigma}^{\text{\rm sub}}\to V_{\rho,K,\Sigma}^{\text{\rm sub}}\to V_{\rho'',K,\Sigma}^{\text{\rm sub}}\to 0
\end{equation*}
of sheaves on $\Ca{X}_{K}^{\text{\rm min}}$.
\item If $\rho$ is an algebraic representation of $M$ on  a finite $R$-module and $k$ is any integer then
\begin{equation*}
V_{\rho\otimes\det^k L_0^\vee,K}^{\text{\rm sub}}=V_{\rho,K}^{\text{\rm sub}}\otimes\omega_K^{\otimes k}.
\end{equation*}
\end{enumerate}
\end{prop}
\begin{proof}
Part 1 follows from part 2 of Proposition \ref{P:autprop} and Theorem \ref{T:HLTT}.  Part 2 follows from part 3 of \ref{P:autprop}, Proposition \ref{P:omegaamp}, and the projection formula.
\end{proof}

We note however that the analog of part 1 of Propositions \ref{P:autprop} and \ref{P:autpropext} is no longer true.  Even if $\rho$ is an algebraic representation on a finite free $R$-module, $V_{\rho,K}^{\text{sub}}$ needn't be locally free.  

Next we consider Hecke actions.
\begin{prop}\label{P:minhecke}
Let $K,K'\subset G(\hat{\Z}^{(p)})$ be neat open compact subgroups and $g\in G(\A^{\infty,p})$ be such that $g^{-1}Kg\subset K'$ so that there is a map
\begin{equation*}
[g]:\Ca{X}_{K}^{\text{\rm min}}\to\Ca{X}_{K''}^{\text{\rm min}}
\end{equation*}
as in point 4 in theorem \ref{T:toroidal}.  Then the isomorphism
\begin{equation*}
g:[g]^*V_{\rho,K'}\to V_{\rho,K}.
\end{equation*}
of Proposition \ref{P:hecke} extends to a morphism
\begin{equation*}
g:[g]^*V_{\rho,K'}^{\text{\rm sub}}\to V_{\rho,K}^{\text{\rm sub}}.
\end{equation*}
\end{prop}

\begin{proof}
Let $\Sigma$ (resp. $\Sigma'$) be a good compatible family of cone decompositions at level $K$ (resp. $K'$) such that $\Sigma$ is a $g$-refinement of $\Sigma'$ so that there is a commutative diagram
\begin{equation*}
\begin{tikzcd}
\Ca{X}_{K,\Sigma}^{\text{\rm tor}}\arrow{r}{[g]}\arrow{d}{\pi_{K,\Sigma}}&\Ca{X}_{K',\Sigma'}^{\text{\rm tor}}\arrow{d}{\pi_{K',\Sigma'}}\\
\Ca{X}_K^{\text{\rm min}}\arrow{r}{[g]}&\Ca{X}_{K'}^{\text{\rm min}}.
\end{tikzcd}
\end{equation*}
as in part 5 of Theorem \ref{T:minimal}.  Then by Proposition \ref{P:exthecke} there is a canonical morphism
\begin{equation*}
g:[g]^*V_{\rho,K',\Sigma'}^{\text{sub}}\to V_{\rho,K,\Sigma}^{\text{sub}}.
\end{equation*}
Push it forward by $\pi_{K,\Sigma}$ and consider the composition
\begin{equation*}
[g]^*V_{\rho,K'}^{\text{sub}}=[g]^*{\pi_{K',\Sigma'}}_*V_{\rho,K',\Sigma'}^{\text{sub}}\to\pi_{K,\Sigma,*}[g]^*V_{\rho,K',\Sigma'}^{\text{sub}}\to\pi_{K,\Sigma,*}V_{\rho,K,\Sigma}^{\text{sub}}=V_{\rho,K}^{\text{sub}}.
\end{equation*}
One can verify that this is independent of the choice of $\Sigma$ and $\Sigma'$ with a similar argument as that in Definition \ref{D:minext}.
\end{proof}

\subsection{Hecke Action on Coherent Cohomology}\label{S:heckeaction}

The goal of this section is to define actions of Hecke algebras on the coherent cohomology of the extensions of automorphic vector bundles on compactifications considered in the previous sections.  First we will define the Hecke algebras we will consider.

\begin{defn}
Let $K\subset G(\hat{\Z}^{(p)})$.
\begin{enumerate}
\item The (universal, prime to $p$) Hecke algebra $\Bf{T}_K$ of level $K$ is the $R$-algebra of compactly supported $R$-valued bi-$K$-invariant functions on $G(\A^{\infty,p})$ with multiplication being convolution (with the Haar measure on $G(\A^{\infty,p})$ normalized so that $K$ has measure 1.)
\item Let $S$ be a finite set of places of $\Q$ including $p$, $\infty$, and all other primes $l$ for which $G(\Z_l)$ is not a hyperspecial maximal compact subgroup of $G(\Q_l)$ and let $K^S=\prod_{l\not\in S}G(\Z_l)\subset G(\A^S)$, an open compact subgroup.  We let $\Bf{T}^S$ be the $R$-algebra of compactly supported bi-$K^S$-invariant functions on $G(\A^S)$ with multiplication being convolution (with the Haar measure on $G(\A^S)$ normalized so that $K^S$ has measure 1.)  If $K^S\subset K$ then there is a homomorphism of $R$-algebras $\Bf{T}^S\to\Bf{T}_K$ defined by sending the characteristic function of $K_SgK_S$ to the characteristic function of $KgK$.
\end{enumerate}
\end{defn}

Next we consider some generalities on trace maps.

\begin{lem}\label{L:tracelem}
Let $\pi:X\to Y$ be a generically finite, separable, and proper map of reduced noetherian schemes with $Y$ normal. Then there is a trace map
\begin{equation*}
\tr_{\pi}:\pi_*\O_X\to \O_Y
\end{equation*}
which is characterized by the fact that for each generic point $\eta$ of $Y$,
\begin{equation*}
(\pi_*\O_X)_{\eta}\to\O_{Y,\eta}
\end{equation*}
is the trace map from the finite separable $\O_{Y,\eta}$-algebra $(\pi_*\O_X)_{\eta}$ to $\O_{Y,\eta}$.

If $Z\subset Y$ is a reduced closed subscheme with ideal sheaf $\I_Z$ and $Z'$ is  is the set theoretic pre image $\pi^{-1}(Z)$ with the reduced induced subscheme structure with ideal sheaf $\I_{Z'}$, then $\tr_\pi$ maps the sub sheaf $\pi_*\I_{Z'}$ of $\pi_*\O_X$ into $\I_Z$, i.e. there is a map of sheaves
\begin{equation*}
\tr_\pi:\pi_*\I_{Z'}\to\I_Z.
\end{equation*}
\end{lem}
\begin{proof}
By considering the Stein factorization of $\pi$
\begin{equation*}
X\to \underline{\spec}(\pi_*\O_X)\to Y
\end{equation*}
we may reduce to considering the case where $\pi$ is finite.  We may also assume that $Y$ is affine and irreducible.

So now we are reduced to the following problem: we have a normal domain $A$ with fraction field $K$ and a finite $A$-algebra $B$ such that $B\otimes_A K$ is a finite separable $K$ algebra.  What we want to show is that if $b\in B$ then $\tr_{B\otimes_AK/K}(b\otimes 1)$ actually lies in $A$.  But now as $A$ is normal, it suffices to show that for every height 1 prime $\Fr{p}$ of $A$, $\tr_{B\otimes_AK/K}(b\otimes 1)$ lies in $A_\p$.  Then the image $B'$ of $B\otimes_K A_\p$ in $B\otimes_A K$ is a finite and torsion free as an $A_\p$-module, and hence finite free as $A_\p$ is a DVR.  Hence
\begin{equation*}
\tr_{B\otimes_AK/K}(b\otimes 1)=\tr_{B'/A_\p}(b\otimes 1)\in A_\p.
\end{equation*}
\end{proof}

\begin{defn}\label{D:tracedef}
Let $K\subset K'\subset G(\hat{\Z}^{(p)})$ be neat open compact subgroups and let $\Sigma$ (resp. $\Sigma'$) be a good compatible family of cone decompositions at level $K$ (resp. $K'$) such that $\Sigma$ is a $[1]$ refinement of $\Sigma'$, so that there is a map
\begin{equation*}
[1]:\Ca{X}_{K,\Sigma}^{\text{tor}}\to\Ca{X}_{K',\Sigma'}^{\text{tor}}.
\end{equation*}
Let $\rho$ be an algebraic representation of $M$ on either a finite free $R$ module or a finite free $R/\pi^r$-module for some $r$. 
\begin{enumerate}
\item
As in Lemma \ref{L:tracelem} we have a trace map
\begin{equation*}
\tr:[1]_*\O_{\Ca{X}_{K,\Sigma}^{\text{tor}}}\to\O_{\Ca{X}_{K',\Sigma'}^{\text{tor}}}
\end{equation*}
Tensoring with $V_{\rho,K',\Sigma'}^{\text{can}}$ we obtain
\begin{equation*}
([1]_*\O_{\Ca{X}_{K,\Sigma}^{\text{tor}}})\otimes V_{\rho,K',\Sigma'}^{\text{can}}\to V_{\rho,K',\Sigma'}^{\text{can}}
\end{equation*}
and we also have isomorphisms
\begin{equation*}
([1]_*\O_{\Ca{X}_{K,\Sigma}^{\text{tor}}})\otimes V_{\rho,K',\Sigma'}^{\text{can}}\simeq [1]_*([1]^*V_{\rho,K',\Sigma'}^{\text{can}})\simeq[1]_*V_{\rho,K,\Sigma}^{\text{can}}
\end{equation*}
by the projection formula and the isomorphism of Proposition \ref{P:exthecke}.  Composing we obtain a trace map
\begin{equation*}
\tr:[1]_*V_{\rho,K,\Sigma}^{\text{can}}\to V_{\rho,K',\Sigma'}^{\text{can}}.
\end{equation*}
Similarly from tensoring the trace map 
\begin{equation*}
\tr:[1]_*\I_{D_{K,\Sigma}}\to\I_{D_{K',\Sigma'}}.
\end{equation*}
with $V_{\rho,K',\Sigma'}^{\text{can}}$ we obtain a trace map
\begin{equation*}
\tr:[1]_*V_{\rho,K,\Sigma}^{\text{sub}}\to V_{\rho,K',\Sigma'}^{\text{sub}}.
\end{equation*}
\item
Now consider the diagram
\begin{equation*}
\begin{tikzcd}
\Ca{X}_{K,\Sigma}^{\text{\rm tor}}\arrow{r}{[1]}\arrow{d}{\pi_{K,\Sigma}}&\Ca{X}_{K',\Sigma'}^{\text{\rm tor}}\arrow{d}{\pi_{K',\Sigma'}}\\
\Ca{X}_K^{\text{\rm min}}\arrow{r}{[1]}&\Ca{X}_{K'}^{\text{\rm min}}.
\end{tikzcd}
\end{equation*}
as in part 5 of Theorem \ref{T:minimal}.  Applying ${\pi_{K',\Sigma'}}_*$ to the trace map
\begin{equation*}
\tr:[1]_*V_{\rho,K,\Sigma}^{\text{sub}}\to V_{\rho,K',\Sigma'}^{\text{sub}}.
\end{equation*}
we obtain a map
\begin{equation*}
[1]_*V_{\rho,K}^{\text{sub}}={\pi_{K',\Sigma'}}_*[1]_*V_{\rho,K,\Sigma}^{\text{sub}}\to{\pi_{K',\Sigma'}}_*V_{\rho,K',\Sigma'}^{\text{sub}}=V_{\rho,K'}^{\text{sub}}
\end{equation*}
which we also denote by $\tr$.  One may show that this is independent of the original choice of $\Sigma$ and $\Sigma'$ by choosing common refinements as in Definition \ref{D:minext}.
\end{enumerate}
\end{defn}

Now we consider coherent cohomology.  For the rest of the section let $K\subset G(\hat{\Z}^{(p)})$ be a neat open compact subgroup and let $\rho$ be an algebraic representation of $M$ on a finite free $R$-module or a finite free $R/\pi^r$-module for some $r$.  Let $\Sigma$ and $\Sigma'$ be good compatible families of cone decompositions at level $K$ such that $\Sigma$ is a refinement of $\Sigma'$.  By considering the Leray spectral sequence for the map
\begin{equation*}
[1]:\Ca{X}_{K,\Sigma}^{\text{tor}}\to \Ca{X}_{K,\Sigma'}^{\text{tor}}
\end{equation*}
and applying Corollary \ref{C:autpush} we conclude that for each $i$ the pullback maps
\begin{equation*}
H^i(\Ca{X}_{K',\Sigma'}^{\text{tor}},V_{\rho,K',\Sigma'}^{\text{can}})\to H^i(\Ca{X}_{K,\Sigma}^{\text{tor}},V_{\rho,K,\Sigma}^{\text{can}})\quad\text{and}\quad H^i(\Ca{X}_{K',\Sigma'}^{\text{tor}},V_{\rho,K',\Sigma'}^{\text{sub}})\to H^i(\Ca{X}_{K,\Sigma}^{\text{tor}},V_{\rho,K,\Sigma}^{\text{sub}})
\end{equation*}
are isomorphisms.  Consequently to make something canonical we may define
\begin{equation*}
H^i(\Ca{X}_K^{\text{tor}},V_{\rho,K}^{\text{can}}):=\varinjlim_{\Sigma}H^i(\Ca{X}_{K,\Sigma}^{\text{tor}},V_{\rho,K,\Sigma}^{\text{can}})
\end{equation*}
and
\begin{equation*}
H^i(\Ca{X}_K^{\text{tor}},V_{\rho,K}^{\text{sub}}):=\varinjlim_{\Sigma}H^i(\Ca{X}_{K,\Sigma}^{\text{tor}},V_{\rho,K,\Sigma}^{\text{sub}})
\end{equation*}
where the limit is taken over the directed system of all good compatible families of cone decompositions at level $K$ under refinement.

By considering the Leray spectral sequence for the maps
\begin{equation*}
\pi_{K,\Sigma}:\Ca{X}_{K,\Sigma}^{\text{tor}}\to\Ca{X}_{K,\Sigma'}^{\text{min}}
\end{equation*}
and applying Theorem \ref{T:HLTT} we conclude that the pullback maps
\begin{equation*}
H^i(\Ca{X}_{K}^{\text{min}},V_{\rho,K}^{\text{sub}})\to H^i(\Ca{X}_K^{\text{tor}},V_{\rho,K}^{\text{sub}})
\end{equation*}
are isomorphisms.

Let $g\in G(\A^{\infty,p})$.  Let $\Sigma$ be a good compatible family of cone decompositions at level $K$ and let $\Sigma'$ be a good compatible family of cone decompositions of level $gKg^{-1}\cap K$ which is both a 1-refinement and a $g$-refinement of $\Sigma$, so that we have a Hecke correspondence
\begin{equation*}
\begin{tikzcd}
\Ca{X}_{K,\Sigma}^{\text{tor}}&\Ca{X}_{gKg^{-1}\cap K,\Sigma'}^{\text{tor}} \arrow{l}[swap]{[g]}\arrow{r}{[1]}&\Ca{X}_{K,\Sigma}^{\text{tor}}.
\end{tikzcd}
\end{equation*}
We define an endomorphism $T_g$ of $H^i(\Ca{X}_{K,\Sigma}^{\text{tor}},V_{\rho,K,\Sigma}^\circ)$ for $\circ$ either $\text{can}$ or $\text{sub}$ as the composition of
\begin{equation*}
H^i(\Ca{X}_{K,\Sigma}^{\text{tor}},V_{\rho,K,\Sigma}^\circ)\overset{[g]^*}{\to} H^i(\Ca{X}_{gKg^{-1}\cap K,\Sigma'}^{\text{tor}},[g]^*V_{\rho,K,\Sigma}^\circ)\overset{g}{\to}H^i(\Ca{X}_{gKg^{-1}\cap K,\Sigma'}^{\text{tor}},V_{\rho,gKg^{-1}\cap K,\Sigma'}^\circ)
\end{equation*}
and
\begin{equation*}
H^i(\Ca{X}_{gKg^{-1}\cap K,\Sigma'}^{\text{tor}},V_{\rho,gKg^{-1}\cap K,\Sigma'}^\circ)\simeq H^i(\Ca{X}_{K,\Sigma}^{\text{tor}},[1]_*V_{\rho,gKg^{-1}\cap K,\Sigma'}^\circ)\overset{\tr}{\to}H^i(\Ca{X}_{K,\Sigma}^{\text{tor}},V_{\rho,K,\Sigma}^\circ)
\end{equation*}
where the isomorphism in the second displayed equation comes from the degeneration of the Leray spectral sequence for $[1]_*$ by Corollary \ref{C:autpush}.  By considering refinements one shows that this yields an endomorphism $T_g$ of $H^i(\Ca{X}_K^{\text{tor}},V_{\rho,K}^\circ)$ which is independent of the choices of $\Sigma$ and $\Sigma'$.

Similarly we may define an endomorphism $T_g$ of $H^i(\Ca{X}_K^{\text{min}},V_{\rho,K}^{\text{sub}})$ as follows: we have a Hecke Correspondence
\begin{equation*}
\begin{tikzcd}
\Ca{X}_K^{\text{min}}&\Ca{X}_{gKg^{-1}\cap K}^{\text{min}} \arrow{l}[swap]{[g]}\arrow{r}{[1]}&\Ca{X}_K^{\text{min}}
\end{tikzcd}
\end{equation*}
and we define $T_g$ to be the composition of
\begin{equation*}
H^i(\Ca{X}_{K,\Sigma}^{\text{min}},V_{\rho,K}^{\text{min}})\overset{[g]^*}{\to}H^i(\Ca{X}_{gKg^{-1}\cap K}^{\text{min}},[g]^*V_{\rho,K,\Sigma}^{\text{sub}})\overset{g}{\to}H^i(\Ca{X}_{gKg^{-1}\cap K}^{\text{min}},V_{\rho,gKg^{-1}\cap K}^{\text{sub}})
\end{equation*}
and
\begin{equation*}
H^i(\Ca{X}_{gKg{^-1}\cap K}^{\text{min}},V_{\rho,gKg^{-1}\cap K}^{\text{sub}})\simeq H^i(\Ca{X}_K^{\text{min}},[1]_*V_{\rho,gKg^{-1}\cap K}^{\text{sub}})\overset{\tr}{\to}H^i(\Ca{X}_{K,\Sigma}^{\text{min}},V_{\rho,K}^{\text{min}}).
\end{equation*}
Moreover one readily checks that the isomorphism $H^i(\Ca{X}_K^{\text{min}},V_{\rho,K}^{\text{sub}})\simeq H^i(\Ca{X}_K^{\text{tor}},V_{\rho,K}^{\text{sub}})$ is compatible with $T_g$.

\section{Well Positioned Subschemes and Sections}\label{S:wellpos}

Throughout this section fix a neat open compact subgroup $K\subset G(\hat{\Z}^{(p)})$ and $\Sigma$ a good compatible family of cone decompositions at level $K$.  The fact that the line bundle $\omega_K/\Ca{X}_K^{\text{min}}$ is ample plays a crucial role in the construction of congruences in Chapter \ref{cong}.  However the fact that $\Ca{X}_K^{\text{min}}$ is usually not smooth over $R$ is the source of some complications.  The goal of this section is to develop some tools to deal with these difficulties.

\subsection{Subschemes Well Positioned at the Boundary}

Let $\Sc{C}\in\text{Cusp}_K$ be a cusp label.  Corresponding to $\Sc{C}$ we have we have the reduced locally closed subscheme $\Ca{X}_{K,\Sigma,\Sc{C}}^{\text{tor}}$ of $\Ca{X}_{K,\Sigma}^{\text{tor}}$ and we denoted the formal completion of the latter along the former by $\hat{\Ca{X}}_{K,\Sigma,\Sc{C}}^{\text{tor}}$.  Via the isomorphism of formal schemes
\begin{equation*}
\hat{\Ca{X}}_{K,\Sigma,\Sc{C}}^{\text{tor}}\simeq\Fr{X}_{\Sc{C},\Sigma_{\Sc{C}}}/\Gamma_{\Sc{C}}
\end{equation*}
of part 2 of Theorem \ref{T:toroidal} we have a structural morphism
\begin{equation*}
\hat{\Ca{X}}_{K,\Sigma,\Sc{C}}^{\text{tor}}\to\Ca{X}_{\Sc{C}}.
\end{equation*}

Similarly we have $\Ca{X}_{\Sc{C}}$ viewed as a reduced locally closed subscheme of $\Ca{X}_{K,\Sc{C}}^{\text{min}}$ and we denoted the formal completion of the latter along the former by $\hat{\Ca{X}}_{K,\Sc{C}}^{\text{min}}$.  By part 4 of Theorem \ref{T:minimal} we have a structural morphism
\begin{equation*}
\hat{\Ca{X}}_{K,\Sc{C}}^{\text{min}}\to\Ca{X}_{\Sc{C}}.
\end{equation*}
If $Z\subset\Ca{X}_{\Sc{C}}$ is a closed subscheme then we denote by $(-)_Z$, the base change of a scheme (or formal scheme, or sheaf) over $\Ca{X}_{\Sc{C}}$ to $Z$.

\begin{defn}
\begin{enumerate}
\item We say that a closed subscheme $Z\subset \Ca{X}^{\text{tor}}_{K,\Sigma}$ is well positioned at the boundary if for every cusp label $\Sc{C}\in\text{Cusp}_K$, the formal completion of $Z$ along $\Ca{X}_{K,\Sigma,\Sc{C}}^{\text{tor}}$ is of the form $(\hat{\Ca{X}}_{K,\Sigma,\Sc{C}}^{\text{tor}})_{Z_{\Sc{C}}}$ for some closed subscheme $Z_{\Sc{C}}$ of $\Ca{X}_{\Sc{C}}$.
\item We say that a closed subscheme $Z\subset \Ca{X}^{\text{min}}_K$ is well positioned at the boundary if for every cusp label $\Sc{C}\in\text{Cusp}_K$, the formal completion of $Z$ along $\Ca{X}_\Sc{C}$ is of the form $(\hat{\Ca{X}}_{K,\Sigma}^{\text{min}})_{Z_{\Sc{C}}}$ for some closed subscheme $Z_{\Sc{C}}$ of $\Ca{X}_{\Sc{C}}$ (which must in fact just be the scheme theoretic intersection of $Z$ with $\Ca{X}_{\Sc{C}}$.)
\end{enumerate}
\end{defn}

Intuitively, a closed subscheme of $\Ca{X}_{K,\Sigma}^{\text{tor}}$ or $\Ca{X}_K^{\text{min}}$ is well positioned at the boundary if it is locally cut out by automorphic functions whose Fourier-Jacobi expansions consist of only a constant term.

Here is the main result we will prove regarding these definitions.
\begin{thm}\label{T:wellpos}
There is a correspondence between subschemes $Z^{\text{\rm min}}\subset \Ca{X}^{\text{\rm min}}_{K}$ well positioned at the boundary and subschemes $Z^{\text{\rm tor}}\subset\Ca{X}^{\text{\rm tor}}_{K,\Sigma}$ well positioned at the boundary characterized by the fact that for each cusp label $\Sc{C}\in\text{\rm Cusp}_K$, the corresponding closed subschemes $Z_\Sc{C}$ are the same.  Moreover if $Z^{\text{\rm min}}$ and $Z^{\text{\rm tor}}$ correspond then
\begin{enumerate}
\item $\pi_{K,\Sigma}^{-1}(Z^{\text{\rm min}})=Z^{\text{\rm tor}}$ scheme theoretically.
\item $\pi_{K,\Sigma,*}\O_{Z^{\text{\rm tor}}}=\O_{Z^{\text{\rm min}}}$.
\item $Z^{\text{\rm min}}$ is reduced if and only if $Z^{\text{\rm tor}}$ is reduced if and only if $Z_{\Sc{C}}$ is reduced for every cusp label $\Sc{C}\in\text{\rm Cusp}_K$.
\end{enumerate}
\end{thm}

Before proving the theorem we will prove the following lemma.  One essentially finds the proof in Faltings-Chai \cite[p. 154-155]{FC90} and in Lan \cite[Proposition 7.2.4.3]{La13}.
\begin{lem}\label{L:bclem}
Let $\Sc{C}\in\text{\rm Cusp}_K$ and let $\pi:\Fr{X}_{\Sc{C},\Sigma_{\Sc{C}}}/\Gamma_{\Sc{C}}\to\Ca{X}_{\Sc{C}}$ be the canonical map of formal schemes.  Let $Z\subset\Ca{X}_{\Sc{C}}$ be a closed subscheme.  Then
\begin{equation*}
\pi_*(\O_{(\Fr{X}_{\Sc{C},\Sigma_{\Sc{C}}}/\Gamma_{\Sc{C}})_Z})=\pi_*(\O_{\Fr{X}_{\Sc{C},\Sigma_{\Sc{C}}}/\Gamma_{\Sc{C}}})\otimes \O_Z
\end{equation*}
as sheaves of adic $\O_Z$-algebras.
\end{lem}
\begin{proof}
If we let $\tilde{\pi}:\Fr{X}_{\Sc{C},\Sigma_{\Sc{C}}}\to\Ca{X}_{\Sc{C}}$ denote the canonical projection, then as the quotient map
\begin{equation*}
\Fr{X}_{\Sc{C},\Sigma_{\Sc{C}}}\to\Fr{X}_{\Sc{C},\Sigma_{\Sc{C}}}/\Gamma_{\Sc{C}}
\end{equation*}
is a local isomorphism we have
\begin{equation*}
\pi_*(\O_{(\Fr{X}_{\Sc{C},\Sigma_{\Sc{C}}}/\Gamma_{\Sc{C}})_Z})\cong(\tilde{\pi}_*\O_{(\Fr{X}_{\Sc{C},\Sigma_{\Sc{C}}})_Z})^{\Gamma_{\Sc{C}}}
\end{equation*}
We may factor $\tilde{\pi}$ as
\begin{equation*}
\Fr{X}_{\Sc{C},\Sigma_{\Sc{C}}}\overset{\pi_1}{\to} C_{\Sc{C}}\overset{\pi_2}{\to} \Ca{X}_{\Sc{C}}.
\end{equation*}
First we claim that
\begin{equation*}
\pi_{1,*}\O_{(\Fr{X}_{\Sc{C},\Sigma_{\Sc{C}}})_Z}\cong\prod_{l\in  P_{\Sc{C}}^\vee}\Psi_{\Sc{C}}(l)_Z
\end{equation*}

Indeed, we have
\begin{equation*}
\pi_{1,*}\O_{(\Fr{X}_{\Sc{C},\Sigma_{\Sc{C}}})_Z}=\bigcap_{\sigma\in\Sigma_{\Sc{C}},\sigma\subset P^+_{\Sc{C}}}\pi_{1,*}\O_{\Fr{X}_{\Sc{C}}(\sigma)}\subset\prod_{l\in P_{\Sc{C}}^\vee}\Psi_{\Sc{C}}(l)_Z
\end{equation*}
the intersection being taken inside of the huge sheaf $\prod_{l\in\Bf{S}_{\Sc{C}}}\Psi_{\Sc{C}}(l)_Z$, and where the inclusion follows from the fact that
\begin{equation*}
\bigcap_{\sigma\in\Sigma_{\Sc{C}},\sigma\subset P_{\Sc{C}}^+}\sigma^\vee=P^\vee_{\Sc{C}}
\end{equation*}
To show the other inclusion we need to show that for each $\sigma\in\Sigma_{\Sc{C}}$ with $\sigma\subset P_{\Sc{C}}^+$ we have
\begin{equation*}
\prod_{l\in P_{\Sc{C}}^\vee}\Psi_{\Sc{C}}(l)_Z\subset\pi_{1,*}\O_{\Fr{X}_{\Sc{C}}(\sigma)_Z}=\hat{\bigoplus_{l\in\sigma^\vee}}\Psi_{\Sc{C}}(l)_Z
\end{equation*}
or in other words that for each $n$, all but finitely many of the terms in the product on the left are in $\I_\partial^n$.

Hence we have
\begin{align*}
\pi_*(\O_{(\Fr{X}_{\Sc{C},\Sigma_{\Sc{C}}}/\Gamma_{\Sc{C}})_Z})&\cong\left(\pi_{2,*}\prod_{l\in  P_{\Sc{C}}^\vee}\Psi_{\Sc{C}}(l)_Z\right)^{\Gamma_{\Sc{C}}}\\
&\cong\left(\prod_{l\in  P_{\Sc{C}}^\vee}\pi_{2,*}(\Psi_{\Sc{C}}(l)_Z)\right)^{\Gamma_{\Sc{C}}}.
\end{align*}
In order to complete the proof of the lemma we must show that
\begin{equation}\label{E:twoisos}
\left(\prod_{l\in  P_{\Sc{C}}^\vee}\pi_{2,*}(\Psi_{\Sc{C}}(l)_Z)\right)^{\Gamma_{\Sc{C}}}\simeq\left(\prod_{l\in  P_{\Sc{C}}^\vee}(\pi_{2,*}\Psi_{\Sc{C}}(l))_Z\right)^{\Gamma_{\Sc{C}}}\simeq\left(\prod_{l\in  P_{\Sc{C}}^\vee}\pi_{2,*}\Psi_{\Sc{C}}(l)\right)^{\Gamma_{\Sc{C}}}\otimes\O_Z\tag{*}
\end{equation}

Now we recall some facts found in the proof of Proposition 7.2.4.3 of \cite{La13}.  For each $l\in P_{\Sc{C}}^\vee$ let $\Gamma_{\Sc{C},l}$ denote the stabilizer of $l$ in $\Gamma_{\Sc{C}}$ and let $\Gamma_{\Sc{C},l}'$ denote the finite index subgroup of $\Gamma_{\Sc{C},l}$ which acts trivially on $\tilde{\Ca{X}}_{\Sc{C}}$.  Then as in the proof of Proposition 7.2.4.3 of \cite{La13} we may factor $\pi_2$ as
\begin{equation*}
C_{\Sc{C}}\overset{\pi_{2,1}}{\to} C(l)\overset{\pi_{2,2}}{\to}\tilde{\Ca{X}}_{\Sc{C}}\overset{\pi_{2,3}}{\to}\Ca{X}_{\Sc{C}}
\end{equation*}
where both $\pi_{2,1}$ and $\pi_{2,2}$ are abelian scheme torsors and $\pi_{2,3}$ is the finite \'{e}tale cover of point 7 in the list at the beginning of section \ref{S:charts} and there is an action of $\Gamma_{\Sc{C},l}$ on $C(l)$ for which $\pi_{2,1}$ and $\pi_{2,2}$ are equivariant.  Moreover there is a $\Gamma_{\Sc{C},l}$ equivariant line bundle $\Psi'(l)/C(l)$, relatively ample over $\tilde{\Ca{X}}_{\Sc{C}}$ with a $\Gamma_{\Sc{C},l}$ equivariant isomorphism $\Psi_{\Sc{C}}(l)\simeq \pi_{2,1}^*\Psi'(l)$, and such that the $\Gamma_{\Sc{C},l}'$ action on $\pi_{2,2,*}\Psi'(l)$ is trivial.  Additionally, for use in the proof of Proposition \ref{P:autminbound} below, we note that when $l\in P_{\Sc{C}}^{\vee,+}$ then $C(l)=C_{\Sc{C}}$ so that $\Psi_{\Sc{C}(l)}$ is already relatively ample over $\tilde{X}_{\Sc{C}}$ and $\Gamma_{\Sc{C},l}$ is trivial.

Next we recall that if $\pi:C\to X$ is an abelian scheme torsor and $\L/C$ is a line bundle then the formation of $\pi_*\L$ is compatible with arbitrary base change in either of the following two cases:
\begin{enumerate}
\item $\L=\O_C$ (Indeed, $\pi_*\O_C=\O_X$.)
\item $\L$ is relatively ample for $C/X$.  Indeed, this follows from the theorem on cohomology and base change \cite[III, 7.7.5]{EGA} and the fact that for an abelian variety over a field, the higher coherent cohomology of an ample line bundle vanishes \cite{Mu70}.
\end{enumerate}

For the first isomorphism in (\ref{E:twoisos}) we must show that for each $l\in P_{\Sc{C}}^\vee$ we have 
\begin{equation*}
\pi_{2,*}(\Psi_{\Sc{C}}(l)_Z)\simeq (\pi_{2,*}\Psi_{\Sc{C}}(l))_Z
\end{equation*}
or in other words, we must show that base change to $Z$ commutes with push forward by $\pi_{2,i,*}$ for $i=1,2,3$.  For $\pi_{2,1}$ this follows by the projection formula and point 1 above.  For $\pi_{2,2}$ it follows from point 2 above.  For $\pi_{2,3}$ it follows from the fact that $\pi_{2,3}$ is affine.

For the second isomorphism in (\ref{E:twoisos}) we must show that the formation of $\Gamma_{\Sc{C}}$ invariants commutes with base change.  First note that the product over all $l\in P_{\Sc{C}}^\vee$ breaks up into a product over $\Gamma_{\Sc{C}}$ orbits on $P_{\Sc{C}}^\vee$ of terms of the form
\begin{equation*}
\text{Ind}_{\Gamma_{\Sc{C},l}}^{\Gamma_{\Sc{C}}}(\pi_{2,*}\Psi_{\Sc{C}}(l))_Z
\end{equation*}
but then
\begin{equation*}
\left(\text{Ind}_{\Gamma_{\Sc{C},l}}^{\Gamma_{\Sc{C}}}(\pi_{2,*}\Psi_{\Sc{C}}(l))_Z\right)^{\Gamma_{\Sc{C}}}=((\pi_{2,*}\Psi_{\Sc{C}}(l))_Z)^{\Gamma_{\Sc{C},l}}=((\pi_{2,*}\Psi_{\Sc{C}}(l))_Z)^{\Gamma_{\Sc{C},l}/\Gamma_{\Sc{C},l}'}
\end{equation*}
as $\Gamma_{\Sc{C},l}'$ acts trivially on $\pi_{2,*}\Psi_{\Sc{C}}(l)$.  Finally by \'{e}tale descent we have an isomorphism
\begin{equation*}
((\pi_{2,*}\Psi_{\Sc{C}}(l))_Z)^{\Gamma_{\Sc{C},l}/\Gamma_{\Sc{C},l}'}\simeq (\pi_{2,*}\Psi_{\Sc{C}}(l))^{\Gamma_{\Sc{C},l}/\Gamma_{\Sc{C},l}'}\otimes\O_Z.
\end{equation*}
Indeed this reduces to the fact that of $A\to B$ is finite \'{e}tale Galois with Galois group $G$, and $M$ is a $B$-module with a semi linear $G$-action then by \'{e}tale descent there is a $G$-equivariant isomorphism
\begin{equation*}
M^G\otimes_AB\simeq M.
\end{equation*}
Then if $A'$ is any $A$-algebra, tensoring with $A'$ and taking $G$-invariants we obtain an isomorphism
\begin{equation*}
M^G\otimes_A A'\simeq (M\otimes_AA')^G.
\end{equation*}
\end{proof}

\begin{proof}[Proof of Theorem \ref{T:wellpos}]
Let $Z\subset \Ca{X}_{K,\Sigma}^{\text{tor}}$ be well positioned at the boundary.  Our first aim is show that the natural map of coherent sheaves on $\Ca{X}_K^{\text{min}}$
\begin{equation*}
\O_{\Ca{X}_K^{\text{min}}}=\pi_{K,\Sigma,*}\O_{\Ca{X}_{K,\Sigma}^{\text{tor}}}\to \pi_{K,\Sigma,*}\O_Z
\end{equation*}
is surjective (where the first equality holds by part 3 of Theorem \ref{T:minimal}).   It suffices to prove this after formally completing along the locally closed subschemes $\Ca{X}_{\Sc{C}}$ for each cusp label $\Sc{C}\in\text{Cusp}_K$.  We may compute the formal completions of these sheaves along $\Ca{X}_{\Sc{C}}$ using \cite[III, 4.1.5]{EGA}.  We have $\pi_{K,\Sigma}^{-1}(\Ca{X}_{\Sc{C}})=\Ca{X}_{K,\Sigma,\Sc{C}}^{\text{tor}}$ and we denoted the formal completion of $\Ca{X}_{K,\Sigma}^{\text{tor}}$ along this by $\hat{\Ca{X}}_{K,\Sigma,\Sc{C}}^{\text{tor}}$, so that we have a map of formal schemes
\begin{equation*}
\hat{\pi}_{K,\Sigma}:\hat{\Ca{X}}_{K,\Sigma,\Sc{C}}^{\text{tor}}\to\hat{\Ca{X}}_{K,\Sc{C}}^{\text{min}}
\end{equation*}
As $\pi_{K,\Sigma}$ is proper, we may identify the formal completion of
\begin{equation*}
\pi_{K,\Sigma,*}\O_{\Ca{X}_{K,\Sigma}^{\text{tor}}}\to \pi_{K,\Sigma,*}\O_Z
\end{equation*}
along $\Ca{X}_{\Sc{C}}$ with
\begin{equation*}
\hat{\pi}_{K,\Sigma,*}\O_{\hat{\Ca{X}}_{K,\Sigma,\Sc{C}}^{\text{tor}}}\to\hat{\pi}_{K,\Sigma,*}\O_{\hat{Z}}
\end{equation*}
where $\hat{Z}$ denotes the formal completion of $Z$ along $\Ca{X}_{K,\Sigma,\Sc{C}}^{\text{tor}}$.  Now recall from part 4 of Theorem \ref{T:minimal} that we have a commutative diagram of formal schemes
\begin{equation*}
\begin{tikzcd}
\hat{\Ca{X}}_{K,\Sigma,\Sc{C}}^{\text{\rm tor}}\arrow{r}\arrow{d}{\hat{\pi}_{K,\Sigma}}&\Fr{X}_{\Sc{C},\Sigma_{\Sc{C}}}/\Gamma_{\Sc{C}}\arrow{d}{\pi}\\
\hat{\Ca{X}}_{K,\Sc{C}}^{\text{\rm min}}\arrow{r}& \Ca{X}_{\Sc{C}}
\end{tikzcd}
\end{equation*}
The top row is an isomorphism and the bottom row, while certainly not an isomorphism of formal schemes, does induce an isomorphism of underlying topological spaces.  Via these isomorphisms the map of sheaves we are considering can be identified with
\begin{equation*}
\pi_*\O_{\Fr{X}_{\Sc{C},\Sigma_{\Sc{C}}}/\Gamma_{\Sc{C}}}\to\pi_*\O_{(\Fr{X}_{\Sc{C},\Sigma_{\Sc{C}}}/\Gamma_{\Sc{C}})_{Z_{\Sc{C}}}}
\end{equation*}
where $Z_{\Sc{C}}$ is the closed subscheme of $\Ca{X}_{\Sc{C}}$ determining $\hat{Z}$.  This map is surjective by Lemma \ref{L:bclem}.

Hence we may define a closed subscheme $Z^{\text{min}}\subset\Ca{X}_{\Sc{C}}^{\text{min}}$ to be the closed subscheme with structure sheaf $\pi_{K,\Sigma,*}\O_Z$.  Then it follows from Lemma \ref{L:bclem} again that the formal completion of $Z^{\text{min}}$ along $\Ca{X}_{\Sc{C}}$ is $(\hat{\Ca{X}}_{K,\Sc{C}}^{\text{min}})_{Z_{\Sc{C}}}$.

For part 3, the reducedness of $Z^\text{tor}$ implies the reducedness of $Z^{\text{min}}$ by part 2.  Next we claim that if $Z^{\text{min}}$ is reduced then so is each $Z_{\Sc{C}}$.  First recall that by \cite[IV, 7.8.3]{EGA}, the formal completion of a reduced excellent ring is reduced.  Hence if $Z^{\text{min}}$ is reduced then so is the formal scheme $\hat{Z}^{\text{min}}_{\Sc{C}}$ obtained by formally completing $Z^{\text{min}}$ along $\Ca{X}_{\Sc{C}}$ (by which we mean that the structure sheaf is a sheaf of reduced rings).  But by the computation in lemma \ref{L:bclem} we saw that the structure sheaf of $\O_{Z_{\Sc{C}}}$ occurs as a direct factor of $\O_{\hat{Z}^{\text{min}}_{\Sc{C}}}$ (as the factor corresponding to $l=0$) and hence $Z_{\Sc{C}}$ is reduced.  Finally we we claim that if $Z_{\Sc{C}}$ is reduced, then so is $Z^{\text{tor}}$.  To show that $Z^{\text{tor}}$ is reduced, it suffices to show that for each cusp label $\Sc{C}$, the formal completion of $Z^{\text{tor}}$ along $\Ca{X}_{\Sc{C}}^{\text{tor}}$ is reduced.  But this is a formal completion of $(\Xi_{\Sc{C},\Sigma_{\Sc{C}}})|_{Z_{\Sc{C}}}$ and the map $\Xi_{\Sc{C},\Sigma_{\Sc{C}}}\to \Ca{X}_{\Sc{C}}$ is smooth.
\end{proof}

\subsection{Sections of $\omega^k$ Near the Boundary}\label{S:welpossec}

Let $\Sc{C}\in\text{Cusp}_K$ be a cusp label.  Over the toroidal boundary chart $\Xi_{\Sc{C},\Sigma_{\Sc{C}}}$ we have the semiabelian scheme $\tilde{A}$ which sits in an exact sequence
\begin{equation*}
0\to T\to\tilde{A}_{\Sc{C}}\to A_{\Sc{C}}\to 0
\end{equation*}
where $T$ is the constant torus with character group $X$ and $A_{\Sc{C}}$ is the pullback of the universal abelian scheme on $\Ca{X}_{\Sc{C}}$.  Then taking co-lie algebras we have a short exact sequence of locally free sheaves on $\Xi_{\Sc{C},\Sigma_{\Sc{C}}}$
\begin{equation*}
0\to\omega_{A_{\Sc{C}}}\to\omega_{\tilde{A}_{\Sc{C}}}\to \omega_T\to 0.
\end{equation*}

We will denote the determinant of the Hodge bundle of $A_{\Sc{C}}$ by $\omega_{\Sc{C}}/\Ca{X}_{\Sc{C}}$ and the determinant of $\omega_{\tilde{A}_{\Sc{C}}}$ by $\tilde{\omega}_{\Sc{C}}$.  Then we have an isomorphism
\begin{equation*}
\omega_T\simeq X\otimes\O_{\Xi_{\Sc{C},\Sigma_{\Sc{C}}}}
\end{equation*} 
given by sending $x\in X$ to $\frac{dx}{x}$.  Hence we have a $\Gamma_{\Sc{C}}$-equivariant isomorphism

\begin{equation*}
\tilde{\omega}_{\Sc{C}}\simeq \det X\otimes \pi^*\omega_{\Sc{C}}
\end{equation*}
where $\pi:\Xi_{\Sc{C},\Sigma_{\Sc{C}}}\to\Ca{X}_\Sc{C}$ denotes the canonical map.

Now we will consider completions.  Recall that we have an isomorphism of formal completions
\begin{equation*}
\hat{\Ca{X}}_{K,\Sigma,\Sc{C}}^{\text{tor}}\simeq\Fr{X}_{\Sc{C},\Sigma_{\Sc{C}}}/\Gamma_{\Sc{C}}.
\end{equation*}
We will denote the formal completion of $\omega_K/\Ca{X}_{K,\Sigma,\Sc{C}}^{\text{tor}}$ along $\Ca{X}_{K,\Sigma,\Sc{C}}^{\text{tor}}$ by $\hat{\omega}_{K,\Sc{C}}$.  We will also denote the formal completion of $\tilde{\omega}_{\Sc{C}}/\Xi_{\Sc{C},\Sigma_{\Sc{C}}}$ along $\partial_{\Sc{C},\Sigma_{\Sc{C}}}$ by $\hat{\tilde{\omega}}_{\Sc{C}}/\Fr{X}_{\Sc{C},\Sigma_{\Sc{C}}}$.  We use the same symbol for its quotient by $\Gamma_{\Sc{C}}$, a line bundle on the formal scheme $\Fr{X}_{\Sc{C},\Sigma_{\Sc{C}}}/\Gamma_\Sc{C}$.  Then we have
\begin{equation*}
\hat{\tilde{\omega}}_\Sc{C}\simeq \det X\otimes\pi^*\omega_\Sc{C}
\end{equation*}
where $\pi:\Fr{X}_{\Sc{C},\Sigma_{\Sc{C}}}/\Gamma_\Sc{C}\to\Ca{X}_{\Sc{C}}$ is the canonical map.  Indeed, even though $\Gamma_{\Sc{C}}$ acts non trivially on $X$, the neatness of $\Gamma_{\Sc{C}}$ implies that $\Gamma_{\Sc{C}}$ acts trivially on $\det X$.

\begin{prop}\label{P:torcompare}
With notation as above, for each cusp label $\Sc{C}\in\text{\rm Cusp}_K$ we have a canonical isomorphism
\begin{equation*}
\hat{\omega}_{K,\Sc{C}}\simeq \hat{\tilde{\omega}}_{\Sc{C}}\simeq \det X\otimes \pi^*\omega_{\Sc{C}}
\end{equation*}
of line bundles over the formal scheme $\hat{\Ca{X}}_{K,\Sigma,\Sc{C}}^{\text{\rm tor}}$ where $\pi:\hat{\Ca{X}}_{K,\Sigma,\Sc{C}}^{\text{\rm tor}}\to\Ca{X}_{\Sc{C}}$ is the structural morphism.
\end{prop}
\begin{proof}
The first isomorphism comes from the isomorphism between the formal completions of $A$ and $\tilde{A}_\Sc{C}$ (see part 2 of Theorem \ref{T:toroidal}.)
\end{proof}

We have a similar result for the minimal compactification.  We will denote by $\hat{\omega}_{K,\Sc{C}}/\hat{\Ca{X}}^{\text{min}}_K$ the formal completion of $\omega_K/\Ca{X}^{\text{min}}_{K,\Sc{C}}$ along $\Ca{X}_K$.

\begin{prop}\label{P:mincompare}
For each cusp label $\Sc{C}\in\text{\rm Cusp}_K$ we have a canonical isomorphism
\begin{equation*}
\hat{\omega}_{K,\Sc{C}}\simeq\det X\otimes\pi^*\omega_{\Sc{C}}
\end{equation*}
of line bundles on $\hat{\Ca{X}}_{K,\Sc{C}}^{\text{\rm min}}$, where $\pi:\hat{\Ca{X}}_{K,\Sc{C}}^{\text{\rm min}}\to\Ca{X}_\Sc{C}$ is the structural morphism of part 4 of Theorem \ref{T:minimal}.
\end{prop}
\begin{proof}
This follows from Proposition \ref{P:torcompare} combined with the theorem on formal functions \cite[III, 4.1.5]{EGA} and the projection formula.
\end{proof}

Now we have the following definition.
\begin{defn}\label{D:wellpossec}
\begin{enumerate}
\item
Let $Z^{\text{tor}}\subset\Ca{X}_{K,\Sigma}^{\text{tor}}$ be a closed subscheme which is well positioned at the boundary corresponding to closed subschemes $Z_{\Sc{C}}\subset\Ca{X}_{\Sc{C}}$ for each cusp label $\Sc{C}$.  Then a section
\begin{equation*}
A\in H^0(Z^{\text{tor}},\omega_K^{\otimes k}|_{Z^{\text{tor}}})
\end{equation*}
is said to be well positioned at the boundary if for each cusp label $\Sc{C}$ there is a section
\begin{equation*}
A_{\Sc{C}}\in H^0(Z_{\Sc{C}},\omega_{\Sc{C}}^{\otimes k}|_{Z_{\Sc{C}}})
\end{equation*}
such that the section
\begin{equation*}
\hat{A}\in H^0((\hat{\Ca{X}}_{K,\Sigma}^{\text{tor}})_{Z_{\Sc{C}}},\hat{\omega}_K)
\end{equation*}
obtained by formally completing $A$ along $\Ca{X}_{K,\Sigma,\Sc{C}}^{\text{tor}}$ is of the form
\begin{equation*}
\alpha^{\otimes k}\otimes \pi^*A_{\Sc{C}}
\end{equation*}
under the isomorphism of Proposition \ref{P:torcompare} restricted to $(\hat{\Ca{X}}_{K,\Sigma}^{\text{tor}})_{Z_{\Sc{C}}}$, where 
\begin{equation*}
\pi:(\hat{\Ca{X}}_{K,\Sigma,\Sc{C}})_{Z_{\Sc{C}}}\to Z_{\Sc{C}}
\end{equation*}
 is the structural morphism and $\alpha$ is a choice of generator for $\det X$.
\item
Let $Z^{\text{min}}\subset\Ca{X}_K^{\text{min}}$ be a closed subscheme which is well positioned at the boundary corresponding to closed subschemes $Z_{\Sc{C}}\subset\Ca{X}_{\Sc{C}}$ for each cusp label $\Sc{C}$.  Then a section
\begin{equation*}
A\in H^0(Z^{\text{min}},\omega_K^{\otimes k}|_{Z^{\text{min}}})
\end{equation*}
is said to be well positioned at the boundary if for each cusp label $\Sc{C}$ there is a section
\begin{equation*}
A_{\Sc{C}}\in H^0(Z_{\Sc{C}},\omega_{\Sc{C}}^{\otimes k}|_{Z_{\Sc{C}}})
\end{equation*}
such that the section
\begin{equation*}
\hat{A}\in H^0((\hat{\Ca{X}}_{K,\Sigma}^{\text{min}})_{Z_{\Sc{C}}},\hat{\omega}_K)
\end{equation*}
obtained by formally completing $A$ along $\Ca{X}_{K,\Sigma,\Sc{C}}^{\text{min}}$ is of the form
\begin{equation*}
\alpha^{\otimes k}\otimes \pi^*A_{\Sc{C}}
\end{equation*}
under the isomorphism of Proposition \ref{P:mincompare} restricted to $(\hat{\Ca{X}}_{K,\Sigma}^{\text{min}})_{Z_{\Sc{C}}}$, where
\begin{equation*}
\pi:(\hat{\Ca{X}}_{K,\Sigma,\Sc{C}})_{Z^{\text{min}}}\to Z_{\Sc{C}}
\end{equation*}
is the structural morphism and $\alpha$ is a choice of generator for $\det X$.
\end{enumerate}
\end{defn}

\begin{rem}\label{R:wellposrem}
\begin{enumerate}
\item
The choice of the generator $\alpha$ of $\det X$ is unique up to multiplication by $-1$.  Hence the sections $A_{\Sc{C}}$ for cusp labels $\Sc{C}$ associated to $A$ are possibly only determined up to multiplication by $-1$.  On the other hand if either $k$ is even or $2=0$ on $Z$ then replacing $\alpha$ by $-\alpha$ does not change $A_{\Sc{C}}$.  In our applications one of these conditions will always be satisfied, and hence we will speak of the $A_{\Sc{C}}$ as if they are canonically associated to $A$.
\item
Now suppose that $Z^{\text{tor}}\subset\Ca{X}_{K,\Sigma}^{\text{\rm tor}}$ and $Z^{\text{min}}\subset\Ca{X}_K^{\text{min}}$ are closed subschemes which are well positioned at the boundary and correspond as in Theorem \ref{T:wellpos}.  Then by part 2 of that Theorem, the pullback map
\begin{equation*}
H^0(Z^{\text{min}},\omega_K^{\otimes k})\to H^0(Z^{\text{tor}},\omega_K^{\otimes k})
\end{equation*}
is an isomorphism.  It is clear from the definition that this induces a bijection between sections well positioned at the boundary in each space and under this bijection the corresponding $A_\Sc{C}$ for cusp labels $\Sc{C}$ are the same.
\item
If $Z$ is a subscheme of $\Ca{X}_{K,\Sigma}^{\text{tor}}$ or $\Ca{X}_K^{\text{min}}$ which is well positioned at the boundary and $A\in H^0(Z,\omega^{\otimes k}|_Z)$ is well positioned at the boundary then it is clear from the definitions that the (scheme theoretic) vanishing locus $V(A)$ of $A$ is also well positioned at the boundary.
\end{enumerate}
\end{rem}

\subsection{Automorphic Vector Bundles Near the Boundary}

In Chapter \ref{cong} we will need to show that a certain sequence of sections of powers of $\omega$ (restricted to suitable subschemes) is a regular sequence on $V_{\rho,K}^{\text{sub}}$.  This is complicated by the fact that $\Ca{X}_K^{\text{min}}$ is not usually Cohen-Macaulay and $V_{\rho,K}^{\text{sub}}$ is not usually locally free.  Our aim is to prove that the situation is better when the sections are well positioned at the boundary in the sense of the previous section.

\begin{prop}\label{P:autminbound}
Suppose we have integers $r,m\geq 0$ and a sequence $A_0,A_1,\ldots,A_m$ where
\begin{equation*}
A_0\in H^0(\Ca{X}_K^{\text{\rm min}}\times_R R/\pi^r,\omega^{\otimes k_0}|_{\Ca{X}_K^{\text{\rm min}}\times_R R/\pi^r})
\end{equation*}
and for $i=1,\ldots,m$,
\begin{equation*}
A_i\in H^0(V(A_{i-1}),\omega^{\otimes k_i}|_{V(A_{i-1})}).
\end{equation*}
Suppose that for $i=0,\ldots,m$, $A_i$ and $V(A_i)$ are well positioned at the boundary (in fact this is only a condition on the $A_i$ by part 3 of Remark \ref{R:wellposrem}.)  Suppose further that for each cusp label $\Sc{C}$, the associated sequence of sections $A_{\Sc{C},0},A_{\Sc{C},1},\ldots,A_{\Sc{C},m}$ on subschemes of $\Ca{X}_\Sc{C}$ is a regular sequence.  Then for each representation $\rho$ of $M$ on a finite free $R/\pi^r$-module, the sequence $A_0,A_1,\ldots,A_m$ is $V_{\rho}^{\text{\rm sub}}$-regular.
\end{prop}
\begin{proof}
It suffices to prove the corresponding statement after formally completing along $\Ca{X}_\Sc{C}\subset\Ca{X}_K^{\text{min}}$ for each cusp label $\Sc{C}\in\text{Cusp}_K$.  Let $\hat{V}_{\rho,K,\Sc{C}}^{\text{sub}}$ denote the formal completion of $V_{\rho,K}^{\text{sub}}$ along $\Ca{X}_{\Sc{C}}$.

From part 4 of Theorem \ref{T:minimal} that we have a commutative diagram of formal schemes
\begin{equation*}
\begin{tikzcd}
\hat{\Ca{X}}_{K,\Sigma,\Sc{C}}^{\text{\rm tor}}\arrow{r}\arrow{d}{\hat{\pi}_{K,\Sigma}}&\Fr{X}_{\Sc{C},\Sigma_{\Sc{C}}}/\Gamma_{\Sc{C}}\arrow{d}{\pi}\\
\hat{\Ca{X}}_{K,\Sc{C}}^{\text{\rm min}}\arrow{r}& \Ca{X}_{\Sc{C}}
\end{tikzcd}
\end{equation*}
in which the top horizontal arrow is an isomorphism, and the bottom horizontal arrow is not usually an isomorphism, but does induce an isomorphism of underlying topological spaces.

As $A_0,\ldots,A_i$ are well positioned at the boundary, we have
\begin{equation*}
\hat{V}_{\rho,K}^{\text{sub}}|_{V(A_i)}=\hat{V}_{\rho,K}^{\text{sub}}\otimes_{\O_{\hat{\Ca{X}}_{K,\Sc{C}}^{\text{min}}}}\O_{V(A_i)}=\hat{V}_{\rho,K}^{\text{sub}}\otimes_{\O_{\Ca{X}_{\Sc{C}}}}\O_{V(A_{\Sc{C},i})}=\hat{V}_{\rho,K}^{\text{sub}}|_{V(A_{\Sc{C},i})}
\end{equation*}
(note that the equalities on the far left and far right are just the definition of the restriction, but we want to emphasize that in this formula the restrictions correspond to tensor products over different ringed spaces!)  Hence in order to show that $A_i$ is a non zero divisor on $\hat{V}_{\rho,K}^{\text{sub}}|_{V(A_{i-1})}$ we need to show that $A_{\Sc{C},i}$ is a non zero divisor on the (not usually quasi-coherent!) sheaf of $\O_{V(A_{\Sc{C},i-1})}$-modules $\hat{V}_{\rho,K}^{\text{sub}}|_{V(A_{\Sc{C},i})}$.  To complete the proof, we will show that as a sheaf of $\O_{\Ca{X}_{\Sc{C}}}$-modules, $\hat{V}_{\rho,K}^{\text{sub}}$ is a (usually infinite) product of locally free sheaves of $\O_{\Ca{X}_{\Sc{C}}\times_RR/\pi^r}$-modules.

By the theorem on formal functions \cite[III, 7.7.5]{EGA}, 
\begin{equation*}
\hat{V}_{\rho,K}^{\text{sub}}=\hat{\pi}_{K,\Sigma,*}\hat{V}_{\rho,K,\Sigma}^{\text{sub}}
\end{equation*}
where $\hat{V}_{\rho,K,\Sigma}^{\text{sub}}$ denotes the formal completion of $V_{\rho,K,\Sigma}^{\text{sub}}$ along $\Ca{X}_{K,\Sigma,\Sc{C}}^{\text{tor}}$.  

Now we have the local isomorphism
\begin{equation*}
p:\Fr{X}_{\Sc{C},\Sigma_{\Sc{C}}}\to\Fr{X}_{\Sc{C},\Sigma_{\Sc{C}}}/\Gamma_{\Sc{C}}.
\end{equation*}
Denote $\tilde{\pi}=\pi p$.  Then as $\O_{\Ca{X}_{\Sc{C}}}$-modules
\begin{equation*}
\pi_*\hat{V}_{\rho,K,\Sigma}^{\text{sub}}=(\tilde{\pi}_*p^*V_{\rho,K,\Sigma}^{\text{sub}})^{\Gamma_{\Sc{C}}}
\end{equation*}

We will study this in a manner similar to the proof of Lemma \ref{L:bclem}.  The map $\tilde{\pi}$ factors as a composition
\begin{equation*}
\Fr{X}_{\Sc{C},\Sigma_{\Sc{C}}}\overset{\pi_1}{\to}C_{\Sc{C}}\overset{\pi_2}{\to}\tilde{\Ca{X}}_{\Sc{C}}\overset{\pi_3}{\to}\Ca{X}_{\Sc{C}}
\end{equation*}
First we recall that by Proposition 5.6 of \cite{La14},
\begin{enumerate}
\item There is a $\Gamma_{\Sc{C}}$ equivariant sheaf $V/C_{\Sc{C}}$ such that
\begin{equation*}
p^*\hat{V}_{\rho,K,\Sigma}^{\text{sub}}=\pi_1^*V\otimes\I_{\partial}
\end{equation*}
\item $V$ has a filtration
\begin{equation*}
V=V^0\supset V^1\supset\cdots\supset V^d=0
\end{equation*}
such that for each $i$, $V^i/V^{i+1}$ is of the form $\pi_2^*V_i'$ for $V_i'/\tilde{X}_{\Sc{C}}$ a locally free sheaf of $\O_{\tilde{X}_{\Sc{C}}\times_RR/\pi^r}$-modules.
\end{enumerate}

By an argument similar to that in the proof of Lemma \ref{L:bclem} we have
\begin{equation*}
\pi_{1,*}p^*\hat{V}_{\rho,K,\Sigma}^{\text{sub}}=\prod_{l\in P_{\Sc{C}}^{\vee,+}}\Psi_{\Sc{C}}(l)\otimes_{\O_{C_{\Sc{C}}}}V
\end{equation*}
and hence
\begin{equation*}
\tilde{\pi}_*p^*\hat{V}_{\rho,K,\Sigma}^{\text{sub}}=\prod_{l\in P_{\Sc{C}}^{\vee,+}}\pi_{3,*}\pi_{2,*}(\Psi_{\Sc{C}}(l)\otimes V)
\end{equation*}
Recall from the proof of Lemma \ref{L:bclem} that for $l\in P_{\Sc{C}}^{\vee,+}$, $\Psi_{\Sc{C}}(l)$ is relatively ample over $\tilde{X}_{\Sc{C}}$ and hence $\pi_{2,*}\Psi_{\Sc{C}}(l)$ is locally free and $R^1\pi_{2,*}\Psi_{\Sc{C}}(l)=0$.  Then it follows from this and point 2 above that $\pi_{2,*}(\Psi_{\Sc{C}}(l)\otimes V)$ is a locally free sheaf of $\O_{\tilde{\Ca{X}}_{\Sc{C}}\times_RR/\pi^r}$-modules.  Hence each of the terms in the product above are locally free sheaves of $\O_{\Ca{X}_{\Sc{C}}\times_RR/\pi^r}$-modules.

It remains to analyze what happens when we take $\Gamma_{\Sc{C}}$-invariants.  For $l\in P_{\Sc{C}}^{\vee,+}$, the stabilizer of $l$ in $\Gamma_{\Sc{C}}$ is trivial (see the proof of Lemma \ref{L:bclem}.)  Thus the the product over the terms corresponding to the $\Gamma_{\Sc{C}}$-orbit of $l$ is
\begin{equation*}
\text{Ind}_{\{1\}}^{\Gamma_{\Sc{C}}}\pi_{3,*}\pi_{2,*}(\Psi_{\Sc{C}}(l)\otimes V).
\end{equation*}
Hence $(\tilde{\pi}_*p^*V_{\rho,K,\Sigma}^{\text{sub}})^{\Gamma_{\Sc{C}}}$ is a product of locally free sheaves of $\O_{\Ca{X}_{\Sc{C}}\times_RR/\pi^r}$-modules as desired.
\end{proof}

\chapter{Ekedahl-Oort Stratification and Generalized Hasse Invariants}\label{EOPEL}

The first goal of this chapter is to recall the Ekedahl-Oort stratification of the special fibers of PEL type Shimura varieties.  It was introduced and studied first in the Siegel case by Ekedahl and Oort \cite{Oo01}.  For Hilbert modular varieties it was studied by Goren-Oort \cite{GO00}.  The general PEL case was taken up by Moonen and Wedhorn \cite{Mo01} \cite{Mo04} \cite{We01} (see also \cite{MW04} and \cite{VW13}.)  Then we will introduce our ``generalized Hasse invariants'' on the open Ekedahl-Oort strata.  Their definition is directly inspired by the ``generalized Raynaud trick'' of Ekedahl and Oort \cite{Oo01}.  We will then formulate the first main result of this thesis, Theorem \ref{T:hasseext}, which states that some power of these Hasse invariants extends to the closed Ekedahl-Oort stratum and vanishes on the complement of the open stratum.  We reduce the proof of this theorem to the Siegel case, which will be completed in the next chapter.

Let us now give a more detailed overview of this chapter.  In section \ref{S:btdefs} we define 1-truncated Barsotti-Tate with various additional structures (polarizations and endomorphisms.)  The key examples will be the $p$-torsion subgroup schemes of the abelian schemes parameterized by the PEL modular varieties introduced in Chapter \ref{PEL}.  We also introduce a more general notion of ``partial $\BT$s'' which includes the $p$-torsion of of the semiabelian schemes over the toroidal compactifications of Chapter \ref{compact}.  In section \ref{S:canfil} we will recall the theory of the canonical filtration, due to Ekedahl and Oort.  The canonical filtration plays a central role in the theory of $\BT$s and in this thesis.  In section \ref{S:moonen} we will recall the classification of $\BT$ with polarization and endomorphisms over algebraically closed fields of characteristic $p$, due to Kraft (unpublished work,) Oort \cite{Oo01}, Moonen \cite{Mo01}, and Moonen-Wedhorn \cite{MW04}.  Essentially the classification shows that such a $\BT$ is ``determined by its canonical filtration.''  In section \ref{S:EO} we will introduce the Ekedahl-Oort stratification of a PEL modular variety.  We will define it using the theory of canonical filtrations, as in \cite{Oo01}, but for general PEL modular varieties.  But using the results of Section \ref{S:moonen} we can show that our definition agrees with the usual one.  As a consequence of our approach we are able to prove Theorem \ref{T:opensieg} which is perhaps the first new result of this chapter: it states that under the maps
\begin{equation*}
\phi_{K,\tilde{K}}:X_K\to X_{\tilde{K}}
\end{equation*}
of section \ref{S:morsieg} from a general PEL modular variety to a Siegel modular variety, each EO strata of $X_K$ is open in the pre image of some EO stratum of $X_{\tilde{K}}$ under $\phi_{K,\tilde{K}}$.  Finally in section \ref{S:pelhasse} we turn to generalized Hasse invariants.  They are first constructed as non vanishing sections of a power of the determinant of the Hodge bundle on the open Ekedahl-Oort strata, using the canonical filtration.  Then we state our main Theorem \ref{T:hasseext} on the existence of generalized Hasse invariants: it states that some power of these these sections extend to the closed Ekedahl-Oort strata and vanish on the complement of the open strata.  In this chapter we will explain how to reduce it to the Siegel case.  The proof in the Siegel case will be given in the next chapter.

Throughout this chapter we work over a base $S$ which is assumed locally noetherian and of characteristic $p$.

\section{Definitions}\label{S:btdefs}

\subsection{Truncated Barsotti-Tate Groups}

\begin{defn}
A (1-)\emph{truncated Barsotti-Tate group} (or $\BT$ for short) is a finite flat group scheme $G/S$ such that
\begin{equation*}
\begin{tikzcd}
G\arrow{r}{F}&G^{(p)}\arrow{r}{V}&G
\end{tikzcd}
\end{equation*}
is exact.
\end{defn}

Note in particular that if $G/S$ is a $\BT$, then $[p]_G=0$ and $G[F]\subset G$ is finite flat.

\begin{rem}
More generally one can define $n$-truncated Barsotti-Tate groups for any integer $n$ over schemes which aren't necessarily of characteristic $p$.  However we won't need these notions.
\end{rem}

The $p$-torsion subgroup of an abelian scheme or $p$-divisible group over $S$ is a $\BT$.  However we will also need to consider the $p$-torsion subgroup of semiabelian schemes, which are not finite in general.  Hence we introduce the following non-standard notion.

\begin{defn}
A \emph{partial} $\BT$ is a quasi-finite, flat, separated, group scheme $G/S$ such $[p]_G=0$ and for every $s\in S$, the fiber $G_s/k(s)$ is a $\BT$.
\end{defn}

We remark that it really is necessary to assume that $[p]_G=0$ (consider the kernel of $F^2$ on the universal characteristic $p$ first order deformation of a supersingular elliptic curve.)  We now prove some lemmas to show that this is a reasonable definition.

\begin{lem}\label{L:gpfinite}
Let $G/S$ be a quasi-finite group scheme such that for every $s\in S$, the fiber $G_s/k(s)$ is connected.  Then $G/S$ is finite.
\end{lem}
\begin{proof}
The hypothesis implies that $G\to S$ is a universal homeomorphism, and so the conclusion follows from \cite[IV 8.11.6]{EGA}.
\end{proof}

\begin{lem}\label{L:omegafree}
Let $G/S$ be a partial $\BT$.  Then $G[F]/S$ is finite flat.
\end{lem}
\begin{proof}
The relative Frobenius $F$ gives a bijection on points, so each fiber $G[F]_s$ consists of a single point.  Hence by lemma \ref{L:gpfinite}, $G[F]/S$ is finite.  We need to show that $G[F]/S$ is flat.

By assumption we have $[p]_{G^{(p)}}=FV=0$ on $G$ and hence we can consider the map
\begin{equation*}
V:G^{(p)}\to G[F].
\end{equation*}
Now for every $s\in S$, $G_s/k(s)$ is a $\BT$ and so the map on fibers
\begin{equation*}
V:G_s^{(p)}\to G_s[F]
\end{equation*}
is a surjective map of finite groups schemes over a field, and hence flat.  By the fiberwise criteria for flatness \cite[IV 11.3.11]{EGA} we conclude that $G[F]$ is flat.
\end{proof}

\begin{cor}
Let $G/S$ be a partial $\BT$.  Then $\omega_G=e^*\Omega^1_{G/S}$ is locally free of finite rank equal to the height of $G[F]$.
\end{cor}

In particular the rank of $\omega_G$ is locally constant on $S$.  We (abusively) call it the \emph{dimension} of $G$.

\begin{cor}
Let $G/S$ be a partial $\BT$ which is finite over $S$.  Then $G$ is a $\BT$.
\end{cor}

If $\pi:G\to S$ is a finite flat group scheme then the degree of $G$ is defined to be the rank of $\pi_*\O_G$ as a locally free $\O_S$-module (a locally constant function on $S$.)  The degree of a $\BT$ is a power of $p^h$ where $h$ is the \emph{height}.

Defining the height of a partial $\BT$ is somewhat more subtle.  We recall the following lemma 
%TODO: which I should give a proof of or find a reference for.  It follows from Zariski's main theorem
\begin{lem}
Let $\pi:X\to S$ be flat, separated, quasi-finite.  Consider the degree function
\begin{align*}
d:S&\to\Bf{N}\\
s&\mapsto \deg X_s.
\end{align*}
where $\deg X_s=\dim_{k(s)} A_s$ where $X_s=\spec A_s$ (we remind the reader that the fiber $X_s$ of the quasi-finite map $\pi$ is a finite $k(s)$-scheme.)
\begin{enumerate}
\item The function $d$ on $S$ is lower semicontinuous.
\item If $d$ is locally constant then $f$ is finite.
\end{enumerate}
\end{lem}

Given $G/S$ a partial $\BT$ we can consider the ``finite height'' function
\begin{align*}
f_G:S&\to \Bf{N}\\
s&\mapsto\log_p\deg G_s.
\end{align*}
Then $f_G$ is lower semicontinous by the above lemma.  We will say that $G$ has height $\leq h$ if $f_G(s)\leq h$ for all $s\in S$.  The reason for calling this the finite height will become clear in the next section when we consider quasi-polarizations.

\subsection{Quasi-Polarizations on $\BT$s}

In this section we define principal quasi-polarizations on partial $\BT$'s.  There are some subtleties in characteristic 2.  Our definition is somewhat ad-hoc as a result.

If $G/S$ is a finite flat group scheme killed by $p$, then its cartier dual $G^D$ is the finite flat group scheme representing the functor
\begin{equation*}
S'\mapsto \hom(G(S'),\Bf{G}_m(S'))=\hom(G(S'),\mu_p(S')).
\end{equation*}
It is (contravariantly) functorial and compatible with arbitrary base change.  In particular $(G^{(p)})^D=(G^D)^{(p)}$.  Cartier duality interchanges Frobenius and Verschiebung in the sense that
\begin{equation*}
F_{G^D}=(V_G)^D:G^D\to (G^D)^{(p)}
\end{equation*}
and
\begin{equation*}
V_{G^D}=(F_G)^D:(G^D)^{(p)}\to G^D
\end{equation*}
Moreover there is a canonical evaluation homomorphism $G\to (G^D)^D$ which is an isomorphism.

We recall the following well known fact:
\begin{prop}\label{P:dualbt}
Let $G/S$ be a $\BT$.
\begin{enumerate}
\item $G^D$ is also a $\BT$.
\item If we let $h$ denote the height of $G$, $d$ the dimension of $G$ and $d'$ the dimension of $G^D$ (sometimes called the codimension) then we have an equality
\begin{equation*}
h=d+d'
\end{equation*}
of locally constant functions on $S$.
\end{enumerate}
\end{prop}
\begin{proof}
The first part is an immediate consequence of the definition and the exactness of cartier duality.  For the second, just note that the Cartier dual of $\ker (F:G^D\to (G^{(p)})^D)$ is $\coker (V:G^{(p)}\to G)$.  Hence
\begin{equation*}
\text{ht}(G)=d'+\text{ht}(\im(V:G^{(p)}\to G))=d'+\text{ht}(\ker F:G\to G^{(p)})=d'+d.
\end{equation*}
\end{proof}

Giving a bilinear pairing
\begin{equation*}
\lambda:G\times G\to \mu_p
\end{equation*}
is the same as giving a group homomorphism
\begin{equation*}
\lambda:G\to G^D.
\end{equation*}

Now let $G/S$ be a separated, quasi-finite, flat group scheme killed by $p$.  A pairing
\begin{equation*}
\lambda:G\times G\to\mu_p
\end{equation*}
is said to be antisymmetric if $\lambda=\tilde{\lambda}^{-1}$ if $\tilde{\lambda}$ denotes the pairing defined by exchanging the two factors.  If $G$ is finite, this is equivalent to the corresponding map $\lambda:G\to G^D$ satisfying $\lambda^D=-\lambda$, where $(G^D)^D$ has been identified with $G^D$ via the canonical map described above.

If $G/S$ is finite flat and endowed with an antisymmetric pairing $\lambda:G\times G\to\mu_p$ then we let
\begin{equation*}
\ker\lambda:=\ker(\lambda:G\to G^D).
\end{equation*}
We note that if $\ker\lambda$ is trivial, then $\lambda:G\to G^D$ is an isomorphism.  Indeed $\lambda$ is then injective and $G$ and $G^D$ have the same degree.  It is clear that for any $S$ scheme $S'$ and any $x\in (\ker \lambda)(S')$ and $y\in G(S')$ we have
\begin{equation*}
\lambda(x,y)=\lambda(y,x)=1
\end{equation*}
Then if $\ker\lambda$ is flat, $G/\ker\lambda$ is representable by a finite flat group scheme and it is endowed we have a pairing
\begin{equation*}
\lambda:G/\ker\lambda\times G/\ker\lambda\to\mu_p
\end{equation*}
which has trivial kernel.

\begin{defn}
\begin{enumerate}
\item
A principal quasi-polarized partial $\BT$ is a pair $(G,\lambda)$ consisting of a partial $\BT$ $G$ and a skew symmetric pairing
\begin{equation*}
\lambda:G\times G\to\mu_p
\end{equation*}
such that for each $s\in S$, the following conditions are satisfied on the fiber $(G_s,\lambda_s)$:
\begin{enumerate}
\item The kernel $\ker\lambda_s$ is a group of multiplicative type.
\item If $p=2$ then $(G_{\overline{s}},\lambda_{\overline{s}})$ denote the base change of $(G_s,\lambda_s)$ to some algebraic closure $\overline{k(s)}$ of $k(s)$, then the pairing
\begin{equation*}
D(G_{\overline{s}}/\ker \lambda_{\overline{s}})\times D(G_{\overline{s}}/\ker \lambda_{\overline{s}})\to\overline{k(s)}
\end{equation*}
induced by $\lambda_{\overline{s}}$ is alternating, where $D$ is the (contravariant) Diuedonne module functor.
\end{enumerate}
\item By a principally quasi polarized $\BT$ we mean a $\BT$ $G$ along with an anti symmetric pairing $\lambda:G\times G\to\mu_p$ such that the corresponding homomorphism $\lambda:G\to G^D$ is an isomorphism and if $p=2$, the condition of (b) in the definition above is satisfied.
\end{enumerate}
\end{defn}

Let us make some remarks about this definition.

\begin{rem}
\begin{enumerate}
\item If $(G,\lambda)$ is a principally quasi-polarized partial $\BT$ and $G$ is in fact finite (i.e. it is actually a $\BT$) then it is not necessarily true that $(G,\lambda)$ is a principally quasi-polarized $\BT$ in the sense of 2 above because $\lambda$ may have a kernel.  However if $\ker\lambda$ is trivial, then as remarked above $\lambda:G\to G^D$ is an isomorphism so $(G,\lambda)$ is a principally quasi-polarized $\BT$.

\item Let us now make some remarks on the ``correctness'' of this definition.  Our notion of partial $\BT$s is already nonstandard, so we will only discuss principal quasi-polarizations on $\BT$s.  When $p\not=2$ our definition agrees with that in \cite{We01}.  Wedhorn shows for example, that the ``truncation'' functor from the stack of principally quasi-polarized $p$-divisible groups to principally quasi-polarized $\BT$s is formally smooth, and this justifies this being the ``correct'' notion of a principle quasi-polarized $\BT$ in families.

When $p=2$ there is some trouble which has been discussed in \cite{Oo01} and \cite{Mo01}.  Our definition has been rigged with the following two considerations in mind:
\begin{enumerate}
\item If $(G,\lambda)$ is a principally quasi-polarized $\BT$ over an algebraically closed field $k$ of characteristic 2 then there is a principally quasi-polarized $2$-divisible group $(X,\lambda')/k$ with $(G,\lambda)$ as its 2-torsion.  This would not be true without the extra condition on the pairing on the Dieudonne module!
\item If $S$ is any base of characteristic 2 and $(A,\lambda)/S$ is a prime to $2$ quasi-polarized abelian scheme then $A[2]$ with the $\lambda$-Weil pairing $\lambda:A[2]\times A[2]\to\mu_2$ is a principally quasi-polarized $\BT$.
\end{enumerate}

It should be clear that this condition on the Dieudonne modules of the geometric fibers would not be suitable for the study of families.  We do not know if there exists a good notion of principally quasi-polarized $\BT$s over general bases of characteristic 2!  The reader who finds this to be a headache should just assume that $p\not=2$.
\end{enumerate}
\end{rem}

We have several numerical functions on $S$ for a principally quasi polarized partial $\BT$ $(G,\lambda)$:
\begin{enumerate}
\item The dimension $d(s)=\text{ht}G_s[F]$.
\item The height $h(s)=2 d(s)$.
\item The finite height $f(s)=\text{ht}G_s$.
\item The toral height $t(s)=\text{ht}\ker\lambda_s$.
\item The abelian height $a(s)=\text{ht}(G_s/\ker\lambda_s)$
\end{enumerate}

We note that when $(G,\lambda)$ is a principally quasi-polarized $\BT$ this definition of height agrees with that of the previous section by Proposition \ref{P:dualbt}.  We have the relations
\begin{equation*}
f+t=a+2t=h
\end{equation*}
from which we see that given the height $h$, any one of $f$, $t$ and $a$ determine the rest.  The height and dimension are locally constant, while the finite height and abelian height are lower semicontinous and the total height is upper semicontinuous.

\subsection{$\BT$s With Extra Endomorphisms}\label{SS:extra}

Let $\overline{\O}$ be a finite dimensional semisimple $\FF_p$-algebra.  Let $\FF$ be its center.  We denote by $\Ca{T}$ the set of all embeddings $\tau:\FF\to\overline{\FF}_p$.  Absolute Frobenius acts on $\Ca{T}$ and we denote it by $F$.  For $\tau\in\Ca{T}$ let $[\tau]$ denote its orbit under Frobenius.  Then we have a decomposition
\begin{equation*}
\FF=\prod_{[\tau]}\FF_{[\tau]}
\end{equation*}
where $\FF_\tau$ are finite fields.  We have a corresponding decomposition
\begin{equation*}
\overline{\O}=\prod_{[\tau]}M_{r_{[\tau]}}(\FF_{[\tau]})
\end{equation*}
and we denote the idempotent in the $[\tau]$ factor by $e_{[\tau]}$.

Let $k$ be a subfield of $\overline{\FF}_p$ containing the image of every embedding $\FF\to\overline{\FF}_p$.  Then we have decompositions
\begin{equation*}
\FF\otimes k=\prod_{\tau} k_\tau
\end{equation*}
and
\begin{equation*}
\overline{\O}\otimes k=\prod_\tau M_{r_{[\tau]}}(k_\tau)
\end{equation*}
where $k_\tau$ denotes $k$ with an $\O$-action via $\tau$.  We let $e_\tau$ denote the idempotent in the $\tau$ factor.

For the rest of the chapter we will assume that our base $S$ is actually a $k$-scheme.  Let $\Ca{E}$ be a finite locally free $\O_S$-module with an $\O_S$-linear $\overline{\O}$ action on either the left or the right.  Then we obtain a decomposition
\begin{equation*}
\Ca{E}=\bigoplus_\tau \Ca{E}_\tau=\bigoplus_\tau e_\tau\cdot\Ca{E}.
\end{equation*}
Note that $e_\tau$ lies in the center of $\overline{\O}\otimes k$ so this formula makes sense even when $\overline{\O}$ acts on the right.  Each summand $\Ca{E}_\tau$ is a summand of a finite locally free $\O_S$-module and hence itself finite locally free.  We call the vector of locally constant functions $(\text{rk} (\Ca{E}_\tau)/r_{[\tau]})_\tau$ the multi rank of $\Ca{E}$.  If $\Ca{E}^{(p)}$ denotes the pullback by absolute Frobenius $F:S\to S$ with its induced $\O_S$-linear $\O$ action, then
\begin{equation*}
(\Ca{E}_{F\tau})^{(p)}=(\Ca{E}^{(p)})_{\tau}.
\end{equation*}
Note also that if $\Ca{E}$ has an $\O_S$ linear $\overline{\O}$ action on the left, then it has the same $\overline{\O}$-multirank as its dual $\Ca{E}$ (which has a right $\overline{\O}$ action.)

\begin{defn}
By a partial $\BT$ with $\overline{\O}$ action, we mean a partial $\BT$ $G/S$ equipped with a ring homomorphism $i:\overline{\O}\to \End_S(G)$.
\end{defn}

Let $G/S$ be a partial $\BT$ with $\overline{\O}$ action.  By Lemma \ref{L:omegafree}, $\omega_G=\omega_{G[F]}$ is a locally free $\O_S$-module which inherits an $\O_S$-linear right $\overline{\O}$ action.  We denote its multi rank by $(d_\tau)_{\tau\in\Ca{T}}$ and call it the multi dimension of $G$.

Next note that the decomposition of $\overline{\O}$ into simple factors induces a decomposition
\begin{equation*}
G=\prod_{[\tau]}e_{[\tau]}\cdot G=\prod_{[\tau]} G_{[\tau]}
\end{equation*}
When $G/S$ is finite (i.e. when it is a $\BT$) let $h_{[\tau]}=\text{ht}(G)/(r_{[\tau]}[\Bf{F}_\tau:\FF_p])$.  Then we call the tuple $(h_{[\tau]})$ the multi height of $G$.

\begin{prop}\label{P:hindep}
Let $G/S$ be $\BT$ with $\overline{O}$-action.  Then for each $[\tau]$, $h_{[\tau]}$ is an integer.
\end{prop}
\begin{proof}
We clearly may as well assume that $\overline{\O}=M_r(\FF)$ is simple.  By the usual Morita equivalence we reduce to the case that $r=1$.  Thus what we need to show is that if $\FF$ is a finite field and $G$ is a $\BT$ with a $\FF$ action, then its height is a multiple of $[\FF:\FF_p]$.  As the height is locally constant and compatible with base change, it suffices to treat the case that $S=\spec k'$ where $k'$ is an algebraically closed field of characteristic $p$, with a chosen embedding $k\to k'$.

We now utilize Dieudonne theory.  Let $D$ be the contravariant Diuedonne module of the $\BT$ $G/k'$ with a $\FF$ action.  Then the height of $G$ is $k'$ dimension of $D$.  There are maps
\begin{equation*}
F:D^{(p)}\to D\qquad V:D\to D^{(p)}
\end{equation*}
As $G$ is a $\BT$ we have $\ker F=\im V$ and $\im V=\ker F$.  Hence there are short exact sequences
\begin{equation*}
0\to \ker F\to D\to \ker V\to 0
\end{equation*}
and
\begin{equation*}
0\to \ker V\to D^{(p)}\to \ker F\to 0.
\end{equation*}
Moreover, as $\FF$ acts on $G$, there is a $k'$ linear $\FF$ action on $D$ which commutes with $F$ and $V$ which commutes with $F$ and $V$.  We may then decompose $D$, $D^{(p)}$, $\ker F$ and $\ker V$ into isotypic pieces as above, and we see that for each $\tau\in\Ca{T}$
\begin{equation*}
\dim D_\tau=\dim (\ker F)_\tau+\dim(\ker V)_\tau=\dim(D^{(p)})_\tau=\dim D_{F\tau}.
\end{equation*}
Hence the height of $G$ is $\dim D=\sum_{\tau\in\Ca{T}}\dim D_\tau$ is a multiple of $[\FF:\FF_p]$.
\end{proof}

Now let $*$ be an involution of $\Ca{\O}$.  Then $(\overline{\O},*)$ can be factored as a product of simple algebras with involution.  We recall that there is a rough classification of simple algebras with involution as follows: let $(\overline{\O},*)$ be a simple algebra with involution with center $\Bf{F}$ and let $\Bf{F}^+=\Bf{F}^{*=\text{id}}$.  Then $(\overline{\O},*)$ has one of the following types:
\begin{enumerate}
\item Type A split: $\Bf{F}=\Bf{F}^+\times\Bf{F}^+$ with $\Bf{F}^+$ a field, and $\overline{\O}=M_r(\Bf{F}^+)\times M_r(\Bf{F}^+)^{\text{op}}$ with $*(x,y)=(y,x)$.
\item Type A non split: $\Bf{F}/\Bf{F}^+$ a quadratic extension of fields, with $\overline{\O}=M_r(\Bf{F})$.
\item Type C: $\Bf{F}=\Bf{F}^+$ and $\overline{\O}=M_r(\Bf{F})$ with $*$ given by conjugation with respect to a non degenerate symmetric form.
\item Type D: $\Bf{F}=\Bf{F}^+$ and $\overline{\O}=M_r(\Bf{F})$ with $*$ given by conjugation with respect to a non degenerate alternating form.
\end{enumerate}

As $*$ acts on the center $\FF$ it acts on the set of embeddings $\Ca{T}$.

\begin{defn}
By a principally quasi-polarized $\BT$ with $(\overline{\O},*)$ action we mean a principally quasi-polarized partial $\BT$ $(G,\lambda)/S$ along with $i:\overline{O}\to\End_S(G)$ which satisfies
\begin{equation*}
\lambda(x\cdot-,-)=\lambda(-,x^*\cdot-):G\times G\to\mu_p
\end{equation*}
for all $x\in\overline{\O}$.  
\end{defn}

Note that when $(G,\lambda)$ is a principally quasi-polarized $\BT$ then the condition in the definition is the same as asking that the isomorphism $\lambda:G\to G^D$ satisfies $\lambda i(x)=i(x^*)^D\lambda$.

\begin{prop}
Let $(G,\lambda,i)$ be a principally quasi-polarized partial $\BT$ with $\overline{\O}$-action.  Then
\begin{equation*}
d_\tau+d_{\tau*}=h_{[\tau]}.
\end{equation*}
\end{prop}

\subsection{Mod $p$ PEL Data}\label{S:modppel}

In the last section we saw that principally quasi-polarized $\BT$ with $\overline{\O}$ action have certain discrete invariants: the multi height $(h_{[\tau]})$ and multi dimension $(d_\tau)$.  If we consider the special fiber of a PEL modular variety as in Chapter \ref{PEL}, the $p$-torsion of the universal abelian scheme will be a $\BT$ with extra structure, and we should be able to read off these discrete invariants from the PEL datum defining the moduli problem.  The goal of this section, which is pure linear algebra, is to explain how this works.

\begin{defn}
Let $k'$ be a field of characteristic $p$.  By a symplectic $\overline{\O}\otimes k'$-module we mean a finite dimensional $k'$ vector space $V$ equipped with a $k'$ linear left $\overline{\O}$ action and a non degenerate alternating pairing $\langle\cdot,\cdot\rangle:V\times V\to k'$ satisfying
\begin{equation*}
\langle xv,w\rangle=\langle v,x^*w\rangle
\end{equation*}
for all $x\in\overline{\O}$, and $v,w\in V$.  We say that two symplectic $\overline{\O}\otimes k'$-modules $(V,\langle\cdot,\cdot\rangle)$ and $(V',\langle\cdot,\cdot\rangle')$ are isomorphic if there is a $\overline{\O}\otimes k'$-linear isomorphism $f:V\to V'$ and a constant $c\in {k'}^\times$ such that for all $v,w\in V$, we have $\langle f(v),f(w)\rangle'=c\langle v,w\rangle$.
\end{defn}

The following result is basic.
\begin{prop}\label{P:algclosedunique}
Let $k'$ be an algebraically closed field of characteristic $p$ with an embedding $k\to k'$.  Then two symplectic $\overline{\O}\otimes k'$-modules $V$ and $V'$ are isomorphic if and only if their $\overline{\O}$-multiranks are the same.
\end{prop}

From now on, we assume that $(\overline{\O},*)$ has no simple factors of type D.  Let $(V,\langle\cdot,\cdot\rangle)$ be a symplectic $\overline{\O}$-module.  Then to $V$ we can associate the algebraic group
\begin{equation*}
\overline{G}(A)=\{(g,a)\in\End_{\O\otimes A}(V\otimes A)\times A^\times\mid \langle gv,gw\rangle=a\langle v,w\rangle\}
\end{equation*}

\begin{prop}\label{P:hequiv}
Assume $(\overline{\O},*)$ has no simple factors of type D.  Then there is a bijection between the set of symplectic $\overline{\O}$-modules up to isomorphism and tuples $(h_{[\tau]})$ such that 
\begin{enumerate}
\item $h_{[\tau]}$ is even if $[\tau]$ corresponds to a factor of $\overline{\O}$ of type C.
\item $h_{[\tau]}=h_{[\tau]*}$ for each $[\tau]$.
\end{enumerate}
The bijection is given by sending $(V,\langle\cdot,\cdot\rangle)$ to the $\overline{\O}$ multi rank of $V\otimes_{\FF_p} k$.
\end{prop}
\begin{proof}
To see that a symplectic $\overline{\O}$-module $(V,\langle\cdot,\cdot\rangle)$ is determined up to isomorphism by the $V\otimes k$ multi rank, note that because $(\overline{\O},*)$ has no factors of type D, the algebraic group $\overline{G}$ is connected, and so the result follows by Lang's theorem and Proposition \ref{P:algclosedunique}.
\end{proof}

\begin{prop}\label{P:dequiv}
Let $(V,\langle\cdot,\cdot\rangle)$ be a Let $k'/k$ be an extension which is either a finite field or an algebraically closed field.  Then two $\O$-stable maximal isotropic subspaces of $V$ are in the same $\overline{G}(k')$ orbit if and only if they have the same $\overline{\O}$ multirank.  Moreover, a tuple $(d_\tau)$ occurs as the $\overline{\O}$-multirank of such a maximal isotropic if and only if it satisfies
\begin{equation*}
d_\tau+d_{\tau*}=h_{[\tau]}
\end{equation*}
for all $\tau\in\Ca{T}$.
\end{prop}

The failure of these two propositions when $(\overline{\O},*)$ has simple factors of type D is one of the reasons for excluding this case in this thesis.

Now we come to the main definition of this section.
\begin{defn}\label{D:modppel}
By a mod $p$ PEL datum we mean a semisimple $\FF_p$-algebra with involution $(\overline{\O},*)$ along with one of the following equivalent sets of data (by Propositions \ref{P:hequiv} and \ref{P:dequiv}.)
\begin{enumerate}
\item A symplectic $\overline{\O}$-module along with a $G(k)$ orbit of maximal isotropic $\overline{\O}$-submodules $N\subset V\otimes k$.
\item A pair of tuples of integers $(h_{[\tau]})$ and $(d_\tau)$ satisfying
\begin{equation*}
d_\tau+d_{\tau*}=h_{[\tau]}
\end{equation*}
for all $\tau\in\Ca{T}$.  (Note that this condition implies the two conditions on $(h_{[\tau]})$ in Proposition \ref{P:hequiv}.)
\end{enumerate}
\end{defn}

\begin{defn}\label{D:modpcons}
Given an integral PEL datum $(\overline{\O},*,L,\langle\cdot,\cdot\rangle,h)$ with no factors of type D such that $p$ is a good prime, we define a mod $p$ PEL datum as follows:
\begin{itemize}
\item Take $\overline{\O}=\O\otimes\FF_p$, which is a semisimple $\FF_p$-algebra as $p$ is a good prime.  Take $*$ to be the induced involution.  $(\overline{\O},*)$ has no simple factors of type D because $(\O,*)$ does.
\item Take $(V,\langle\cdot,\cdot\rangle)$ to be $L\otimes\FF_p$, and the pairing induced by $\langle\cdot,\cdot\rangle$ after picking a choice of an isomorphism $\Z(1)\simeq \Z$ and reducing mod $p$.  The resulting pairing on $V$ is non degenerate because $p$ is a good prime.
\item Take $N\subset V\otimes k$ to be a maximal isotropic with the same $\overline{\O}$-multirank as $L_0\otimes_R k$.
\end{itemize}
\end{defn}

We finish this section by introducing some notation that will be used in Section \ref{S:moonen}.  Let $\Ca{D}$ be a mod $p$ PEL datum.  Let $\overline{G}$ be the associated group as defined above.  Fix a maximal torus and Borel $T\subset B\subset G$ so that we get a set of simple roots $\Delta$.  Let $P_N\subset G$ be the parabolic fixing $N\subset V\otimes k$ and let $I\subset \Delta$ be the corresponding set of simple roots.  Let $W$ be the Weyl group of $G$, and let $W_I\subset W$ be the parabolic subgroup generated by the simple reflections in $I$.  Let $l$ denote the length function on $W$.  Let $W^I$ denote the set of minimal length coset representatives for $W/W_I$, so that for each $w\in W^I$ we have
\begin{equation*}
l(ww')>l(w)\qquad\forall w'\in W_I.
\end{equation*}

\section{The Canonical Filtration}\label{S:canfil}

\subsection{Definition and Basic Properties}

Let $G,G'$ be a quasi-finite, flat, separated $S$-group schemes and let $f:G\to G'$ be a homomorphism.  In general, $\ker f$ exists as a quasi-finite, separated $S$-group scheme, but it need not be flat.  Meanwhile $\im f$ needn't even be representable.  However we recall the following crucial fact: if $\ker f$ is in fact finite and flat then $\im f$ is representable by a separated, quasi-finite, flat closed subgroup scheme of $G'$ which is finite if $G$ is.

Let us introduce some definitions.
\begin{defn}
Let $G$ be a partial $\BT$ and let $H\subset G$ is a finite flat closed subgroup scheme.  
\begin{enumerate}
\item If $F^{-1}(H^{(p)})\subset G$ is finite flat then we denote it (abusively) by $F^{-1}(H)$ and say that ``$F^{-1}(H)$ exists.''  If $\im (V:H^{(p)}\to G)$ exists as a finite flat group scheme then we denote it by $V(H)$ and say ``$V(H)$ exists.''  As $G$ itself might not be finite, we also let $F^{-1}(G)=G$ and $V(G)=G[F]$.  This is consistent with the case that $G$ is finite and the above definitions apply.
\item Let $\Ca{R}$ be the set of words in the symbols $F^{-1}$ and $V$.  Given $R=R_1\cdots R_n\in\Ca{R}$ where each $R_i$ is either $F^{-1}$ or $V$, we say that $R(H)$ exists if $R_n(H),R_{n-1}R_n(H),$ $\ldots,R_1R_2\cdots R_n(H)$ all exist.  By the convention above, we may also make sense of $R(G)$, even if $G$ is not finite.
\end{enumerate}
\end{defn}

We now observe the following easy but crucial fact.

\begin{lem}\label{L:fil}
If $R,R'\in\Ca{R}$ and $R(G[F])$ and $R'(G[F])$ both exist, then either $R(G[F])\subset R'(G[F])$ or $R'(G[F])\subset R(G[F])$.
\end{lem}
\begin{proof}
For any finite flat subgroup $H\subset G$, we have $G[F]\subset F^{-1}(H)$ and $V(H)\subset G[F]$ provided they exist.  This implies the result if one of $R$ or $R'$ is empty.

Otherwise, write $R=R_1\tilde{R}$ and $R'=R_1'\tilde{R}'$ with $R_1$ and $R_1'$ each either $F^{-1}$ or $V$.  If $R_1=R_1'$ the result follows by induction.  If not, then without loss of generality $R_1=V$ and $R_1'=F^{-1}$.  But then
\begin{equation*}
R(G[F])\subset G[F]\subset R'(G[F]).
\end{equation*}
\end{proof}

\begin{defn}
Let $G/S$ be a partial $\BT$.  We say that $G$ \emph{admits a canonical filtration} if for each $R\in\Ca{R}$, $R(G[F])$ exists.  If $G$ admits a canonical filtration then we say that it has \emph{constant type} if for each $R\in\Ca{R}$, the (locally constant) height of $R(G[F])$ is constant on $S$.
\end{defn}

Let us explain the definition.  Suppose that $G/S$ is a partial $\BT$ which admits a canonical filtration of constant type.  Then by \ref{L:fil}, the subgroups of the form $R(G[F])$ for $R\in\Ca{R}$ form a filtration of $G$.  Moreover, if two groups $R(G[F])$ and $R'(G[F])$ have the same (constant) height, they must be equal.  As the height of any $R(G[F])$ is bounded by the height of any of the fibers $G_s$, we see that the set $\{R(G[F])\mid R\in\Ca{R}\}$ must in fact be finite.  Ordering them by inclusion, and adding $0$ and $G$ if necessary, we arrive at a filtration
\begin{equation*}
0=G_0\subset G_1\subset\cdots\subset G_c=G[F]\subset \cdots\subset G_n=G
\end{equation*}
of $G$ by finite flat closed subgroup schemes.  It is called the \emph{canonical filtration.}  By construction it has the following property: for $i=0,\ldots,n$ both $F^{-1}(G_i)$ and $V(G_i)$ exist and are terms in the filtration.  Moreover, it is the coarsest filtration of $G$ with this property.

Let us continue to assume that $G/S$ is a partial $\BT$ which admits a canonical filtration of constant type.  For $i=1,\ldots,n$, $G_{i-1}\subset G_i$ is a finite flat closed subgroup scheme, and hence we may form the quotient $G_i/G_{i-1}$, which is separated, quasi-finite, flat, and even finite expect possibly when $i=n$ (if $G$ itself is not finite.)  Our next goal is to study how $F$ and $V$ behave on the ``associated gradeds" of the canonical filtration.

The following theorem, due to Ekedahl and Oort, summarizes the main properties of the canonical filtration.
\begin{thm}\label{T:strcan}
Let $G/S$ be a partial $\BT$ which admits a canonical filtration of constant type.
\begin{enumerate}
\item For $i=1,\ldots c$, there exists some $1\leq j\leq n$ with $G_i=V(G_j)$.  Let $\sigma(i)$ be the smallest such $j$.  Then $V(G_{\sigma(i)-1})=G_{i-1}$ and
\begin{equation*}
V:(G_{\sigma(i)}/G_{\sigma(i)-1})^{(p)}\to G_i/G_{i-1}
\end{equation*}
is an isomorphism.
\item For $i=c+1,\ldots,n$, there exists some $1\leq j\leq n$ with $G_i=F^{-1}(G_j)$.  Let $\sigma(i)$ be the smallest such $j$.  Then $F^{-1}(G_{\sigma(i)-1})=G_{i-1}$ and
\begin{equation*}
F:G_i/G_{i-1}\to (G_{\sigma(i)}/G_{\sigma(i)-1})^{(p)}
\end{equation*}
is an isomorphism.
\item The map $\sigma:\{1,\ldots,n\}\to\{1,\ldots,n\}$ defined in parts 1 and 2 is a bijection and satisfies
\begin{equation*}
\sigma(1)<\sigma(2)<\cdots<\sigma(c)
\end{equation*}
and
\begin{equation*}
\sigma(c+1)<\sigma(c+2)<\cdots<\sigma(n)
\end{equation*}
\end{enumerate}
\end{thm}
\begin{proof}
We first prove the first sentences of parts 1 and 2.  We have that $F^{-1}(G_n)=G_n$ and $V(G_n)=G_c$.  For $i=1,\ldots,n-1$ it follows from the definition of the canonical filtration that $G_i=R(G[F])$ for some $R\in\Ca{R}$ and so there exists $j$ such that either $G_i=F^{-1}(G_j)$ or $G_i=V(G_j)$.  But if $i<c$ then $G_i\subset G_c=G[F]$ so we must have $G_i=V(G_j)$, and if $i>c$ then $G[F]\subset G_i$ so we must have $G_i=F^{-1}(G_j)$.

Now if $1\leq i<i'\leq c$ then $V(G_{\sigma(i)})=G_i\subset V(G_{\sigma(i')})= G_{i'}$.  Thus $G_{\sigma(i')}\not\subset G_{\sigma(i)}$ so we must have $G_{\sigma(i)}\subset G_{\sigma(i')}$ and hence $\sigma(i)<\sigma(i')$.  Thus for $i=1,\ldots,c$, $\sigma(i-1)\leq\sigma(i)-1$ and hence 
\begin{equation*}
G_{i-1}=V(G_{\sigma(i-1)})\subset V(G_{\sigma(i)-1}).
\end{equation*}
Thus if $V(G_{\sigma(i)-1})=G_j$ then $j\geq i-1$.  But also $j<i$ by the definition of $\sigma(i)$, and hence $j=i-1$.  Consequently we have a map
\begin{equation*}
V:(G_{\sigma(i)}/G_{\sigma(i)-1})^{(p)}\to G_i/G_{i-1}.
\end{equation*}
which is surjective.

Similarly if $c<i<i'\leq n$ then $F^{-1}(G_{\sigma(i)})=G_i\subset F^{-1}(G_{\sigma(i')})=G_{i'}$.  Thus $G_{\sigma(i')}\not\subset G_{\sigma(i)}$ and so we must have $G_{\sigma(i)}\subset G_{\sigma(i')}$ and hence $\sigma(i)<\sigma(i')$.  Thus for $i=c+1,\ldots,n$, $\sigma(i-1)\leq \sigma(i)-1$ and hence
\begin{equation*}
G_{i-1}=F^{-1}(G_{\sigma(i-1)})\subset F^{-1}(G_{\sigma(i)-1}).
\end{equation*}
Thus if $F^{-1}(G_{\sigma(i)-1})=G_j$ then $j\geq i-1$.  But also $j<i$ by the definition of $\sigma(i)$, and hence $j=i-1$.
Consequently we have a map
\begin{equation*}
F:G_i/G_{i-1}\to (G_{\sigma(i)}/G_{\sigma(i)-1})^{(p)}.
\end{equation*}
which is injective.

Now for $i=1,\ldots,n$ we have a short exact sequence
\begin{equation*}
0\to F^{-1}(G_i)/F^{-1}(G_{i-1})\overset{F}{\to} G_i/G_{i-1}\overset{V}{\to} V(G_i)/V(G_{i-1})\to 0.
\end{equation*}
The $F$ is non zero if and only if have $i=\sigma(j)$ where $G_j=F^{-1}(G_i)$.  Likewise $V$ is non zero if and only $i=\sigma(j)$ where $G_j=V(G_i)$.  One of $F$ or $V$ must be non zero, as $G_i/G_{i-1}$ is.  Hence $\sigma$ is surjective.  Consequently it is also injective.  We conclude that in the above exact sequence exactly one of $F$ or $V$ is zero, and the other is an isomorphism.
\end{proof}

Next we deduce something about the structure of the sub quotients for the canonical filtration.

\begin{prop}
Let $G/S$ be a partial $\BT$ which admits a canonical filtration.  If $1\leq i\leq n$ satisfies $\sigma(i)=i$ then
\begin{enumerate}
\item If $i\leq c$ then $i=1$, in which case $G_1$ is of multiplicative type.
\item If $i>c$ then $i=n$, in which case $G_n/G_{n-1}$ is \'{e}tale.
\end{enumerate}
If $1\leq i\leq n$ is such that $\sigma(i)\not=i$ then $G_i/G_{i-1}$ is an $\alpha$-group (i.e. both $F$ and $V$ are 0.)
\end{prop}
\begin{proof}
Let us suppose that $\sigma(i)=i$.  Let us first consider the possibility that $i\leq c$.  Suppose $i>1$.  From the definition of the canonical filtration, we must have $G_{i-1}=V^r(G_j)$ for some $j\geq c$ and some $r\geq0$.  But $G_i\subset G_j$ and hence
\begin{equation*}
G_i=V^r(G_i)\subset V^r(G_j)=G_{i-1}
\end{equation*}
a contradiction.  Hence we must have $i=1$.  Then by Theorem \ref{T:strcan} we have that
\begin{equation*}
V:G_1^{(p)}\to G_1
\end{equation*}
is an isomorphism, and hence $G_1$ is multiplicative.

Now let us consider the case that $i>c$.  Suppose $i<n$.  From the definition of the canonical filtration, we must have $G_i=F^{-r}(G_j)$ for some $j\leq c$ and $r\geq 0$.  But $G_j\subset G_{i-1}$ and hence
\begin{equation*}
G_i=F^{-r}(G_j)\subset F^{-r}(G_{i-1})=G_{i-1}
\end{equation*}
a contradiction.  Hence we must have $i=n$.  Then by Theorem \ref{T:strcan} we have that
\begin{equation*}
F:G_n/G_{n-1}\to (G_n/G_{n-1})^{(p)}
\end{equation*}
is an isomorphism, and hence $G_n/G_{n-1}$ is \'{e}tale.

Now we prove the last statement of the proposition.  For $i=1,\ldots n$ we have $V(G_i)=G_j$ with $0\leq j\leq c$.  We clearly must have $j\leq i$.  Suppose $i=j$.  Then $j>0$ so by Theorem \ref{T:strcan} $V(G_{\sigma(j)})=G_j$ and $V(G_{\sigma(j)-1})=G_{j-1}$.  Hence we must have $G_{\sigma(j)}\subset G_i$ and so
\begin{equation*}
\sigma(j)\leq i=j
\end{equation*}
But for $1\leq j\leq c$ we have $\sigma(j)\geq j$ and hence $\sigma(j)=j$.  Thus unless $\sigma(i)=i$ we must have $i<j$.  In other words we have shown that unless $\sigma(i)=i$, the map
\begin{equation*}
V:G_i^{(p)}\to G_i
\end{equation*}
factors through $G_j\subset G_{i-1}$, and hence
\begin{equation*}
V:(G_i/G_{i-1})^{(p)}\to G_i/G_{i-1}
\end{equation*}
is 0.

Now we consider $F$.  If $0< i\leq c$ then $G_i\subset G[F]=G_c$ so certainly $F$ is 0 on $G_i/G_{i-1}$.  Hence we assume that $i>c$.  Now by Theorem \ref{T:strcan} we have $G_i=F^{-1}(G_{\sigma(i)})$.  Now $\sigma(i)\leq i$ so unless $\sigma(i)=i$ we have $\sigma(i)<i$ and hence
\begin{equation*}
F:G_i\to G_i^{(p)}
\end{equation*}
factors through $G_{\sigma(i)}^{(p)}\subset G_{i-1}^{(p)}$.  Thus
\begin{equation*}
F:G_{i}/G_{i-1}\to (G_i/G_{i-1})^{(p)}
\end{equation*}  
is 0.
\end{proof}

\begin{cor}\label{C:omegafree}
Let $G/S$ be a partial $\BT$ which admits a canonical filtration of constant type.  Then for $i=1,\ldots n$, $\omega_{G_i/G_{i-1}}$ is locally free.  It is trivial if and only if $i=n$, $n>c$ and $\sigma(n)=n$, in which case $G_i/G_{i-1}$ is \'{e}tale.  Otherwise it has rank equal to the height of $G_i/G_{i-1}$.
\end{cor}
\begin{proof}
This follows from a general fact about finite flat group schemes killed by $F$.
\end{proof}

It is certainly not true that any partial $\BT$ over a general base admits a canonical filtration.  The next theorem, also due to Ekedahl and Oort, describes how we can decompose the base $S$ into locally closed subschemes in such a way that the restriction of $G$ to each piece admits a canonical filtration.  This decomposition will ultimately give the construction of the Ekedahl-Oort stratification for the special fibers of Siegel modular varieties and their toroidal compactifications.

In preparation we record the following well known lemma.
\begin{lem}\label{L:semiflat}
Let $f:X\to S$ be finite map.  And let
\begin{align*}
S&\to\Bf{N}\\
s&\mapsto\deg X_s
\end{align*}
be the degree map.  Then
\begin{enumerate}
\item The function $d$ is upper semicontinuous on $S$.
\item If $d$ is constant and $S$ is reduced then $f$ is flat.
\end{enumerate}
\end{lem}

\begin{thm}\label{T:candecomp}
Let $G/S$ be a partial $\BT$ of height $\leq h$.  Then there is a coarsest set theoretic decomposition
\begin{equation*}
S=\coprod_\alpha S_\alpha
\end{equation*}
into finitely many reduced locally closed subschemes such that for each $\alpha$, $G|{S_\alpha}$ admits a canonical filtration of constant type.
\end{thm}

\begin{proof}
We will construct a decomposition $S=\coprod_\alpha S_\alpha$ into reduced locally closed subschemes such that the following two properties hold
\begin{enumerate}
\item For each $\alpha$, $G|_{S_\alpha}$ admits a canonical filtration.
\item Two points $s,s'\in S$ lie in the same $S_\alpha$ if and only if $R(G_s[F])$ and $R(G_{s'}[F])$ have the same height for all $R\in\Ca{R}$.
\end{enumerate}
It is clear that such a stratification satisfies the conditions of the theorem.

We make the following preliminary observation: if $R\in\Ca{R}$ and $R(G[F])$ exists, then the fibers of $R(G[F])$ consist (set theoretically) of single points.  Indeed if this is the case for some finite flat $H\subset G$, then the same is true for $V(H)$ and $F^{-1}(H)$ if they exist.

Suppose we have $G/S$ a partial $\BT$ and $R\in\Ca{R}$ such that $R(G[F])$ exists and has constant height.  Let $R_1$ be either $V$ or $F^{-1}$.  As a step in the construction of the decomposition in the theorem, we will explain how to decompose $S=\coprod S_\alpha$ into into finitely many reduced locally closed subschemes such that
\begin{enumerate}
\item For each $\alpha$, $R_1(R(G[F])|_{S_\alpha})$ exists.
\item Two points $s,s'\in S$ lie in the same $S_\alpha$ if and only $R_1R(G_s[F])$ and $R_1R(G_{s'}[F])$ have the same height.
\end{enumerate}

First suppose $R_1=F^{-1}$.  Then $F^{-1}R(G[F])$ is quasi-finite and separated, but not necessarily flat.  But by Lemma \ref{L:gpfinite} it is in fact finite.  Thus by Lemma \ref{L:semiflat} we see that the subset $S_n\subset S$ of points $s$ such that $\text{ht}(F^{-1}R(G[F]))_s=n$ is locally closed.  Give it the reduced induced subscheme structure.  Then Lemma \ref{L:semiflat} again implies that $F^{-1}R(G[F])|_{S_n}$ is flat.

Now we argue similarly when $R_1=V$.  In this case, consider $H=\ker(V: (R(G[F]))^{(p)}\to G)$.  Then the subset $S_n\subset S$ of points $s$ such that $\text{ht}(H_s)=n$ is locally closed, and if we give it the reduced induced subscheme structure, $H|_{S_n}$ is flat by \ref{L:semiflat}, and hence $V(R(G[F])|_{S_n})$.  Again there are at most $h+1$ values of $n$ for which $S_n$ is nonempty, and so we have the desired decomposition.

Let $\Ca{R}_n\subset \Ca{R}$ be the set of words of length at most $n$.  Then by iterating the above two steps we may construct a decomposition $S=\coprod_\alpha S_\alpha$ into finitely many reduced locally closed subschemes such that the following two properties hold:
\begin{enumerate}
\item For each $\alpha$, $G|_{S_\alpha}$ admits a canonical filtration.
\item Two points $s,s;\in S$ lie in the same $S_\alpha$ if and only if $R(G_s[F])$ and $R(G_{s'}[F])$ have the same height for all $R\in\Ca{R}_n$.
\end{enumerate}

To complete the proof we need to show that this decomposition is independent of $n$ for $n$ sufficiently large.  But in fact, this is the case for $n\geq h$ by Lemma \ref{L:fil}.
\end{proof}

\begin{rem}
Note that we make no attempt to give a ``scheme theoretic'' definition of the decomposition in the theorem.  In particular one might as well assume that the base $S$ in the theorem is reduced.  In fact, we could have put scheme structures on the $S_\alpha$ by at each step, considering the flattening stratification of the relevant finite, but not necessarily flat group scheme.  However, we don't know of any application of this.
\end{rem}

Finally we say something about how the canonical filtration behaves under base change.
\begin{prop}\label{P:canfilbc}
Let $G/S$ be a partial $\BT$ and let $f:S'\to S$ be any morphism with $S'$ nonempty.
\begin{enumerate}
\item If $G$ admits a canonical filtration
\begin{equation*}
0=G_0\subset G_1\subset\cdots\subset G_c=G[F]\subset\cdots\subset G_n=G
\end{equation*}
then $G_{S'}$ admits a canonical filtration which is
\begin{equation*}
0=(G_0)_{S'}\subset (G_1)_{S'}\subset\cdots\subset (G_c)_{S'}=G_{S'}[F]\subset\cdots\subset (G_n)_{S'}=G_{S'}
\end{equation*}
except that it may may happen that $(G_{n-1})_{S'}=(G_n)_{S'}$ in which case the canonical filtration of $G_{S'}$ ends at $(G_{n-1})_{S'}$.  This never happens if $G$ is a $\BT$.
\item Suppose $G$ has bounded height.  If $S=\coprod_\alpha S_\alpha$ is the decomposition of $S$ into reduced locally closed subschemes as in Theorem \ref{T:candecomp} then the decomposition of $S'$ for $G_{S'}$ is
\begin{equation*}
S'=\coprod_\alpha f^{-1}(S_\alpha)
\end{equation*}
where the union is over those $\alpha$ for which $f^{-1}(S_\alpha)$ is nonempty, and the locally closed set $f^{-1}(S_\alpha)$ is given its reduced induced subscheme structure.
\end{enumerate}
\end{prop}
\begin{proof}
Both parts follow from the fact that $F$, $V$, and the formation of images and inverse images (when they exist as finite flat group schemes) are all compatible with base change.
\end{proof}

Let us try to demystify the exception in part 1 of the above proposition through a simple and typical example.
\begin{exmp}
Let $E/\FF_p[[q]]$ be the semiabelian extension of the Tate curve.  Let $G=E[p]$ be its $p$-torsion subgroup.  This is a partial $\BT$ which admits a two step canonical filtration
\begin{equation*}
0\subset \mu_p\subset G
\end{equation*}
where $G/\mu_p$ is a quasi-finite \'{e}tale group with trivial special fiber.  On the other hand the special fiber $G_{\FF_p}$ is just $\mu_p$, and hence it only has a one step canonical filtration.  Of course this is because the \'{e}tale quotient $G/\mu_p$ has trivial special fiber.
\end{exmp}
As we will see below, in the presence of a principal quasi-polarization, partial $\BT$ ``remembers'' whether or not it is missing an \'{e}tale part and so we can fix this defect in Convention \ref{convpartial} below.

\subsection{The canonical Filtration of a Principally Quasi-Polarized partial $\BT$}

In this section we will study the canonical filtration in the presence of a principal quasi-polarization.

\begin{defn}
Let $(G,\lambda)/S$ be a principally quasi-polarized $\BT$.  Let $H\subset G$ be closed finite flat subgroup scheme.  Then we let
\begin{equation*}
H^\perp=\lambda^{-1}(\ker(G^D\to H^D)).
\end{equation*}
The antisymmetry of $\lambda$ implies that $(H^\perp)^\perp=H$.
\end{defn}

Next we observe the following easy lemma:
\begin{lem}\label{L:fvperp}
Let $(G,\lambda)/S$ be a principally quasi-polarized $\BT$ and let $H\subset G$ be a finite flat subgroup scheme.  Then
\begin{enumerate}
\item $F^{-1}(H)$ exists if and only if $V(H^\perp)$ does, in which case
\begin{equation*}
(F^{-1}(H))^\perp=V(H^\perp).
\end{equation*}
\item $V(H)$ exists if and only if $F^{-1}(H^\perp)$ does, in which case
\begin{equation*}
(V(H))^\perp=F^{-1}(H^\perp).
\end{equation*}
\end{enumerate}
\end{lem}

\begin{proof}
Indeed, more generally let $(G,\lambda)$ and $(G',\lambda')$ be finite flat group schemes equipped with isomorphisms $\lambda:G\to G^D$ and $\lambda':G'\to (G')^D$ and $f:G\to G'$ is such that $\lambda=f^D\lambda'f$.  Then if $H\subset G$ is a finite flat subgroup scheme then $f(H):=\im(f:H\to G')$ exists as a finite flat if and only $f^{-1}(H^\perp)$ is finite flat, in which case
\begin{equation*}
(f(H))^\perp=f^{-1}(H^\perp).
\end{equation*}
The lemma follows from this and the fact that cartier duality exchanges Frobenius and Verschiebung.
\end{proof}

\begin{prop}\label{P:candual}
Let $(G,\lambda)$ be a principally quasi-polarized $\BT$ which admits a canonical filtration of constant type.  Let
\begin{equation*}
0=G_0\subset G_1\subset \cdots\subset G_c=G[F]\subset\cdots\subset G_n=G.
\end{equation*}
be the canonical filtration.  Then $n=2c$ and the filtration is self dual in the sense that for $i=0,\ldots,2c$
\begin{equation*}
G_i=G_{2c-i}^\perp.
\end{equation*}
and for $0\leq i< j\leq 2c$, $\lambda$ induces an isomorphism
\begin{equation*}
\lambda:G_j/G_i\simeq (G_{2c-i}/G_{2c-j})^D
\end{equation*}
Moreover the permutation $\sigma:\{1,\ldots,2c\}\to\{1,\ldots,2c\}$ satisfies
\begin{equation*}
\sigma(2c+1-i)=2c+1-\sigma(i)
\end{equation*}
\end{prop}

\begin{proof}
For any $R\in\Ca{R}$, let $R'\in\Ca{R}$ denote the element obtained by exchanging $F^{-1}$ and $V$.  Then as $G[F]=G[F]^\perp$, Lemma \ref{L:fvperp} implies that $R(G[F])^\perp=R'(G[F])$.  Hence the terms of the canonical filtration are stable under $H\mapsto H^{\perp}$.  As this is inclusion reversing, we conclude that we must have $n=2c$ and $G_i=G_{2c-i}^\perp$.  The second two statements follow immediately.
\end{proof}

Now we want to explain how to extend Proposition \ref{P:candual} to principally quasi-polarized \emph{partial} $\BT$s.  First we have to deal with a small notational inconvenience.  Note that if $(G,\lambda)$ is a principally quasi-polarized $\BT$, the by the previous proposition, if $G_1$ is multiplicative, then $G_n/G_{n-1}$, its Cartier dual, is \'{e}tale.  On the other hand $(G_{n-1},\lambda)$ is a principally quasi-polarized such that $G_1$ is multiplicative, but such that the ``top'' sub quotient in the canonical filtration, $G_{n-1}/G_{n-2}$, is not \'{e}tale.
\begin{conv}\label{convpartial}
If $(G,\lambda)$ is a principally quasi-polarized partial $\BT$ with $G_1$ multiplicative but $G_n/G_{n-1}$ not \'{e}tale, then we will increment $n$ by one and extend the canonical filtration so that $G_{n-1}=G_n$ and $G_n/G_{n-1}$ is \'{e}tale (and trivial.)
\end{conv}
As remarked above, this convention makes no change when $(G,\lambda)$ is a principally quasi-polarized $\BT$ and the reader may verify that with this new canonical filtration all the results of the previous section remain trivially valid.  Moreover with this convention, the canonical filtration is compatible with base change without the exception in part 1 of Proposition \ref{P:canfilbc}.

Now we have the following analog of Proposition \ref{P:candual}
\begin{prop}
Let $(G,\lambda)$ be a principally quasi-polarized partial $\BT$ which admits a canonical filtration of constant type.  Let
\begin{equation*}
0=G_0\subset G_1\subset \cdots\subset G_c=G[F]\subset\cdots\subset G_n=G.
\end{equation*}
be the canonical filtration, observing Convention \ref{convpartial}.  Then $n=2c$ and for $i=1,\ldots 2c-1$, $G_i$ and $G_{2c-i}$ are orthogonal for $\lambda$ and for $1\leq i< j\leq 2c-1$, $\lambda$ induces an isomorphism
\begin{equation*}
\lambda:G_j/G_i\simeq (G_{2c-i}/G_{2c-j})^D
\end{equation*}
Moreover the permutation $\sigma:\{1,\ldots,2c\}\to\{1,\ldots,2c\}$ satisfies
\begin{equation*}
\sigma(2c+1-i)=2c+1-\sigma(i)
\end{equation*}
\end{prop}

\subsection{The Canonical Filtration of a partial $\BT$ With Extra Endomorphisms}

We begin with the following important observation: if $(G,i)$ is a partial $\BT$ with $\overline{\O}$-action and $R\in\Ca{R}$ is such that $R(G[F])$ exists, then $R(G[F])$ is stable under the action of $\overline{\O}$.  Indeed this follows immediately from the functoriality of $F$ and $V$.  Thus if $G$ admits a canonical filtration
\begin{equation*}
0=G_0\subset G_1\subset\cdots\subset G_c=G[F]\subset\cdots\subset G_n=G
\end{equation*}
it must be stable by $\overline{\O}$.  In particular $\overline{\O}$ acts on the sub quotients $G_i/G_{i-1}$ and their co-lie algebras $\omega_{G_i/G_{i-1}}$ (on the right) which are locally free $\O_S$-modules by Corollary \ref{C:omegafree}.

\begin{defn}\label{D:otype}
Let $(G,i)/S$ be a partial $\BT$ with $\overline{\O}$-action.  We say that $(G,i)$ admits a canonical filtration with constant $\overline{\O}$-type if $G$ admits a canonical filtration of constant type
\begin{equation*}
0=G_0\subset G_1\subset\cdots\subset G_c=G[F]\subset\cdots\subset G_n=G
\end{equation*}
and moreover, for each $i=1,\ldots,n$ the $\overline{\O}$ multi rank of $\omega_{G_i/G_{i-1}}$ is constant.  
\end{defn}

The following theorem will ultimately lead to the construction of the Ekedahl-Oort stratification on PEL type modular varieties and their toroidal compactifications.
\begin{thm}\label{T:candecompO}
Let $(G,i)$ be a partial $\BT$ with $\overline{\O}$ action with height $\leq h$.  Then there is a coarsest (set theoretic) decomposition
\begin{equation*}
S=\coprod_\beta S_\beta 
\end{equation*}
into finitely many reduced locally closed subschemes such that for each $\beta$, $G|_{S_{\beta}}$ admits a canonical filtration of constant $\overline{\O}$-type.  Moreover if $S=\coprod_\alpha S_\alpha$ is the decomposition from Theorem \ref{T:candecomp} then for each $\beta$ there is some $\alpha$ such that $S_\beta$ is an open and closed subscheme of $S_\alpha$.
\end{thm}
\begin{proof}
It is clear that the decomposition $S=\coprod_\beta S_\beta$ in the theorem must refine the decomposition $S=\coprod_\alpha S_\alpha$ of Theorem \ref{T:candecomp}.  So consider one of the locally closed subschemes $S_\alpha$ from Theorem \ref{T:candecomp}.  Then $G|_{S_\alpha}$ admits a canonical filtration
\begin{equation*}
0=G_0\subset G_1\subset\cdots\subset G_c=G[F]\subset\cdots\subset G_n=G|_{S_\alpha}.
\end{equation*}
Consider for each $i$, the $\overline{\O}$ multi ranks of $\omega_{G_i/G_{i-1}}$.  These are locally constant functions on $S_\alpha$, and consequently there is a coarsest decomposition of $S_\alpha$ into open and closed subschemes on which they are all constant.  This gives the decomposition in the theorem.
\end{proof}

\section{Classification Over an Algebraically Closed Field}\label{S:moonen}

The goal of this section is to recall the classification of principally quasi-polarized $\BT$s with $\overline{\O}$-action over an algebraically closed field.  This is due to Oort \cite{Oo01} when $\overline{\O}=\FF_p$ and Moonen \cite{Mo01} in general (at least when $p>2$.)

Let $\Ca{D}=(\overline{\O},*,(h_{[\tau]}),(d_\tau))$ be a mod $p$ PEL datum with no factors of type D.  Let $k'$ be an algebraically closed field of characteristic $p$, with an embedding $k\to k'$.  Let ${\BT}^\Ca{D}_{/k'}$ denote the set of isomorphism classes of principally quasi polarized $\BT$ with $\overline{\O}$ action of type over $k'$ of type $\Ca{D}$.

Let $(V,\langle\cdot,\cdot\rangle)$ be the symplectic $\Ca{O}$-module corresponding to $\Ca{D}$ and let $\overline{G}$ be the corresponding group.

Let $G$ be an element of ${\BT}^\Ca{D}_{/k'}$.  $G$ admits a canonical filtration
\begin{equation*}
0=G_0\subset\cdots\subset G_c=G[F]\subset\cdots\subset G_{2c}=G.
\end{equation*}
Now let $D=D(G)$ be the contravarient Dieudonne module of $G$.  It is a $k'$ vector space with a $k'$-linear right action of $\Ca{O}$, a pairing $\langle\cdot,\cdot\rangle$ and an $F$ and $V$.  Let $h_{\tau}$ be its multi degree.  Then by the proof of Proposition we have $h_{\tau}=h_{[\tau]}$ for each $\tau\in\Ca{T}$.

We now define two flags on $D(G)$.  The first is
\begin{equation*}
0\subset \ker F\subset D
\end{equation*}
and let $P_N\subset G_{k'}$ be the parabolic stabilizing this flag.  We have $\dim_{k'}(\ker F)_\tau=d_\tau$ and moreover $\ker F$ is maximal isotropic.  Hence $P_N$ has type $I$.

The other flag comes from the canonical filtration.  Let
\begin{equation*}
D_i=\ker(D(G)\to D(G_{2c-i}))
\end{equation*}
so that we have an $\overline{\O}$-stable, self dual, flag
\begin{equation*}
0=D_0\subset D_1\subset\cdots\subset D_c=\im F\subset\cdots\subset D_{2c}=D.
\end{equation*}
Let $P_C\subset G_{k'}$ be the parabolic which stabilizes it.  Then following Moonen, we define
\begin{equation*}
w(G)=\text{relpos}(P_C,P_N)\in W^I.
\end{equation*}

The following theorem is due to Oort \cite{Oo01} when $\overline{\O}=\FF_p$, Moonen \cite{Mo01} for $\overline{\O}$ general and $p>2$, and Moonen-Wedhorn \cite{MW04} in general.  The formulation given here is due to Moonen.
\begin{thm}
The map defined above gives a bijection
\begin{align*}
{\BT}^\Ca{D}_{/k'}&\to W^I\\
G&\mapsto w(G)
\end{align*}
\end{thm}

Now we want to connect this classification with the notions introduced in section \ref{S:canfil}.  Here is the key proposition.
\begin{prop}\label{P:candetermines}
Let the notation be as above.  Then $w(G)$, and hence $G$ up to isomorphism, is determined by the permutation $\sigma$ as in theorem \ref{T:strcan} and the $\overline{\O}$-multiranks of $\omega_{G_i/G_{i-1}}$ for $i=0,\ldots,2c$.
\end{prop}
\begin{proof}
The relative position $w(G)$ is determined by the $\overline{\O}$-multiranks of

\begin{equation*}
\ker F\cap D_i
\end{equation*}
for $i=0,\ldots,2c$.

By Theorem \ref{T:strcan}, translated into the language of Dieudonne modules, we have that for $i=1,\ldots,c$,
\begin{equation*}
F:(D_{\sigma(i)}/D_{\sigma(i)-1})^{(p)}\to D_i/D_{i-1}
\end{equation*}
is an isomorphism, while for $i=c+1,\ldots,2c$,
\begin{equation*}
V:D_i/D_{i-1}\to (D_{\sigma(i)}/D_{\sigma(i)-1})^{(p)}
\end{equation*}
is an isomorphism.  Moreover we have canonical isomorphism
\begin{equation*}
\omega_{G_{2c+1-i}/G_{2c-i}}\simeq D_i/(D_{i-1}+FD_i)
\end{equation*}
and we have $FD_i\subset D_{i-1}$ except when $i=1$ and $\sigma(1)=1$.

Now for $i=1,\ldots,2c$ there are two possibilities.  If $\sigma^{-1}(i)\leq c$ then as we have an isomorphism
\begin{equation*}
F:(D_i/D_{i-1})^{(p)}\to D_{\sigma^{-1}(i)}/D_{\sigma^{-1}(i-1)}
\end{equation*}
we have
\begin{equation*}
\ker F\cap D_i=\ker F\cap D_{i-1}.
\end{equation*}
On the other hand when $\sigma^{-1}(i)>c$ then we have an isomorphism
\begin{equation*}
V:D_{\sigma^{-1}(i)}/D_{\sigma^{-1}(i)-1}\to(D_i/D_{i-1})^{(p)}
\end{equation*}
As $\im V=\ker F$ it follows that the $\overline{\O}$-multirank of $\ker F\cap D_i$ is that of $\ker F\cap D_{i-1}$ plus that of $D_i/D_{i-1}$.
\end{proof}

\section{The Ekedahl-Oort Stratification of a PEL Type Modular Variety}\label{S:EO}

Now we return to the setting of PEL modular varieties.  Let $(\Ca{\O},*,L,\langle\cdot,\cdot\rangle,h)$ be an integral PEL datum with no factors of type D for which $p$ is a good prime.  Let $\Ca{D}=(\overline{\Ca{O}},*,(h_{[\tau]}),d_\tau)$ be the corresponding mod $p$ PEL datum.

Let $K\subset G(\A^{\infty,p})$ be a neat open compact subgroup.  For any $(A,\lambda,i,\alpha_K)$ in the universal isogeny class over $X_K$ we get a principally quasi polarized $\BT$ of type $\Ca{D}$ by taking the $p$ torsion $A[p]$ along with the Weil pairing $\lambda:A[p]\times A[p]\to \mu_p$ determined by $\lambda$, and the $\overline{\O}$ action induced by the action of $\overline{\O}$ on $A$.

As $K$ is neat, for any other choice $(A',\lambda',i',\alpha_K')$ in the universal isogeny class over $X_K$, there is a unique prime to $p$ quasi-isogeny $f:(A,\lambda,i,\alpha_K)\to(A',\lambda',i',\alpha_K')$ as in Definition \ref{D:modisog}.  Hence there is a canonical isomorphism $f:(A[p],i)\to (A'[p],i')$ of $\BT$ with $\overline{\O}$ action.  In this way we get a canonical $\BT$ with $\overline{\O}$-action over $X_K$, which we denote by $(G,i)$ along with a principal quasi-polarization $\lambda$ which is only canonical up to $\FF_p^\times$ similitude.

Now from Theorem \ref{T:candecompO} and Proposition \ref{P:candetermines} we get a set theoretic decomposition
\begin{equation*}
X_K=\coprod_{w\in W^I} X_{K,w}
\end{equation*}
of $X_K$ into reduced locally closed subschemes $X_{K,w}/k$ which has the following properties:
\begin{enumerate}
\item For each $w\in W^I$, $(G,i)|_{X_{K,w}}$ admits a canonical filtration of constant $\overline{\O}$ type.
\item For any $k'/k$ algebraically closed extension and $x\in X_K(k')$ let $w=w(G_x)$ be as in section \ref{S:moonen}.  Then $x\in X_{K,w}(k')$.
\end{enumerate} 
This is the \emph{Ekedahl-Oort} stratification of $X_K$.  We will also denote the Zariski closure of $X_{K,w}$ by $\overline{X}_{K,w}$.

We now list some of the basic properties of the Ekedahl-Oort stratification, due to Oort, Moonen, Wedhorn, and Wedhorn-Viehman.
\begin{thm}\label{T:EOprop}
\begin{enumerate}
\item For each $w\in W^I$, $X_{K,w}$ is nonempty, smooth, and of dimension $l(w)$.
\item There is a partial order $\preceq$ on $W^I$, which we emphasize is not necessarily the Bruhat order, such that (set theoretically)
\begin{equation*}
\overline{X}_{K,w}=\coprod_{w'\preceq w}X_{K,w'}.
\end{equation*}
\end{enumerate}
\end{thm}

Next we show that the Ekedahl-Oort stratification is prime to $p$ Hecke stable.
\begin{prop}\label{P:EOhecke}
Let $g\in G(\A^{\infty,p})$ and let $K,K'\subset G(\A^{\infty,p})$ be neat compact open subgroups with $g^{-1}Kg\subset K'$ so that we have a map
\begin{equation*}
[g]:X_K\to X_K'.
\end{equation*}
Then for each $w\in W^I$,
\begin{equation*}
[g]^{-1}(X_{K',w})=X_{K,w}\qquad\text{and}\qquad [g]^{-1}(\overline{X}_{K',w})=\overline{X}_{K,w}.
\end{equation*}
Note that these pullbacks can be interpreted either set theoretically or scheme theoretically as $[g]$ is \'{e}tale.
\end{prop}
\begin{proof}
If we let $G_{K'}$ (resp. $G_K$) denote the canonical $\BT$ with $\overline{\O}$ action on on $X_{K'}$ (resp. $X_K$) then by the definition of $[g]$ we have a canonical isomorphism $G_{K'}\times_{X_{K'}}X_K\simeq G_K$, and so the first result follows by Proposition \ref{P:canfilbc}.  The statement about the closures follows from the fact that $[g]$ is flat.
\end{proof}

Finally we consider the relation between the Ekedahl-Oort stratification on $X_K$ and that on a Siegel modular variety.  We recall the notation of Section \ref{S:morsieg}.  We have another PEL datum $(\Z,\text{id},L,\langle\cdot,\cdot\rangle,h)$ which defines a group $\tilde{G}$ with $G\subset\tilde{G}$ and with corresponding PEL modular varieties $\tilde{X}_{\tilde{K}}$ for $\tilde{K}\subset \tilde{G}(\A^{\infty,p})$ neat open compact.  We also get a new mod $p$ PEL datum $\tilde{\Ca{D}}$.  We note that a $\BT$ of type $\tilde{\Ca{D}}$ is just a principally quasi polarized $\BT$ of height $\dim L$.  We let $(\tilde{W},\tilde{I})$ be the corresponding Weyl group and parabolic type.

For any algebraically closed extension $k'/k$ we get a commutative diagram
\begin{equation*}
\begin{tikzcd}
{\BT}_{k'}^{\Ca{D}}\arrow{r}\arrow{d}&W^I\arrow[dotted]{d}{\theta}\\
{\BT}_{k'}^{\tilde{\Ca{D}}}\arrow{r}& \tilde{W}^{\tilde{I}}
\end{tikzcd}
\end{equation*}
where the horizontal arrows are the bijections of Section \ref{S:moonen} and the left vertical arrow is ``forget the $\overline{\O}$ action.''  The dotted map $\theta$ is defined by the commutativity of the diagram.  It is independent of $k'$.

Now for $\tilde{K}\subset\tilde{G}(\A^{p,\infty})$ neat open compact, we have an Ekedahl-Oort stratification
\begin{equation*}
\tilde{X}_{\tilde{K}}=\coprod_{\tilde{w}\in\tilde{W}^{\tilde{I}}} \tilde{X}_{\tilde{K},\tilde{w}}.
\end{equation*}

Finally for $K\subset G(\A^{\infty,p})$ and $\tilde{K}\subset \tilde{G}(\A^{\infty,p})$ neat open compacts, with $K\subset\tilde{K}$ we have a map
\begin{equation*}
\phi_{K,\tilde{K}}:X_K\to X_{\tilde{K}}
\end{equation*}
from Proposition \ref{P:morsieg}.

\begin{thm}\label{T:opensieg}
With notation as above, for each $w\in\tilde{W}^{\tilde{I}}$
\begin{equation*}
\phi_{K,\tilde{K}}^{-1}(\tilde{X}_{\tilde{K},w})=\coprod_{w\in W^I, \theta(w)=\tilde{w}}X_{K,w}
\end{equation*}
where here the disjoint union is actually as schemes, i.e. each $X_{K,w}$ is open and closed in $\phi_{K,\tilde{K}}^{-1}(\tilde{X}_{\tilde{K},w}$).
\end{thm}

\begin{proof}
Let $\tilde{G}$ be the canonical $\BT$ on $\tilde{X}_{\tilde{K}}$ and let $(G,i)$ be the canonical $\BT$ with $\overline{\O}$-action.  Then it follows from the definition of $\phi_{K,\tilde{K}}$ that $G=\tilde{G}\times_{X_{\tilde{K}}} X_K$.  The Ekedahl-Oort stratification of $\tilde{X}_{\tilde{K}}$ is the decomposition from Theorem \ref{T:candecomp} applied to $\tilde{G}/\tilde{X}_{\tilde{K}}$.  Then by Proposition \ref{P:canfilbc}, the decomposition
\begin{equation*}
X_K=\coprod_{\tilde{w}\in\tilde{W}^{\tilde{I}}} \phi_{K,\tilde{K}}^{-1}(X_{\tilde{K},\tilde{w}})
\end{equation*}
is that given by applying Theorem \ref{T:candecomp} to $G/X_{K}$, except that some of the terms may be empty.  On the other hand the Ekedahl-Oort stratification of $X_K$ is exactly the decomposition given by Theorem \ref{T:candecompO} applied to $(G,i)/X_K$.  Moreover Theorem \ref{T:candecompO} tells us that each $X_{K,w}$ must be open and closed in some $\phi_{K,\tilde{K}}^{-1}(X_{\tilde{K},\tilde{w}})$.  But by considering a geometric point of $X_{K,w}$ it follows from the definition of $\theta$ that $X_{K,w}$ must lie in $\phi_{K,\tilde{K}}^{-1}(X_{\tilde{K},\theta(w)})$.
\end{proof}

\begin{rem}
The observation that each $X_{K,w}$ is open in $\phi^{-1}_{K,\tilde{K}}(\tilde{X}_{\tilde{K},w})$ seems to be new to this thesis.  The reader might find it more surprising in light of the fact that the dimensions of the $X_{K,w}$ with $\theta(w)=\tilde{w}$ aren't even all the same.  In principle there should be a completely combinatorial proof of this fact: indeed the map $\theta:W^I\to\tilde{W}^I$ and the partial order $\preceq$ on $W^I$ both have purely combinatorial descriptions, and the problem is to show that $\theta^{-1}(\tilde{w})$ is discrete for $\preceq$ for each $\tilde{w}$.
\end{rem}

\section{Generalized Hasse Invariants}\label{S:pelhasse}
\subsection{Generalized Hasse Invariants for Partial $\BT$ With Canonical Filtration}\label{S:hassebt}

In this section let $G/S$ be a partial $\BT$ which admits a canonical filtration of constant type
\begin{equation*}
0=G_0\subset G_1\subset\cdots\subset G_c=G[F]\subset\cdots\subset G_{n}=G.
\end{equation*}

By Theorem \ref{T:strcan} and Corollary \ref{C:omegafree} there is a permutation $\sigma:\{1,\ldots,n\}\to\{1,\ldots,n\}$ such that:
\begin{enumerate}
\item For $i=1,\ldots c$ we have $V(G_{\sigma(i)})=G_i$ and $V(G_{\sigma(i)-1})=G_{i-1}$ and
\begin{equation*}
V:(G_{\sigma(i)}/G_{\sigma(i)-1})^{(p)}\to G_i/G_{i-1}
\end{equation*}
is an isomorphism.
\item For $i=c+1,\ldots,n$ we have $F^{-1}(G_{\sigma(i)})=G_i$ and $F^{-1}(G_{\sigma(i)-1})=G_{i-1}$ and
\begin{equation*}
F:G_i/G_{i-1}\to (G_{\sigma(i)}/G_{\sigma(i)-1})^{(p)}
\end{equation*}
is an isomorphism.
\item The co-lie algebra $\omega_{G_i/G_{i-1}}$ is finite locally free.  It is trivial if $i=n$ and $\sigma(n)=n$ (in which case $G_n/G_{n-1}$ is \'{e}tale) and has rank equal to the height of $G_i/G_{i-1}$ otherwise.
\end{enumerate}

Now for $i=1,\ldots,n$ and unless $i=n=\sigma(n)$ we have line bundles
\begin{equation*}
\omega_i=\det\omega_{G_i/G_{i-1}}.
\end{equation*}
on $S$.  We also have the line bundle $\omega=\det\omega_G$.  Moreover there is a natural isomorphism
\begin{equation*}
\omega=\otimes_{i=1}^c\omega_i
\end{equation*}
which comes about as follows: the filtration
\begin{equation*}
0=G_0\subset G_1\subset \cdots\subset G_c=G[F]
\end{equation*}
of $G[F]$ by finite flat subgroup schemes induces a filtration of $\omega_G=\omega_{G[F]}$ by locally free subsheaves where the sub quotients are $\omega_{G_i/G_{i-1}}$ for $i=1,\ldots,c$.

Now for $i=1,\ldots,c$ Verschiebung defines isomorphisms of co-lie algebras
\begin{equation*}
V^*:\omega_{G_i/G_{i-1}}\to\omega_{G_{\sigma(i)}/G_{\sigma(i)-1}}^{(p)}
\end{equation*}
and hence upon taking determinants an isomorphism
\begin{equation*}
A_i:\omega_i\to\omega_{\sigma(i)}^{(p)}\simeq\omega_{\sigma(i)}^{\otimes p}.
\end{equation*}
Similarly for $i=c+1,\ldots,n$ Frobenius defines isomorphisms of co-lie algebras
\begin{equation*}
F^*:\omega_{G_{\sigma(i)}/G_{\sigma(i)-1}}^{(p)}\to \omega_{G_i/G_{i-1}}
\end{equation*}
and unless $i=n$ and $\sigma(n)=n$ upon taking determinants we obtain an isomorphism
\begin{equation*}
B_i:\omega_{\sigma(i)}^{\otimes p}\to\omega_i.
\end{equation*}
We will also let
\begin{equation*}
A_i=B_i^{-1}:\omega_i\to\omega_{\sigma(i)}^{\otimes p}
\end{equation*}
We may also think of the $A_i$ as non vanishing sections
\begin{equation*}
A_i\in H^0(S,\omega_i^{\otimes -1}\otimes\omega_{\sigma(i)}^{\otimes p}).
\end{equation*}

We like to think of the $A_i$ defined above as being some sort of ``partial Hasse invariants.''  By combining them suitably we may create a ``total Hasse invariant'' as follows: let $N$ be the least common multiple of the orders of the cycles of $\sigma$.  For $i=1,\ldots,c$ consider the composition
\begin{equation*}
\omega_i\overset{A_i}{\to}\omega_{\sigma(i)}^{\otimes p}\overset{A_{\sigma(i)}^p}{\to}\omega_{\sigma^2(i)}^{\otimes p^2}\to\cdots\to\omega_{\sigma^{N-1}(i)}^{\otimes p^{N-1}}\overset{A_{\sigma^{N-1}(i)}^{p^n}}{\to}\omega_{\sigma^N(i)}^{p^N}
\end{equation*}
multiplying these together over $i=1,\ldots c$ and using the fact that $\sigma^N(i)=i$ we get an isomorphism
\begin{equation*}
\omega\to \omega^{p^N}
\end{equation*}
which can be viewed as a non vanishing section
\begin{equation*}
A'\in H^0(S,\omega^{p^N-1}).
\end{equation*}
We call this the total Hasse invariant for the partial $\BT$ $G$.

We record the fact that the formation of this total Hasse invariant is compatible with base change.
\begin{prop}\label{P:hassebc}
Let $G/S$ be a partial $\BT$ which admits a canonical filtration, and let $S'\to S$.  Let $A'\in H^0(S,\omega^{p^{N}-1})$ be the total Hasse invariant for $G/S$ as defined above.  
\end{prop}

\subsection{Generalized Hasse Invariants on Open Ekedahl-Oort Strata}\label{S:hasse}

Now we return to the setting of PEL modular varieties.  We retain the notation of Section \ref{S:EO}.  In particular for each $K\subset G(\A^{\infty,p})$ let $G/X_K$ be the canonical $\BT$ of section \ref{S:EO}.  Then we have a canonical isomorphism $\omega_G=\Ca{E}_K$ and hence $\det\omega_G=\omega_K$.

\begin{defn}
For each $w\in W^I$ let
\begin{equation*}
A'_{K,w}\in H^0(X_{K,w},\omega^{N_w'})
\end{equation*}
be the non vanishing section attached to the $\BT$ $G|_{X_{K,w}}$ as constructed in the last section.  Here $N_w'=p^N-1$ where $N$ is the least common multiple of the cycles in the permutation $\sigma$ associated to $G$.
\end{defn}

We record the behavior of these Hasse invariants under the Hecke action and the maps to Siegel modular varieties.

\begin{prop}\label{P:openhasse}
Let the notation be as above.
\begin{enumerate}
\item If $g\in G(\A^{\infty,p})$ and $K,K'\subset G(\A^{\infty,p})$ are open compact subgroups with $g^{-1}Kg\subset K'$ then
\begin{equation*}
[g]^*A_{K',w}'=A_{K,w}'
\end{equation*}
under the canonical isomorphism $[g]^*\omega_{K',A}\simeq\omega_{K,A}$ restricted to $X_{K,w}$.
\item Let notation be as in Sections \ref{S:morsieg} and \ref{S:EO}.  Let $\tilde{w}=\theta(w)\in\tilde{W}^{\tilde{I}}$.  For $K\subset G(\A^{\infty,p})$ and $\tilde{K}\subset G(\A^{\infty,p})$ open compact subgroups with $K\subset\tilde{K}$, we have
\begin{equation*}
\phi_{K,\tilde{K}}^*A'_{\tilde{K},\tilde{w}}=A'_{K,w}
\end{equation*}
under the canonical isomorphism $\phi_{K,\tilde{K}}^*\omega_{\tilde{K}}\simeq\omega_{K}$.
\end{enumerate}
\end{prop}
\begin{proof}
Both statements are immediate consequences of the behavior of the total Hasse invariant under base change in Proposition \ref{P:hassebc}.
\end{proof}

\subsection{Extensions: Reduction to Siegel case}

Now we are in a position to state the first main result of this thesis.

\begin{thm}\label{T:hasseext}
For each $w\in W^I$ there is an integer $N_w>0$ such that for each $K\subset G(\A^{\infty,p})$ neat open compact, there is a unique section
\begin{equation*}
A_{K,w}\in H^0(\overline{X}_{K,w},\omega_K^{\otimes N_w})
\end{equation*}
with the following properties
\begin{enumerate}
\item $A_{K,w}$ is non vanishing precisely on $X_{K,w}\subset\overline{X}_{K,w}$
\item There is some integer $n>0$ such that $A_{K,w}|_{X_{K,w}}=(A'_{K,w})^n$.
\item If $g\in G(\A^{\infty,p})$ and $K,K'\subset G(\A^{\infty,p})$ are open compact subgroups with $g^{-1}Kg\subset K'$ then we have
\begin{equation*}
[g]^*A_{K',w}=A_{K,w}
\end{equation*}
under the canonical isomorphism $[g]^*\omega_{K'}\simeq\omega_K$ restricted to $\overline{X}_{K,w}$.
\end{enumerate}
\end{thm}

Here we give some reductions.  The proof of the theorem will be completed in the next chapter.

\begin{proof}
First note that once an $A_{K,w}$ satisfying 2 has been constructed, the Hecke stability in 3 is automatic by the density of $X_{K,w}$ in $\overline{X}_{K,w}$ and the Hecke stability of $A'_{K,w}$ from Proposition \ref{P:openhasse}.

Next we claim that it suffices to show that $N_w$ and $A_{K,w}$ satisfying 1 and 2 for a single level $K\subset G(\A^{\infty,p})$.  Indeed, first note that if $K\subset K'\subset G(\A^{\infty,p})$ are neat open compacts so that we have a map
\begin{equation*}
[1]:X_K\to X_{K'}
\end{equation*}
Then observe that
\begin{enumerate}
\item For $n>0$, $(A_{K,w}')^n$ extends to a (necessarily unique) element of $H^0(\overline{X}_{K,w},\omega_K^{nN_w'})$ if and only if $(A_{K',w}')^n$ extends to an element of $H^0(\overline{X}_{K',w},\omega_{K'}^{nN_w'})$.
\item If such extensions exist, the extension of $(A_{K,w}')^n$ is non vanishing precisely on $X_{K,w}$ if and only if the extension of $(A_{K',w}')^n$ is non vanishing precisely on $X_{K',w}$.
\end{enumerate}
Point 1 follows from the fact that $[1]$ is finite \'{e}tale, and point 2 is clear.  Then the reduction is complete upon noting that for any pair $K,K'\subset G(\A^{\infty,p})$ of neat open compact subgroups, we have $K\cap K'\subset G(\A^{\infty,p})$ open compact with $K\cap K'\subset K$ and $K\cap K'\subset K'$.

Next we claim that it suffices to prove the theorem for Siegel modular varieties.  Indeed, with notation as in Sections \ref{S:morsieg} and \ref{S:EO}, suppose we have $K\subset G(\A^{\infty,p})$ and $\tilde{K}\subset G(\A^{\infty,p})$ neat open compacts and suppose we have shown that for some $n>0$, $(A'_{\tilde{K},\theta(w)})^n$ extends to a section
\begin{equation*}
A_{\tilde{K},\theta(w)}\in H^0(\overline{\tilde{X}}_{\tilde{K},\theta(w)},\omega_{\tilde{K}}^{N_{\theta(w)}})
\end{equation*}
which is non vanishing precisely on $\tilde{X}_{\tilde{K},\theta(w)}$.  Then take $N_w=N_{\theta(w)}$ and
\begin{equation*}
A_{K,w}=\phi_{K,\tilde{K}}^*A_{\tilde{K},\theta(w)}\in H^0(\overline{X}_{K,w},\omega_K^{\otimes N_w})
\end{equation*}
under the canonical isomorphism $\phi_{K,\tilde{K}}^*\omega_{\tilde{K}}\simeq\omega_K$.  Note that by Proposition \ref{P:openhasse}, $A_{K,w}$ extends $(A_{K,w}')^n$.  Also note that by Theorem \ref{T:opensieg} we have
\begin{equation*}
\overline{X}_{K,w}\cap \phi_{K,\tilde{K}}^{-1}(\tilde{X}_{\tilde{K},\theta(w)})=X_{K,w}.
\end{equation*}
It follows that $A_{K,w}$ is non vanishing precisely on $X_{K,w}$.

The Siegel case will be treated in Chapter \ref{siegel}.
\end{proof}

\chapter{Extension of Hasse Invariants}\label{siegel}

The goal of this chapter is to complete the proof of Theorem \ref{T:hasseext}.  We refer the reader to the introduction for an overview of the strategy.

We now introduce some notation that will be used in this chapter.  Fix a positive integer $g$.  Let $W$ be the Weyl group of type $C_g$.  We realize $W$ as the subgroup of permutations on $w\in S_{2g}$, the permutation group on $\{1,\ldots,2g\}$, satisfying
\begin{equation*}
w(2g+1-i)=2g+1-w(i).
\end{equation*}
For notational convenience we will adopt the convention that $w(0)=0$.

Inside $W$ we have the simple reflections
\begin{equation*}
s_i=(i\ i+1)(2g+1-i\ 2g-i)\qquad i=1,\ldots,g-1
\end{equation*}
and
\begin{equation*}
s_g=(g\ g+1).
\end{equation*}

We denote the usual length function on $W$ by $l$ and the Bruhat order by $\leq$.  For a subset $J\subset\{1,\ldots,g\}$ we denote by $W_J$ the parabolic subgroup of $W$ generated by $s_i$ for $i\in J$, $i\not=0$.  For two subsets $J$ and $J'$ it is known that the cosets $wW_J$, $W_Jw$, and $W_JwW_{J'}$ each contain a unique element of minimal length, and we denote the corresponding sets of minimal coset representatives by ${}^JW$, $W^J$, and ${}^JW^{J'}$.

Throughout we fix $I=\{1,\ldots,g-1\}$.  Then $W_I\subset W$ is the symmetric group on $g$ letters.  We also fix the Weyl group element
\begin{equation*}
x=(1\  g+1)(2\  g+2)\cdots(g\  2g).
\end{equation*}
We note that $x\in {}^IW^I$ and in fact it is the longest element in this set.

If $w\in W$ we denote by ${}^wJ$ the set
\begin{equation*}
{}^wJ=\{i\mid s_i=ws_jw^{-1}\text{ for some $j\in J$}\}.
\end{equation*}

Throughout most of this chapter we will fix another subset $J\subset \{1,\ldots,g-1\}$.  We will then use the following notation:
\begin{itemize}
\item $\{0,\ldots,g\}-J=\{k_0,\ldots,k_c\}$ with $0=k_0<k_1<\ldots<k_c=g$.
\item $k_i=2g-k_{2r-i}$ for $i=c,\ldots,2c$.
\item $\tilde{J}={}^xJ=\{g-i\mid i\in J\}$.
\item $\{0,\ldots,g\}-\tilde{J}=\{\tilde{k}_0,\ldots,\tilde{k}_r\}$ with $0=\tilde{k}_0<\tilde{k}_1<\ldots<\tilde{k}_c=g$.
\item $\tilde{k}_i=2g-\tilde{k}_{2r-i}$ for $i=c,\ldots,2c$.
\end{itemize}

We note the formula
\begin{equation*}
\tilde{k}_i=g-k_{c-i}
\end{equation*}
for $i=0,\ldots,r$.

When this notation becomes burdensome the reader should focus on the ``generic'' case when $J=\emptyset$ so that $\tilde{J}=\emptyset$, $k_i=\tilde{k}_i=i$, and $c=g$.  The subsets $J$ and $\tilde{J}$ will determine the types of certain partial flags we will consider, and this will correspond to the case where all the flags are complete.

Let $V=\FF_p^{2g}$ with basis $e_1,\ldots,e_{2g}$ and fix the symplectic form $\psi$ on $V$ given by
\begin{equation*}
\psi(e_i,e_j)=0
\end{equation*}
and
\begin{equation*}
\psi(e_i,e_{2g+1-j})=\delta_{ij}
\end{equation*}
for $1\leq i,j\leq g$.  For this section only we will let $G=\GSp(V,\psi)=\GSp_{2g}/\FF_p$.

\section{Weyl Groups and Schubert Varieties}

\subsection{Admissible pairs}

We begin this section with some lemmas about Weyl group cosets.

\begin{lem}\label{L:W1}
$J,J'\subset \{1,\ldots,g\}$ be arbitrary and let $w\in W^{J'}$.  Then the following are equivalent
\begin{enumerate}
\item $w\in {}^JW$ and $W_JwW_{J'}=wW_{J'}$
\item $J\subset {}^wJ'$
\end{enumerate}
\end{lem}
\begin{proof}
First we prove that 2 implies $w\in {}^JW$.  Indeed, 2 implies that for all $i\in J$, $s_i=ws_jw^{-1}$ for $j\in J'$.  But then
\begin{equation*}
l(s_iw)=l(ws_j)=l(w)+1
\end{equation*}
because $w\in w^{J'}$, and hence $w\in {}^JW$.

Next note that $W_JwW_{J'}=wW_{J'}$ if and only if for ever $i\in I$, $s_iwW_{J'}=wW_{J'}$.  This is equivalent to
\begin{equation*}
s_i\in wW_{J'}w^{-1}.
\end{equation*}
From this it is clear that 2 implies 1.  To show the converse what we need to show is that if
\begin{equation*}
s_iw=wv
\end{equation*}
for some $v\in W_{J'}$ then $v$ is a simple reflection.  But in fact, as $w\in{}^JW^{J'}$
\begin{equation*}
1+l(w)=l(s_iw)=l(wv)=l(w)+l(v)
\end{equation*}
and hence $l(v)=1$.
\end{proof}

\begin{lem}\label{L:W2}
Let $w\in W^I$ and $J\subset\{1,\ldots,g-1\}$ be such that ${}^wJ=J$.  Then
\begin{enumerate}
\item For all $i=1,\ldots,2c$, there exists $0<j\leq 2c$ such that $w(k_i)=k_j$.
\item If $i=1,\ldots,2c$ and $k_{i-1}<a\leq k_i$ then $w(a)=w(k_i)-(k_i-a)$.
\item There is a unique permutation $\sigma\in S_{2c}$ satisfying $w(k_i)=k_{\sigma(i)}$.  We have $\sigma(2c+1-i)=2c+1-\sigma(i)$ for all $i=1,\ldots,2c$ and $\sigma(1)<\sigma(2)<\cdots<\sigma(c)$.
\item For $i=1,\ldots,2c$, $k_i-k_{i-1}=k_{\sigma(i)}-k_{\sigma(i)-1}$
\end{enumerate}
\end{lem}
\begin{proof}
Let $i\in J$.  Then by hypothesis there is $j\in J$ with
\begin{equation*}
s_j=ws_iw^{-1}
\end{equation*}
and hence
\begin{equation*}
(j\ j+1)(2g+1-j\ 2g-j)=(w(i)\ w(i+1))(w(2g+1-i)\ w(2g-i)).
\end{equation*}
Moreover as $w\in W^I$ we have
\begin{equation*}
w(i)<w(i+1)\qquad\text{and}\qquad w(2g-i)<w(2g+1-i)
\end{equation*}
(indeed $w\in W^I$ is equivalent to $w(k)<w(k+1)$ for all $k\not=g$).  From these three formulas we conclude that one of the following two possibilities occurs:
\begin{itemize}
\item $w(i)=j$, $w(i+1)=j+1$, $w(2g-i)=2g-j$, $w(2g+1-i)=2g+1-j$ or
\item $w(i)=2g-j$, $w(i+1)=2g+1-j$, $w(2g-i)=j$, $w(2g+1-i)=j+1$.
\end{itemize}
Either way we see that the set
\begin{equation*}
S=J\cup (2g-J)
\end{equation*}
is stable by $w$.  But by definition we had
\begin{equation*}
\{1,\ldots,k_{2c}\}=\{1,\ldots,2g\}-S
\end{equation*}
and so this set is stable by $w$ as well, proving 1.  Moreover these formulas show that for $i\in S$, $w(i+1)=w(i)+1$ and so 2 follows from this and induction.  Point 3 is clear.

To prove 4, note that by 2, $k_{\sigma(i)-1}\leq k_{\sigma(i)}-(k_i-k_{i-1})$.  Hence
\begin{equation*}
k_i-k_{i-1}\leq k_{\sigma(i)}-k_{\sigma(i)-1}\leq k_{\sigma^2(i)}-k_{\sigma^2(i)-1}\leq \cdots.
\end{equation*}
But $\sigma^n(i)=i$ for some $n>0$ and hence all of these inequalities are equalities. 
\end{proof}

\begin{defn}
We say that the pair $(w,J)$ with $w\in W$ and $J\subset \{1,\ldots,g-1\}$ is admissible if $w\in W^I$ and ${}^w\tilde{J}=J$
\end{defn}

Note that $(w,\emptyset)$ for $w\in W^I$ is always admissible.  Moreover if $(w,J)$ and $(w,J')$ are admissible then so is $(w,J\cup J')$.  Hence if $w\in W^I$ there is a maximal $J$ such that $(w,J)$ is admissible, and we denote it by $J_w$.

\begin{prop}\label{P:adm}
Let $(w,J)$ be an admissible pair.  Then
\begin{enumerate}
\item $w\in{}^JW$ and $W_JwW_{\tilde{J}}=wW_{\tilde{J}}$.
\item For all $i=1,\ldots,2c$, there exists $0<j\leq 2c$ such that $w(\tilde{k}_i)=k_j$.
\item If $i=1,\ldots,2c$ and $\tilde{k}_{i-1}<a\leq k_i$ then $w(a)=w(\tilde{k}_i)-(k_i-a)$.
\item There is a unique permutation $\sigma\in S_{2c}$ satisfying $wx(k_i)=k_{\sigma(i)}$ for $i=1,\ldots,2c$.  Moreover $\sigma(2c+1-i)=2c+1-\sigma(i)$ for all $i=1,\ldots,2c$ and $\sigma(1)<\sigma(2)<\cdots<\sigma(c)$.
\item For $i=1,\ldots,2c$, $k_i-k_{i-1}=k_{\sigma(i)}-k_{\sigma(i)-1}$.
\end{enumerate}
\end{prop}
\begin{proof}
Point 1 follows from Lemma \ref{L:W1} while 2 through 5 follow from Lemma \ref{L:W2} applied to wx.
\end{proof}

For notational purposes we also introduce the permutation $\tau\in S_{2c}$ given by
\begin{equation*}
\tau(i)=\begin{cases}\sigma(i+c)&1\leq i\leq c\\ \sigma(i-c)&c+1\leq i\leq 2c\end{cases}
\end{equation*}
Then by parts 4 of the Proposition we have that for $i=1,\ldots,2c$, $w(\tilde{k}_i)=k_{\tau(i)}$.  Moreover $\tau(2c+1-i)=2c+1-\tau(i)$ for $i=1,\ldots,2c$ and $\tau(1)<\tau(2)<\cdots<\tau(c)$.  By part 5 of the proposition we have $\tilde{k}_i-\tilde{k}_{i-1}=k_{\tau(i)}-k_{\tau(i)-1}$.

\subsection{Schubert varieties}

Let $\text{Fl}_{\tilde{J}}/\FF_p$ denote the variety of symplectic flags of type $\tilde{J}$ in $V$.  More precisely, this space parameterizes flags
\begin{equation*}
0=F_0\subset F_1\subset\cdots\subset F_c\subset\cdots\subset F_{2c}=V\otimes\O_S
\end{equation*}
where the $F_i$ are local direct summands, $F_i=F_{2c-i}^\perp$, and $\dim F_i=\tilde{k}_i$ for $i=0,\ldots, 2c$.

$\text{Fl}_{\tilde{J}}$ has stratifications by certain Schubert varieties which we now recall.  We define standard flags
\begin{align*}
E_i&=\langle e_1,\ldots,e_{k_i}\rangle\\
\tilde{E}_i&=\langle e_1,\ldots,e_{\tilde{k}_i}\rangle\\
\tilde{E}_i^B&=\langle e_1,\ldots,e_i\rangle
\end{align*}
and let $P$, $\tilde{P}$, and $B$ be the parabolics of $G$ stabilizing the flags $(E_i)$, $(\tilde{E}_i)$, and $(E_i^B)$ respectively.

The Schubert varieties we will consider correspond to the orbits of $P$ and $B$ on $\text{Fl}_{\tilde{J}}$.  We have two decompositions into reduced locally closed subvarieties
\begin{equation*}
\text{Fl}_{\tilde{J}}=\coprod_{w\in{}^JW^{\tilde{J}}}Y_w
\end{equation*}
and
\begin{equation*}
\text{Fl}_{\tilde{J}}=\coprod_{w\in W^{\tilde{J}}}Y_w^B
\end{equation*}

For $w\in {}^JW^{\tilde{J}}$, $0\leq i\leq 2c$, and $0\leq j\leq 2c$ let
\begin{equation*}
d_w(i,j)=\dim(w\tilde{E}_i\cap E_j)=\#(w\{1,\ldots,\tilde{k}_i\}\cap\{1,\ldots,k_j\})
\end{equation*}
and for $w\in W^{\tilde{J}}$, $0\leq i\leq 2c$, and $0\leq j\leq 2g$ let
\begin{equation*}
d_w^B(i,j)=\dim(w\tilde{E}_i\cap E_j^B)=\#(w\{1,\ldots,k_i\}\cap\{1,\ldots,j\}).
\end{equation*}

Let $k/\FF_p$ be a field.  Then $(F_i)\in \text{Fl}_{\tilde{J}}(k)$ lies in $Y_w(k)$ for $w\in {}^JW^{\tilde{J}}$ if and only if the Schubert condition
\begin{equation*}
\dim(F_i\cap E_j\otimes k)=d_w(i,j)\qquad\text{for all $0\leq i\leq 2c$, $0\leq j\leq 2c$}
\end{equation*}
is satisfied.  Similarly $(F_i)$ lies in $Y^B_w(k)$ if and only if
\begin{equation*}
\dim(F_i\cap E^B_j\otimes k)=d_w^B(i,j)\qquad\text{for all $0\leq i\leq 2c$, $0\leq j\leq 2g$}.
\end{equation*}

Let $\overline{Y}_w$ and $\overline{Y}_w^B$ denote the Zariski closures of $Y_w$ and $Y_w^B$ respectively.  Their points can also be characterized by Schubert conditions.  We have $(F_i)\in \overline{Y}_w(k)$ if and only if
\begin{equation*}
\dim(F_i\cap E_j\otimes k)\geq d_w(i,j)\qquad\text{for all $0\leq i\leq 2c$, $0\leq j\leq 2c$}
\end{equation*}
and $(F_i)\in\overline{Y}^B_w(k)$ if and only if
\begin{equation*}
\dim(F_i\cap E_j^B\otimes k)\geq d_w^B(i,j)\qquad\text{for all $0\leq i\leq 2c$, $0\leq j\leq 2g$}.
\end{equation*}

We remark that for the Schubert varieties $Y_w$, the Schubert condition for the pair $(i,j)$ is equivalent to that for the pair $(2c-i,2c-j)$.  Similarly for the Schubert varieties $Y_w^B$ the Schubert condition for the pair $(i,j)$ is equivalent to that for the pair $(2c-i,2g-j)$.

The following proposition regarding the Schubert varieties $\overline{Y}_w^B$ is well known.
\begin{prop}\label{P:bschub}
For $v,w\in W^{\tilde{J}}$
\begin{enumerate}
\item $\dim Y^B_w=\dim \overline{Y}_w^B=l(w)$.
\item $\overline{Y}^B_w$ is normal.
\item $Y_v^B\subset \overline{Y}_w^B$ if and only if $v\leq w$ in the Bruhat order.
\item The complement
\begin{equation*}
\overline{Y}_v^B-Y_v^B=\bigcup_{v\in D_w^B}Y_w^B
\end{equation*}
is a union of irreducible divisors where
\begin{equation*}
D_w^B=\{v\in W^{\tilde{J}}\mid \text{$v\leq w$ and $l(v)=l(w)-1$}\}
\end{equation*}
is the set of Bruhat descendants of $w$.
\end{enumerate}
\end{prop}

As $B\subset P$ we can decompose each $P$ orbit $Y_w$ for $w\in {}^JW^{\tilde{J}}$ into $B$ orbits as
\begin{equation*}
Y_w=\coprod_{v\in (W_JwW_{\tilde{J}})\cap W^{\tilde{J}}}Y_v^B
\end{equation*}
As $X_w$ is irreducible and the decomposition is finite, there must be a unique open orbit.  That is there is a maximal element $\dot{w}\in(W_JwW_{\tilde{J}})\cap W^{\tilde{J}}$ for the Bruhat order (this can also be seen purely combinatorially) and $Y^B_{\dot{w}}\subset Y_w$ is open dense.

We now have the following Corollary to proposition \ref{P:bschub}:
\begin{cor}\label{P:schub}
For $w\in {}^JW^{\tilde{J}}$
\begin{enumerate}
\item $\dim Y_w=\dim \overline{Y}_w=l(\dot{w})$
\item $\overline{Y}_w$ is normal.
\item If $W_JwW_{\tilde{J}}=wW_{\tilde{J}}$ then the complement
\begin{equation*}
\overline{Y}_w-Y_w=\bigcup_{v\in D_w^B}\overline{Y}_v^B=\bigcup_{v\in D_w}\overline{Y}_v
\end{equation*}
is a union of irreducible divisors where
\begin{equation*}
D_w=W_JD_w^B\cap {}^JW^{\tilde{J}}.
\end{equation*}
\end{enumerate}
\end{cor}
\begin{proof}
1 and 2 follow from the fact that $\overline{Y}_w=\overline{Y}_{\dot{w}}^B$.  For 3, note that when $W_JwW_{\tilde{J}}=wW_{\tilde{J}}$ we have $Y_w=Y_w^B$.  Hence
\begin{equation*}\tag{*}
\overline{Y}_w-Y_w=\overline{Y}_w^B-Y_w^B=\bigcup_{v\in D_w^B}\overline{Y}_v^B.
\end{equation*}
For the last equality, for each $v\in D_v^B$, let $\tilde{v}$ be the shortest element of $W_JvW_{\tilde{J}}$.  Then what we need to show is that $\overline{Y}_v^B=\overline{Y}_{\tilde{v}}$, or in other words that $\dot{\tilde{v}}=v$.  There is a purely combinatorial proof of this but here we give a geometric one.  First note that $Y_v^B\subset Y_{\tilde{v}}$ which implies one inclusion, and also using the fact that $Y_{\tilde{v}}$ is $P$-stable by definition we have 
\begin{equation*}
Y_{\tilde{v}}\subset P\cdot Y_v^B\subset P\cdot\overline{Y}_v^B.
\end{equation*}
Hence we would be done if we knew that $\overline{Y}_v^B$ was $P$-stable.  But indeed the left hand side of (*) is $P$-stable and hence so is the right.  Moreover, as $P$ is irreducible it must stabilize each of the $\overline{Y}_v^B$, $v\in D_w$ individually, so we are done.
\end{proof}

For the rest of the section we assume $(w,J)$ is admissible and so in particular $w\in {}^JW^{\tilde{J}}$.  We first observe that the Schubert conditions defining $\overline{Y}_w$ are very easy to describe.  

\begin{prop}\label{P:schubcon}
If $(w,J)$ is admissible and $k/\FF_p$ is a field then
\begin{equation*}
\overline{Y}_w(k)=\{(F_i)\in \text{\rm Fl}_{\tilde{J}}(k)\mid F_i\subset E_{\tau(i)}\otimes k,\text{ for $i=1,\ldots, c$}\}.
\end{equation*}
\end{prop}
\begin{proof}
Note that for $i=1,\ldots,c$ we have
\begin{equation*}
d_w(i,\tau(i))=\tilde{k}_i
\end{equation*}
and hence the Schubert condition for the pair $(i,\tau(i))$ is
\begin{equation*}
F_i\subset E_{\tau(i)}\otimes k.
\end{equation*}

If $(i,j)$ satisfies $1\leq i\leq c$ and $j>\tau(i)$ then $d_w(i,j)=\tilde{k}_i$ and so the condition for the pair $(i,j)$ is
\begin{equation*}
F_i\subset E_j\otimes k.
\end{equation*}
As $E_{\tau(i)}\subset E_j$ this condition is implied by the Schubert condition for $(i,\tau(i))$.

Next we consider pairs $(i,j)$ where $i=1,\ldots,c$ and $j<\tau(i)$.  We claim that $d_w(i,j)=\tilde{k}_{i'}$ where $i'$ is the largest integer for which $\tau(i')\leq j$, or equivalently, for which $w(\tilde{k}_{i'})\leq k_j$.  For this it suffices to show that $w(\tilde{k}_{i'}+1)>k_j$.  Indeed
\begin{equation*}
w(\tilde{k}_{i'}+1)=w(\tilde{k}_{i'+1})-(\tilde{k}_{i'+1}-\tilde{k}_{i'}-1)=k_{\tau(i'+1)}-(k_{\tau(i'+1)}-k_{\tau(i'+1)-1})+1>k_j
\end{equation*}
where the first equality is by Proposition \ref{P:adm} 3, the second equality is by Proposition \ref{P:adm} 5, and the last inequality uses that $k_{\tau(i'+1)-1}\geq k_j$ as $\tau(i'+1)>j$.

The Schubert condition for the pair $(i,j)$ is
\begin{equation*}
\dim F_i\cap E_{j}\otimes k\geq \tilde{k}_{i'}
\end{equation*}
But if the Schubert condition for $(i',\tau(i'))$ holds then
\begin{equation*}
F_{i'}\subset F_i\cap E_j
\end{equation*}
and hence the Schubert condition for the pair $(i,j)$ holds.
\end{proof}

Let $(\F_i)/\text{Fl}_{\tilde{J}}$ be the universal flag of type $\tilde{J}$ on $V\otimes\O_{\text{Fl}_{\tilde{J}}}$.  Define $\Ca{E}_i$ by
\begin{equation*}
\Ca{E}_i=E_i\otimes\O_{\text{Fl}_{\tilde{J}}}.
\end{equation*} 
We will use the same symbols for their restrictions to any subvariety of $\text{Fl}_{\tilde{J}}$ when this will cause no confusion.

By the proposition, over $\overline{Y}_w$ we have
\begin{equation*}
\F_i\subset\Ca{E}_{\tau(i)}
\end{equation*}
and
\begin{equation*}
\F_{i-1}\subset\Ca{E}_{\tau(i-1)}\subset\Ca{E}_{\tau(i)-1}
\end{equation*}
for $i=1,\ldots,c$.  Hence we may form the maps of vector bundles (of the same rank by Proposition \ref{P:adm} 5)
\begin{equation*}
\F_i/\F_{i-1}\to \Ca{E}_{\tau(i)}/\Ca{E}_{\tau(i)-1}
\end{equation*}
and then
\begin{equation*}
C_i:\det(\F_i/\F_{i-1})\to\det(\Ca{E}_{\tau(i)}/\Ca{E}_{\tau(i)-1}).
\end{equation*}

We may think of $C_i$ as a section of the line bundle
\begin{equation*}
\underline{\hom}_{\O_{\overline{Y}_w}}(\det(\F_i/\F_{i-1}),\det(\Ca{E}_{\tau(i)}/\Ca{E}_{\tau(i)-1}))\simeq\det(\F_i/\F_{i-1})^\vee
\end{equation*}
where the last isomorphism is non canonical and depends on a choice of a trivialization of the trivial line bundle $\det(\Ca{E}_{\tau(i)}/\Ca{E}_{\tau(i)-1})$.

For the rest of this section, the goal is to study the vanishing locus of $C_i$, $i=1,\ldots,c$.  First we give an explicit description of the set of Bruhat descendants $D_w^B$.  Let $w=s_{i_1}\cdots s_{i_l}$ be a reduced expression for $w$ in terms of simple reflections.  Then any $v\in D_w^B$ can be written in the form
\begin{equation*}
v=s_{i_1}\cdots \hat{s}_{i_j}\cdots s_{i_l}
\end{equation*}
for some $j$.  In other words
\begin{equation*}
v=w(s_{i_{j+1}}\cdots s_{i_l})^{-1}s_{i_j}(s_{i_{j+1}}\cdots s_{i_l}).
\end{equation*}
Hence any $v\in D_w^B$ has the form $v=ws$ where $s\in W$ is a reflection (recall that a reflection is just an element conjugate to a simple reflection.)  Since it is easy to enumerate all reflections in $W$, in order to explicitly determine the set $D_w^B$ we should just check whether or not $ws$ is in $D_w^B$ for each reflection $s$.

We first consider the case that $s$ is conjugate to $s_g$, i.e. $s=(i\ 2g+1-i)$ for $i=1,\ldots,g$.  Then one easily checks that $ws\leq w$ if and only if $w(i)>g$ and that in this case one always has $l(ws)=l(w)-1$.  Now we need to determine when $ws\in W^{\tilde{J}}$.  We have $ws\in W^{\tilde{J}}$ if and only if for all $j=1,\ldots,c$ and all $\tilde{k}_{j-1}<l<m\leq \tilde{k}_j$, $w(l)<w(m)$.  Let $j$ be such that $\tilde{k}_{j-1}<i\leq \tilde{k}_j$.  Then if $i>\tilde{k}_{j-1}+1$ we have by Proposition \ref{P:adm} that $w(i-1)=w(i)-1>g$ (we cannot have $w(i-1)=g$ as $g=k_c$ but $i-1$ isn't of the form $\tilde{k}_l$ for any $l$.)  Hence
\begin{equation*}
ws(i-1)=w(i-1)>g\geq ws(i)
\end{equation*}
and so $ws\not\in W^{\tilde{J}}$.  On the other hand if $i=k_{j-1}+1$ then to show that $ws\in W^{\tilde{J}}$ there is nothing to check if $i=k_j$ and otherwise we only have to check that $ws(i)<ws(i+1)$.  But by Proposition \ref{P:adm} we have $ws(i+1)=w(i+1)=w(i)+1>g$ whereas $ws(i)\leq g$.  Finally we note that if $i=k_{j-1}+1$ then $w(i)>g$ if and only if $w(k_j)>g$, again by Proposition \ref{P:adm}.

Next consider the case where $s$ is conjugate to $s_i$, $i\not=g$.  Then $s=(a\ b)(2g+1-a\ 2g+1-b)$.  We can assume without loss of generality that $a<\text{min}(b,2g+1-b)$.  Then one checks that $ws\leq w$ if and only if $w(a)>w(b)$, and a slightly tedious calculation shows that under this condition, $l(ws)=l(w)-1$ if in addition $w(a)\leq g$.  A calculation similar to the one above determines when $ws\in W^{\tilde{J}}$.

To summarize we have
\begin{prop}\label{P:vcases}
Let $(w,J)$ be admissible.  Then $D_w^B$ consists of
\begin{itemize}
\item $v_i:=w(\tilde{k}_{i-1}+1\ 2g-\tilde{k}_{i-1})$ for $1\leq i\leq c$ for which $\tau(i)>c$.
\item $v_{a,b}:=w(\tilde{k}_{a-1}+1\ \tilde{k}_b)(2g-\tilde{k}_{a-1}\ 2g+1-\tilde{k}_b)$ for $1\leq a\leq c$, $c+1\leq b\leq 2c$ for which $\tau(b)<\tau(a)$ and $\tau(a)\leq c$.
\end{itemize}
\end{prop}

As $\overline{Y}_w$ is normal and hence regular in codimension 1, it makes sense to talk about the order vanishing of a section of a line bundle along an irreducible divisor.

\begin{prop}\label{P:van}
If $1\leq i\leq c$ then $C_i$ is non vanishing on $Y_w$.  Moreover for $v\in D_w^B$ we have
\begin{equation*}
\text{\rm ord}_{\overline{Y}^B_v}(C_i)=\begin{cases}1&\text{if $v=v_i$, or if $v=v_{a,b}$ with $i=a$ or $2c+1-b$}  \\ 0&\text{\rm otherwise.}\end{cases}
\end{equation*}
\end{prop}
\begin{proof}
For the first statement consider the point $(wF_i)\in Y_w(\FF_p)$.  By Proposition \ref{P:adm}, 
\begin{equation*}
w\{\tilde{k}_{i-1}+1,\ldots,\tilde{k}_i\}=\{k_{\tau(i)-1}+1,\ldots,k_{\tau(i)}\}
\end{equation*}
and hence
\begin{equation*}
wF_i/wF_{i-1}\to E_{\tau(i)}/E_{\tau(i)-1}
\end{equation*}
is an isomorphism.  This shows that $C_i$ doesn't vanish on the point $(wF_i)$ and by considering the $P$ action we see that $C_i$ is non vanishing on all of $Y_w$.

Now let $v\in D_w^B$.  An inspection of the cases in \ref{P:vcases} shows that for all $1\leq j\leq g$, we have $v(j)=w(j)$ except
\begin{itemize}
\item if $v=v_l$ and $j=\tilde{k}_{l-1}+1$,
\item or if $v=v_{a,b}$ and $j=\tilde{k}_{a-1}+1$ or $\tilde{k}_{2g-b}+1$.
\end{itemize}
and moreover in these exceptional cases we have $v(j)<w(j)=k_{\tau(i)-1}+1$.

Consider the point $(vF_i)\in Y_v^B(\FF_p)$.  Then unless $i=l$ in the first case, or $i=a$ or $2g+1-b$ in the second case, we still have that
\begin{equation*}
vF_i/vF_{i-1}\to E_{\tau(i)}/E_{\tau(i)-1}
\end{equation*}
is an isomorphism.  Now in the exceptional cases we see that
\begin{equation*}
v\{\tilde{k}_{i-1}+1,\ldots,\tilde{k}_i\}=\{v(\tilde{k}_{i-1}+1),k_{\tau(i)-1}+2,\ldots,k_{\tau(i)}\}
\end{equation*}
and so $\det (vF_i/vF_{i-1})$ is generated by the class of
\begin{equation*}
e_{v(\tilde{k}_{i-1}+1)}\wedge e_{k_{\tau(i)-1}+2}\wedge\cdots\wedge e_{k_{\tau(i)}}
\end{equation*}
but this maps to 0 in $E_{\tau(i)}/E_{\tau(i)-1}$ as $v(k_{i-1}+1)\leq k_{\tau(i)-1}$ and hence $e_{v(k_{i-1}+1)}\in E_{\tau(i)-1}$.  Now by considering the action of $B$ we can conclude that $C_i$ vanishes on all of $Y_v^B$ if and only if it vanishes at the point $(vF_i)$.

Thus to complete the proof of the proposition it suffices to show that when $C_i$ vanishes on $\overline{Y}_v^B$, it vanishes to order 1.  In order to do this it suffices to find a tangent vector $(F_j')\in\overline{Y}_w(\FF_p[\epsilon]/\epsilon^2)$ such that the underlying $\FF_p$ point of $(F_j')$ lies in $Y^B_v$ and such that the restriction of $C_i$ to $(F_j')$ is nonzero.

Indeed take the tangent vector to the point $vF_i$ given by
\begin{equation*}
F_j'=\langle e_{v(1)}+\epsilon e_{w(1)},\ldots,e_{v(\tilde{k}_j)}+\epsilon e_{w(\tilde{k}_j)}\rangle
\end{equation*}
Then $\det F'_i/F'_{i-1}$ is a free $k[\epsilon]/\epsilon^2$-module of rank 1 generated by the image of
\begin{equation*}
(e_{v(\tilde{k}_{i-1}+1)}+\epsilon e_{k_{\tau(i)-1}+1})\wedge (1+\varepsilon)e_{k_{\tau(i)-1}+2}\wedge\cdots\wedge(1+\varepsilon)e_{k_{\tau(i)}}
\end{equation*}
and this maps to
\begin{equation*}
\epsilon\cdot e_{k_{\tau(i)-1}+1}\wedge e_{k_{\tau(i)-1}+2}\wedge\cdots\wedge e_{k_{\tau(i)}}\in \det(E_{\tau(i)}/E_{\tau(i)-1})\otimes \FF_p[\epsilon]/\epsilon^2
\end{equation*}
under $C_i$.
\end{proof}

\subsection{An Inequality}\label{S:ineq}
Let $(w,J)$ be an admissible pair.  First we define some constants $c_i$, $i=1,\ldots, c$.  Let $N$ be the least common multiple of the lengths of the cycles of the permutation $\sigma$.  We let
\begin{equation*}
c_i=\sum_{j=0}^{N-1}\epsilon_{i,j} p^j
\end{equation*}
where for $j=1,\ldots,N$,
\begin{equation*}
\epsilon_{i,N-j}=\begin{cases}1&\text{if $\sigma^{j}(i)\leq c$}\\ -1&\text{otherwise.} \end{cases}
\end{equation*}

Now we prove
\begin{prop}\label{P:ineq}
If $(w,J)$ is admissible then for all $v\in D_w^B$ we have
\begin{equation*}
\sum_{i=1}^g c_i\text{\rm ord}_{\overline{Y}^B_v}(C_{c+1-i})>0
\end{equation*}
\end{prop}
\begin{proof}
We begin with a trivial remark.  If $r(x)=\sum_{j=0}^n a_jx^j$ is a polynomial with $a_j\in\{0,1,-1\}$ for all $j$, then if $a_n=1$, $r(p)>0$ for any prime number $p$.  Indeed, this is so because $p^n-p^{n-1}-\cdots-1>0$ for any prime number $p$.

Now we split up into two cases according to the two possibilities for $v$ in Proposition \ref{P:vcases}.  First suppose we have $v=v_j$ so that $1\leq j\leq c$ and $\tau(j)>g$.  Then by Proposition \ref{P:van}
\begin{equation*}
\sum_{i=1}^g c_i\text{\rm ord}_{\overline{Y}_v^B}(C_{c+1-i})=c_{c+1-j}.
\end{equation*}
In order to win we just need to check that the leading coefficient $\varepsilon_{j,N-1}$ of $c_j$ is $1$.  But $\sigma(c+1-j)=\tau(2c+1-j)=2c+1-\tau(j)\leq c$.

Now consider the case where $v=v_{a,b}$ with $1\leq a\leq c$, $c+1\leq b\leq 2c$, $\tau(b)<\tau(a)$ and $\tau(a)\leq c$.  Then by Proposition \ref{P:van},
\begin{equation*}
\sum_{i=1}^g c_i\text{\rm ord}_{\overline{Y}_v^B}(C_{c+1-i})=c_{c+1-a}+c_{b-c}.
\end{equation*}
This will be a polynomial in $p$ all of whose coefficients are either $0,2,$ or $-2$.  First we note that the coefficients of this polynomial are not identically zero.  Indeed, we have
\begin{equation*}
\sigma^N(c+1-a)=c+1-a\leq c\quad\text{and}\quad\sigma^N(b-c)=b-c\leq c
\end{equation*}
and hence
\begin{equation*}
\epsilon_{a,0}=\epsilon_{2c+1-b,0}=1
\end{equation*}
and hence the constant term is 2.

We need to prove that the leading nonzero term is positive.  From the recipe above, we see that the leading term will be $\pm 2p^{N-i}$ where $i$ is the smallest positive integer for which the pair of numbers
\begin{equation*}
\sigma^i(c+1-a),\quad\text{and}\quad \sigma^i(b-c)
\end{equation*}
are either both $\leq c$ or both $>c$.  By what we have said above such an $i$ exists (and is $\leq N$.)  Moreover the sign will be positive of they are both $\leq c$ and negative otherwise.  Let $r=\sigma(c+1-a)$, and $s=\sigma(b-c)$.  Then
\begin{equation*}
r=\sigma(c+1-a)=\tau(2c+1-a)=2c+1-\tau(a)
\end{equation*}
and
\begin{equation*}
s=\sigma(b-c)=\tau(b)
\end{equation*}
and so the conditions on $a$ and $b$ show that we have
\begin{equation*}
r>c\geq s
\end{equation*}
and
\begin{equation*}
r+s<2c+1.
\end{equation*}

What we want now follows by repeated application of the following lemma:
\begin{lem}
Let $r$ and $s$ satisfy
\begin{equation*}
r>c\geq s
\end{equation*}
and
\begin{equation*}
r+s<2c+1.
\end{equation*}
Then exactly one of the following three possibilities holds:
\begin{itemize}
\item $\sigma(r),\sigma(s)\leq g$.
\item The pair $(r',s')=(\sigma(r),\sigma(s))$ satisfies the hypotheses of the lemma.
\item The pair $(r',s')=(\sigma(s),\sigma(r))$ satisfies the hypotheses of the lemma.
\end{itemize}
\end{lem}
\begin{proof}
Note that the conclusion of the lemma is equivalent to the claim that $\sigma(r)+\sigma(s)< 2c+1$.  But the hypothesis on $r$ and $s$ imply that $s<2c+1-r\leq g$.  But then
\begin{equation*}
\sigma(s)<\sigma(2g+1-r)=2c+1-\sigma(r)
\end{equation*}
which is what we wanted.
\end{proof}
\end{proof}

\section{Moduli of Abelian Varieties with Parahoric Level, Local Models, and Kottwitz-Rapoport Stratification}

In this section we will recall the theory of moduli spaces of abelian varieties with parahoric level structure and their local models.  A basic reference for this subject is the original paper of de Jong \cite{dJ93}.  We also refer to the book of Rapoport-Zink \cite{RZ96} and the survey article of Haines \cite{Ha05}.

\subsection{The moduli problem}
We denote by $X_J$ the moduli space of principally polarized abelian schemes over a base of characteristic $p$ with parahoric level structure of type $J$, and an auxiliary prime to $p$ level structure which will be suppressed throughout this section because it plays no role except to rigidify our moduli problem.  Here are two equivalent descriptions of the moduli problem
\begin{enumerate}
\item $X_J$ parameterizes tuples $(\{A_i\},\{\phi_i:A_i\to A_{i+1}\},\lambda,\lambda')$ where
\begin{enumerate}
\item $A_i$ for $i=0,\ldots,c$ are abelian schemes of dimension $g$.
\item $\lambda:A_0\to \hat{A}_0$ and $\lambda':A_{c}\to\hat{A}_{c}$ are principal polarizations.
\item $\phi_i:A_i\to A_{i+1}$ for $i=0,\ldots,c-1$ are isogenies of degree $p^{k_i-k_{i-1}}$
\end{enumerate}
Moreover we impose the condition that the composition
\begin{equation*}
A_0\overset{\phi_0}{\to} A_1\to\cdots \to A_{c-1}\overset{\phi_{c-1}}{\to}A_c\overset{\lambda'}{\to}\hat{A}_c\overset{\phi_{c-1}^\vee}{\to}\hat{A}_{c-1}\to\cdots\to \hat{A}_1\overset{\phi_0^\vee}\to \hat{A}_0
\end{equation*}
is $p\lambda$.
\item $X_J$ parameterizes tuples $(A_0,\lambda,\{G_i\})$ where $(A_0,\lambda)$ is a principally polarized abelian scheme and $G_i$ for $i=1,\ldots,c$ is a finite flat subgroup scheme of $A_0[p]$ of order $p^{k_i}$ which satisfy $G_i\subset G_{i+1}$ for $i=1,\ldots, c-1$, and $G_c$ is isotropic for the Weil pairing on $A_0[p]$ induced by $\lambda$.
\end{enumerate}

To pass from description 1 to description 2, one takes $G_i=\ker \phi_i\phi_{i-1}\cdots \phi_0$.  To go in the other direction, take $A_i=A_0/G_i$ and let $\phi_i$ be the canonical map $A_0/G_{i-1}\to A_0/G_i$.  Finally to get $\lambda'$ consider the diagram
\begin{equation*}
\begin{CD}
A_0@>p\lambda>>\hat{A}_0\\
@VVV @AAA\\
A_c @. \hat{A}_c
\end{CD}
\end{equation*}
and argue that the existence of a bottom arrow making the diagram commute is equivalent to $G_c$ being isotropic.  

When we consider points of $X_J$ we will use all of the notation in 1 and 2 above, as well as the following additional notation.  We let $A_i=\hat{A}_{2c-i}$ for $c<i\leq 2c$.  We let $\phi_c=\phi_{c-1}^\vee\circ \lambda':A_c\to \hat{A}_{c-1}$, and we let $\phi_i=\phi_{2c-i}^\vee: A_i\to A_{i+1}$ for $c<i<2c$.  Using this notation, the chain of isogenies
\begin{equation*}
A_0\to A_1\to\cdots\to A_c\to\cdots\to A_{2c}
\end{equation*}
is self dual in the sense that the diagram
\begin{equation*}
\begin{CD}
A_0@>>>\cdots @>>> A_c@>>>\cdots @>>> A_{2c}\\
@| @. @VV\lambda'V @. @|\\
\hat{A}_{2c}@>>> \cdots @>>> \hat{A}_c @>>> \cdots @>>> \hat{A}_0
\end{CD}
\end{equation*}
commutes (here the top row is the original sequence, the bottom row is its dual, and all the vertical maps are the defining equalities except for the one in the middle which is $\lambda'$.)  Next we let $G_i=\ker(A_0\to A_i)$ for $c<i\leq 2c$.  Then the sequence of finite flat subgroup schemes
\begin{equation*}
0=G_0\subset G_1\subset\cdots\subset G_c\subset\cdots\subset G_{2c}=A_0[p]
\end{equation*}
is self dual for the Weil pairing on $A[p]$ induced by $\lambda$.  In particular we obtain isomorphisms $G_i/G_{i-1}\simeq (G_{2c+1-i}/G_{2c-i})^D$, where we recall that $G^D$ denotes the Cartier dual of the finite flat group scheme $G$.

We denote the universal point over $X_J$ by $(\{\Ca{A}_i,\{\phi_i\},\lambda,\lambda')$ and the universal subgroups by $\Ca{G}_i$.  As usual we will also use the same symbols to for the restrictions of these objects to subschemes of $X_J$, when the meaning is clear from context.

\subsection{Local models}
For $i=0,\ldots,2c-1$ we define a map
\begin{equation*}
\phi_i:V\to V
\end{equation*}
by
\begin{equation*}
\phi_i(e_j)=\begin{cases}0&\text{if $k_i<j\leq k_{i+1}$}\\ e_j&\text{otherwise.}\end{cases}
\end{equation*}
Then we define the local model $X^{\text{loc}}_J/\FF_p$ as the moduli space of tuples $(W_i)_{i=0,\ldots,c}$ where $W_i$ is a local direct summand of $V\otimes\O_S$ of rank $g$ which satisfy:
\begin{enumerate}
\item $\phi_i(W_i)\subset W_{i+1}$ for $i=0,\ldots,c-1$.
\item $W_0$ and $W_c$ are isotropic for $\psi$.
\end{enumerate}

For $i=c,\ldots,2c$ let $W_i=W_{2c-i}^\perp$.  With this notation, 1 is equivalent to $\phi_i(W_i)\subset W_{i+1}$ for $i=c,\ldots,2c-1$.

In this paper we will be especially interested in the irreducible component $X^{\text{loc},0}_J\subset X^{\text{loc}}_J$ where $W_c=\langle e_{g+1},\ldots e_{2g}\rangle \otimes\O_S$, or equivalently the locus where
\begin{equation*}
(\phi_{2c-1}\cdots\phi_{c+1}\phi_{c})|_{W_g}=0.
\end{equation*}

We now explain how to identify $X^{\text{loc},0}_J$ with a certain Schubert variety.  Indeed we will construct a map
\begin{equation*}
\psi:X^{\text{loc},0}_J\to \text{Fl}_{\tilde{J}}
\end{equation*}
and show that it defines an isomorphism $\psi:X^{\text{loc},0}\simeq \overline{Y}_x$.  Given a point $(W_i)$ of $X^{\text{loc},0}_J(S)$ we define
\begin{equation*}
F_i=\phi_{2c-1}\cdots \phi_{2c-i}(W_{2c-i}).
\end{equation*}
for $i=0,\ldots,c$.  Then one checks that $F_i\subset V\otimes\O_S$ is a local direct summand of rank $\tilde{k}_i$, and moreover one has inclusions $0=F_0\subset F_1\subset\cdots\subset F_c$.  Moreover, $F_c=W_{2c}$ is isotropic.  Hence we can define $F_{2c-i}=F_i^\perp$ for $i=0,\ldots c$ and we obtain a flag of type $\tilde{J}$.  It is clear from the definition that $F_i\subset E_{k_{i+c}}\otimes\O_S$ for $i=1,\ldots,c$ and hence we have defined a map $\psi:X_J^{\text{loc},0}\to \overline{Y}_x$.

This map has an inverse defined as follows: given a point $(F_i)\in \overline{Y}_x(S)$ we take
\begin{equation*}
W_{2c-i}=(\phi_{2c-1}\cdots\phi_{2c-i})^{-1}(F_i)=F_i\oplus \langle e_{k_{g+i}+1},\ldots,e_{2g}\rangle\otimes\O_S
\end{equation*}
for $i=0,\ldots,c$, and
\begin{equation*}
W_i=W_{2c-i}^\perp
\end{equation*}
for $i=0,\ldots,c$.  Hence we have an isomorphism $\psi:X_J^{\text{loc},0}\simeq\overline{Y}_x$.

Let $\overline{P}_J/k$ be the the subgroup
\begin{equation*}
\overline{P}_J\subset \text{GSp}(V,\psi)\times \GL(V)^{c-1}\times\text{GSp}(V,\psi)
\end{equation*}
of elements $(g_0,g_1,\ldots,g_c)$ satisfying
\begin{equation*}
\phi_ig_i=g_{i+1}\phi_i
\end{equation*}
for $i=0,\ldots,c-1$, and such that $g_0$ and $g_c$ have the same multiplier.  Then it is known that $\overline{P}_J$ is smooth (it is the special fiber of a parahoric group scheme.)  Moreover, $\overline{P}_J$ acts on $X_J^{\text{loc}}$ via $(g_i)\cdot(W_i)=(g_iW_i)$.  It is known that the action of $\overline{P}_J$ has finitely many orbits on $X_g^{\text{loc}}$ and they are naturally indexed by a certain finite set of affine Weyl group double cosets.  We will describe the orbits of $\overline{P}_J$ on the component $X_J^{\text{loc},0}$ explicitly in terms of (ordinary) Schubert varieties below, and this is all that will be used in this paper.

For a point $(g_i)_{i=0,\ldots,c}$ of $\overline{P}_J$ we denote
\begin{equation*}
g_i=\mu (g_{2c-i}^*)^{-1}
\end{equation*}
for $i=c+1,\ldots,2c$, where the adjoint is taken with respect to $\psi$ and $\mu$ is the common multiplier of $g_0$ and $g_c$.  With this definition we have $g_0=g_{2c}$ and $\phi_ig_i=g_{i+1}\phi_i$ for $i=c,\ldots,2c-1$.

There is a surjective homomorphism
\begin{align*}
\overline{P}_J&\to P\\
(g_i)&\mapsto g_0
\end{align*}
To see that $g_0$ actually lies in $P$ note the identity
\begin{equation*}
\phi_{i-1}\cdots\phi_0g_0=g_i\phi_{i-1}\cdots\phi_0.
\end{equation*}
The kernel of right hand side is $E_i$, while the kernel of the left hand side is $g_0^{-1}E_i$.  Hence $g_0E_i=E_i$.

Similarly, one shows that $g_c$ leaves $\langle e_{g+1},\ldots,e_{2g}\rangle$ invariant, and hence $X^{\text{loc},0}_J$ is stable under $\overline{P}_J$.  Moreover it is clear that the map $X_J^{\text{loc},0}\to \text{Fl}_{\tilde{J}}$ defined above is $\overline{P}_J$ equivariant, where we let $\overline{P}_J$ act on $\text{Fl}_{\tilde{J}}$ via its map to $P$.

Then we get a decomposition
\begin{equation*}
X^{\text{loc},0}_J=\coprod_{w\in {}^JW^{\tilde{J}}, w\leq x}X^{\text{loc}}_w
\end{equation*}
into irreducible locally closed subvarieties by transferring the stratification of $\overline{Y}_x$ by Schubert varieties.  Moreover this is precisely the decomposition of $X^{\text{loc},0}_J$ into its orbits under $\overline{P}_J$.  As usual we denote the Zariski closure of $X_{J,w}^{\text{loc}}$ by $\overline{X}_{J,w}^{\text{loc}}$.

Next we want to translate some of our constructions and results from the last section to $X^{\text{loc},0}$.  This is pure bookkeeping.  First we observe that for any point $(W_i)\in X^{\text{loc},0}(k)$, $k$ a field of characteristic $p$, the natural maps
\begin{equation*}
W_i\to W_j
\end{equation*}
for $c\leq i\leq j\leq 2c$ are of rank $g-(k_j-k_i)$.  Then for $i=0,\ldots,c$ we define
\begin{equation*}
U_i=\im( W_{2c-i}\to W_{2c})\qquad V_i=\coker(W_{2c-i}\to W_{2c})
\end{equation*}
and
\begin{equation*}
U_i'=\ker(W_c\to W_{2c-i})\qquad V_i'=\im(W_c\to W_{2c-i}).
\end{equation*}

Over all of $X^{\text{loc}}_J$ we have the tautological vector bundles $\Ca{W}_i$, $i=0,\ldots,2c$ corresponding to the universal point of $X^{\text{loc}}_J$.  By the observation above, over $X_J^{\text{loc},0}$ we can define locally free sheaves
\begin{equation*}
\Ca{U}_i=\im(\Ca{W}_{2c-i}\to \Ca{W}_{2c})\qquad \Ca{V}_i=\coker(\Ca{W}_{2c-i}\to \Ca{W}_{2c})
\end{equation*}
and
\begin{equation*}
\Ca{U}_i'=\ker(\Ca{W}_c\to \Ca{W}_{2c-i})\qquad \Ca{V}_i'=\im(\Ca{W}_c\to \Ca{W}_{2c-i}).
\end{equation*}
which, by considering the definition of the map $\psi$, fit in diagrams
\begin{equation*}
\begin{CD}
0@>>> \Ca{U}_i@>>>\Ca{W}_{2c}@>>>\Ca{V}_i@>>> 0\\
@. @VVV @VVV @VVV @.\\
0@>>> \psi^*\Ca{F}_{c-i}@>>> \psi^*\Ca{F}_c@>>> \psi^*(\Ca{F}_c/\Ca{F}_{c-i})@>>> 0
\end{CD}
\end{equation*}
and
\begin{equation*}
\begin{CD}
0@>>> \Ca{U}_i'@>>> \Ca{W}_c@>>> \Ca{V}_i' @>>> 0\\
@. @VVV @VVV @VVV @.\\
0@>>> \psi^*\Ca{E}_{2c-i}/\Ca{E}_c@>>> \psi^*\Ca{E}_{2c}/\Ca{E}_c@>>> \psi^*\Ca{E}_{2c}/\Ca{E}_{2c-i}@>>> 0
\end{CD}
\end{equation*}
where the rows are exact and where all the vertical maps are canonical isomorphisms induced by $\psi$.  We remark that in the second diagram, all the vector bundles are trivial.

Next we record:
\begin{prop}\label{P:less}
For $0\leq i,j\leq c$ and any point $(W_i)\in X_J^{\text{\rm loc},0}(k)$ consider the diagram
\begin{equation*}
\begin{CD}
0@>>> U_i@>>>W_0@>>> V_i@>>>0\\
@. @. @VVV @. @.\\
0@>>> U_j' @>>> W_c @>>> V_j'@>>>0
\end{CD}
\end{equation*}
Then the following are equivalent
\begin{enumerate}
\item We can fill in an arrow $U_i\to U_j'$ in the above diagram.
\item We can fill in an arrow $V_i\to V_j'$ in the above diagram.
\item If $(F_i)=\psi((W_i))$ is the corresponding point of $\text{\rm Fl}_{\tilde{J}}(k)$ then $F_{c-i}\subset E_{2c-j}\otimes k$.
\end{enumerate}
\end{prop}
\begin{proof}
The equivalence of the first two is obvious.  For the first and third, note that passing to the description in terms of $\text{Fl}_{\tilde{J}}$ we are asking whether the diagram
\begin{equation*}
\begin{CD}
F_{c-i}@>>> F_c\\
@. @VVV\\
E_{2c-j}/E_c@>>> E_{2c}/E_c
\end{CD}
\end{equation*}
can be filled in.
\end{proof}

Next for $i=1,\ldots,c$ we define the reduced closed subscheme $X_J^{\text{loc},i}\subset X_J^{\text{loc},0}$ by specifying that a field valued point $(W_i)\in X_J^{\text{loc},0}(k)$ lies in $X_J^{\text{loc},i}$ if and only if $\ker \phi_{c-1}\cdots\phi_{c-i}=\langle e_{k_{c-i}+1},\ldots,e_g\rangle\subset W_{c-i}$.

We want to understand what this condition means in terms of $(F_i)=\psi((W_i))$.  Then we note that
\begin{equation*}
\langle e_{k_{c-i}+1},\ldots,e_g\rangle\subset W_{c-i}
\end{equation*}
if and only if
\begin{equation*}
W_{c+i}=W_{c-i}^\perp\subset \langle e_{k_{c-i}+1},\ldots,e_g\rangle^\perp\subset\langle e_1,\ldots,e_g\rangle\oplus\langle e_{k_{c+i}+1},\ldots,e_{2g}\rangle
\end{equation*}
which holds if and only if
\begin{equation*}
F_i\subset\langle e_1,\ldots,e_g\rangle.
\end{equation*}

To summarize, we have shown
\begin{prop}\label{P:i}
For $i=1,\ldots,c$ we have
\begin{equation*}
X_J^{\text{\rm loc},i}=\coprod_{w\in{}^JW^{\tilde{J}},w\leq x,w\{1,\ldots,\tilde{k}_i\}\subset\{1,\ldots,g\}} X^{\text{\rm loc}}_{J,w}
\end{equation*}
\end{prop}
In particular $X_J^{\text{loc},i}$ is $\overline{P}_J$ stable (but it is easy to see this directly.)

Now let $c<i\leq 2c$.  For a point $(W_i)\in X_J^{\text{loc},i-c}(k)$ we note that the map $W_{2c-i}\to W_c$ has rank $k_c-k_{2c-i}$ (by the definition of $X_J^{\text{loc},i-c}$.)  Hence we can define
\begin{equation*}
U_i=\im(W_{2c-i}\to W_c)\qquad V_i=\coker(W_{2c-i}\to W_{c})
\end{equation*}
and similarly the map $\Ca{W}_{2c-i}|_{X_J^{\text{loc},2c-i}}\to\Ca{W}_{c}|_{X_J^{\text{loc},2c-i}}$ has constant rank and so we can define vector bundles on $X_J^{\text{loc},2c-i}$ by
\begin{equation*}
\Ca{U}_i=\im(\Ca{W}_{2c-i}\to \Ca{W}_c)\qquad \Ca{V}_i=\coker(\Ca{W}_{2c-i}\to \Ca{W}_{c})
\end{equation*}
Note that we cannot define vector bundles on all of $X_{J}^{\text{loc},0}$ by the same formulas.

At the level of points we have
\begin{align*}
W_i&=F_{i-c}\oplus\langle e_{k_i+1},\ldots,e_{2g}\rangle
\end{align*}
and hence
\begin{equation*}
W_{2c-i}=W_i^\perp=F_{3c-i}\cap\langle e_{k_{2c-i}+1},\ldots,e_{2g}\rangle
\end{equation*}
We have a diagram of sheaves on $X_J^{\text{loc},i}$
\begin{equation*}
\begin{CD}
0@>>> \Ca{U}_i@>>>\Ca{W}_c@>>> \Ca{V}_i@>>> 0\\
@. @VVV @VVV @VVV @.\\
0@>>> \psi^*\Ca{F}_{3c-i}/\Ca{E}_c@>>> \psi^*\Ca{E}_{2c}/\Ca{E}_c@>>> \psi^*\Ca{E}_{2c}/\Ca{F}_{3c-i}@>>>0
\end{CD}
\end{equation*}

Then we have
\begin{prop}\label{P:great}
For $c<i\leq 2c$ and $0\leq j\leq c$ and any point $(W_i)\in X_{J}^{\text{\rm loc},i-c}$ consider the diagram
\begin{equation*}
\begin{CD}
0@>>> U_j'@>>> W_c@>>> V_j'@>>>0\\
@. @. @V\text{\rm id}VV @. @.\\
0@>>> U_i@>>> W_c@>>> V_i@>>>0
\end{CD}
\end{equation*}
Then the following are equivalent
\begin{enumerate}
\item We can fill in an arrow $U_j'\to U_i$ in the above diagram.
\item We can fill in an arrow $V_j'\to V_i$ in the above diagram.
\item If $(F_i)=\psi((W_i))$ is the corresponding point of $\text{Fl}_{\tilde{J}}(k)$ then $E_{2c-j}\subset F_{3c-i}$, or equivalently $F_{i-c}\subset E_{j}$.
\end{enumerate}
\end{prop}
\begin{proof}
The equivalence of the first two points is obvious.  For the equivalence of 1 and 3, we note that upon applying $\psi$, 1 is equivalent to asking whether
\begin{equation*}
\begin{CD}
E_{2c-j}/E_c@>>>E_{2c}/E_c\\
@. @VVV\\
F_{3c-i}/E_c@>>> E_{2c}/E_c
\end{CD}
\end{equation*}
can be filled in.
\end{proof}

Now let $(w,J)$ be admissible.  Combining Propositions \ref{P:i}, \ref{P:less}, and \ref{P:great} with Proposition \ref{P:schubcon}
we conclude
\begin{prop}\label{P:admchar}
Let $(w,J)$ be admissible and let $i_0$ be the largest integer with $i_0\leq c$ and $w(\tilde{k}_{i_0})\leq c$.  Then $\overline{X}_{J,w}^{\text{\rm loc}}\subset X^{\text{\rm loc},i_0}_J$, and a point $(W_i)\in X_J^{\text{\rm loc},i_0}(k)$ lies in $\overline{X}_{J,w}^{\text{\rm loc}}$ if and only if
\begin{enumerate}
\item For $1\leq i\leq c$ with $\sigma(i)\leq c$ we can fill in
\begin{equation*}
\begin{CD}
W_0@>>> V_{i-1}\\
@VVV @.\\
W_c@>>> V_{\sigma(i)-1}'
\end{CD}
\end{equation*}
with an arrow $V_{i-1}\to V_{\sigma(i)-1}$.
\item For $1\leq i\leq c$ with $\sigma(i)>c$ we can fill in
\begin{equation*}
\begin{CD}
W_c@>>> V_{\sigma(2c+1-i)}'\\
@VVV @.\\
W_c@>>> V_{2c+1-i}
\end{CD}
\end{equation*}
\end{enumerate}
\end{prop}
\begin{proof}
By what we have proved, this is just a matter of keeping track of indices.  Let $(F_i)=\psi((W_i))$ By Proposition \ref{P:less} we see that 1 is equivalent to
\begin{equation*}
F_{c+1-i}\subset E_{\sigma(2c+1-i)}\otimes k=E_{\tau(c+1-i)}\otimes k
\end{equation*}
For 2, first note that $\sigma(i)>c$ if and only if $\tau(c+1-i)=\sigma(2c+1-i)\leq c$ and hence $V_{2c+1-i}$ is defined and we may apply proposition \ref{P:great}.  Then it says that 2 is equivalent to
\begin{equation*}
F_{c+1-i}\subset E_{\sigma(2c+1-i)}\otimes k=E_{\tau(c+1-i)}\otimes k
\end{equation*}
but these two conditions together are precisely the Schubert conditions for $\overline{Y}_w$ according to proposition \ref{P:schubcon}.
\end{proof}

\begin{prop}\label{P:schubcomp}
Let $(w,J)$ be admissible and let $(W_i)\in \overline{X}_{J,w}^{\text{\rm loc}}(k)$
\begin{enumerate}
\item If $1\leq i\leq c$ and $\sigma(i)\leq c$ then we can fill in the diagram
\begin{equation*}
\begin{CD}
W_0@>>> V_i\\
@VVV @.\\
W_c@>>> V'_{\sigma(i)}
\end{CD}
\end{equation*}
with a map $V_i\to V'_{\sigma(i)}$.  Thus over $\overline{X}_{J,w}^{\text{\rm loc}}$ we have a commutative diagram of locally free sheaves
\begin{equation*}
\begin{CD}
\Ca{W}_0@>>> \Ca{V}_i @>>> \Ca{V}_{i-1}\\
@VVV @VVV @VVV\\
\Ca{W}_c@>>> \Ca{V}'_{\sigma(i)} @>>> \Ca{V}'_{\sigma(i)-1}
\end{CD}
\end{equation*}
and hence we obtain a map of sheaves on $\overline{X}_{J,w}^{\text{\rm loc}}$
\begin{equation*}
A_i^{\text{\rm loc}}:\det\ker(\Ca{V}_i\to\Ca{V}_{i-1})\to\det\ker(\Ca{V}'_{\sigma(i)}\to\Ca{V}'_{\sigma(i)-1})
\end{equation*}
which, fits in the commutative diagram
\begin{equation*}
\begin{CD}
\det\ker(\Ca{V}_i\to\Ca{V}_{i-1})@>A_i^{\text{\rm loc}}>>\det\ker(\Ca{V}'_{\sigma(i)}\to\Ca{V}'_{\sigma(i)-1})\\
@VVV @VVV\\
\psi^*\det(\Ca{F}_{c+1-i}/\Ca{F}_{c-i})@>\psi^*C_{c+1-i}>>\psi^*\det(\Ca{E}_{\tau(c+1-i)}/\Ca{E}_{\tau(c-i)})
\end{CD}
\end{equation*}
where the vertical maps are isomorphisms induced by $\psi$.
\item If $1\leq i\leq c$ and $\sigma(i)>c$ then we can fill in the diagram
\begin{equation*}
\begin{CD}
W_c@>>> V_{\sigma(2c+1-i)-1}'\\
@VVV @.\\
W_c @>>> V_{2c-i}
\end{CD}
\end{equation*}
with an arrow $V_{\sigma(2c+1-i)-1}'\to V_{2c-i}$.  Thus over $\overline{X}_{J,w}^{\text{\rm loc}}$ we have a commutative diagram of locally free sheaves
\begin{equation*}
\begin{CD}
\Ca{W}_c@>>> \Ca{V}_{\sigma(2c+1-i)}'@>>> \Ca{V}_{\sigma(2c+1-i)-1}'\\
@VVV @VVV @VVV\\
\Ca{W}_c@>>> \Ca{V}_{2c+1-i}@>>> \Ca{V}_{2c-i}
\end{CD}
\end{equation*}
and hence we obtain a map of sheaves on $\overline{X}_{J,w}^{\text{loc}}$
\begin{equation*}
B_{2c+1-i}^{\text{\rm loc}}:\det\ker(\Ca{V}'_{\sigma(2c+1-i)}\to \Ca{V}'_{\sigma(2c+1-i)-1})\to\det\ker(\Ca{V}_{2c+1-i}\to\Ca{V}_{2c-i})
\end{equation*}
which fits in a commutative diagram
\begin{equation*}
\begin{CD}
\det\ker(\Ca{V}'_{\sigma(2c+1-i)}\to \Ca{V}'_{\sigma(2c+1-i)-1})@>B_{2c+1-i}^{\text{\rm loc}}>> \det\ker(\Ca{V}_{2c+1-i}\to\Ca{V}_{2c-i})\\
@VVV @VVV\\
\psi^*\det(\Ca{E}_{\sigma(i)}/\Ca{E}_{\sigma(i)-1})@>\psi^*C_{c+1-i}^\vee>>\psi^*\det(\F_{c+i}/\F_{c+i-1})
\end{CD}
\end{equation*}
where the vertical maps are isomorphisms induced by $\psi$.
\end{enumerate}
\end{prop}
\begin{proof}
Let $(F_i)=\psi((W_i))$.  If $\sigma(i)\leq c$ then by Proposition \ref{P:less}, to prove the first part of 1 it suffices to show that
\begin{equation*}
F_{c-i}\subset  E_{2c-\sigma(i)}.
\end{equation*}
Similarly if $\sigma(i)>c$ then by Proposition \ref{P:great}, to prove the first part of 2 it also suffices to show that
\begin{equation*}
F_{c-i}\subset  E_{2c-\sigma(i)}.
\end{equation*}
In either case, if $i=c$ there is nothing to prove.  Otherwise $\tau(c-i)\leq \tau(c+1-i)-1=2c-\sigma(i)$ and
\begin{equation*}
F_{c-i}\subset E_{\tau(c-i)}\otimes k\subset E_{2c-\sigma(i)}\otimes k.
\end{equation*}
where the first inclusion follows from the Schubert condition for $\overline{Y}_w$.  The rest of the proposition is a diagram chase.
\end{proof}

\subsection{The local model diagram and the Kottwitz-Rapoport stratification.}
We now explain the connection between $X_J$ and the local model $X^{\text{loc}}_J$ introduced in the last section.  Let $\tilde{X}_J$ be the moduli space of $(\{A_i\},\{\phi_i:A_i\to A_{i+1}\},\lambda,\lambda',\{\alpha_i\})$ where $(\{A_i\},\{\phi_i\},\lambda,\lambda')$ gives a point of $X_J$ and for $i=0,\ldots,c$, $\alpha_i:H^1_{dR}(\hat{A}_{2c-i}/S)\to V\otimes\O_S$ are isomorphisms which satisfy:
\begin{enumerate}
\item $\alpha_0$ and $\alpha_c$ send the poincare pairings induced by $\lambda$ and $\lambda'$ to $\psi$.
\item The following diagram commutes
\begin{equation*}
\begin{CD}
H^1_{dR}(A_{2c}/S)@>\phi_{2c-1}^*>> H^1_{dr}(A_{2c-1}/S)@>\phi_{2c-2}^*>> \cdots @>\phi_c^*>> H^1_{dR}(A_c/S)\\
@VVV @VVV @. @VVV\\
V\otimes\O_S@>\phi_0>> V\otimes\O_S@>\phi_1>> \cdots @>\phi_{c-1}>> V\otimes\O_S.
\end{CD}
\end{equation*}
\end{enumerate}

We have a diagram
\begin{equation*}
X_J\overset{p_1}{\leftarrow}\tilde{X}_J\overset{p_2}{\to}X_J^{\text{loc}}
\end{equation*}
where $p_1$ is given by ``forget the $\alpha$'s'' and $p_2$ sends $(\{A_i\},\{\phi_i:A_i\to A_{i+1}\},\lambda,\lambda',\{\alpha_i\})$ to $(\alpha_i(\omega_{A_{2c-i}}))_{i=0,\ldots,c}$.  We have an action of the group $\overline{P}_J$ on $\tilde{X}_J$ where $(g_i)_{i=0,\ldots,c}$ acts by replacing $\{\alpha_i\}$ with $\{g_i\alpha_i\}$.  The maps $p_1$ and $p_2$ are equivariant for this action where we let $\overline{P}_J$ act trivially on $X_J$.  Then the basic result of the theory of local models is (see \cite{dJ93} or \cite{Ha05})
\begin{thm}\label{T:locmodel}
\begin{enumerate}
\item The map $p_1$ is a torsor for $\overline{P}_J$.  In particular it is smooth and surjective.
\item The map $p_2$ is smooth.
\end{enumerate}
\end{thm}

Given any locally closed subvariety $Z^{\text{loc}}\subset X_J^{\text{loc}}$ we can define $Z=p_1(p_2^{-1}(Z^{\text{loc}}))\subset X_J$.  If $Z^{\text{loc}}$ is $\overline{P}_J$ stable then $p_1^{-1}(Z)=p_2^{-1}(Z^{\text{loc}})$.  In this case we can often transfer ``smooth local'' information from $Z$ to $Z'$.  For example we have the following Proposition.

\begin{prop}\label{P:trans}
Let $Z^{\text{\rm loc}}\subset X_J^{\text{\rm loc}}$ be $\overline{P}_J$ stable and let $Z=p_1(p_2^{-1}(Z^{\text{\rm loc}}))$ be the corresponding subvariety of $X_J$.
\begin{enumerate}
\item If $Z^{\text{\rm loc}}$ is normal then so is $Z$.
\item Let $W^{\text{\rm loc}}\subset Z^{\text{\rm loc}}$ be a $\overline{P}_J$ stable prime divisor such that the local ring of $Z^{\text{\rm loc}}$ at the generic point of $W^{\text{\rm loc}}$ is regular.  Then $W\subset Z$ is a divisor with the property that the local ring of $Z$ at every generic point of $W$ is regular.  Moreover if there are line bundles $\L$ on $Z$ and $\L^{\text{\rm loc}}$ on $Z^{\text{\rm loc}}$, sections $s\in H^0(Z,\L)$ and $s^{\text{\rm loc}}\in H^0(Z^{\text{\rm loc}},\L^{\text{\rm loc}})$, and an isomorphism $p_1^*\L\simeq p_2^*\L^{\text{\rm loc}}$ on $p_1^{-1}(Z)=p_2^{-1}(Z^{\text{\rm loc}})$ sending $s$ to $s^{\text{\rm loc}}$ then
\begin{equation*}
\text{\rm ord}_W(s)=\text{\rm ord}_{W^{\text{\rm loc}}}(s^{\text{\rm loc}})
\end{equation*}
where the left hand side is to be interpreted as the order of vanishing on every irreducible component of $W$.
\end{enumerate}
\end{prop}

We now record the following well known lemma.
\begin{lem}\label{L:Ffactor}
Let $\phi:A\to B$ be an isogeny of abelian varieties over a field $k$ of characteristic $p$ of degree $p^i$ and such that
\begin{equation*}
\dim_k(\ker \phi^*:\omega_B\to\omega_A)=i.
\end{equation*}
Then Frobenius $F:A\to A^{(p)}$ factors through $\phi$:
\begin{equation*}
A\overset{\phi}{\to}B\to A^{(p)} 
\end{equation*}
\end{lem}

For $i=0,\ldots,c$ we let $X_J^i\subset X_J$ be the subvariety associated with $X_J^{\text{loc,i}}\subset X_J^{\text{loc}}$.  Then we have
\begin{prop}\label{P:idesc}
Let $x=(\{A_i\},\{\phi_i\},\lambda,\lambda')\in X_J(k)$ be a field valued point of $X_J$.
\begin{enumerate}
\item We have $x\in X_J^0(k)$ if and only if the isogeny $A_0\to A_c$ factors
\begin{equation*}
A_0\overset{F}{\to} A_0^{(p)}\simeq A_c.
\end{equation*}
\item If this is the case then we further have $x\in X_J^i(k)$ if and only if $G_{c+i}/G_c$ is killed by $F$.
\end{enumerate}
\end{prop}
\begin{proof}
From the definition we have $x\in X_J^0(k)$ if and only if the map $\omega_0\to\omega_c$ is 0.  But $A_0\to A_c$ has degree $p^g$, and hence 1 follows from Lemma \ref{L:Ffactor}.  For 2, note that $G_c/G_{c+i}=\ker A_c\to A_{c+i}$.  Hence this is killed by frobenius if and only if $A_c\to A_{c+i}$ factors through Frobenius.  On the other hand, by definition we have $x\in X_J^i(k)$ if and only if $\dim_k(\ker(\omega_{A_c}\to\omega_{A_{c+i}}))=k_{c+i}-k_c$.  As the degree of $A_c\to A_{c+i}$ is $p^{k_{c+i}-k_c}$, the result follows again rom Lemma \ref{L:Ffactor}.
\end{proof}

As $X_J^0$ is reduced, we conclude from the proposition that for the universal family over $X_J^0$, the isogeny $\Ca{A}_0\to\Ca{A}_c$ factors
\begin{equation*}
\Ca{A}_0\overset{F}{\to}\Ca{A}_0^{(p)}\overset{\gamma}{\to}\Ca{A}_c,
\end{equation*}
where $\gamma$ is an isomorphism.  Here as before we are abusing notation by not writing restrictions when they are clear from context.  Similarly we conclude that over $X_J^i$, $\Ca{G}_{i+c}/\Ca{G}_i$ is killed by Frobenius.

We now recall the following important theorem.
\begin{thm}\label{T:fkill}
Let $S$ be a scheme over $\FF_p$.  The functor
\begin{equation*}
G\mapsto (\omega_G,V^*:\omega_G\to\omega_G^{(p)})
\end{equation*}
defines an anti equivalence between the category of finite, locally free group schemes $G/S$ killed by Frobenius, and the category of pairs $(\Ca{M},V:\Ca{M}\to\Ca{M}^{(p)})$ with $\Ca{M}$ a locally free sheaf on $S$ 
\end{thm}

Now for $i=1,\ldots,c$ we have an exact sequence of group schemes on $X_J$
\begin{equation*}
0\to \Ca{G}_i\to \Ca{A}_0\to \Ca{A}_i.
\end{equation*}
Over $X_J^0$ we have $\Ca{G}_i\subset \Ca{A}_0[F]=G_c$ and hence over $X_J^0$ we have an exact sequence of finite flat groups schemes killed by Frobenius
\begin{equation*}
0\to \Ca{G}_i\to\Ca{A}_0[F]\to\Ca{A}_i[F].
\end{equation*}
Hence we have an exact sequence
\begin{equation*}
\omega_{\Ca{A}_i}\to\omega_{\Ca{A}_0}\to \omega_{\Ca{G}_i}\to 0.
\end{equation*}
Over $\tilde{X}_J^0$ we have a commutative diagram
\begin{equation*}
\begin{tikzcd}
p_1^*\omega_{\Ca{A}_i}\arrow{r}\arrow{d}& p_1^*\omega_{\Ca{A}_0}\arrow{r}\arrow{d}&p_1^*\omega_{\Ca{G}_i}\arrow{r}\arrow[dotted]{d}& 0\\
p_2^*\Ca{W}_{2c-i}\arrow{r}&p_2^*\Ca{W}_{2c}\arrow{r}&p_2^*\Ca{V}_i\arrow{r}&0
\end{tikzcd}
\end{equation*}
where the first two vertical arrows are isomorphisms induced by the $\alpha$'s and the dotted arrow is the induced isomorphism.

Next, for $c<i\leq 2c$ we have an exact sequence
\begin{equation*}
0\to \Ca{G}_i/\Ca{G}_c\to \Ca{A}_c\to \Ca{A}_i.
\end{equation*}
By Proposition \ref{P:idesc}, over $X_J^{i-c}$ we have $\Ca{G}_i/\Ca{G}_c\subset\Ca{A}_c[F]$ and hence we have an exact sequence of finite flat group schemes killed by Frobenius
\begin{equation*}
0\to \Ca{G}_i/\Ca{G}_c\to\Ca{A}_c[F]\to\Ca{A}_i[F].
\end{equation*}
Then we get an exact sequence of sheaves
\begin{equation*}
\omega_{\Ca{A}_i}\to\omega_{\Ca{A}_c}\to \omega_{\Ca{G}_i/\Ca{G}_c}\to 0.
\end{equation*}
Over $\tilde{X}_J^{i-c}$ we get a commutative diagram
\begin{equation*}
\begin{tikzcd}
p_1^*\omega_{\Ca{A}_i}\arrow{r}\arrow{d}& p_1^*\omega_{\Ca{A}_c}\arrow{r}\arrow{d}&p_1^*\omega_{\Ca{G}_i/\Ca{G}_c}\arrow{r}\arrow[dotted]{d}& 0\\
p_2^*\Ca{W}_{2c-i}\arrow{r}&p_2^*\Ca{W}_{c}\arrow{r}&p_2^*\Ca{V}_i\arrow{r}&0
\end{tikzcd}
\end{equation*}
where the first two vertical arrows are isomorphisms induced by the $\alpha$'s and the dotted arrow is the induced isomorphism.

Now for $0<i\leq c$ note that
\begin{equation*}
\Ca{G}_c/\Ca{G}_i=\ker(\Ca{A}_i\to\Ca{A}_c).
\end{equation*}
Over $X_J^0$, $\Ca{G}_c/\Ca{G}_i$ is killed by Frobenius and hence over $X_J^0$ we also have
\begin{equation*}
\Ca{G}_c/\Ca{G}_i=\ker(\Ca{A}_i[F]\to\Ca{A}_c[F]).
\end{equation*}
Hence the the image of the map $\Ca{A}_i[F]\to\Ca{A}_c[F]$ is representable by a finite flat subgroup scheme $\Ca{H}_i\subset \Ca{A}_c[F]$.  By Proposition \ref{P:idesc} there is a canonical isomorphism
\begin{equation*}
\gamma:\Ca{A}_0^{(p)}\overset{\sim}{\to}\Ca{A}_c
\end{equation*}
and hence
\begin{equation*}
\gamma:\Ca{A}_0^{(p)}[F]\overset{\sim}{\to}\Ca{A}_c[F].
\end{equation*}
Under this isomorphism we have
\begin{equation*}
\gamma:\Ca{G}_i^{(p)}\overset{\sim}{\to}\Ca{H}_i.
\end{equation*}

Hence we have natural isomorphisms
\begin{equation*}
\omega_{\Ca{G}_i^{(p)}}\simeq\im(\omega_{\Ca{A}_c}\to\omega_{\Ca{A}_i})
\end{equation*}
But on $\tilde{X}_J^0$ we have a commutative diagram
\begin{equation*}
\begin{tikzcd}
p_1^*\omega_{\Ca{A}_i}\arrow{r}\arrow{d}&p_1^*\omega_{\Ca{A}_c}\arrow{d}\\
p_2^*\Ca{W}_{2c-i}\arrow{r}&p_2^*\Ca{W}_c
\end{tikzcd}
\end{equation*}
where the vertical maps are isomorphisms induced by the $\alpha$'s.  Hence we have an induced isomorphism
\begin{equation*}
p_1^*\omega_{\Ca{G}_i^{(p)}}\simeq p_2^*\Ca{V}_i'.
\end{equation*}

Now we record
\begin{prop}\label{P:less2}
Let $x=(\{A_i\},\{\phi_i\},\lambda,\lambda')\in X_J^0(k)$ be a field valued point.  Let $0\leq i,j\leq c$.  Then the following are equivalent
\begin{enumerate}
\item $V(G_j^{(p)})\subset G_i$
\item $F(G_{2c-i})\subset G_{2c-j}^{(p)}$
\item We can fill in the dotted arrow in the diagram
\begin{equation*}
\begin{tikzcd}
\omega_{A_{2c}}\arrow[two heads]{r}\arrow{d}&\omega_{G_i}\arrow[dotted]{d}\\
\omega_{A_c}\arrow[two heads]{r}&\omega_{G_j^{(p)}}
\end{tikzcd}
\end{equation*}
\end{enumerate} 
\end{prop}
\begin{proof}
The equivalence of the first two points follows from Cartier Duality.  For the equivalence of 1 and 3 note that by Theorem \ref{T:fkill} we have $V(G_j^{(p)})\subset G_i$ if and only if we can fill in the diagram
\begin{equation*}
\begin{tikzcd}
\omega_{A_0[F]}\arrow[two heads]{r}\arrow{d}{V^*}&\omega_{G_i}\arrow[dotted]{d}\\
\omega_{A_0^{(p)}[F]}\arrow[two heads]{r}&\omega_{G_j^{(p)}}
\end{tikzcd}
\end{equation*}
But now we have a commutative diagram
\begin{equation*}
\begin{tikzcd}
\omega_{A_{2c}}\arrow{r}\arrow{d}&\omega_{A_0[F]}\arrow{d}{V^*}\\
\omega_{A_c}\arrow{r}&\omega_{A_0^{(p)}[F]}
\end{tikzcd}
\end{equation*}
where the top horizontal arrow is the isomorphism induced by $\lambda:A_0\to A_{2c}=\hat{A}_0$, the bottom horizontal arrow is induced by $\gamma$, and the commutativity of the diagram follows from the fact that $A_c\to A_{2c}$ is the dual of $A_0\to A_c$ and hence factors
\begin{equation*}
A_c\overset{\gamma^\vee}{\to}A_{2c}^{(p)}\overset{V}{\to}A_c
\end{equation*}

\end{proof}
Similarly we have
\begin{prop}\label{P:great2}
Let $x=(\{A_i\},\{\phi_i\},\lambda,\lambda')\in X_J^0(k)$ be a field valued point.  Let $0\leq i\leq c$ and $c\leq j\leq 2c$.  Then the following are equivalent
\begin{enumerate}
\item $V(G_j^{(p)})\subset G_i$
\item $F(G_{2c-i})\subset G_{2c-j}^{(p)}$
\item We have $x\in X_J^{c-i}(k)$ and we can fill in the dotted arrow in the diagram
\begin{equation*}
\begin{tikzcd}
\omega_{A_c}\arrow[equals]{d}\arrow[two heads]{r}&\omega_{G_{2c-j}^{(p)}}\arrow[dotted]{d}\\
\omega_{A_c}\arrow[two heads]{r}&\omega_{G_{2c-i}/G_c}
\end{tikzcd}
\end{equation*}
\end{enumerate}
\end{prop}
\begin{proof}
The equivalence of the first two points follows from Cartier duality.  Now we prove the equivalence of 2 and 3.  First note that if $F(G_{2c-i})\subset G_{2c-j}^{(p)}$ then as $G_{2c-j}^{(p)}\subset G_c^{(p)}$ we conclude that $F$ kills $G_{2c-i}/G_c$ and hence that $x\in X_J^{c-i}(k)$.  Now assuming that $x\in X_J^{c-i}(k)$, point 2 is holds if and only if we can fill in the dotted arrow in the diagram
\begin{equation*}
\begin{tikzcd}
\omega_{A_0^{(p)}[F]}\arrow{d}\arrow[two heads]{r}&\omega_{G_{2c-j}^{(p)}}\arrow[dotted]{d}\\
\omega_{A_c[F]}\arrow[two heads]{r}&\omega_{G_{2c-i}/G_c}
\end{tikzcd}
\end{equation*}
where the first vertical arrow is the isomorphism induced by $\gamma$.
\end{proof}

Now from the stratification of $X_J^{\text{loc},0}$ by Schubert varieties we obtain a stratification
\begin{equation*}
X_J^0=\coprod_{w\in{}^JW^{\tilde{J}},w\leq x}X_{J,w}.
\end{equation*}
This is the Kottwitz-Rapoport stratification of $X_J^0$.  In fact there is a Kottwitz-Rapoport stratification of all of $X_J$, obtained from the stratification of all of $X_J^{\text{loc}}$ by $\overline{P}_J$ orbits, but it will not play a role in this thesis.  As usual we denote the Zariski closure of $X_{J,w}$ by $\overline{X}_{J,w}$.

For the rest of this section we fix an admissible $(w,J)$.  We now record the following corollary to Propositions \ref{P:less2} and \ref{P:great2}, the discussion leading up to them, and Proposition \ref{P:admchar}.
\begin{prop}
Let $(w,J)$ be admissible and let $x=(\{A_i\},\{\phi_i\},\lambda,\lambda')\in X_J^0(k)$ be a field valued point.  Then the following are equivalent
\begin{enumerate}
\item $x\in\overline{X}_{J,w}(k)$
\item For $1\leq i\leq c$ we have $V(G_{\sigma(i)-1}^{(p)})\subset G_{i-1}$.
\item For $c+1\leq i\leq 2c$ we have $F(G_{i})\subset G_{\sigma(i)}^{(p)}$.
\end{enumerate}
\end{prop}

As $\overline{X}_{J,w}$ is reduced, we conclude that over $\overline{X}_{J,w}$, for $1\leq i\leq c$:
\begin{itemize}
\item The map $V:\Ca{G}_{\sigma(i)-1}^{(p)}\to \Ca{G}_{\sigma(i)-1}$ factors through $\Ca{G}_{i-1}$.
\item The map $V:\Ca{G}_{\sigma(i)}^{(p)}\to\Ca{G}_{\sigma(i)}$ factors through $\Ca{G}_{i}$.  Indeed, if $i=c$ then $\Ca{G}_c=\im(V:\Ca{G}_{2c}^{(p)}\to\Ca{G}_{2c})$ and there is nothing to prove.  Otherwise this follows from the previous point and the fact that $\Ca{G}_{\sigma(i)}^{(p)}\subset\Ca{G}_{\sigma(i+1)-1}^{(p)}$.
\end{itemize}
Similarly for $c+1\leq i\leq 2c$
\begin{itemize}
\item The map $F:\Ca{G}_i\to \Ca{G}_{i}^{(p)}$ factors through $\Ca{G}_{\sigma(i)}^{(p)}$.
\item The map $F:\Ca{G}_{i-1}\to\Ca{G}_{i-1}^{(p)}$ factors through $\Ca{G}_{\sigma(i)-1}^{(p)}$.  Indeed, if $i=c+1$ there is nothing to prove as $\Ca{G}_c$ is killed by Frobenius.  Otherwise by the previous point it factors through $\Ca{G}_{\sigma(i-1)}^{(p)}\subset\Ca{G}_{\sigma(i)-1}^{(p)}$.
\end{itemize}

We remark that it doesn't necessarily make sense to say something like $V(\Ca{G}_{\sigma(i)-1}^{(p)})\subset\Ca{G}_{i-1}$ because the image on the left may not exist.

\begin{prop}\label{P:omega}
Let $(w,J)$ be admissible.  Then over $\overline{X}_{J,w}$, for $i=1,\ldots,2c$ we have
\begin{enumerate}
\item The map $V:(\Ca{G}_i/\Ca{G}_{i-1})^{(p)}\to \Ca{G}_{i}/\Ca{G}_{i-1}$ is 0 unless $\sigma(i)=i$ and $i\leq c$.
\item The map $F:\Ca{G}_i/\Ca{G}_{i-1}\to (\Ca{G}_{i}/\Ca{G}_{i-1})^{(p)}$ is 0 unless $\sigma(i)=i$ and $i\geq c+1$.
\item $\omega_{\Ca{G}_i/\Ca{G}_{i-1}}$ is locally free of rank $k_i-k_{i-1}$ unless $\sigma(i)=i$ and $i\geq c+1$.
\item We have a canonical isomorphism $\omega_{\Ca{G}_i/\Ca{G}_{i-1}}\simeq \omega_{\Ca{G}_{2c+1-i}/\Ca{G}_{2c-i}}^\vee$ induced by $\lambda$, unless $\sigma(i)=i$.
\end{enumerate}
\end{prop}
\begin{proof}
By Cartier duality, it suffices to prove 1 and 2 for $i=1,\ldots,c$.  Then 2 is clear because $\Ca{G}_i\subset\Ca{G}_c$ is killed by Frobenius.  For 1, let $1\leq j\leq c$ be the unique integer with $\sigma(j)\geq i>\sigma(j-1)$.  Then by the assumption that $\sigma(i)\not=i$ we must have $i>j$ and hence
\begin{equation*}
\Ca{G}_i^{(p)}\subset\Ca{G}_{\sigma(j)}^{(p)}
\end{equation*}
but
\begin{equation*}
V:\Ca{G}_{\sigma(j)}^{(p)}\to\Ca{G}_{\sigma(j)}
\end{equation*}
factors through $\Ca{G}_j\subset\Ca{G}_{i-1}$.

To prove 3, note that by 2, unless $\sigma(i)=i$ and $i\geq c+1$, $\Ca{G}_i/\Ca{G}_{i-1}$ is a finite flat group scheme of order $p^{k_i-k_{i-1}}$ which is killed by Frobenius, and hence $\omega_{\Ca{G}_i/\Ca{G}_{i-1}}$ is locally free of rank $k_i-k_{i-1}$.

Finally to prove 4, we note that by 1 and 2, unless $\sigma(i)=i$, $\Ca{G}_i/\Ca{G}_{i-1}$ is an $\alpha$-group (i.e. both $F$ and $V$ are 0) and hence the isomorphism
\begin{equation*}
\Ca{G}_i/\Ca{G}_{i-1}\simeq(\Ca{G}_{2c+1-i}/\Ca{G}_{2c+1})^D
\end{equation*}
induced by $\lambda$ induces
\begin{equation*}
\omega_{\Ca{G}_i/\Ca{G}_{i-1}}\simeq \omega_{\Ca{G}_{2c+1-i}/\Ca{G}_{2c-i}}^\vee.
\end{equation*}
\end{proof}

For $i=1,\ldots, 2c$, suppose that $\sigma(i)\not=i$ if $i\geq c+1$ and let
\begin{equation*}
\omega_i:=\det\omega_{\Ca{G}_i/\Ca{G}_{i-1}},
\end{equation*}
a line bundle on $\overline{X}_{J,w}$.  Also let $\omega:=\det\omega_{\Ca{A}_0}$.  Then by the proposition, we have canonical isomorphisms
\begin{equation*}
\omega_i\simeq\omega_{2c+1-i}^\vee
\end{equation*}
unless $\sigma(i)=i$.  Moreover from the filtration
\begin{equation*}
0=\Ca{G}_0\subset\Ca{G}_1\subset\Ca{G}_2\subset\cdots\subset\Ca{G}_c=\Ca{A}_0[F]
\end{equation*}
we have an isomorphism
\begin{equation*}
\omega=\det \omega_{\Ca{A}_0[F]}\simeq\bigotimes_{i=1}^c\omega_i.
\end{equation*}

Now we come to the key definition of this thesis.  For $i=1,\ldots c$ we have maps over $\overline{X}_{J,w}$
\begin{equation*}
V:(\Ca{G}_{\sigma(i)}/\Ca{G}_{\sigma(i)-1})^{(p)}\to\Ca{G}_i/\Ca{G}_{i-1}
\end{equation*}
which induce maps
\begin{equation*}
V^*:\omega_{\Ca{G}_i/\Ca{G}_{i-1}}\to\omega_{\Ca{G}_{(\sigma(i)}/\Ca{G}_{\sigma(i)-1})^{(p)}}
\end{equation*}
and hence upon taking determinants
\begin{equation*}
A_i:\omega_i\to\omega_{\sigma(i)}^p.
\end{equation*}
We also view this as a section
\begin{equation*}
A_i\in H^0(\overline{X}_{J,w},\omega_{\sigma(i)}^{p}\otimes\omega_i^{-1}).
\end{equation*}
Similarly for $i=c+1,\ldots,2c$ with $\sigma(i)\not=i$ we have
\begin{equation*}
F:\Ca{G}_i/\Ca{G}_{i-1}\to(\Ca{G}_{\sigma(i)}/\Ca{G}_{\sigma(i)-1})^{(p)}
\end{equation*}
aand we can form
\begin{equation*}
F^*:\omega_{(\Ca{G}_{\sigma(i)}/\Ca{G}_{\sigma(i)-1})^{(p)}}\to \omega_{\Ca{G}_i/\Ca{G}_{i-1}}
\end{equation*}
and hence upon taking determinants
\begin{equation*}
B_i:\omega_{\sigma(i)}^{(p)}\to\omega_i.
\end{equation*}
We also view this as a section
\begin{equation*}
B_i\in H^0(\overline{X}_{J,w},\omega_{i}\otimes\omega_{\sigma(i)}^{-p})
\end{equation*}

If $i=1\ldots,c$ and $\sigma(i)\not=i$ then Cartier duality gives a commutative diagram
\begin{equation*}
\begin{tikzcd}
\omega_{\Ca{G}_i/\Ca{G}_{i-1}}\arrow{r}{V^*}\arrow{d}&\omega_{\Ca{G}_{\sigma(i)}/\Ca{G}_{\sigma(i)-1}}\arrow{d}\\
\omega_{\Ca{G}_{2c+1-i}/\Ca{G}_{2c-i}}^\vee\arrow{r}{(F^*)^\vee}&\omega_{\Ca{G}_{2c+1-\sigma(i)}/\Ca{G}_{2c-\sigma(i)}}^\vee
\end{tikzcd}
\end{equation*}
where the vertical maps are the isomorphisms of Proposition \ref{P:omega} 4.  Hence we conclude that
\begin{equation*}
A_i=B_{2c+1-i}^\vee\in H^0(\overline{X}_{J,w},\omega_{\sigma(i)}^p\otimes\omega_i^{-1}).
\end{equation*}

We want to compute the vanishing locus of the $A_i$ (or equivalently $B_{2c+1-i}$).  In preparation for this we record the following proposition.
\begin{prop}
Let $(w,J)$ be admissible.  Then
\begin{enumerate}
\item $\overline{X}_{J,w}$ is normal.
\item The complement
\begin{equation*}
\overline{X}_{J,w}-X_{J,w}=\bigcup_{v\in D_w}\overline{X}_{J,v}
\end{equation*}
is a union of (not necessarily irreducible) divisors.
\end{enumerate}
\end{prop}
\begin{proof}
This is an immediate consequence of Theorem \ref{T:locmodel} and Propositions \ref{P:trans} and \ref{P:schub}.
\end{proof}

Next we have the following exercise in bookkeeping:
\begin{prop}\label{P:3d}
Let $(w,J)$ be admissible.
\begin{enumerate}
\item If $i=1,\ldots,c$ with $\sigma(i)\leq c$ then over $\overline{\tilde{X}}_{J,w}$ we have a commutative diagram
\begin{equation*}
\begin{tikzcd}[row sep=scriptsize, column sep=0]
& p_1^*\omega_{G_i/G_{i-1}} \arrow{dl}\arrow{rr}\arrow[dotted]{dd} & & p_1^*\omega_{G_i} \arrow{dl}\arrow{dd}\arrow{rr}& &p_1^*\omega_{G_{i-1}} \arrow{dl}\arrow{dd} \\
p_1^*\omega_{(G_{\sigma(i)}/G_{\sigma(i)-1})^{(p)}} \arrow[crossing over]{rr}\arrow[dotted]{dd} & & p_1^*\omega_{G_{\sigma(i)}^{(p)}}  \arrow[crossing over]{rr}\arrow{dd}& &p_1^*\omega_{G_{\sigma(i)-1}^{(p)}} \\
& p_2^*\ker(\Ca{V}_i\to \Ca{V}_{i-1}) \arrow{dl}\arrow{rr} &  & p_2^*\Ca{V}_i\arrow{dl}\arrow{rr} &  & p_2^*\Ca{V}_{i-1} \arrow{dl} \\
p_2^*\ker(\Ca{V}_{\sigma(i)}'\to \Ca{V}_{\sigma(i)-1}') \arrow{rr} & & \Ca{V}_{\sigma(i)}'\arrow{rr}\arrow[crossing over, leftarrow]{uu} & & \Ca{V}_{\sigma(i)-1}'\arrow[crossing over, leftarrow]{uu} \\
\end{tikzcd}
\end{equation*}
where the rows are short exact sequences, the solid vertical maps are induced by the $\alpha$'s as discussed after Theorem \ref{T:fkill} and the dotted arrows are isomorphisms induced by those on the right.  Taking the determinant of the face on the left we conclude we obtain a commutative diagram
\begin{equation*}
\begin{tikzcd}
p_1^*\omega_i\arrow{r}{p_1^*A_i}\arrow{d}&p_1^*\omega_{\sigma(i)}^{(p)}\arrow{d}\\
p_2^*\det\ker(\Ca{V}_i\to\Ca{V}_{i-1})\arrow{r}{p_2^*A_i^{\text{\rm loc}}}&p_2^*\ker(\Ca{V}'_{\sigma(i)}\to\Ca{V}_{\sigma(i)-1}')
\end{tikzcd}
\end{equation*}
where the vertical maps are isomorphisms.
\item If $i=c+1,\ldots,2c$ with $\sigma(i)\leq c$ then over $\overline{\tilde{X}}_{J,w}$ we have a commutative diagram
\begin{equation*}
\begin{tikzcd}[row sep=scriptsize, column sep=0]
& p_1^*\omega_{(G_\sigma(i)/G_{\sigma(i)-1})^{(p)}} \arrow{dl}\arrow{rr}\arrow{dd} & & \omega_{G_{\sigma(i)}^{(p)}} \arrow{dl}\arrow{dd}\arrow{rr}& &\omega_{G_{\sigma(i)-1}^{(p)}} \arrow{dl}\arrow{dd} \\
p_1^*\omega_{G_{i}/G_{i-1}}\arrow[crossing over]{rr}\arrow{dd} & & p_1^*\omega_{G_i/G_c}  \arrow[crossing over]{rr}\arrow{dd}& & p_1^*\omega_{G_{i-1}/G_c}\\
& p_2^*\ker(\Ca{V}_{\sigma(i)}'\to\Ca{V}_{\sigma(i)-1}') \arrow{dl}\arrow{rr} &  & p_2^*\Ca{V}'_{\sigma(i)}\arrow{dl}\arrow{rr} &  & \Ca{V}'_{\sigma(i)-1} \arrow{dl} \\
p_2^*\ker(\Ca{V}_i\to\Ca{V}_{i-1}) \arrow{rr} & & p_2^*\Ca{V}_i\arrow{rr}\arrow[crossing over, leftarrow]{uu} & & p_2^*\Ca{V}_{i-1}\arrow[crossing over, leftarrow]{uu} \\
\end{tikzcd}
\end{equation*}
where the rows are short exact sequences, the solid vertical maps are induced by the $\alpha$'s as discussed after Theorem \ref{T:fkill} and the dotted arrows are isomorphisms induced by those on the right.  Taking the determinant of the face on the left we conclude we obtain a commutative diagram
\begin{equation*}
\begin{tikzcd}
p_1^*\omega_{\sigma(i)}^{(p)}\arrow{r}{p_1^*B_i}\arrow{d}&p_1^*\omega_i\arrow{d}\\
p_2^*\det\ker(\Ca{V}'_{\sigma(i)}\to\Ca{V}'_{\sigma(i)-1})\arrow{r}{p_2^*B_i^{\text{\rm loc}}}&p_2^*\ker(\Ca{V}_i\to\Ca{V}_{i-1})
\end{tikzcd}
\end{equation*}
where the vertical maps are isomorphisms.
\end{enumerate}
\end{prop}
As an immediate corollary we can compute the order of vanishing of the $A_i$.
\begin{cor}\label{C:krord}
Let $(w,J)$ be admissible.  For $i=1,\ldots,c$ and $v\in D_w$ we have
\begin{equation*}
\text{\rm ord}_{\overline{X}_{J,v}}(A_i)=\text{\rm ord}_{\overline{Y}_v}(C_{c+1-i})
\end{equation*}
\end{cor}
\begin{proof}
If $\sigma(i)\leq c$ then we have
\begin{equation*}
\text{ord}_{\overline{X}_{J,v}}(A_i)=\text{ord}_{\overline{X}_{J,v}}^{\text{loc}}(A_i^{\text{loc}})=\text{ord}_{\overline{Y}_v}(C_{c+1-i})
\end{equation*}
where the first equality is by Propositions \ref{P:3d} and \ref{P:trans}, and Theorem \ref{T:locmodel} and the second equality is by Proposition \ref{P:schubcomp}.  On the other hand if $\sigma(i)>c$ then we have
\begin{equation*}
\text{ord}_{\overline{X}_{J,v}}(A_i)=\text{ord}_{\overline{X}_{J,v}}(B_{2c+1-i})=\text{ord}_{\overline{X}_{J,v}}^{\text{loc}}(B_{2c+1-i}^{\text{loc}})=\text{ord}_{\overline{Y}_v}(C_{c+1-i})
\end{equation*}
by the same list of results.
\end{proof}

For $1\leq i\leq c$ with $\sigma(i)<c$ let $A_i'=A_i$ while if $\sigma(i)>c$ let
\begin{equation*}
A_i'\in H^0(\overline{X}_{J,w},\omega_{2c+1-\sigma(i)}^{-p}\otimes\omega_i^{-1})
\end{equation*}
be $A_i$ after applying the isomorphism $\omega_{\sigma(i)}\simeq\omega_{2c+1-\sigma(i)}$.  Recall the definition of the numbers $N$ and $c_i$ from Section \ref{S:ineq}.  Then we define
\begin{equation*}
A_{J,w}=\prod_{i=1}^c{A_i'}^{c_i}\in H^0(X_{J,w},\omega^{p^N-1})
\end{equation*}

\begin{thm}
The section $A_{J,w}$ above extends to a section
\begin{equation*}
A_{J,w}\in H^0(\overline{X}_{J,w},\omega^{p^N-1})
\end{equation*}
whose vanishing locus is precisely $\overline{X}_{J,w}-X_{J,w}$.
\end{thm}
\begin{proof}
This follows from Corollary \ref{C:krord} combined with Proposition \ref{P:ineq}.
\end{proof}

\section{Ekedahl-Oort and Kottwitz-Rapoport Strata}\label{S:cansec}

Let $X$ be the moduli space of principally polarized abelian schemes over a base of characteristic $p$ (with suitable prime to $p$ level structure which we omit from the notation.)  We have an Ekedahl-Oort stratification
\begin{equation*}
X=\coprod_{w\in W^I}X_w
\end{equation*}
as in Section \ref{S:EO}.

Now fix some $w\in W^I$.  Let $A$ be the universal abelian scheme over $X$.  Then over $X_w$, the principally quasi-polarized $\BT$ $A[p]$ has a canonical filtration
\begin{equation*}
0=G_0\subset G_1\subset \cdots\subset G_c\subset\cdots\subset G_{2c}=A[p]|_{X_w}.
\end{equation*}
as in Section \ref{S:canfil}.  It is self dual by proposition \ref{P:candual}.  Let $J$ be the type of this filtration, i.e. choose $J$ such that
\begin{equation*}
k_i=\text{ht}G_i
\end{equation*}
for $i=0,\ldots,2c$.  Then the canonical filtration defines a canonical section
\begin{equation*}
s_w:X_w\to X_J
\end{equation*}
to the projection
\begin{equation*}
\pi_J:X_J\to X
\end{equation*}
given by forgetting the level structure.  We note that $\pi_J$ is proper by \cite{dJ93}.

We have the following theorem of G\"{o}rtz and Hoeve \cite[Theorem 5.3]{GH12} (a related result can be found in \cite{EG09}).
\begin{thm}
Let the notation be as above.  Then $(w,J)$ is admissible and the section $s_w$ defines an isomorphism
\begin{equation*}
s_w:X_w\to X_{J,w}.
\end{equation*}
\end{thm}

Now we are in the following situation.  We have a proper surjective map
\begin{equation*}
\pi: \overline{X}_{J,w}\to\overline{X}_w
\end{equation*}
which restricts to an isomorphism from $X_{J,w}$ to $X_w$.  We have a section
\begin{equation*}
A_w'\in H^0(X_w,\omega^{\otimes N_w'})
\end{equation*}
as constructed in Section \ref{S:hasse}.  On the other hand, $s_w^* A_w'$ is the section $A_{J,w}$ which extends to
\begin{equation*}
A_{J,w}\in H^0(\overline{X}_{J,w},\omega^{\otimes N_w'})
\end{equation*}
which vanishes precisely on $\overline{X}_{J,w}-X_{J,w}$.  We would like to conclude from this that there is some $n>0$ such that $(A_w')^n$ extends to an element of $H^0(X_w,\omega^{\otimes nN_w'})$ which vanishes precisely on $\overline{X}_w-X_w$, and hence complete the proof of Theorem \ref{T:hasseext}.  This follows from the following lemma.

\begin{lem}\label{L:extlem}
Let $f:X\to Y$ be a proper surjective morphism of reduced noetherian schemes.  Let $\L$ be a line bundle on $Y$ and let $U\subset Y$ be a dense open subscheme with the property that $f^{-1}(U)\to U$ is an isomorphism and let $Z=X-U$ be its (set theoretic) complement.  Suppose we are given a section $s\in H^0(X,f^*\L)$ which vanishes (set theoretically) on $f^{-1}(Z)$.  Then there is an integer $n>0$ and a section $t\in H^0(Y,\L^n)$ which vanishes set theoretically on $Z$ and pulls back to $s^n$ under $f$.
\end{lem}
\begin{proof}
First we reduce to the case that $f$ is finite.  We consider the stein factorization
\begin{equation*}
X\overset{g}{\to} X'=\underline{\spec}\,f_*\O_X\overset{f'}{\to} Y.
\end{equation*}
As $g_*\O_X=\O_X'$ we have that $s$ gives a section of
\begin{equation*}
g_*f^*\L={f'}^*\L
\end{equation*}
and so it suffices to prove the theorem for $f'$ which is finite.

Next we observe that the question is local on $Y$.  Indeed if for some finite affine open cover $Y=\cup V_i$ we have integers $n_i$ and sections $t_i\in H^0(V_i,\L^{n_i})$ such that $f^* t_i=s^{n_i}|_{f^{-1}(V_i)}$ then form $t_i'=t_i^{\prod_{j\not=i}n_i}$.  Then for each $i,j$, $t_i'|_{V_i\cap V_j}=t_j'|_{V_i\cap V_j}$ because $f$ is surjective, $Y$ is reduced, and they both pull back to $s^n$, $n=\prod_i n_i$.  Hence they glue to a section $t\in H^0(Y,\L^n)$ which pulls back to $s^n$.

Hence we may assume that $Y=\spec A$ and $X=\spec B$ are affine, $A\to B$ is finite, and $\L$ is trivial.  Let $I$ be the radical ideal corresponding to $Z\subset \spec A$ and let $I'=\sqrt{IB}$ be the radical ideal corresponding to $f^{-1}(Z)$.  Then $t$ is just a regular function $b\in B$ and the fact that it vanishes on $f^{-1}(Z)$ means that $b\in I'$.

Form the exact sequence of $A$ modules
\begin{equation*}
0\to A\to B\overset{\alpha}{\to} C\to 0
\end{equation*}
where the map $A\to B$ is injective because $f$ is surjective and $A$ is reduced.  For $\p\in U\subset\spec A$, $A_\p\to B_\p$ is an isomorphism by assumption, and hence $C_\p=0$.  Thus $\text{supp} C\subset Z$ and hence there is an integer $n_1$ such that $I^{n_1}C=0$.  Now $b\in I'=\sqrt{BI}$ and so there is some $n_2>0$ with $b^{n_2}\in BI$ and hence we can write
\begin{equation*}
b^{n_2}=b'a
\end{equation*}
with $a\in I$.  Then
\begin{equation*}
\alpha (b^{n_1n_2})=a^{n_1}\alpha(b'^{n_2})=0
\end{equation*}
and hence $b^{n_1n_2}$ lies in the image of $A$.
\end{proof}

\chapter{Generalized Hasse Invariants at the Boundary}\label{boundaryhasse}

The goal of this section is to study the Ekedahl-Oort stratification and generalized Hasse invariants of chapters \ref{EOPEL} and \ref{siegel} near the boundary of the compactifications recalled in chapter \ref{compact}.  As we have seen in chapter \ref{compact}, suitable formal neighborhoods of the boundary in either a toroidal or minimal compactification ``fiber over'' smaller PEL modular varieties.  Roughly speaking, these ``structural morphisms'' give the non degenerating abelian part of an abelian variety near the boundary.  What we will show is that in these formal neighborhoods, the Ekedahl-Oort stratification is just the pullback of the Ekedahl-Oort stratification of the smaller PEL modular variety, and similarly on each Ekedahl-Oort strata, the Hasse invariant is the pullback of a suitable Hasse invariant on the smaller Shimura variety. 

Let us now summarize the contents of this Chapter.  In Section \ref{S:eobound} we study the Ekedahl-Oort stratification at the boundary.  We treat toroidal compactifications first.  On a toroidal compactification $X_{K,\Sigma}^{\text{tor}}$, we have a semiabelian scheme $A$ with $\O$-action extending the abelian scheme on the interior.  Its $p$ torsion $A[p]$ gives a principally quasi-polarized partial $\BT$ with $\overline{\O}$-action, and hence the general construction of chapter \ref{EOPEL} we may extend the stratification from the interior to the entire toroidal compactification.  In order to understand it at the boundary, we also consider the ``Raynaud Extensions'' $\tilde{A}$ on the boundary charts $\Xi_{\Sc{C},\Sigma_{\Sc{C}}}\times_Rk$.  The $p$-torsion of $\tilde{A}_{\Sc{C}}$ defines an Ekedahl-Oort stratification of $\Xi_{\Sc{C},\Sigma_{\Sc{C}}}\times_Rk$, which from the definition of $\tilde{A}_{\Sc{C}}$ is seen to be the pullback of the Ekedahl-Oort stratification of $X_{\Sc{C}}$ via the map $\Xi_{\Sc{C},\Sigma_{\Sc{C}}}\times_Rk\to X_{\Sc{C}}$.  Hence we obtain two stratifications of our formal scheme $\hat{X}_{K,\Sigma,\Sc{C}}^{\text{tor}}\simeq(\Fr{X}_{\Sc{C},\Sigma_{\Sc{C}}}/\Gamma_{\Sc{C}})\times_Rk$ (see \ref{T:toroidal} part 2) and our aim is to show that they are the same.

Let us illustrate how this goes in the simplest possible case.  Consider what happens at the cusp of the modular curve.  We have the semiabelian Tate curve $E/\FF_p[[q]]$ and the Raynaud extension $\Bf{G}_m/\FF_p[[q]]$.  They are related by the fact that their formal completions along their special fibers are canonically isomorphic.  Now in general given a semiabelian scheme $A/\FF_p[[q]]$ we cannot expect to recover the quasi-finite flat group scheme $A[p]$ from the formal completion $\hat{A}$.  However we can recover the finite part $A[p]^f$ (see Lemma \ref{L:finitepart} and the surrounding discussion for this notion) and this is enough to determine the Ekedahl-Oort strata of both the generic and special fiber of $A$.  Returning to the Tate curve $E$, we can conclude that 
\begin{equation*}
E[p]^f\simeq\Bf{G}_m[p]=\mu_p
\end{equation*}
and hence the generic fiber of the Tate curve is ordinary.  The general case is just an elaboration of this argument.

The main properties of the Ekedahl-Oort stratification of a toroidal compactification are summarized in Theorem \ref{T:EOtor}.  From this it is easy to define an Ekedahl-Oort stratification of the minimal compactification, and we summarize its properties in Theorem \ref{T:EOmin}.

Next, in section \ref{S:hassebound} we turn to Hasse invariants.  In theorem \ref{T:hassetor} we show that for each Ekedahl-Oort stratum $X_{K,w}$, the Hasse invariant on $\overline{X}_{K,w}$ extends to $\overline{X}_{K,w}^{\text{tor}}$, and on a suitable formal neighborhood of the boundary, it is just the pullback of a suitable Hasse invariant on an Ekedahl-Oort stratum of a smaller PEL modular variety.  Again this will be done by comparing the semiabelian scheme $A/X_{K,\Sigma}^{\text{tor}}$ with the Raynaud extension $\tilde{A}_{\Sc{C}}$.  The latter is an extension
\begin{equation*}
0\to T\to \tilde{A}_{\Sc{C}}\to A_\Sc{C}\to 0
\end{equation*}
where $A_{\Sc{C}}$ is the pullback of the universal abelian scheme on $X_{\Sc{C}}$ and $T$ is a torus.  Then it follows easily from the definition that on an Ekedahl-Oort stratum $\Xi_{\Sc{C},\Sigma_{\Sc{C}},w}$ of the boundary chart, the $w$ Hasse invariant of $\tilde{A}_{\Sc{C}}$ is just the product of a suitable Hasse invariant and a power of the determinant of
\begin{equation*}
V^*:\omega_T\to\omega_T^{(p)},
\end{equation*}
which is nothing but the canonical generator of $\det\omega_T^{p-1}$.

In the case of the Tate curve considered above, the isomorphism $E[p]^f\simeq\Bf{G}_m[p]$ gives a canonical generator (up to sign) $\pm\frac{dT}{T}$ of $\omega_{E}$ where $T$ is one of the generators of the character lattice of $\Bf{G}_m$.  Moreover we have
\begin{equation*}
V^*\frac{dT}{T}=\left(\frac{dT}{T}\right)^p
\end{equation*}
and hence over $\Z[[q]]$ the classical Hasse invariant is nothing but $A=(\frac{dT}{T})^{p-1}$ (note that this formula does not depend on the choice of $T$.)  The reader will recognize this as nothing but the classical calculation of the $q$-expansion of the Hasse invariant.  We like to interpret Theorem \ref{T:hassetor} as saying that ``each Fourier-Jacobi expansion of the generalized Hasse invariant $A_w$ has a suitable generalized Hasse invariant at the boundary as a constant term, and vanishing non constant terms.''

Finally we show in Theorem \ref{T:hassemin} that the Hasse invariants also extend to the minimal compactification (this is easily deduced from the Toroidal case using the results of Section \ref{S:wellpos}).  As an immediate consequence we deduce in Corollary \ref{C:affine} that the minimally compactified Ekedahl-Oort strata $X_{K,w}^{\text{min}}$ are affine.  In the Siegel case, this answers a question of Oort \cite[14.2]{Oo01}.

We remark that for Siegel modular varieties, extensions of the Ekedahl-Oort stratification to compactifications were already considered by Oort, and the results of section \ref{S:eobound} should be compared with \cite[\S 6]{Oo01}.

Throughout this chapter we fix an integral PEL datum $(\Ca{O},*,L,\langle\cdot,\cdot,\rangle,h)$ without factors of type D satisfying Condition \ref{extcond}, a neat open compact subgroup $K\subset G(\hat{\Z}^{(p)})$, and $\Sigma$ a good compatible family of cone decompositions at level $K$.  We let $\Ca{D}$ be the corresponding mod $p$ PEL datum as in Definition \ref{D:modpcons} and we denote by $W^I$ the associated set of Weyl group cosets (see the end of section \ref{S:modppel}).

\section{Ekedahl-Oort Stratification at the Boundary}\label{S:eobound}

\subsection{Ekedahl-Oort Stratification of the Boundary Charts}
Fix a cusp label $\Sc{C}\in\text{Cusp}_K$.  Associated to it is an integral PEL datum $(\O,*,L_{\Sc{C}},\langle\cdot,\cdot\rangle_{\Sc{C}},h_{\Sc{C}})$ and hence an associated mod $p$ PEL datum $\Ca{D}_{\Sc{C}}$.  We let $W_{\Sc{C}}^{I_{\Sc{C}}}$ be the associated set of Weyl group cosets.  Both $\Ca{D}$ and $\Ca{D}_{\Sc{C}}$ consist of the same semisimple $\FF_p$-algebra with involution $(\overline{\O},*)$.  We denote the numerical invariants of $\Ca{D}$ by $(h_{[\tau]})$ and $(d_\tau)$ and those of $\Ca{D}_{\Sc{C}}$ by $(h_{\Sc{C},[\tau]})$ and $(d_{\Sc{C},\tau})$ (see definition \ref{D:modppel}.)  We also let $(t_\tau=t_{[\tau]})$ be the $\overline{\O}$-multirank of $X\otimes_\Z k$.  Then it follows from the construction of $(\O,*,L_{\Sc{C}},\langle\cdot,\cdot\rangle_{\Sc{C}},h_{\Sc{C}})$ (see 5.4.2.6 of \cite{La13}) that for each $\tau$
\begin{equation*}
h_{[\tau]}=h_{\Sc{C},[\tau]}+2t_{[\tau]}
\end{equation*}
and
\begin{equation*}
d_{[\tau]}=d_{\Sc{C},[\tau]}+t_{[\tau]}.
\end{equation*}

Now let $k'$ be an algebraically closed field of characteristic $p$ with an embedding $k\to k'$.  Let $G_0/k'$ be the $\BT$ with $\overline{\O}$-action of multiplicative type with character group $X\otimes\FF_p$.  It has multi height and multi dimension $(t_{[\tau]})$.  We form
\begin{equation*}
G_1=G_0\times G_0^D
\end{equation*}
which we make into a principally quasi-polarized $\BT$ with $\overline{\O}$-action in the obvious way.

We define a map
\begin{equation*}
\iota_{\Sc{C}}:W_{\Sc{C}}^{I_{\Sc{C}}}\to W^I
\end{equation*}
by the commutativity of the diagram
\begin{equation*}
\begin{tikzcd}
{\BT}_{k'}^{\Ca{D}_{\Sc{C}}}\arrow{r}\arrow{d}&W_{\Sc{C}}^{I_{\Sc{C}}}\arrow[dotted]{d}{\iota_{\Sc{C}}}\\
{\BT}_{k'}^{\Ca{D}}\arrow{r}& W^{I}
\end{tikzcd}
\end{equation*}
where the horizontal arrows are the bijections of Section \ref{S:moonen} and the left vertical arrow is the map
\begin{equation*}
G\mapsto G\times G_1.
\end{equation*}
It is clear that $\iota_{\Sc{C}}$ is injective and independent of the choice of $k'$.

For $G\in{\BT}_{k'}^{\Ca{D}}$ we may also define a principally quasi polarized partial $\BT$ with $\overline{\O}$-action of type $\Ca{D}$ with the group with its $\overline{\O}$-action given by $G\times G_0$, and with the quasi-polarization extended from $G$ to be zero on $G_0$.  Then we note that $G\times G_0$ and $G\times G_1$ have canonical filtrations with the same $\overline{\O}$-type in the sense of Definition \ref{D:otype}.

Then as in section \ref{S:EO} the PEL modular variety $X_{\Sc{C}}$ has an Ekedahl-Oort stratification
\begin{equation*}
X_{\Sc{C}}=\coprod_{w_{\Sc{C}}\in W_{\Sc{C}}^{I_{\Sc{C}}}}X_{\Sc{C},w_{\Sc{C}}}.
\end{equation*}
For notational purposes we define for $w\in W^I$
\begin{equation*}
X_{\Sc{C},w}=\begin{cases}X_{\Sc{C},w_{\Sc{C}}}&\text{if there exists $w_{\Sc{C}}\in W_{\Sc{C}}^{I_{\Sc{C}}}$ with $\iota_{\Sc{C}}(w_{\Sc{C}})=w$}\\ \emptyset&\text{otherwise}\end{cases}.
\end{equation*}

Now consider the special fiber of the toroidal boundary chart
\begin{equation*}
\Xi_{\Sc{C},\Sigma_{\Sc{C}}}\times_R k\to X_{\Sc{C}}.
\end{equation*}
Over $\Xi_{\Sc{C},\Sigma_{\Sc{C}}}\times_R k$ we have a semiabelian scheme $\tilde{A}_{\Sc{C}}$ which sits in an exact sequence
\begin{equation*}
0\to T\to \tilde{A}_{\Sc{C}}\to A_\Sc{C}\to 0
\end{equation*}
where $T$ is the split torus with character group $X$, and $A_{\Sc{C}}$ is the pullback of the universal abelian scheme on $X_{\Sc{C}}$.  
Then $\tilde{A}_{\Sc{C}}[p]$ is a $\BT$ with $\overline{\O}$-action.  We make it into a principally quasi polarized $\BT$ with $\overline{\O}$-action of type $\Ca{D}$ by extending the principal quasi polarization on $A_{\Sc{C}}[p]$ to be zero on $T[p]$.

As a consequence of Theorem \ref{T:candecompO} applied to $\tilde{A}_{\Sc{C}}[p]$ and Proposition \ref{P:candetermines} we have a decomposition
\begin{equation*}
\Xi_{\Sc{C},\Sigma_{\Sc{C}}}\times_Rk=\coprod_{w\in W^I}\Xi_{\Sc{C},\Sigma_{\Sc{C}},w}.
\end{equation*}
Now consider a geometric point $x\in \Xi_{\Sc{C},\Sigma_{\Sc{C}},w}(k')$.  We have
\begin{equation*}
\tilde{A}_{\Sc{C}}[p]_x\simeq A_{\Sc{C}}[p]_x\times G_0.
\end{equation*}
Hence by the definition of $\iota_{\Sc{C}}$ and what we observed above, if $w_{\Sc{C}}$ corresponds to $A_{\Sc{C}}[p]_x$ then $\iota(w_{\Sc{C}})=w$, or in other words under the map $\pi:\Xi_{\Sc{C},\Sigma_{\Sc{C}}}\times_Rk\to X_{\Sc{C}}$, the point $x$ maps into $X_{\Sc{C},w}$.  Hence we have
\begin{equation*}
\pi^{-1}(X_{\Sc{C},w})=\Xi_{\Sc{C},\Sigma_{\Sc{C}},w}
\end{equation*}
at least set theoretically.  But in fact as $\pi$ is smooth, $\pi^{-1}(X_{\Sc{C},w})$ is reduced and so this holds scheme theoretically as well.  As $\pi$ is flat, we also have
\begin{equation*}
\pi^{-1}(\overline{X}_{\Sc{C},w})=\pi^{-1}(\overline{\Xi}_{\Sc{C},\Sigma_{\Sc{C}},w})
\end{equation*}
where $\overline{X}_{\Sc{C},w}$ (resp. $\overline{\Xi}_{\Sc{C},\Sigma_{\Sc{C}},w}$) denotes the Zariski closure of $X_{\Sc{C},w}$ (resp. $\Xi_{\Sc{C},\Sigma_{\Sc{C}},w}$).

Now consider the canonical filtration
\begin{equation*}
0=G_0\subset G_1\subset \cdots\subset G_c\subset\cdots\subset G_{2c}=\tilde{A}[p]|_{\Xi_{\Sc{C},\Sigma_{\Sc{C}},w}}
\end{equation*}
of the principally quasi-polarized partial $\BT$ $\tilde{A}[p]|_{\Xi_{\Sc{C},\Sigma_{\Sc{C}},w}}$ where we remind the reader that according to Convention \ref{convpartial} for canonical filtrations for partial $\BT$ we may have $G_{2c-1}=G_{2c}$ if $\tilde{A}[p]|_{\Xi_{\Sc{C},\Sigma_{\Sc{C}},w}}$ has no \'{e}tale part.

Now note that as $T[p]|_{\Xi_{\Sc{C},\Sigma_{\Sc{C}},w}}$ is multiplicative we have
\begin{equation*}
V(T[p]|_{\Xi_{\Sc{C},\Sigma_{\Sc{C}},w}}^{(p)})=T[p]|_{\Xi_{\Sc{C},\Sigma_{\Sc{C}},w}}
\end{equation*}
and thus
\begin{equation*}
T[p]|_{\Xi_{\Sc{C},\Sigma_{\Sc{C}},w}}\subset G_1.
\end{equation*}
and
\begin{equation*}
0\subset G_1/T[p]\subset G_2/T[p]\subset\cdots G_c/T[p]\subset\cdots\subset G_{2c}/T[p]=A_{\Sc{C}}[p]|_{\Xi_{\Sc{C},\Sigma_{\Sc{C}},w}}
\end{equation*}
is a canonical filtration for the $\BT$ $A_{\Sc{C}}[p]|_{\Xi_{\Sc{C},\Sigma_{\Sc{C}},w}}$, except that if the fibers of $A_{\Sc{C}}[p]|_{\Xi_{\Sc{C},\Sigma_{\Sc{C}},w}}$ are connected-connected then $G_1/T[p]=0$ and $G_{2c}/T[p]=G_{2c-1}/T[p]$ and the outermost two terms should be removed (we call this the exceptional case in what follows.)

Now we turn to Hasse invariants.  Applying the construction of section \ref{S:hassebt} of a Hasse invariant associated to a $\BT$ with canonical filtration we obtain
\begin{equation*}
\tilde{A}'_{\Sc{C},w}\in H^0(\Xi_{\Sc{C},\Sigma_{\Sc{C},w}},\tilde{\omega}_{\Sc{C}}^{\otimes p^N-1})
\end{equation*}
associated to $\tilde{A}_{\Sc{C}}[p]|_{\Xi_{\Sc{C},\Sigma_{\Sc{C},w}}}$ and
\begin{equation*}
A'_{\Sc{C},w}\in H^0(\Xi_{\Sc{C},\Sigma_{\Sc{C}},w},\pi^*\omega_{\Sc{C}}^{\otimes p^N-1})
\end{equation*}
associated to $A_{\Sc{C}}[p]|_{\Xi_{\Sc{C},\Sigma_{\Sc{C},w}}}$ where the notation is as in the beginning of section \ref{S:welpossec}.  Note that the same integer $N$ occurs in both expressions because the canonical filtrations for $\tilde{A}_{\Sc{C}}[p]|_{\Xi_{\Sc{C},\Sigma_{\Sc{C},w}}}$ and $A_{\Sc{C}}[p]|_{\Xi_{\Sc{C},\Sigma_{\Sc{C},w}}}$ have the same associated permutation $\sigma$, except in the exceptional case when that for the latter is missing two cycles of length 1 (which doesn't change $N$.)

We want to compare $\tilde{A}'_{\Sc{C},w}$ and $A_{\Sc{C},w}'$.  From the definition we have
\begin{equation*}
\tilde{A}'_{\Sc{C},w}=\tilde{A}_1^{\frac{p^N-1}{p-1}}B,\qquad A_{\Sc{C},w}'=A_1^{\frac{p^N-1}{p-1}}B
\end{equation*}
where
\begin{equation*}
B\in H^0(\Xi_{\Sc{C},\Sigma_{\Sc{C}},w},(\omega_2\otimes\cdots\otimes\omega_{c})^{p^N-1})
\end{equation*}
is the product of the terms for $i=2,\ldots,c$ in the definitions of $\tilde{A}'_{\Sc{C},w}$ and $A_{\Sc{C},w}'$ while
\begin{equation*}
\tilde{A}_1\in H^0(\Xi_{\Sc{C},\Sigma_{\Sc{C}},w},\omega_1^{\otimes p-1})
\end{equation*}
comes from
\begin{equation*}
\det V^*:\det\omega_{G_1}\to(\det\omega_{G_1})^{\otimes p}
\end{equation*}
while
\begin{equation*}
A_1\in H^0(\Xi_{\Sc{C},\Sigma_{\Sc{C}},w},(\det \omega_{G_1/T[p]})^{\otimes p-1})
\end{equation*}
comes from
\begin{equation*}
\det V^*:\det\omega_{G_1/T[p]}\to (\det\omega_{G_1/T[p]})^{\otimes p}.
\end{equation*}

Now consider
\begin{equation*}
V^*:\det\omega_{T[p]}\to(\det\omega_{T[p]})^{\otimes p}
\end{equation*}
If $x_1,\ldots,x_m$ form a basis for $X$, then $\det\omega_{T[p]}=\det\omega_{T}$ is generated by
\begin{equation*}
\alpha=\frac{dx_1}{x_1}\wedge\cdots\wedge\frac{dx_m}{x_m}
\end{equation*}
and
\begin{equation*}
V^*\alpha=\alpha^{\otimes p}
\end{equation*}
and hence
\begin{equation*}
\tilde{A}_1=\alpha^{\otimes p-1}A_1.
\end{equation*}
We note that the section
\begin{equation*}
\alpha^{\otimes p-1}\in H^0(\Xi_{\Sc{C},\Sigma_{\Sc{C}}}\times_Rk,\omega_T^{\otimes p-1})=H^0(\Xi_{\Sc{C},\Sigma_{\Sc{C}}}\times_Rk,(\det X\otimes\O_{\Xi_{\Sc{C},\Sigma_{\Sc{C}}}\times_Rk})^{\otimes p-1})
\end{equation*}
is really canonical and independent of the choice of basis $x_1,\ldots,x_m$.  Indeed, $\alpha$ is unique up to multiplication by $-1$, and either $p=2$ or $p-1$ is even.

Now by Theorem \ref{T:hasseext} applied to the PEL modular variety $X_{\Sc{C}}$, for each $w_{\Sc{C}}\in W_{\Sc{C}}^{I_{\Sc{C}}}$ there is a generalized Hasse invariant
\begin{equation*}
A_{\Sc{C},w_{\Sc{C}}}\in H^0(\overline{X}_{\Sc{C},w_{\Sc{C}}},\omega_\Sc{C}^{\otimes N_{w_{\Sc{C}}}})
\end{equation*}
which satisfies
\begin{equation*}
{A'_{\Sc{C},w}}^{\frac{N_{w_{\Sc{C}}}}{p^N-1}}=\pi^*A_{\Sc{C},w_{\Sc{C}}}|_{\Xi_{\Sc{C},\Sigma_{\Sc{C}},w}}.
\end{equation*}
For the rest of this chapter we will adopt the following convention.

\begin{conv}
For each $w\in W^I$ and each $\Sc{C}\in\text{Cusp}_K$ for which there is a $w_{\Sc{C}}\in W_{\Sc{C}}^{I_{\Sc{C}}}$ with $\iota_{\Sc{C}}(w_{\Sc{C}})=w$ then we assume that $N_{w_{\Sc{C}}}=N_w$.  If this is not already the case it may be arrange by replacing $A_w$ and the $A_{\Sc{C},{w_{\Sc{C}}}}$ with suitable powers.
\end{conv}

For the rest of this chapter for $w\in W^I$ such that there exists $w_{\Sc{C}}\in W_{\Sc{C}}^{I_{\Sc{C}}}$ with $\iota_{\Sc{C}}(w_{\Sc{C}})=w$ we will denote the section $A_{\Sc{C},w_{\Sc{C}}}$ by $A_{\Sc{C},w}\in H^0(\overline{X}_{\Sc{C},w},\omega_\Sc{C}^{\otimes N_w})$.

Let us summarize what we have seen in this section in the following proposition.

\begin{prop}\label{P:boundary}
With notation as above, for each cusp label $\Sc{C}\in\text{\rm Cusp}_K$ and each $w\in W^I$ we have
\begin{equation*}
\pi^{-1}(X_{\Sc{C},w})=\Xi_{\Sc{C},\Sigma_{\Sc{C}},w},\qquad \pi^{-1}(\overline{X}_{\Sc{C},w})=\overline{\Xi}_{\Sc{C},\Sigma_{\Sc{C}},w}
\end{equation*}
where $\pi:\Xi_{\Sc{C},\Sigma_{\Sc{C}}}\times_Rk\to X_{\Sc{C}}$ is the canonical map.  Moreover if
\begin{equation*}
\tilde{A}'_{\Sc{C},w}\in H^0(\Xi_{\Sc{C},\Sigma_{\Sc{C},w}},\tilde{\omega}_{\Sc{C}}^{\otimes p^N-1})
\end{equation*}
is the Hasse invariant associated to the $\BT$ with canonical filtration $\tilde{A}_{\Sc{C}}[p]|_{\Xi_{\Sc{C},\Sigma_{\Sc{C},w}}}$ by the construction of section \ref{S:hassebt} then under the isomorphism
\begin{equation*}
\tilde{\omega}_{\Sc{C}}\simeq\det X\otimes\pi^*\omega_{\Sc{C}}
\end{equation*}
of section \ref{S:welpossec} we have
\begin{equation*}
{\tilde{A}{}'_{\Sc{C},w}}^{\frac{N_w}{p^N-1}}=\alpha^{\frac{N_w}{p-1}}\otimes\pi^*A_{\Sc{C},w}|_{\Xi_{\Sc{C},\Sigma_{\Sc{C}},w}}
\end{equation*}
where $\alpha\in\det X$ is a generator and $A_{\Sc{C},w}\in H^0(\overline{X}_{\Sc{C},w},\omega_{\Sc{C}}^{\otimes N_w})$ is the generalized Hasse invariant of Theorem \ref{T:hasseext}.
\end{prop}

\subsection{Ekedahl-Oort Stratification of the Toroidal Compactification}\label{S:eotor}
Let $A$ be an adic noetherian ring with ideal of definition $I$.  We let $\overline{A}=A/I$ and for a scheme $X/A$ we let $\overline{X}=X\times_A\overline{A}$.  Given a scheme $X/A$ we let $\hat{X}$ denote its formal completion along $\overline{X}$.  This defines a functor $X\mapsto\hat{X}$ from the category of finite type schemes over $A$ to the category of formal schemes, adic and finite type over $\text{Spf}(A)$.  This functor is in general far from being an equivalence.  Nonetheless, it does induce an equivalence between the category of schemes finite over $\spec(A)$ and formal schemes finite over $\text{Spf}(A)$ (we remind the reader that a formal scheme $\Fr{X}/\text{Spf}(A)$ is finite if it is adic over $\text{Spf}(A)$ and $\overline{\Fr{X}}$ is finite over $\overline{A}$.)

We now consider the theory of ``finite parts'' of quasi-finite schemes over $A$.  We recall the following well known lemma, whose proof we sketch because we don't know of a reference for this exact statement.

\begin{lem}\label{L:finitepart}
Let $X/A$ be quasi-finite and separated, and assume that $\overline{X}$ is finite over $\overline{A}$.  Then there is a unique (scheme theoretic) decomposition
\begin{equation*}
X=X^f\coprod X'
\end{equation*}
with $X^f/A$ finite and $\overline{X}'$ empty.
\end{lem}
\begin{proof}
By Zariski's main theorem one may factor $X\to\spec A$ as
\begin{equation*}
X\to \tilde{X}\to\spec A
\end{equation*}
with $X\to\tilde{X}$ an open immersion and $\tilde{X}\to\spec A$ finite.  Then
\begin{equation*}
\overline{X}\to\overline{\tilde{X}}
\end{equation*}
is an open immersion because $X\to\tilde{X}$ is, and it is also closed because $\overline{X}$ is finite over $\overline{A}$.  By \cite[XI Prop. 1]{Ra70} and Hensel's lemma we then have
\begin{equation*}
\tilde{X}=\tilde{X}_1\coprod\tilde{X}_2
\end{equation*}
where $\overline{\tilde{X}_1}=\overline{X}$.  Then take $X^f=\tilde{X}_1\cap X$ and $X'=\tilde{X}_2\cap X$.
\end{proof}

\begin{rem}
The same result (and proof) holds if one only assumes that $(A,I)$ is a Henselian couple in the sense of \cite[XI]{Ra70}.
\end{rem}

We call $X^f$ as in the lemma the finite part of $X$.  From the lemma it follows that $\hat{X}=\hat{X}^f$.  In other words, we may recover $X^f$ from the formal completion $\hat{X}$ via the equivalence of categories discussed above.

Formation of the finite part is clearly functorial in $X$ and compatible with fiber products over $A$.  In particular if $G/\spec A$ is a quasi-finite, separated, group scheme with $\overline{G}/\overline{A}$ finite, then $G^f/\spec A$ is a finite group scheme.  Moreover $G^f$ is flat over $A$ if $G$ is.  

Now we turn to the problem of extending the Ekedahl-Oort stratification at the boundary of a toroidal compactification.  By Theorem \ref{T:toroidal} there is a semiabelian scheme $A/X_{K,\Sigma}^{\text{tor}}$ with an action of $\O$ which extends the universal abelian scheme $A$ over $X_K$.  Then $A[p]$ a principally quasi-polarized partial $\BT$ with $\overline{\O}$-action.  Hence by Theorem \ref{T:candecompO} and Proposition \ref{P:candetermines} we obtain a set theoretic decomposition 

\begin{equation*}
X_{K,\Sigma}^{\text{tor}}=\coprod_{w\in W^I}X_{K,\Sigma,w}^{\text{tor}}.
\end{equation*}
As usual we will also denote by $\overline{X}_{K,\Sigma,w}^{\text{tor}}$ the Zariski closure of $X_{K,\Sigma,w}$.

Let $\Sc{C}\in\text{Cusp}_K$ be any cusp label.  We may cover the formal scheme 
\begin{equation*}
\hat{X}_{K,\Sigma,\Sc{C}}^{\text{tor}}\simeq (\Fr{X}_{\Sc{C},\Sigma_{\Sc{C}}}\times_R k)/\Gamma_{\Sc{C}}
\end{equation*}
by affine formal schemes $\Fr{U}$ with the following properties:
\begin{enumerate}
\item $\Fr{U}$ lifts to an open in the formal scheme $(\Fr{X}_{\Sc{C},\Sigma_{\Sc{C}}}\times_Rk)$ (use the fact that $\Fr{X}_{\Sc{C},\Sigma_{\Sc{C}}}\to\Fr{X}_{\Sc{C},\Sigma_{\Sc{C}}}/\Gamma_{\Sc{C}}$ is Zariski locally an isomorphism.)
\item $\Fr{U}$ arises as the formal completion of an affine open in $X_{K,\Sigma}^{\text{tor}}$, as well as from the formal completion of an affine open in $\Xi_{\Sc{C},\Sigma_{\Sc{C}}}\times_Rk$.
\end{enumerate}
We have $\Fr{U}=\text{Spf}\, A$ for $A$ an adic noetherian ring with some ideal of definition $I$.  We may also consider the scheme $U=\spec A$.  As the formal completion of a noetherian ring is flat, we have flat maps of schemes
\begin{equation*}
f_1:U\to X_{K,\Sigma}^{\text{tor}}
\end{equation*} 
and
\begin{equation*}
f_2:U\to \Xi_{\Sc{C},\Sigma_{\Sc{C}}}\times_Rk.
\end{equation*}
We also know that $U$ is reduced because it is the formal completion of a reduced excellent ring (see \cite[IV, 7.8.3]{EGA}).

We denote by $(A,i)/U$ the pullback of the semiabelian scheme with $\O$-action $(A,i)/X_{K,\Sigma}^{\text{tor}}$ by $f_1$, and we denote by $(\tilde{A},i)/U$ the pullback of the semiabelian scheme with $\O$-action $(\tilde{A},i)/(\Xi_{\Sc{C},\Sigma_{\Sc{C}}}\times_Rk)$ by $f_2$.  By part 3 of Theorem \ref{T:toroidal}, the $I$-adic completions $\hat{A}$ and $\hat{\tilde{A}}$ are isomorphic, compatibly with the $\O$-action.  In particular we conclude from Lemma \ref{L:finitepart} and the remarks following it we conclude that we have an isomorphism of $\BT$ with $\O$-action over $U$
\begin{equation*}
A[p]^f\simeq\tilde{A}[p]^f=\tilde{A}[p]
\end{equation*}
where the second equality is because $\tilde{A}[p]$ is already finite.  Now applying the construction of theorem \ref{T:candecompO} and its compatibility with base change \ref{P:canfilbc} (noting that it makes no difference whether we apply it to $A[p]^f$ or $A[p]$) we conclude that set theoretically,
\begin{equation*}
f_1^{-1}(X^{\text{tor}}_{K,\Sigma,w})=f_2^{-1}(\Xi_{\Sc{C},\Sigma_{\Sc{C}},w}).
\end{equation*}
Now as $f_1$ and $f_2$ are flat we conclude that again set theoretically
\begin{equation*}
f_1^{-1}(\overline{X}_{K,\Sigma,w}^{\text{tor}})=f_2^{-1}(\overline{\Xi}_{\Sc{C},\Sigma_{\Sc{C}},w}).
\end{equation*}
But $\overline{X}_{K,\Sigma,w}^{\text{tor}}$ and $\overline{\Xi}_{\Sc{C},\Sigma_{\Sc{C}},w}$ are reduced by definition, and hence by \cite[IV, 7.8.3]{EGA} again, $f_1^{-1}(\overline{X}_{K,\Sigma,w}^{\text{tor}})$ and $f_2^{-1}(\overline{\Xi}_{\Sc{C},\Sigma_{\Sc{C}},w})$ are reduced as well, and hence the equality holds as schemes.

We can now conclude our main theorem on the Ekedahl-Oort stratification on a toroidal compactification.

\begin{thm}\label{T:EOtor}
Let $K\subset G(\hat{\Z}^{(p)})$ be neat open compact and let $\Sigma$ be a good compatible family of cone decompositions at level $K$.  For each $w\in W^I$ we have
\begin{enumerate}
\item $\overline{X}_{K,\Sigma,w}^{\text{\rm tor}}$ is well positioned at the boundary and for each cusp label $\Sc{C}\in\text{\rm Cusp}_K$ the corresponding subscheme is $\overline{X}_{\Sc{C},w}\subset X_\Sc{C}$.
\item We have (set theoretically)
\begin{equation*}
\overline{X}_{K,\Sigma,w}^{\text{\rm tor}}=\coprod_{w'\preceq w}X_{K,\Sigma,w'}^{\text{\rm tor}}.
\end{equation*}
where $\preceq$ is the partial order of Theorem \ref{T:EOprop}.
\item Let $g\in G(\A^{\infty,p})$ and let $K,K'\subset G(\hat{\Z}^{(p)})$ be neat open compact subgroups with $g^{-1}Kg\subset K'$ and let $\Sigma$ (resp. $\Sigma'$) be a good compatible family of cone decompositions at level $K$ (resp. $K'$) such that $\Sigma$ is a $g$-refinement of $\Sigma'$.  Then
\begin{equation*}
[g]^{-1}(X_{K',\Sigma',w}^{\text{\rm tor}})=X_{K,\Sigma,w}^{\text{\rm tor}}\qquad\text{and}\qquad [g]^{-1}(\overline{X}_{K',\Sigma',w}^{\text{\rm tor}})=\overline{X}_{K,\Sigma,w}^{\text{\rm tor}}.
\end{equation*}
\end{enumerate}
\end{thm}
\begin{proof}
Recall that by definition, part 1 just means that for each cusp label $\Sc{C}\in\text{Cusp}_K$, the formal completion of $\overline{X}_{K,\Sigma,w}^{\text{tor}}$ along $\Ca{X}_{K,\Sigma,\Sc{C}}^{\text{tor}}$ is
\begin{equation*}
(\Fr{X}_{\Sc{C},\Sigma_{\Sc{C}}}/\Gamma_{\Sc{C}})_{\overline{X}_{\Sc{C},w}}
\end{equation*}
under the isomorphism
\begin{equation*}
\hat{\Ca{X}}^{\text{tor}}_{K,\Sigma,\Sc{C}}\simeq \Fr{X}_{\Sc{C},\Sigma_{\Sc{C}}}/\Gamma_{\Sc{C}}.
\end{equation*}
But this is exactly what we showed in the preceding paragraphs after intersecting with each affine open $\Fr{U}$.

Part 2 follows from Theorem \ref{T:EOprop} and part 3 follows from \ref{P:EOhecke}.
\end{proof}

\subsection{Ekedahl-Oort Stratification of the Minimal Compactification}

Now we would like to define an Ekedahl-Oort stratification on the minimal compactification $X_K^{\text{min}}$.  There is no natural partial $\BT$ over $X_K^{\text{min}}$ to use to construct it directly.  Instead we use the construction of the previous section for the toroidal compactification.

Indeed we have a map $\pi_{K,\Sigma}:X_{K,\Sigma}^{\text{tor}}\to X_K^{\text{min}}$, and the fibers of $\pi_{K,\Sigma}$ are contained entirely inside single Ekedahl-Oort strata $X_{K,\Sigma,w}^{\text{tor}}$.  Hence we may simply define
\begin{equation*}
X_{K,w}^{\text{min}}=\pi_{K,\Sigma}(X_{K,\Sigma,w}^{\text{tor}})
\end{equation*}
and
\begin{equation*}
\overline{X}_{K,w}^{\text{min}}=\pi_{K,\Sigma}(\overline{X}_{K,\Sigma,w}^{\text{tor}}).
\end{equation*}

Here is the main theorem on the Ekedahl-Oort stratification of the minimal compactification.

\begin{thm}\label{T:EOmin}
Let $K\subset G(\hat{\Z}^{(p)})$ be neat open compact.  For each $w\in W^I$ we have
\begin{enumerate}
\item $\overline{X}_{K,w}^{\text{\rm min}}$ is well positioned at the boundary and for each cusp label $\Sc{C}\in\text{\rm Cusp}_K$ the corresponding subscheme is $\overline{X}_{\Sc{C},w}\subset X_\Sc{C}$.
\item We have (set theoretically)
\begin{equation*}
\overline{X}_{K,w}^{\text{\rm min}}=\coprod_{w'\preceq w}X_{K,w'}^{\text{\rm min}}.
\end{equation*}
where $\preceq$ is the partial order of Theorem \ref{T:EOprop}.
\item Let $g\in G(\A^{\infty,p})$ and let $K,K'\subset G(\hat{\Z}^{(p)})$ be neat open compact subgroups with $g^{-1}Kg\subset K'$ then
\begin{equation*}
[g]^{-1}(X_{K',w}^{\text{\rm min}})=X_{K,w}^{\text{\rm min}}\qquad\text{and}\qquad [g]^{-1}(\overline{X}_{K',w}^{\text{\rm min}})=\overline{X}_{K,w}^{\text{\rm min}}.
\end{equation*}
\end{enumerate}
\end{thm}
\begin{proof}
The first point follows from Theorems \ref{T:wellpos} and \ref{T:EOtor}.  Part 2 follows from Theorem \ref{T:EOprop} and part 3 follows from \ref{P:EOhecke}.
\end{proof}

\section{Extension of Hasse Invariants to the Boundary}\label{S:hassebound}

In this section we explain how to extend the generalized Hasse invariants from the interior to the entire toroidal and minimal compactifications.

As a preliminary we prove a lemma which permits us to deduce the regularity of a rational function after passing to a formal completion.  If $A$ is a ring then we denote by $K(A)$ its total ring of fractions.  By definition $K(A)=S^{-1}A$ for the set $S$ of non zero divisors in $A$.  If $A\to B$ is flat, then a non zero divisor in $A$ remains a non zero divisor in $B$ and so there is an induced map $K(A)\to K(B)$.  We claim that if in fact $A\to B$ is faithfully flat, then
\begin{equation*}
A=B\cap K(A),
\end{equation*}
the intersection occurring inside $K(B)$.  Indeed suppose $a/s\in K(A)\cap B$.  Then $aB\subset sB$ and hence
\begin{equation*}
a\in (sB)\cap A=sA
\end{equation*}
where the last equality holds because $A\to B$ is faithfully flat.

Now consider the following situation.  Suppose $U_0=\spec A_0$ is a reduced noetherian affine scheme, and $Z\subset U_0$ is a closed subset defined by an ideal $I_0\subset A_0$, which we assume does not contain any generic point .  Let $A$ be the $I_0$-adic completion of $A_0$.  Let us assume that $Z$ meets every irreducible component of $U_0$ so that $A_0\to A$ is injective by Krull's theorem.  Moreover $A_0\to A$ is flat so we have a map $K(A_0)\to K(A)$ which is also injective.
\begin{lem}\label{L:compext}
With notation as above, we have
\begin{equation*}
A_0=A\cap H^0(U_0-Z,\O_{U_0})
\end{equation*}
the intersection taking place inside $K(A)$.  
\end{lem}
\begin{proof}
Let $\p$ be a prime ideal contained in $Z$.  Let $A_{0,\p}$ be the localization of $A_0$ at $\p$, and let $\hat{A}_\p$ be the $\p$-adic completion of $A_0$, which is also the $\p A$-adic completion of $A$.  Then we have maps
\begin{equation*}
A\to A_\p,\quad H^0(U_0-Z,\O_{U_0})\to K(A_{0,\p}),\quad\text{and},\quad K(A)\to K(\hat{A}_\p).
\end{equation*}
Now suppose $x\in A\cap H^0(U_0-Z,\O_{U_0})$.  In order to show that $x\in A$ we need to show that $x$ is regular at each point of $Z$, i.e. for each $\p$ as above, the image of $x$ in $K(A_{0,\p})$ actually lies in $A_{0,\p}$.  But $A_{0,\p}\to \hat{A}_\p$ is faithfully flat, and hence
\begin{equation*}
A_{0,\p}=\hat{A}_\p\cap K(A_{0,\p})
\end{equation*}
by what we said above.
\end{proof}

\begin{thm}\label{T:hassetor}
For each $w\in W^I$, each $K\subset G(\hat{\Z}^{(p)})$ neat open compact, and each $\Sigma$ a good compatible family of cone decompositions at level $K$ there is a unique section
\begin{equation*}
A_{K,\Sigma,w}^{\text{\rm tor}}\in H^0(\overline{X}_{K,\Sigma,w}^{\text{\rm tor}},\omega_K^{\otimes N_w})
\end{equation*}
with the following properties
\begin{enumerate}
\item The restriction of $A_{K,w}^{\text{\rm tor}}$ to $\overline{X}_{K,w}$ is the section $A_{K,w}$ of Theorem \ref{T:hasseext}.
\item $A_{K,\Sigma,w}^{\text{\rm tor}}$ is non vanishing precisely on $X_{K,\Sigma,w}^{\text{\rm tor}}\subset \overline{X}_{K,\Sigma,w}^{\text{\rm tor}}$.
\item $A_{K,\Sigma,w}^{\text{\rm tor}}$ is well positioned at the boundary in the sense of definition \ref{D:wellpossec} where for each cusp label $\Sc{C}$ the corresponding section is $A_{\Sc{C},w}\in H^0(\overline{X}_{\Sc{C},w},\omega_{\Sc{C}}^{\otimes N_w})$.
\item If $g\in G(\A^{\infty,p})$ and $K,K'\subset G(\hat{\Z}^{(p)})$ are open compact subgroups with $g^{-1}Kg\subset K'$ and $\Sigma$ (resp. $\Sigma'$) is a good compatible family of cone decompositions at level $K$ (resp. $K'$) and $\Sigma$ is a $g$-refinement of $\Sigma'$ then
\begin{equation*}
[g]^*A_{K',\Sigma',w}^{\text{\rm tor}}=A_{K,\Sigma,w}^{\text{\rm tor}}
\end{equation*}
under the canonical isomorphism $[g]^*\omega_{K'}\simeq\omega_K$ restricted to $\overline{X}_{K,\Sigma,w}^{\text{\rm tor}}$.
\end{enumerate}
\end{thm}
\begin{proof}
Fix $w\in W^I$.  Let $S\subset\text{Cusp}_K$ be a set of cusp labels with the property that if $\Sc{C}\in S$ and $\Sc{C}'\in\text{Cusp}_K$ with $\Sc{C}\leq \Sc{C}'$ then $\Sc{C}'\in S$.  Then let
\begin{equation*}
\Ca{X}_{K,\Sigma,S}^{\text{tor}}=\coprod_{\Sc{C}\in S}\Ca{X}_{K,\Sigma,\Sc{C}}.
\end{equation*}
Then by the assumption on $S$, $\Ca{X}_{K,\Sigma,S}^{\text{tor}}$ is open in $\Ca{X}_{K,\Sigma}^{\text{tor}}$.  We will prove by induction on $S$ that
\begin{equation*}
A_{K,w}\in H^0(\overline{X}_{K,w},\omega_K^{\otimes N})
\end{equation*}
extends to $\Ca{X}_{K,\Sigma,S}^{\text{tor}}\cap\overline{X}_{K,\Sigma,w}^{\text{tor}}$.  So suppose we have an extension for $S'$ and we want to extend it to $S=S'\cup\{\Sc{C}\}$.

We recall some notation from section \ref{S:eotor}.  We may cover the formal schemes
\begin{equation*}
\hat{X}_{K,\Sigma,\Sc{C}}^{\text{tor}}\simeq(\Fr{X}_{\Sc{C},\Sigma_{\Sc{C}}}\times_R k)/\Gamma_{\Sc{C}}
\end{equation*}
by affine opens $\Fr{U}=\text{spf}\,A$ satisfying the conditions listed there.  We may in particular assume that for each $\Fr{U}$ there is an affine open $U_0=\spec A_0$ in $X_{K,\Sigma,S}^{\text{tor}}$ such that the formal completion of $U_0$ along $U_0\cap\Ca{X}_{K,\Sigma,\Sc{C}}^{\text{tor}}$ is $\Fr{U}$.

As in section \ref{S:eotor} we also denote $U=\spec A$, and let $\overline{U}_w=\spec A_w$ be the closed Ekedahl-Oort strata.  We also let $\overline{U_0},w=\spec A_{0,w}$ be the closed Ekedahl-Oort strata of $U_0$.

Then we have our partially extended section
\begin{equation*}
A_{K,w,S}\in H^0(\overline{U}_{0,w}-(\overline{U}_{0,w}\cap \Ca{X}^{\text{tor}}_{K,\Sigma,\Sc{C}}),\omega^{\otimes N_w})
\end{equation*}
which, after formally completing along $\Ca{X}^{\text{tor}}_{K,\Sigma,\Sc{C}}\cap U_0$ is regular by Proposition \ref{P:boundary}.  Hence by lemma \ref{L:compext} we get the desired extension to all of $\overline{U}_{0,w}$.

This proves the existence of $A_{K,\Sigma,w}^{\text{tor}}$ extending $A_{K,w}$ satisfying property 3.  Property 2 follows from Theorem \ref{T:hasseext}.  Finally property 4 follows from Theorem \ref{T:hasseext} and the fact that $\overline{X}_{K,w}$ is Zariski dense in $\overline{X}_{K,\Sigma,w}^{\text{tor}}$.
\end{proof}

As a consequence we deduce the existence of generalized Hasse invariants on the minimal compactification.

\begin{thm}\label{T:hassemin}
For each $w\in W^I$ and each $K\subset G(\hat{\Z}^{(p)})$ neat open compact there is a unique section
\begin{equation*}
A_{K,w}^{\text{\rm min}}\in H^0(\overline{X}_{K,w}^{\text{\rm min}},\omega_K^{\otimes N_w})
\end{equation*}
with the following properties
\begin{enumerate}
\item The restriction of $A_{K,w}^{\text{\rm min}}$ to $\overline{X}_{K,w}$ is the section $A_{K,w}$ of Theorem \ref{T:hasseext}.
\item $A_{K,w}^{\text{\rm min}}$ is non vanishing precisely on $X_{K,w}^{\text{\rm min}}\subset \overline{X}_{K,w}^{\text{\rm min}}$.
\item $A_{K,w}^{\text{\rm min}}$ is well positioned at the boundary in the sense of definition \ref{D:wellpossec} where for each cusp label $\Sc{C}$ the corresponding section is $A_{\Sc{C},w}\in H^0(\overline{X}_{\Sc{C},w},\omega_{\Sc{C}}^{\otimes N_w})$.
\item If $g\in G(\A^{\infty,p})$ and $K,K'\subset G(\hat{\Z}^{(p)})$ are open compact subgroups with $g^{-1}Kg\subset K'$ then we have
\begin{equation*}
[g]^*A_{K',w}^{\text{\rm min}}=A_{K,w}^{\text{\rm min}}
\end{equation*}
under the canonical isomorphism $[g]^*\omega_{K'}\simeq\omega_K$ restricted to $\overline{X}_{K,w}^{\text{\rm min}}$.
\end{enumerate}
\end{thm}
\begin{proof}
The existence of $A_{K,w}^{\text{min}}$ and the first three properties follow immediately from Theorem \ref{T:hassetor} and part 2 of Remark \ref{R:wellposrem}.  The uniqueness and property 4 follow from \ref{T:hasseext} and the fact that $\overline{X}_{K,w}$ is Zariski dense in $\overline{X}_{K,w}^{\text{min}}$ and the latter is reduced.
\end{proof}

We record the following corollary.

\begin{cor}\label{C:affine}
For each $w\in W^I$, the Ekedahl-Oort stratum in the minimal compactification $X_{K,w}^{\text{\rm min}}$ is affine.
\end{cor}
\begin{proof}
This follows immediately from the fact that it is the non vanishing locus of the section $A_{K,w}^{\text{min}}$ of the ample line bundle $\omega_K^{\otimes N_w}$ on the proper scheme $\overline{X}_{K,w}^{\text{min}}$.
\end{proof}

\chapter{Construction of Congruences}\label{cong}

In this section we explain how to use the generalized Hasse invariants of the previous chapters in order to produce congruences between coherent cohomology classes of automorphic vector bundles on PEL modular varieties.  The underlying idea is quite simple, but the argument is complicated by considerations at the boundary.  We refer the reader to the introduction for an overview of how we construct congruences, and we advise the reader to first understand the argument when the PEL modular variety is compact.

Throughout this chapter we fix an integral PEL datum $(\Ca{O},*,L,\langle\cdot,\cdot,\rangle,h)$ satisfying Condition \ref{extcond} and a neat open compact subgroup $K\subset G(\hat{\Z}^{(p)})$.  We fix an algebraic representation $\rho'$ of $M$ on a finite free $R$-module $W$, as well as an integer $r>0$ and a filtration
\begin{equation*}
0=W_0\subset W_1\subset W_2\subset W_3=W/\pi^rW
\end{equation*}
by $M$-stable submodules such that $W_i/W_{i-1}$ is a free $R/\pi^r$-module for $i=1,2,3$.  We denote the representation of $M$ on $W_2/W_1$ by $\rho$.  For example we may just take $\rho$ to be the reduction mod $\pi^r$ of $\rho'$, so that $W_1=W_0$ and $W_2=W_3$, but we also want to allow representations $\rho$ which do not admit lifts to characteristic 0.

The goal of this section is to prove the following theorem, which may be regarded as the main result of this thesis.

\begin{thm}\label{T:abscong}
For all integers $n\geq0$ and $C$ there is an integer $k\geq C$ such that
\begin{equation*}
H^n(\Ca{X}_K^{\text{\rm min}},V_{\rho,K}^{\text{\rm sub}})
\end{equation*}
is a subquotient of
\begin{equation*}
H^0(\Ca{X}_K^{\text{\rm min}},V_{\rho',K}^{\text{\rm sub}}\otimes\omega_K^{\otimes k})
\end{equation*}
as $\Bf{T}_K$-modules.
\end{thm}

\begin{rem}
The role of $C$ is to ensure that we can take the number $k$ guaranteed by the theorem to be as large as we like.  This will ensure that $H^0(\Ca{X}_K^{\text{min}},V_{\rho',K}^{\text{sub}}\otimes\omega_K^{\otimes k})\otimes E$ can be computed in terms of automorphic representations of $G$ which we expect to be able to attach Galois representations to by other means.
\end{rem}
\begin{rem}
Our primary goal in this work was to prove this theorem in the case $n>0$.  However we remark that the case $n=0$ is not without interest.  Indeed even when $n=0$, $C=0$, and $\rho$ is the reduction mod $\pi^r$ of $\rho'$ (so that $V_{\rho',K}^{\text{sub}}=V_{\rho,K}^{\text{sub}}/\pi^r$) the map
\begin{equation*}
H^0(\Ca{X}_K^{\text{min}},V_{\rho',K}^{\text{sub}})\to H^0(\Ca{X}_K^{\text{min}},V_{\rho,K}^{\text{sub}})
\end{equation*}
needn't be surjective.
\end{rem}

In section \ref{S:prelim} we give some preliminaries.  We first prove a general lifting lemma which will allow us to canonically extend a sufficiently large power of a section of a line bundle canonically to an infinitesimal thickening.  Then we will give some setup for the inductive step of the construction of congruences.  In particular we introduce certain unions of Ekedahl-Oort strata in the minimal compactification (denoted $X_i$) and certain Hasse invariants $A_i$ on them which are obtained by ``glueing together'' the Hasse invariants of the previous parts.  Then in section \ref{S:induct} the main line of our argument: we inductively reduce the cohomological degree while increasing the codimension of the subscheme of $\Ca{X}_K^{\text{min}}$ we are working on.  Finally we conclude the proof of Theorem \ref{T:abscong} in section \ref{S:finpf}.

\section{Preliminaries}\label{S:prelim}

\subsection{Lifting Lemma}

The following standard lemma says that we can improve congruences by taking powers.

\begin{lem}\label{L:loclift}
Let $A$ be a ring and $I\subset A$ an ideal.  Then if $x,y\in A$ with $x-y\in I$ then
\begin{equation*}
x^{p^n}-y^{p^n}\in I^{p^n}+pI^{p^{n-1}}+p^2I^{p^{n-2}}+\cdots+p^nI.
\end{equation*}
\end{lem}
\begin{proof}
Write $x=y+b$ for $a\in A$.  Then
\begin{equation*}
x^p=y^p+b^p+\sum_{i=1}^{p-1}{p\choose i}y^ib^{p-i}
\end{equation*}
and hence $x^p-y^p\in I^p+pI$.  The result follows from this and induction, upon noting that if
\begin{equation*}
J=I^{p^{n-1}}+pI^{p^{n-2}}+\cdots+p^{n-1}I
\end{equation*}
then
\begin{equation*}
J^p+pJ\subset I^{p^n}+pI^{p^{n-1}}+\cdots+p^nI.
\end{equation*}
\end{proof}

Globalizing this we have
\begin{lem}\label{L:lift}
Let $X$ be a scheme and $X_0\subset X$ a closed subscheme defined by a sheaf of ideals $\I$.  Assume that there are integers $c$ and $d$ with $p^c=0$ on $X$ and $\I^{p^d}=0$ (and so in particular the sets underlying $X$ and $X_0$ are the same.)  Assume there is a line bundle $\L$ on $X$ and a section $s\in H^0(X_0,\L|_{X_0})$.  Then there exists a section $\tilde{s}\in H^0(X,\L^{p^{c+d-1}})$ with
\begin{equation*}
\tilde{s}|_{X_0}=s^{p^{c+d-1}}\in H^0(X_0,\L^{p^{c+d-1}}|_{X_0})
\end{equation*}
Moreover $\tilde{s}$ is the unique section with the following property: for any Zariski open $U\subset X$ and section $s'\in H^0(U,\L|_{U})$ with $s'|_{U_0}=s$ where $U_0=U\cap X_0$, we have
\begin{equation*}
\tilde{s}|_{U}={s'}^{p^{c+d-1}}.
\end{equation*}
\end{lem}
\begin{proof}
Pick an affine cover $\{U_i\}$ of $X$ such that $\L$ is trivial on each $U_i$.  Let $U_{i,0}=U_i\cap X_0$.  Then for each $i$
\begin{equation*}
H^0(U_i,\L|_{U_i})\to H^0(U_{i,0},\L_{U_{i,0}})
\end{equation*}
is surjective (because $U_{i,0}$ is a closed subscheme of $U_i$ and $\L|_{U_i}$ is trivial,) and so we can pick sections $s_i\in H^0(U_i,\L|_{U_i})$ reducing to $s|_{U_{i,0}}$.  We claim that the sections $s_i^{p^{c+d-1}}$ glue to give a section with the desired property.  For any pair of indices $i$ and $j$, and any affine open $V\subset U_i\cap V_i$ we have two lifts $s_i|_{V}$ and $s_j|_{V}$ of $s|_{V_0}$ where $V_0=V\cap X_0$.  Then by Lemma \ref{L:loclift} we have $s_i^{p^{c+d-1}}|_{V}=s_j^{p^{c+d-1}}|_{V}$.  Hence the sections $s_i^{p^{c+d-1}}$ glue together to give the desired section $\tilde{s}$.  The uniqueness statement is clear.
\end{proof}

\subsection{Setup}\label{S:setup}

For every $g\in G(\A^{\infty,p})$ we have a Hecke correspondence
\begin{equation*}
\begin{tikzcd}
\Ca{X}_K^{\text{min}}&\Ca{X}_{gKg^{-1}\cap K}^{\text{min}} \arrow{l}[swap]{[g]}\arrow{r}{[1]}&\Ca{X}_K^{\text{min}}
\end{tikzcd}
\end{equation*}

\begin{defn}
\begin{enumerate}
\item We say that a closed subscheme $\Ca{Z}\subset \Ca{X}_K^{\text{min}}$ is Hecke stable if for every $g\in G(\A^{\infty,p})$ we have
\begin{equation*}
[g]^{-1}(\Ca{Z})=[1]^{-1}(\Ca{Z})
\end{equation*}
as subschemes of $\Ca{X}_{gKg^{-1}\cap K}^{\text{min}}$.  We denote it by $\Ca{Z}_g$.
\item If $\Ca{Z}\subset\Ca{X}_K^{\text{min}}$ is Hecke stable then we say that a section
\begin{equation*}
A\in H^0(\Ca{Z},\omega_K^{\otimes k}|_{\Ca{Z}})
\end{equation*}
is Hecke stable if
\begin{equation*}
[g]^*A=[1]^*A
\end{equation*}
as sections of $H^0(\Ca{Z}_g,\omega_{gKg^{-1}\cap K}^{\otimes k}|_{\Ca{Z}_g})$ via the isomorphisms
\begin{equation*}
[1]^*\omega_K\simeq \omega_{gKg^{-1}\cap K}
\end{equation*}
and
\begin{equation*}
[g]^*\omega_K\simeq\omega_{gKg^{-1}\cap K}
\end{equation*}
restricted to $\Ca{Z}_g$.
\item If $\tau$ is any algebraic representation of $M$ on either a free $R$-module or a free $R/\pi^s$ module for some $s$, and $V_{\tau,K}^{\text{sub}}/\Ca{X}_K^{\text{min}}$ is the corresponding sub canonically extended automorphic vector bundle as in definition \ref{D:minext} so that we have maps
\begin{equation*}
g:[g]^*V_{\tau,K}^\text{sub}\to V_{\tau,gKg^{-1}\cap K}^{\text{sub}}
\end{equation*}
and
\begin{equation*}
\tr:[1]_*V_{\tau,gKg^{-1}\cap K}^{\text{sub}}\to V_{\tau,K}^{\text{sub}}
\end{equation*}
as in Proposition \ref{P:minhecke} and definition \ref{D:tracedef}.  Upon restricting to $\Ca{Z}$ and $\Ca{Z}_g$ we obtain
\begin{equation*}
g:[g]^*(V_{\tau,K}^{\text{sub}}|_{\Ca{Z}})=([g]^*V_{\tau,K}^{\text{sub}})|_{\Ca{Z}_g}\to V_{\tau,gKg^{-1}\cap K}^{\text{sub}}|_{\Ca{Z}_g}
\end{equation*}
and
\begin{equation*}
\tr:[1]_*(V_{\tau,gKg^{-1}\cap K}^{\text{sub}}|_{\Ca{Z}_g})=([1]_*V_{\tau,gKg^{-1}\cap K}^{\text{sub}})|_{\Ca{Z}}\to V_{\tau,K}^{\text{sub}}|_{\Ca{Z}}
\end{equation*}
where in the first equality we are using the fact that $[1]$ is affine.  Then we define the Hecke operator $T_g$ to be the endomorphism of $H^i(\Ca{Z},V_{\tau,K}^{\text{sub}}|_{\Ca{Z}})$ given by the composition of
\begin{equation*}
H^i(\Ca{Z},V_{\tau,K}^{\text{sub}}|_{\Ca{Z}})\overset{[g]^*}{\to} H^i(\Ca{Z}_g,[g]^*(V_{\tau,K}^{\text{sub}}|_{\Ca{Z}}))\overset{g}{\to}H^i(\Ca{Z}_g,V_{\tau,gKg^{-1}\cap K}^{\text{sub}}|_{\Ca{Z}_g})
\end{equation*}
and
\begin{equation*}
H^i(\Ca{Z}_g,V_{\tau,gKg^{-1}\cap K}^{\text{sub}}|_{\Ca{Z}_g})=H^i(\Ca{Z},[1]_*(V_{\tau,gKg^{-1}\cap K}^{\text{sub}}|_{\Ca{Z}_g}))\overset{\tr}{\to}H^i(\Ca{Z},V_{\tau,K}^{\text{sub}}|_{\Ca{Z}}).
\end{equation*}
\end{enumerate}
\end{defn}

For $i=0,\ldots,\dim(X_K)$ let
\begin{equation*}
X_i=\bigcup_{\substack{w\in W^I\\l(w)=\dim(X_K)-i}} \overline{X}_{K,w}^{\text{min}}\subset X_K^{\text{min}}
\end{equation*}
be the union of the closed codimension $i$ Ekedahl-Oort strata.  Let $N_i'$ be the least common multiple of the $N_w$ with $l(w)=\dim(X_K)-i$.  Consider the map
\begin{equation*}
f:X_i'=\coprod_{\substack{w\in W^I\\l(w)=\dim(X_K)-i}}\overline{X}_{K,w}^{\text{min}}\to X_i
\end{equation*}
this map $f$ is finite and if we consider the open
\begin{equation*}
U=\coprod_{\substack{w\in W^I\\l(w)=\dim(X_K)-i}}X_{K,w}^{\text{min}}\subset X_i
\end{equation*}
then
\begin{equation*}
f^{-1}(U)\to U
\end{equation*}
is an isomorphism and we have a section
\begin{equation*}
A_i'\in H^0(U,\omega_K^{\otimes N_i'}|_{X_i}).
\end{equation*}
whose restriction to $X_{K,w}^{\text{min}}$ is $(A_{K,w}^{\text{min}})^{N_i'/N_w}$.  Then $A_i'$ extends to a section of $\omega^{\otimes N_i'}$ on $X_i'$ that vanishes (set theoretically) on $X_i'-f^{-1}(U)$.  Then applying lemma \ref{L:extlem} there is some integer $M$ such that $A_i'^M$ extends to a section
\begin{equation*}
A_i\in H^0(X_i,\omega_K^{\otimes N_i}|_{X_i})
\end{equation*}
where $N_i=MN_i'$ whose set theoretic vanishing locus is $X_i-U=X_{i+1}$.

We summarize the key properties of the filtration $X_i$ of $X_{K,w}^{\text{min}}$ and the sections $A_i$ in the following proposition.  We remark that no properties of the $X_i$ and $A_i$ beyond those listed here will be used in what follows.
\begin{prop}\label{P:indsetup}
There is a filtration
\begin{equation*}
X_K^{\text{\rm min}}=X_0\supset X_1\supset\cdots\cdots\supset X_{\dim(X_K)}\supset X_{\dim(X_K)+1}=\emptyset
\end{equation*}
of $X_K^{\text{\rm min}}$ by reduced closed subschemes along with sections
\begin{equation*}
A_i\in H^0(X_i,\omega_K^{\otimes N_i}|_{X_i})
\end{equation*}
satisfying the following properties:
\begin{enumerate}
\item For each cusp label $\Sc{C}\in\text{\rm Cusp}_K$, each irreducible component of
\begin{equation*}
X_i\cap X_{\Ca{C}}
\end{equation*}
has codimension $i$ in $X_\Ca{C}$.
\item The subschemes $X_i\subset X_K^{\text{\rm min}}$ are well positioned at the boundary and are Hecke stable.
\item The set theoretic vanishing locus of the section $A_i$ is $X_{i+1}$.  The sections $A_i$ are Hecke stable and well positioned at the boundary.
\end{enumerate}
\end{prop}
\begin{proof}
Points 1 and 2 follow from Theorem \ref{T:EOmin} while point 3 follows from Theorem \ref{T:hassemin}.
\end{proof}

\section{Inductive Step}\label{S:induct}

As the first step in the proof of Theorem \ref{T:abscong} we will prove the following proposition, which is the main induction step in the argument.

\begin{prop}\label{P:congind}
Let $n$ be as in the Theorem \ref{T:abscong}.
\begin{itemize}
\item There exists a filtration
\begin{equation*}
\Ca{X}_K^{\text{\rm min}}\times_R R/\pi^r=\tilde{X}_0\supset\tilde{X}_1\supset\cdots\supset\tilde{X}_n
\end{equation*}
of $\Ca{X}_K^{\text{\rm min}}$ by Hecke stable closed subschemes which are well positioned at the boundary and which satisfy $(\tilde{X}_i)_{\text{\rm red}}=X_i$.  Moreover for each cusp label $\Sc{C}\in\text{\rm Cusp}_K$, the (scheme theoretic) intersection
\begin{equation*}
\tilde{X}_i\cap\Ca{X}_{\Sc{C}}
\end{equation*}
is Cohen Macaulay.
\item For $i=0,\ldots,n$, there exists integers $\tilde{N}_i>0$ and Hecke stable sections
\begin{equation*}
\tilde{A}_i\in H^0(\tilde{X}_i,\omega_K^{\otimes\tilde{N}_i}|_{\tilde{X}_i})
\end{equation*}
which are well positioned at the boundary and have the property that for each cusp label $\Sc{C}\in\text{Cusp}_K$, the restriction of $\tilde{A}_i$ to $\tilde{X}_i\cap\Ca{X}_{\Sc{C}}$ is a non zero divisor.  Moreover for $i=0,\ldots,n-1$ there are integers $k_i>0$ such that $\tilde{X}_{i+1}\subset \tilde{X}_i$ is the vanishing locus of $\tilde{A}_i^{k_i}$.
\item For $i=0,\ldots,n$, there are integers $M_i$ defined by $M_0=0$ and $M_{i+1}=M_i+k_i\tilde{N}_i$ and for $i=0,\ldots,n-1$ there are Hecke equivariant surjections
\begin{equation*}
H^{n-i-1}(\tilde{X}_{i+1},(V_{\rho,K}^{\text{\rm sub}}\otimes\omega_K^{\otimes M_{i+1}})|_{\tilde{X}_{i+1}})\twoheadrightarrow H^{n-i}(\tilde{X}_i,(V_{\rho,K}^{\text{\rm sub}}\otimes\omega_K^{\otimes M_i})|_{\tilde{X}_i})
\end{equation*}
\end{itemize}
\end{prop}
\begin{proof}
This will be proved by induction on $i$.  For the base case we have $\tilde{X}_0=\Ca{X}_K^{\text{min}}\times_R R/\pi^r$.  We clearly have $(\tilde{X}_0)_{\text{red}}=X_0$ and for each cusp label $\Sc{C}\in\text{Cusp}_K$
\begin{equation*}
\tilde{X}_0\cap\Ca{X}_{\Sc{C}}=\Ca{X}_{\Sc{C}}\times_R R/\pi^r
\end{equation*}
which is Cohen Macaulay as $\Ca{X}_{\Sc{C}}/R$ is smooth and hence $\Ca{X}_{\Sc{C}}$ is regular and $\pi^r$ is a non zero divisor.

First we explain how to construct $\tilde{A}_i$ once $\tilde{X}_i$ has been constructed, using Lemma \ref{L:lift}.  Let $\I$ be the ideal sheaf of $X_i$ in $\tilde{X}_i$.  Then as $(\tilde{X}_i)_{\text{red}}=X_i$ we can pick an integer $c$ with $\I^{p^c}=0$ and an integer $d$ with $p^d\subset(\pi^r)$ and hence $p^d$ is zero on $\tilde{X}_i$.  Then let $\tilde{N}_i=p^{c+d-1}N_i$ and let $\tilde{A}_i\in H^0(\tilde{X}_i,\omega_K^{\otimes\tilde{N}_i}|_{\tilde{X}_i})$ be the canonical section guaranteed to exist by Lemma \ref{L:lift} which satisfies
\begin{equation*}
\tilde{A}_i|_{X_i}=A_i^{p^{c+d-1}}\in H^0(X_i,\omega_K^{\otimes\tilde{N}_i}|_{X_i}).
\end{equation*}

We need to explain why $\tilde{A}_i$ is Hecke stable.  For any $g\in G(\A^{\infty,p})$ let
\begin{equation*}
X_{i,g}=[g]^{-1}(X_i)=[1]^{-1}(X_i)\subset X_{gKg^{-1}\cap K}^{\text{min}}
\end{equation*}
and let
\begin{equation*}
\tilde{X}_{i,g}=[g]^{-1}(\tilde{X}_i)=[1]^{-1}(\tilde{X}_i)\subset\Ca{X}_{gKg^{-1}\cap K}^{\text{min}}.
\end{equation*}
Let $\I_g$ be the ideal sheaf of $X_{i,g}$ in $\tilde{X}_{i,g}$.  Then $\I_g^{p^d}=0$ by part 3 of Theorem \ref{T:wellpos} and the fact that if $\Sc{C}\in\text{Cusp}_K$ and $\Sc{C'}\in\text{Cusp}_{gKg^{-1}}$ are such that the restriction of $[1]$ to $\Ca{X}_{\Sc{C}'}$ factors through $\Ca{X}_{\Sc{C}}$ then $[1]:\Ca{X}_{\Sc{C}'}\to\Ca{X}_{\Sc{C}}$ is \'{e}tale.

Now Lemma \ref{L:lift} applied to
\begin{equation*}
[g]^*A_i=[1]^*A_i\in H^0(X_{i,g},\omega_{gKg^{-1}\cap K}^{\otimes N_i}|_{X_{i,g}})
\end{equation*}
gives a canonical section
\begin{equation*}
\tilde{A}_{i,g}\in H^0(\tilde{X}_{i,g},\omega_{gKg^{-1}\cap K}^{\otimes\tilde{N}_i}|_{\tilde{X}_{i,g}})
\end{equation*}
We will use the uniqueness statement of Lemma \ref{L:lift} to show that
\begin{equation*}
[g]^*\tilde{A}_i=\tilde{A}_{i,g}=[1]^*\tilde{A}_i.
\end{equation*}
Let us prove the first equality, the other following by the same argument.  Let $U\subset\tilde{X}_i$ be a Zariski open subset such that there exists a section
\begin{equation*}
A'\in H^0(U,\omega_K^{\otimes N_i}|_{U})
\end{equation*}
with
\begin{equation*}
A'|_{U\cap X_i}=A_i\in H^0(U\cap X_i,\omega_K^{N_i}|_{U\cap X_i}).
\end{equation*}
Then by the uniqueness statement of Lemma \ref{L:lift} we have
\begin{equation*}
(A')^{p^{c+d-1}}=\tilde{A}_i|_{U}\in H^0(U,\omega_K^{\otimes \tilde{N}_i}|_{U})
\end{equation*}
Similarly by the uniqueness statement of Lemma \ref{L:lift} applied to the section
\begin{equation*}
[g]^*(A')\in H^0([g]^{-1}(U),\omega_{gKg^{-1}\cap K}^{\otimes N_i}|_{[g]^{-1}(U)})
\end{equation*}
we have
\begin{equation*}
[g]^*(A')^{p^{c+d-1}}=\tilde{A}_{i,g}|_{[g]^{-1}(U)}\in H^0([g]^{-1}(U),\omega_{gKg^{-1}\cap K}^{\otimes \tilde{N}_i}|_{[g]^{-1}(U)})
\end{equation*}
and hence
\begin{equation*}
[g]^*\tilde{A}_i|_{[g]^{-1}(U)}=[g]^*(A')^{p^{c+d-1}}=\tilde{A}_{i,g}|_{[g]^{-1}(U)}
\end{equation*}
We are done upon noting that we can pick a cover of $\tilde{X}_i$ by such opens $U$.

Next we claim that for each cusp label $\Sc{C}\in\text{Cusp}_K$, the restriction of $\tilde{A}_i$ to $\tilde{X}_i\cap\Ca{X}_{\Sc{C}}$ is a non zero divisor.  As $\tilde{X}_i\cap\Ca{X}_{\Sc{C}}$ is Cohen-Macaulay, it has no embedded primes and so in order to prove the claim it suffices to show that it doesn't vanish set theoretically on any reduced irreducible component of $\tilde{X}_i\cap\Ca{X}_{\Sc{C}}$.  But the set theoretic vanishing locus of $\tilde{A}_i$ is the same of that of $A_i$, namely $X_{i+1}$ by part 3 of Proposition \ref{P:indsetup}.  But by part 1 of \ref{P:indsetup}, no reduced irreducible component of $\tilde{X}_i\cap \Ca{X}_{\Sc{C}}$ is contained in $X_{i+1}\cap\Ca{X}_{\Sc{C}}$.  From this, the inductive hypothesis, and Proposition \ref{P:autminbound} we conclude that $\tilde{A}_i$ is a non zero divisor on $V_{\rho,K}^{\text{sub}}|_{\tilde{X}_i}$.  

Now suppose $i<n$.  As $\tilde{A}_i$ is a non zero divisor on $V_{\rho,K}^{\text{sub}}|_{\tilde{X}_i}$, for any positive integer $k$ we have a short exact sequence
\begin{equation*}
0\to (V_{\rho,K}^{\text{sub}}\otimes\omega_K^{\otimes M_i})|_{\tilde{X}_i}\overset{\cdot\tilde{A}_i^k}{\to}(V_{\rho,K}^{\text{sub}}\otimes\omega_K^{\otimes M_i+k\tilde{N}_i})|_{\tilde{X}_i}\to(V_{\rho,K}^{\text{sub}}\otimes\omega_K^{\otimes M_i+k\tilde{N}_i})|_{V(\tilde{A}_i^k)}\to 0.
\end{equation*}
By Serre vanishing we may pick $k$ sufficiently large such that
\begin{equation*}
H^{n-i}(\tilde{X}_i,(V_{\rho,K}^{\text{sub}}\otimes\omega_K^{\otimes M_i+k\tilde{N}_i})|_{\tilde{X}_i})=0
\end{equation*}
and hence a piece of the long exact sequence in cohomology of the above short exact sequence reads
\begin{equation*}
H^{n-i-1}(\tilde{X}_i,(V_{\rho,K}^{\text{sub}}\otimes\omega_K^{\otimes M_i+k\tilde{N}_i})|_{V(\tilde{A}_i^k)})\overset{\delta}{\to} H^{n-i}(\tilde{X}_i,(V_{\rho,K}^{\text{sub}}\otimes\omega_K^{\otimes M_i})|_{\tilde{X}_i})\to 0.
\end{equation*}

We claim that $\delta$ is Hecke equivariant.  Take any $g\in G(\A^{\infty,p})$.  To save space let $K_g=gKg^{-1}\cap K$.  First note that whenever we have a map $f:X\to Y$ of schemes and a coherent sheaf $\F$ on $Y$, the map
\begin{equation*}
f^*:H^i(Y,\F)\to H^i(X,f^*\F)
\end{equation*}
is by definition the composition
\begin{equation*}
H^i(Y,\F)\overset{f^{-1}}{\to} H^i(X,f^{-1}\F)\to H^i(X,f^*\F)
\end{equation*}
where $f^{-1}(-)$ denotes the pullback of sheaves of abelian groups, and the second map is induced by the map $f^{-1}\F\to f^*\F$ of sheaves of abelian groups on $X$.  Moreover the functor $f^{-1}(-)$ is exact, and $f^{-1}:H^*(Y,-)\to H^*(X,f^{-1}(-))$ is a morphism of $\delta$-functors.  Applying this to $[g]:\tilde{X}_{i,g}\to\tilde{X}_i$ we obtain a commutative square
\begin{equation*}
\begin{tikzcd}[row sep=scriptsize, column sep=tiny]
H^{n-i-1}(\tilde{X}_i,(V_{\rho,K}^{\text{sub}}\otimes\omega_K^{\otimes M_i+k\tilde{N}_i})|_{V(\tilde{A}_i^k)}) \arrow{r}{\delta}\arrow{d}{[g]^{-1}} &H^{n-i}(\tilde{X}_i,(V_{\rho,K}^{\text{sub}}\otimes\omega_K^{\otimes M_i})|_{\tilde{X}_i}) \arrow{d}{[g]^{-1}}\\
H^{n-i-1}(\tilde{X}_{i,g},[g]^{-1}((V_{\rho,K}^{\text{sub}}\otimes\omega_K^{\otimes M_i+k\tilde{N}_i})|_{V(\tilde{A}_i^k)})) \arrow{r}{\delta}&H^{n-i}(\tilde{X}_{i,g},[g]^{-1}((V_{\rho,K}^{\text{sub}}\otimes\omega_K^{\otimes M_i})|_{\tilde{X}_i}))
\end{tikzcd}
\end{equation*}

Next we have a morphism of short exact sequences of sheaves of abelian groups on $X_{i,g}$
\begin{equation*}
\begin{tikzcd}[column sep=tiny]
{[g]}^{-1}(V_{\rho,K}^{\text{sub}}\otimes\omega_K^{\otimes M_i})|_{\tilde{X}_i}\arrow{r}{\cdot\tilde{A}_i^k}\arrow{d}& {[g]}^{-1}(V_{\rho,K}^{\text{sub}}\otimes\omega_K^{\otimes M_i+k\tilde{N}_i})|_{\tilde{X}_i}\arrow{r}\arrow{d}&{[g]}^{-1}(V_{\rho,K}^{\text{sub}}\otimes\omega_K^{\otimes M_i+k\tilde{N}_i})|_{V(\tilde{A}_i^k)}\arrow{d}\\
(V_{\rho,K_g}^{\text{sub}}\otimes\omega_{K_g}^{\otimes M_i})|_{\tilde{X}_{i,g}}\arrow{r}{\cdot\tilde{A}_{i,g}^k}&(V_{\rho,K_g}^{\text{sub}}\otimes\omega_{K_g}^{\otimes M_i+k\tilde{N}_i})|_{\tilde{X}_{i,g}}\arrow{r}&(V_{\rho,K_g}^{\text{sub}}\otimes\omega_{K_g}^{\otimes M_i+k\tilde{N}_i})|_{V(\tilde{A}_{i,g}^k)}
\end{tikzcd}
\end{equation*}
where, for example, the map in the first column is the composition
\begin{equation*}
[g]^{-1}((V_{\rho,K}^{\text{sub}}\otimes\omega_K^{\otimes M_i})|_{\tilde{X}_i})\to [g]^*((V_{\rho,K}^{\text{sub}}\otimes\omega_K^{\otimes M_i})|_{\tilde{X}_i})\overset{g}{\to}(V_{\rho,K_g}^{\text{sub}}\otimes\omega_{K_g}^{\otimes M_i})|_{\tilde{X}_{i,g}}
\end{equation*}
and the other two columns are defined similarly.  From this we get a morphism of long exact sequences in cohomology, and in particular a commutative square
\begin{equation*}
\begin{tikzcd}[column sep=tiny]
H^{n-i-1}(\tilde{X}_i,[g]^{-1}((V_{\rho,K}^{\text{sub}}\otimes\omega_K^{\otimes M_i+k\tilde{N}_i})|_{V(\tilde{A}_i^k)})) \arrow{r}{\delta}\arrow{d} &H^{n-i}(\tilde{X}_{i,g},[g]^{-1}((V_{\rho,K}^{\text{sub}}\otimes\omega_K^{\otimes M_i})|_{\tilde{X}_i})) \arrow{d}\\
H^{n-i-1}(\tilde{X}_{i,g},(V_{\rho,K_g}^{\text{sub}}\otimes\omega_{K_g}^{\otimes M_i+k\tilde{N}_i})|_{V(\tilde{A}_{i,g}^k)}) \arrow{r}{\delta}&H^{n-i}(\tilde{X}_{i,g},(V_{\rho,K_g}^{\text{sub}}\otimes\omega_{K_g}^{\otimes M_i})|_{\tilde{X}_{i,g}}).
\end{tikzcd}
\end{equation*}
Next the functor $[1]_*$ is exact (because $[1]$ is affine) and defines an isomorphism of $\delta$-functors $H^*(\tilde{X}_{i,g},-)\simeq H^*(\tilde{X}_i,[1]_*(-))$ and hence we obtain a commutative square
\begin{equation*}
\begin{tikzcd}[column sep=tiny]
H^{n-i-1}(\tilde{X}_{i,g},(V_{\rho,K_g}^{\text{sub}}\otimes\omega_{K_g}^{\otimes M_i+k\tilde{N}_i})|_{V(\tilde{A}_{i,g}^k)}) \arrow{r}{\delta}\arrow{d} &H^{n-i}(\tilde{X}_{i,g},(V_{\rho,K_g}^{\text{sub}}\otimes\omega_{K_g}^{\otimes M_i})|_{\tilde{X}_{i,g}}) \arrow{d}\\
H^{n-i-1}(\tilde{X}_i,[1]_*(V_{\rho,K_g}^{\text{sub}}\otimes\omega_{K_g}^{\otimes M_i+k\tilde{N}_i})|_{V(\tilde{A}_{i,g}^k)}) \arrow{r}{\delta}&H^{n-i}(\tilde{X}_i,[1]_*(V_{\rho,K_g}^{\text{sub}}\otimes\omega_{K_g}^{\otimes M_i})|_{\tilde{X}_{i,g}}).
\end{tikzcd}
\end{equation*}
Finally we have a morphism of short exact sequences of sheaves on $\tilde{X}_i$
\begin{equation*}
\begin{tikzcd}[column sep=tiny]
{[1]}_*(V_{\rho,K_g}^{\text{sub}}\otimes\omega_{K_g}^{\otimes M_i})|_{\tilde{X}_{i,g}}\arrow{r}{\cdot\tilde{A}_{i,g}^k}\arrow{d}{\tr}&{[1]}_*(V_{\rho,K_g}^{\text{sub}}\otimes\omega_{K_g}^{\otimes M_i+k\tilde{N}_i})|_{\tilde{X}_{i,g}}\arrow{r}\arrow{d}{\tr}&{[1]}_*(V_{\rho,K_g}^{\text{sub}}\otimes\omega_{K_g}^{\otimes M_i+k\tilde{N}_i})|_{V(\tilde{A}_{i,g}^k)}\arrow{d}{\tr}\\
(V_{\rho,K}^{\text{sub}}\otimes\omega_K^{\otimes M_i})|_{\tilde{X}_i}\arrow{r}{\cdot\tilde{A}_i^k}&(V_{\rho,K}^{\text{sub}}\otimes\omega_K^{\otimes M_i+k\tilde{N}_i})|_{\tilde{X}_i}\arrow{r}&(V_{\rho,K}^{\text{sub}}\otimes\omega_K^{\otimes M_i+k\tilde{N}_i})|_{V(\tilde{A}_i^k)}
\end{tikzcd}
\end{equation*}
and hence we obtain a commutative square
\begin{equation*}
\begin{tikzcd}[column sep=tiny]
H^{n-i-1}(\tilde{X}_i,[1]_*(V_{\rho,K_g}^{\text{sub}}\otimes\omega_{K_g}^{\otimes M_i+k\tilde{N}_i})|_{V(\tilde{A}_{i,g}^k)}) \arrow{r}{\delta}\arrow{d}{\tr} &H^{n-i}(\tilde{X}_i,[1]_*(V_{\rho,K_g}^{\text{sub}}\otimes\omega_{K_g}^{\otimes M_i})|_{\tilde{X}_{i,g}}) \arrow{d}{\tr}\\
H^{n-i-1}(\tilde{X}_i,(V_{\rho,K}^{\text{sub}}\otimes\omega_K^{\otimes M_i+k\tilde{N}_i})|_{V(\tilde{A}_i^k)}) \arrow{r}{\delta} &H^{n-i}(\tilde{X}_i,(V_{\rho,K}^{\text{sub}}\otimes\omega_K^{\otimes M_i})|_{\tilde{X}_i}).
\end{tikzcd}
\end{equation*}
Combining the four commutative squares above we conclude that $\delta$ commutes with $T_g$.

Finally (continuing to assume $i<n$) we must construct $\tilde{X}_{i+1}$ and show that it has the required properties.  We take $\tilde{X}_{i+1}=V(\tilde{A}_i^{k_i})$.  Then $(\tilde{X}_{i+1})_{\text{red}}=X_{i+1}$ because the set theoretic vanishing locus of $\tilde{A}_i^{k_i}$ is the same as that of $A_i$ which is $X_{i+1}$ by part 3 of proposition \ref{P:indsetup}.  Moreover we already saw that for each cusp label $\Sc{C}\in\text{Cusp}_K$, the restriction of $\tilde{A}_i^{k_i}$ to $\tilde{X}_i\cap\Ca{X}_{\Sc{C}}$ is a non zero divisor on $\tilde{X}_i\cap\Ca{X}_{\Sc{C}}$ and hence its (scheme theoretic) vanishing locus $\tilde{X}_{i+1}\cap\Ca{X}_{\Sc{C}}$ is also Cohen Macaulay.
\end{proof}

\section{Completing the Proof}\label{S:finpf}

By composing the Hecke equivariant surjections given by Proposition \ref{P:congind} we obtain a Hecke equivariant surjection
\begin{equation*}
H^0(\tilde{X}_n,(V_{\rho,K}^{\text{sub}}\otimes\omega_K^{\otimes M_n})|_{\tilde{X}_n})\twoheadrightarrow H^n(\tilde{X}_0,(V_{\rho,K}^{\text{sub}}\otimes\omega_K^{M_0})|_{\tilde{X}_0})=H^n(\Ca{X}_K^{\text{min}},V_{\rho,K}^{\text{sub}}).
\end{equation*}

Now consider the restriction map
\begin{equation*}
V_{\rho,K}^{\text{sub}}\to V_{\rho,K}^{\text{sub}}|_{\tilde{X}_n}.
\end{equation*}
It is surjective and so it fits into a short exact sequence of sheaves
\begin{equation*}
0\to\F\to V_{\rho,K}^{\text{sub}}\to V_{\rho,K}^{\text{sub}}|_{\tilde{X}_n}\to 0
\end{equation*}
for some coherent sheaf $\F$ on $\Ca{X}_K^{\text{min}}$.  Tensoring this exact sequence with the line bundle $\omega_K^{k\tilde{N}_n+M_n}$ for some nonnegative integer $k$ and taking cohomology we obtain part of a long exact sequence
\begin{equation*}
H^0(\Ca{X}_K^{\text{min}},V_{\rho,K}^{\text{sub}}\otimes\omega_K^{k\tilde{N}_n+M_n})\to H^0(\tilde{X}_n,(V_{\rho,K}^{\text{sub}}\otimes\omega_K^{k\tilde{N}_n+M_n})|_{\tilde{X}_n})\to H^1(\Ca{X}_K^{\text{min}},\F\otimes\omega_K^{k\tilde{N}_n+M_n}).
\end{equation*}

Next from the filtration
\begin{equation*}
0=W_0\subset W_1\subset W_2\subset W_3=W/\pi^rW
\end{equation*}
we obtain by Proposition \ref{P:autpropmin} a filtration of sheaves
\begin{equation*}
0=V_{W_0,K}^{\text{sub}}\subset V_{W_1,K}^{\text{sub}}\subset V_{W_2,K}^{\text{sub}}\subset V_{W_3,K}^{\text{sub}}=V_{\rho',K}^{\text{sub}}/\pi^r
\end{equation*}
such that for $i=1,2,3$
\begin{equation*}
V_{W_i,K}^{\text{sub}}/V_{W_{i-1},K}^{\text{sub}}\simeq V_{W_i/W_{i-1},K}^{\text{sub}}
\end{equation*}

Now as $\omega_K$ is ample, we can pick $k$ sufficiently large so that
\begin{itemize}
\item $H^1(\Ca{X}_K^{\text{min}},\F\otimes\omega_K^{k\tilde{N}_n+M_n})=0$,
\item $H^1(\Ca{X}_K^{\text{min}},V_{W_i/W_{i-1},K}^{\text{sub}}\otimes\omega_K^{k\tilde{N}_n+M_n})=0$ for $i=1,2,3$, and
\item $k\tilde{N}_n+M_n\geq C$, where $C$ is as in the statement of Theorem \ref{T:abscong}.
\end{itemize}
By the first point, and the long exact sequence above, we have a surjective, Hecke equivariant restriction map
\begin{equation*}
H^0(\Ca{X}_K,V_{\rho,K}^{\text{sub}}\otimes\omega_K^{k\tilde{N}_n+M_n})\twoheadrightarrow H^0(\tilde{X}_n,(V_{\rho,K}^{\text{sub}}\otimes\omega_K^{k\tilde{N}_n+M_n})|_{\tilde{X}_n}).
\end{equation*}

Next using the fact that $H^1(\Ca{X}_K^{\text{min}},V_{W_i/W_{i-1}^{\text{sub}},K}\otimes\omega_K^{k\tilde{N}_n+M_n})=0$ for $i=1,2,3$, we conclude that 
\begin{equation*}
H^1(\Ca{X}_K^{\text{min}},V_{\rho',K}^{\text{sub}}\otimes\omega_K^{k\tilde{N}_n+M_n}/\pi^r)=H^1(\Ca{X}_K^{\text{min}},V_{W_3,K}^{\text{sub}}\otimes\omega^{k\tilde{N}_n+M_n}_K)=0
\end{equation*}

and also that for the filtration

\begin{equation*}
0\subset H^0(\Ca{X}_K^{\text{min}},V_{W_1,K}^{\text{sub}}\otimes\omega_K^{k\tilde{N}_n+M_n})\subset H^0(\Ca{X}_K^{\text{min}},V_{W_2,K}^{\text{sub}}\otimes\omega_K^{k\tilde{N}_n+M_n})\subset H^0(\Ca{X}_K^{\text{min}},V_{W_3,K}^{\text{sub}}\otimes\omega_K^{k\tilde{N}_n+M_n})
\end{equation*}
we have
\begin{equation*}
H^0(\Ca{X}_K^{\text{min}},V_{W_i,K}^{\text{sub}}\otimes\omega_K^{k\tilde{N}_n+M_n})/H^0(\Ca{X}_K^{\text{min}},V_{W_{i-1},K}^{\text{sub}}\otimes\omega_K^{k\tilde{N}_n+M_n})\simeq H^0(\Ca{X}_K^{\text{min}},V_{W_i/W_{i-1},K}^{\text{sub}}\otimes\omega_K^{k\tilde{N}_n+M_n})
\end{equation*}
Hecke equivariantly.

Next consider the short exact sequence of sheaves on $X_K^{\text{min}}$
\begin{equation*}
0\to V_{\rho',K}^{\text{sub}}\overset{\cdot\pi^r}{\to}V_{\rho',K}^{\text{sub}}\to V_{\rho',K}^{\text{sub}}/\pi^r\to 0.
\end{equation*}
Tensor with $\omega_K^{\otimes k\tilde{N}_n+M_n}$ and take cohomology.  We see that
\begin{equation*}
H^1(\Ca{X}_K^{\text{min}},V_{\rho',K}^{\text{sub}}\otimes\omega_K^{k\tilde{N}_n+M_n})=0
\end{equation*}
as it is a finitely generated $R$-module and the cokerenel of multiplication by $\pi^r$ on it embeds into
\begin{equation*}
H^1(\Ca{X}_K^{\text{min}},V_{\rho',K}^{\text{sub}}\otimes\omega_K^{k\tilde{N}_n+M_n}/\pi^r)=0.
\end{equation*}
Hence we conclude that the Hecke equivariant reduction mod $\pi^r$ map
\begin{equation*}
H^0(\Ca{X}_K^{\text{min}},V_{\rho',K}^{\text{sub}}\otimes\omega_K^{k\tilde{N}_n+M_n})\twoheadrightarrow H^0(\Ca{X}_K^{\text{min}},V_{\rho',K}^{\text{sub}}\otimes\omega_K^{k\tilde{N}_n+M_n}/\pi^r)
\end{equation*}
is surjective.

Finally from Proposition \ref{P:congind} we have a section
\begin{equation*}
\tilde{A}_n\in H^0(\tilde{X}_n,\omega^{\otimes\tilde{N}_n}|_{\tilde{X}_n})
\end{equation*}
which is Hecke stable and a nonzero divisor on $V_{\rho,K}|_{\tilde{X}_n}$ as in the proof of Proposition \ref{P:congind}.  Hence we have an injective map
\begin{equation*}
H^0(\tilde{X}_n,(V_{\rho,K}^{\text{sub}}\otimes\omega^{M_n})|_{\tilde{X}_n})\overset{\cdot\tilde{A}_n^k}\hookrightarrow H^0(\tilde{X}_n,(V_{\rho,K}^{\text{sub}}\otimes\omega^{k\tilde{N}_n+M_n})|_{\tilde{X}_n})
\end{equation*}
which is Hecke equivariant by the Hecke stability of $\tilde{A}_n$.

To summarize we have:
\begin{itemize}
\item A Hecke equivariant surjection
\begin{equation*}
H^0(\tilde{X}_n,(V_{\rho,K}^{\text{sub}}\otimes\omega_K^{\otimes M_n})|_{\tilde{X}_n})\twoheadrightarrow H^n(\Ca{X}_K^{\text{min}},V_{\rho,K}^{\text{sub}}).
\end{equation*}
\item A Hecke equivariant injection
\begin{equation*}
H^0(\tilde{X}_n,(V_{\rho,K}^{\text{sub}}\otimes\omega^{M_n})|_{\tilde{X}_n})\overset{\cdot\tilde{A}_n^k}\hookrightarrow H^0(\tilde{X}_n,(V_{\rho,K}^{\text{sub}}\otimes\omega^{k\tilde{N}_n+M_n})|_{\tilde{X}_n}).
\end{equation*}
\item A Hecke equivariant surjection
\begin{equation*}
H^0(\Ca{X}_K,V_{\rho,K}^{\text{sub}}\otimes\omega_K^{k\tilde{N}_n+M_n})\twoheadrightarrow H^0(\tilde{X}_n,(V_{\rho,K}^{\text{sub}}\otimes\omega_K^{k\tilde{N}_n+M_n})|_{\tilde{X}_n}).
\end{equation*}
\item A Hecke stable filtration
\begin{equation*}
0\subset M_0\subset M_1\subset H^0(\Ca{X}_K^{\text{min}},V_{\rho',K}^{\text{sub}}\otimes\omega_K^{k\tilde{N}_n+M_n}/\pi^r)
\end{equation*}
for which we have
\begin{equation*}
M_1/M_0\simeq H^0(\Ca{X}_K^{\text{min}},V_{\rho,K}^{\text{sub}}\otimes\omega_K^{k\tilde{N}_n+M_n})
\end{equation*}
Hecke equivariantly.
\item A Hecke equivariant surjection
\begin{equation*}
H^0(\Ca{X}_K^{\text{min}},V_{\rho',K}^{\text{sub}}\otimes\omega_K^{k\tilde{N}_n+M_n})\twoheadrightarrow H^0(\Ca{X}_K^{\text{min}},V_{\rho',K}^{\text{sub}}\otimes\omega_K^{k\tilde{N}_n+M_n}/\pi^r)
\end{equation*}
\end{itemize}
Thus $H^n(\Ca{X}_K^{\text{min}},V_{\rho,K}^{\text{sub}})$ is a Hecke equivariant sub quotient of $H^0(\Ca{X}_K^{\text{min}},V_{\rho',K}^{\text{sub}}\otimes\omega^{k\tilde{N}_n+M_n})$, and hence theorem \ref{T:abscong} is proved.

\singlespacing

\bibliographystyle{plain}

\end{document}